\documentclass[11pt,leqno]{amsart}
\usepackage{comment}
\usepackage{amsmath}
\usepackage{amsthm}
\usepackage{graphicx}
\usepackage{mathtools}
\usepackage{amsfonts}
\usepackage{amssymb}
\usepackage{hyperref}
\usepackage{bm}
\usepackage{dsfont}
\usepackage[T1]{fontenc}
\usepackage[margin=1 in]{geometry}
\usepackage{enumerate}
\usepackage[normalem]{ulem}
\usepackage{tikz}
\usepackage{tikz-cd}
\usepackage{xcolor}
\usetikzlibrary{matrix,arrows,positioning,automata}
  \usepackage{cancel}
  \usepackage{todonotes}
  
\numberwithin{equation}{section}
\newcommand{\N}{\mathbb N}

\newcommand{\assreg}{\hyperlink{ass:reg}{Assumption {\rm ({\bf Regularity})}}}
\newcommand{\assdis}{\hyperlink{ass:dis}{Assumption {\rm ({\bf Discriminating Property})}}}
\newcommand{\R}{\mathbb R}
\def\E{\mathbb E}

\def\XXint#1#2#3{{\setbox0=\hbox{$#1{#2#3}{\int}$}
\vcenter{\hbox{$#2#3$}}\kern-.5\wd0}}

\numberwithin{equation}{section}
\newtheorem{thm}{Theorem}[section]
\newtheorem{lem}[thm]{Lemma}
\newtheorem{meta-thm}[thm]{Meta-Theorem}
\newtheorem{cor}[thm]{Corollary}
\newtheorem{prop}[thm]{Proposition}

\theoremstyle{definition}
\newtheorem{defn}[thm]{Definition}
\newtheorem{rmk}[thm]{Remark}
\newtheorem{ex}[thm]{Example}

\def\smallnegint{\mathop{\int\mkern-13mu
        \raise.5ex\hbox{${\scriptscriptstyle\diagup}$}}\nolimits}

\def\div{\operatorname{div}}

\def\ssetminus{\,\raise.4ex\hbox{$\scriptstyle\setminus$}\,}

\newcommand{\norm}[1]{\left\Vert#1\right\Vert}

\renewcommand{\bar}{\overline}
\renewcommand{\tilde}{\widetilde}
\renewcommand{\hat}{\widehat}
\renewcommand{\ln}{\log}

\newcommand{\Lip}{\mathrm{Lip}}

\def \bd{\boldsymbol}

\title[Genericity of Polyak--Lojasiewicz Inequalities]{Genericity of Polyak--Lojasiewicz Inequalities for Entropic Mean-Field Neural ODEs.}

\author[S. Daudin]{Samuel Daudin}
\address{(S. Daudin) Universit\'e Paris Cit\'e, CNRS, Sorbonne Universit\'e, Laboratoire Jacques-Louis Lions (LJLL), F-75006, Paris, France
}\email{samuel.daudinATu-paris.fr
}

\author[F. Delarue]{ Fran\c{c}ois Delarue}
\address{(F. Delarue) Universit\'e C\^ote d'Azur, CNRS, Laboratoire J.A. Dieudonn\'e, 06108 Nice, France
}\email{francois.delarue@univ-cotedazur.fr}

\date{\today}

\begin{document}

\thanks{S. Daudin and F. Delarue  acknowledge the financial support of the European Research Council (ERC) under the European Union's Horizon Europe research and innovation program (ELISA project, Grant agreement No. 101054746). Views and opinions expressed are however
those of the author(s) only and do not necessarily reflect those of the European Union or the
European Research Council Executive Agency. Neither the European Union nor the granting
authority can be held responsible for them.}

\maketitle

\begin{abstract}
We address the behavior of idealized deep residual neural networks (ResNets), modeled via an optimal control problem set over continuity (or adjoint transport) equations. The continuity equations describe the statistical evolution of the features in the asymptotic regime where the layers of the network form a continuum. The velocity field is expressed through the network activation function, which is itself viewed as a function of the statistical distribution of the network parameters (weights and biases). From a mathematical standpoint, the control is interpreted in a relaxed sense, taking values in the space of probability measures over the set of parameters. We investigate the optimal behavior of the network when the cost functional arises from a regression problem and includes an additional entropic regularization term on the distribution of the parameters.
In this framework, we focus in particular on the existence of stable optimizers --that is, optimizers at which the Hessian of the cost is non-degenerate. We show that, for an open and dense set of initial data, understood here as probability distributions over features and associated labels, there exists a unique stable global minimizer of the control problem.
Moreover, we show that such minimizers satisfy a local Polyak--Lojasiewicz inequality, which can lead to
exponential convergence of the corresponding gradient descent  when the initialization lies sufficiently close to the optimal parameters. This result thus demonstrates the genericity (with respect to the distribution of features and labels) of the Polyak--Lojasiewicz condition in ResNets with a continuum of layers and under entropic penalization.

\end{abstract}

\setcounter{tocdepth}{1}

\tableofcontents

\section{Introduction}

This paper is dedicated to a control problem of mean-field type modeled on the training of certain deep neural networks. Mathematically it takes the following form:
\begin{equation}    \mbox{ minimize } J \bigl( (t_0,\gamma_0), \bd{\nu} \bigr) \mbox{ with respect to } \bd{\nu},
\label{eq:OriginalControlPb24/06}
\end{equation}
where the total cost $J$ is defined by
\begin{equation} 
J \bigl( (t_0, \gamma_0) , {\boldsymbol \nu} \bigr) := \int_{\R^{d_1} \times \R^{d_2}} L(x,y) d\gamma_T(x,y) +  \epsilon \int_{t_0}^T \mathcal{E} \bigl( \nu_t | \nu^{\infty} \bigr) dt,
\label{eq:deftotalcost}
\end{equation}
and the infimum is taken over flows of probability measures $\boldsymbol{\nu} = (\nu_t)_{t \in [t_0,T]}$  on a \textit{parameter} space $A$. In the expression above, $\gamma_T$ is an element of $\mathcal{P}(\R^{d_1} \times \R^{d_2})$, the set of Borel probability measures on the product space $\R^{d_1} \times \R^{d_2}$ (for two integers $d_1,d_2$), and is determined from $\boldsymbol{\nu}$ via the continuity equation
\begin{equation}
\partial_t \gamma_t + \div_x \Bigl( \int_A b( \cdot,a) d\nu_t(a) \gamma_t \Bigr) = 0 \quad \mbox{ in } (t_0,T) \times \R^{d_1} \times \R^{d_2},
\label{eq:Continuitygammanu}
\end{equation}
with initial condition $\gamma_0 \in \mathcal{P}(\R^{d_1} \times \R^{d_2})$ at $t_0 \in [0,T]$ and final horizon $T >0$. Importantly, the vector field $b : \R^{d_1} \times A \rightarrow \R^{d_1}$ driving the solution in \eqref{eq:Continuitygammanu} is assumed to satisfy  some type of universal approximation property. In the machine learning interpretation, $(\gamma_t)_{t \in [t_0,T]}$ describes the evolution of the joint distribution  of the \textit{features} and \textit{labels} along the layers of the neural network whose depth is indexed by the parameter $t \in [t_0,T]$.  

In the control problem \eqref{eq:deftotalcost}, the terminal cost $L: \R^{d_1} \times \R^{d_2} \rightarrow \R$ is determined by the underlying regression task while  $\mathcal{E}(\nu_t | \nu^{\infty})$, in the running cost, denotes the relative entropy of the parameter distribution $\nu_t$ at layer-time $t$ with respect to some prior measure $\nu^{\infty}$ and serves as a regularization. The intensity of this entropic penalization is given by the (small) parameter $\epsilon >0$.

\subsection{From Deep \textit{ResNets} to Mean-Field Optimal Control.}
The cost functionnal
\eqref{eq:deftotalcost}
together with the continuity 
equation 
\eqref{eq:Continuitygammanu}
should indeed be regarded as a
mathematical idealization of 
some regression tasks achieved by 
a certain 
type of deep neural
networks in machine learning. 
Those networks are usually referred to as 
`Residual Neural Networks'
(\textit{ResNets} for short)
and were primarily 
introduced by 
He et al. \cite{he2016deep}
in the analysis of learning procedures for image recognition. 
The very purpose of \textit{ResNets} is 
to overcome the degradation in accuracy that may occur when the number of layers increases. 
To explain such issues, the main hypothesis in 
 \cite{he2016deep}, which is supported by numerical evidences, 
is that mappings like the identity may not be well approximated by iterating non-linear ones. 
In turn, the principle of 
\textit{ResNets}
is to precondition the problem by decomposing each layer of the network as the sum of the identity mapping and of a residual  
standard parametrized activation function. 

Mathematically, this decomposition can be easily reformulated as the elementary step in the discretization of a continuous dynamical system. This is the starting point of a series of works in the literature, initiated 
by E in \cite{E2017}
and Haber and Ruthotto in 
\cite{HaberRuthotto}, in which regression methods based on a \textit{ResNets} architecture 
are studied with tools from optimal control theory in continuous time. 
In this approach, the 
time parameter in the controlled system represents the layer index in the network. 
See for instance 
Agrachev and Sarychev
\cite{AgrachevSarychev},
Li et al. 
\cite{LiLinShen2023} and Scagliotti \cite{AlessandroScagliotti2023}
for further examples. 

A first idealization in this framework is therefore to replace the 
\textit{a priori} discrete in time structure of the network by a continuous in time
architecture, as done here. Formally, this amounts to say that the number of layers is infinite. The resulting controlled system is sometimes called a \textit{Neural ODE}. 

A second idealization is to assume that there are infinitely many neurons per layer 
and that only their common statistical distribution   enters the instantaneous dynamics of the network. This is exactly what we do here through the formalism of relaxed controls, see in particular 
\eqref{eq:Continuitygammanu} where $\nu_t$ is  interpreted as a control. This idea is consistent with the approach introduced 
by
Mei et al.  
\cite{meiMontanariNguyen2018MFNN},
Bach and Chizat \cite{BachChizat}, 
and 
Rotskoff and Van den Eijnden 
\cite{RotskoffVanDenEijnden}
to study mathematically the gradient descent in one-layer neural networks.
Intuitively, the passage from 
a finite to an infinite number 
of neurons relies on an averaging 
principle that is typical of mean-field models (but which is not studied here). 
In the one-layer case, the very benefit of it is to convexify the loss function. 

Let us stress again that in this framework, the probability measure $\gamma_t$ in \eqref{eq:Continuitygammanu} can be viewed as the joint law of the features after layer-time $t$ and the labels (that remain fixed along the layers of the network). In other words, 
if one writes $X_0$
for the random variable modelling the feature in entry of the neural network and 
$Y_0$ for the random variable 
modeling the label that has to be regressed on $X_0$, 
$\gamma_t$ is the law of 
$(X_t,Y_0)$, i.e.,
\begin{equation}
\label{eq:Lagrangian:X0,XT,Y0}
\gamma_t = {\mathbb P} \circ (X_t,Y_0)^{-1},
\end{equation}
where $X_t$ is the value at time $t$ of the ODE 
\begin{equation} 
\dot{X}_t = \int_A b(X_t,a) d\nu_t(a)
\label{eq:ODE:intro}
\end{equation}
initialized  from $X_0$. 
In turn, if $L(x,y)$ is thought as a cost between $x$ and $y$, then the integral $\int_{{\mathbb R}^d \times {\mathbb R}^d} L(x,y) 
d \gamma_T(x,y)$
is nothing but the mean cost between 
the outputs of the network and the labels. 
As for the entropic term, it forces a form of strict convexity of the Hamiltonian 
associated with the 
control problem and also guarantees that optimal solutions inherit some properties (regularity, concentration and functional inequalities) from the measure $\nu^{\infty}$. The same idea can be found  in some of the aforementioned references on the one-layer case: therein, 
the underlying optimization problem becomes strictly 
convex in presence of an additional entropy. 
That said, it is worth emphasizing that, in comparison with the one-layer case,  our problem is not at all convex
because of the successive iterations of the nonlinear activation function.
This makes the analysis much more challenging. 
In particular, this is one of our objective here to show that, despite the lack of convexity, important local stability results 
remain, at least 
for \textit{generic} (that is `many' in a suitable sense) initial conditions.

\subsection{Main Assumptions and Example.}
Throughout this work, the following assumptions are in order. The parameter space $A$ is the euclidean space $\R^{d'}$ for some $d' \geq 1$.
\vskip 4pt

\hypertarget{ass:reg}{\noindent {\bf Assumption (Regularity).}}
The functions $b$, $\ell$ and $L$ introduced above 
satisfy the conditions below:
\begin{enumerate}[(i)]

\item The function 
$b : \R^{d_1} \times A \rightarrow \R^{d_1}$ is smooth in the sense that all the derivatives $\nabla_x^k \nabla^l_ab$ with $0 \leq k+l \leq 4$ exist and are (jointly) continuous. Moreover, there exists a constant $C\geq 0$ such that for all $(x,a) \in \R^d \times A$ and all integers $k,l$ with $0 \leq k+l \leq 4$, it holds
$$ |\nabla_x^k \nabla_a^lb(x,a)| \leq C(1+|a|^{k+1})(1+|x|^l). $$
We also assume that $b(x,0) = 0$ for all $x \in \R^{d_1}$.  
\label{ass:b}
\item The prior measure $\nu^{\infty} \in \mathcal{P}(A)$ takes the form 
\begin{equation} \nu^{\infty}(da) := \frac{1}{z^{\infty}} \exp{\bigl(-\ell(a)\bigr)} da, \quad  \quad z^{\infty} := \int_A e^{-\ell(a)}da, 
\label{eq:defnuinfty03/07}
\end{equation}
for some twice continuously differentiable convex potential $\ell: A \rightarrow \R$ satisfying, for a constant $c>0$ and for all $a \in A$, 
\begin{equation} 
\nabla^2_{aa} \ell (a) \geq c(1+|a|^2) I_{d'}. 
\label{eq:convexityassumptionL}
\end{equation}
In particular this means that $\ell$ grows at least 
as $\vert a \vert^4$. \label{ass:L}

\item The final cost $L: \R^{d_1} \times \R^{d_2} \rightarrow \R$ is bounded from below
and three times differentiable and, together with its derivatives,  satisfies the growth assumption
$$ \norm{L}_{\mathcal{C}^3_{2,1}} := \sup_{(x,y) \in \R^{d_1} \times \R^{d_2}}  \frac{|L(x,y)|}{1+|x|^2+|y|^2} + \frac{|\nabla L(x,y)|}{1+|x|+|y|} + \frac{|\nabla^2L(x,y)|}{1+|x|+|y|} + \frac{|\nabla^3L(x,y)|}{1+|x|+|y|}   < +\infty.$$
\label{ass:g}
\end{enumerate}
\vskip 2pt

We also assume $b$ to be discriminating in the following sense:
\vskip 4pt

\hypertarget{ass:dis}{\noindent \textbf{Assumption (Discriminating Property).}}
For any probability space $(\Omega, \mathcal{F}, \mathbb{P})$ and any two $\R^{d_1}$-valued random variables $X,Z$ on this probability space with $Z \in L^1(\mathbb{P})$,
the following implication holds true:
$$\biggl( 
\forall a \in A, \quad 
\E \bigl[ b(X,a) \cdot Z \bigr] = 0 
\biggr) 
\quad \Rightarrow 
\Bigl( 
\mathbb{P}\mbox{-almost-surely}, 
\quad 
\E \bigl[ Z | X \bigr] = 0 \Bigr),$$
where $\E$ denotes the expectation under $\mathbb{P}$.

\vskip 4pt

\begin{ex}
Here is the prototypical example for the vector field $b$. 
The condition (i) in 
\hyperlink{ass:reg}{Assumption ({\bf Regularity})} is satisfied if 
$A$ is taken as
${\mathbb R}^{d_1} \times {\mathbb R}^{d_1} \times 
{\mathbb R}$ and 
$b$
has the following structure: 
\begin{equation}
\label{eq:prototype:example}
b(x,a) = \sigma (a_1 \cdot x+a_2) a_0 , \quad 
\bigl(x,a=(a_0,a_1,a_2) \bigr) \in {\mathbb R}^{d_1} \times A,
\end{equation}
for a so-called \textit{activation} function $\sigma : {\mathbb R} \rightarrow {\mathbb R}$ that is bounded and smooth with bounded derivatives. 

Moreover, $b$ satisfies \assdis if $\sigma$ verifies a type of universal approximation property typical of activation functions in machine learning. In particular, all our assumptions are satisfied if $\sigma$ is the hyperbolic tangent or the logistic function. See Section \ref{subse:DPandUAT} for the details. 
\label{ex:protoex}
\end{ex}

Our assumptions on $L$ are satisfied if $d_1=d_2$ and $L$ is the quadratic loss ie, $ L(x,y) = \frac{1}{2}|x-y|^2$. The assumptions on $\ell$ are satisfied if $\ell(a) = c_1 |a|^4 + c_2|a|^2 $ for some $c_1,c_2 >0$.

\begin{rmk}
    Some classical sets of data $(\ell,L,b)$ do not fit our assumptions. Within our prototypical example for the vector field $b$, the assumptions are not satisfied if $\sigma$ is the rectified linear unit because it is neither $\mathcal{C}^1$ 
 nor bounded. The convexity condition on $\ell$ is not satisfied if $\ell(a) = \frac{1}{2}|a|^2$. In our setting, we need some stronger form of convexity and coercivity for $\ell$ to cope with the growth of $b$ and its derivatives when we derive some log-Sobolev inequalities for optimal solutions of the control problem, see Lemma \ref{lem:LSI} and Remark \ref{rmk:thirdmoments}. If 
 $b$ had the form $b(x,a) = \sigma(a_1x + a_2)$ with $a_1 \in M_{d_1}(\R), a_2 \in \R^{d_1}$, for $\sigma$ a smooth and bounded activation function acting component-wise, 
then our analysis could go through with $\ell = \frac{1}{2}|a|^2$. In such a case, \assdis  would still be satisfied for classical activation functions verifying a universal approximation property.
Lastly, 
the assumption that $b(x,0) = 0$ for all $x \in \R^{d_1}$ 
(see item (i) of \assreg) 
is used to prove Propositions \ref{prop:uniquenessNOC} 
and \ref{prop:uniquenessLinSyst}, see Step 3 in each case. This assumption is always satisfied in the framework of Example \ref{ex:protoex}.
\end{rmk}

\subsection{Main Results}
\label{subse:2.4}
We now expose the main results of the paper. To ease the understanding and in particular to lighten the notation, we have chosen to provide informal statements at this stage of the document.
For this reason, we call them 
`meta-statements'. 
Full results are given  
in the following sections,  
once the mathematical objects  
supporting our analysis
have been introduced.

In a nutshell, our objective is to understand the 
properties of the minimizers 
of 
\eqref{eq:OriginalControlPb24/06}
together with the behavior 
of the related gradient descent initialized near these minimizers, and this for 
\textit{generic} initial conditions. 
Here, 
the word \textit{generic} 
has a rather vague meaning. 
Below, we give a topological definition to it. 
That said, the philosophy is quite simple to explain. 
Due to the lack of convexity,  
the control problem 
\eqref{eq:OriginalControlPb24/06}
cannot be expected to 
have a unique minimizer
for any initial condition. To wit, it is well-know that non-convex  
finite dimensional control problems may develop singularities
in finite time. 
However, 
it is also known 
that, still in  finite dimension,
singularities
cannot be in fact too numerous. 
In particular, 
from a topological point of view,  they have an empty interior.
The first main result of this paper is to show that,
thanks to \assdis, this picture remains true in our setting. In other words, 
the Universal Approximation Theorem of neural networks translates, in the mean-field \textit{ResNets} setting, in
the form of a robustness property that is true for many initial conditions. 

\begin{meta-thm}
\label{meta-thm:1}
There exists an open dense subset 
${\mathcal O}$ of $[0,T] \times {\mathcal P}_3({\mathbb R}^{d_1} \times {\mathbb R}^{d_2})$ (with 
${\mathcal P}_3({\mathbb R}^{d_1} \times {\mathbb R}^{d_2})$
denoting the space of probability measures on 
${\mathbb R}^{d_1} \times {\mathbb R}^{d_2}$ with a finite third moment)
such that, for any initial condition 
$(t_0,\gamma_0) \in {\mathcal O}$, 
the control problem \eqref{eq:OriginalControlPb24/06}
has a unique minimizer and this minimizer is \emph{stable}.
\end{meta-thm}

We do not provide a precise  definition of stability  at this stage and only mention that it is related to the non-degeneracy of the Hessian of the cost \eqref{eq:deftotalcost}. The exact definition is given in Section \ref{se:3} and the complete version of 
Meta-Theorem \ref{meta-thm:1} is  stated in 
Theorem \ref{thm:main:stability}. Instead, we explain here
two main consequences of these stability properties, which we believe are relevant to machine learning.

To give the reader a clearer picture, it is worth recalling that, in practice, the challenge is to numerically find the optimal parameters of the neural network. 
Usually, this is done by means of a gradient descent algorithm (or a variant thereof).

In the multi-layer setting, 
the algorithm takes  the form of a collection of intertwined gradient descents
indexed by the successive layers of the network. Each of these
descents returns a (possibly inaccurate) approximation of the 
optimal state of the 
neurons 
at the corresponding layer. 
Mathematically, this principle 
can be formalized at follows: 
at any time
$t \in [t_0,T]$, we can construct 
a flow $(\nu_t^s)_{s \geq 0}$, 
depending on a new time parameter
$s$
and following a gradient flow in the space 
${\mathcal P}(A)$ (in the sense of Ambrosio et al. \cite{AGS}). 
In the mean-field approach to \textit{ResNets}, this formalism  
was introduced 
in Jabir et al. \cite{jabir2021meanfieldneuralodesrelaxed}. 
Of course, one relevant objective is to obtain guarantees under which 
$(\nu_t^s)_{s \geq 0}$ converges
to $\nu_t^*$, the optimal state of the network
at time $t$. 
This is precisely where the stability properties of the minimizers become especially useful.

The analysis of the descent goes through 
an explicit formula for 
the $s$-derivative of the cost along the descent:
\begin{equation}
\label{eq:intro:gradient:descent}
\frac{d}{ds} J \bigl( (t_0,\gamma_0,{\boldsymbol \nu}^s \bigr)  =- \mathcal{I} \bigl( (t_0,\gamma_0) ,{\boldsymbol \nu}^s \bigr), 
\end{equation}
for some (non-negative) functional $\mathcal{I}$ that is made explicit in Section \ref{se:5}. Above the symbol
${\boldsymbol \nu}^s$ denotes the entire curve $(\nu^s_t)_{t_0 \le t \le T}$ at descent-time $s$. 
Thanks to the aforementioned stability properties on the minimizers of 
$J((t_0,\gamma_0),\cdot)$  we manage to prove the following functional inequality, which is sometimes referred to as 
a (here local) \textit{Polyak--Lojasiewicz condition}:

\begin{meta-thm}
\label{meta-thm:2}
For every compact subset ${\mathcal K}$  of $\mathcal{O}$, 
there exists a constant $c>0$ such that, for 
any 
initial condition $(t_0,\gamma_0) \in {\mathcal K}$
and any curve ${\boldsymbol \nu} = (\nu_t)_{t \in [t_0,T]}$ that is close enough to the unique 
minimizer $ \boldsymbol{\nu}^*$ of $J((t_0,\gamma_0),\cdot)$, 
\begin{equation*}
{\mathcal I} \bigl( (t_0,\gamma_0),{\boldsymbol \nu} \bigr) \geq 
c \Bigl( J\bigl( (t_0,\gamma_0),{\boldsymbol \nu} \bigr) 
- 
 J \bigl( (t_0,\gamma_0), \boldsymbol{\nu}^* \bigr)
 \Bigr). 
\end{equation*}
\end{meta-thm}

The rigorous version of 
Meta-Theorem 
\ref{meta-thm:2}
is stated in 
Section
\ref{se:5}, see in particular 
Theorem
\ref{prop:main:local:polyak:lojasiewicz}. Therein, 
we clarify in which sense the curve 
${\boldsymbol \nu}$
has to be close to 
the 
minimizer 
${\boldsymbol \nu}^*$.

As expected, 
Meta-Theorem 
\ref{meta-thm:2}
becomes especially relevant in the analysis of the aforementioned gradient descent, provided the latter one be 
initialized in the neighborhood of the optimal control $\bd \nu^*$. 
To make this precise, a reasonable conjecture is that, for any initial condition \( (t_0, \gamma_0) \in \mathcal{O} \) of the optimal control problem~\eqref{eq:OriginalControlPb24/06}, and for any initial condition \( \boldsymbol{\nu}^0 = (\nu_t^0)_{t \in [t_0,T]} \) of the gradient descent flow \( \boldsymbol{\nu}^s = ((\nu_t^s)_{t \in [t_0,T]})_{s \geq 0} \), provided \( \boldsymbol{\nu}^0 \) is close enough to the unique minimizer \( \boldsymbol{\nu}^* \) of \( J((t_0, \gamma_0), \cdot) \), the cost \( J((t_0, \gamma_0), \boldsymbol{\nu}^s) \) converges exponentially fast to the optimal value \( J((t_0, \gamma_0), \boldsymbol{\nu}^*) \) as \( s \to \infty \).
We emphasize that the result would follow from a standard argument once Meta-Theorem \ref{meta-thm:2} has been established, and is omitted here for reasons of length. Here is the underlying principle: Proposition
\ref{prop:isolated} establishes that the unique minimizer $\bd{\nu}^*$ is an isolated critical point of the cost $J ( (t_0,\gamma_0), \cdot )$; combined with the identity 
\eqref{eq:intro:gradient:descent} (the proof of which would however deserve to be expanded), it follows from standard results in dynamical systems (usually referred to as LaSalle's principle) that the gradient descent flow remains in a neighborhood of the minimizer when initialized sufficiently close to it. The conclusion is then obtained by invoking Meta-Theorem
\ref{meta-thm:2}, which is the central result in this context, to deduce the exponential convergence of the cost.

\subsection{Comparison with Existing Literature}

\label{subse:comparisonlit}

The control problem \eqref{eq:OriginalControlPb24/06}
is a mean-field control problem -- see Lions \cite{Lionscollege2}, Chapter $6$ in Carmona and Delarue \cite{CarmonaDelarue_book_I}, and Section 3.7 in Cardaliaguet et al. \cite{CardaliaguetDelarueLasryLions} for introductory material on the subject together with Lacker \cite{Lacker_2016_convergence} for the use of relaxed controls in this context. Indeed,   
the controlled dynamics given by 
the continuity 
equation 
\eqref{eq:Continuitygammanu} 
take values in the space 
of probability measures over ${\mathbb R}^{d_1} \times 
{\mathbb R}^{d_2}$. However, it is non-standard in the sense that the control $(\nu_t)_{0 \le t \le T}$ (even taken in a 
relaxed form) is not a vector field taking values 
in ${\mathcal P}(A)$ --in which case it would take the form $[0,T] \times \R^{d_1} \ni (t,x) \mapsto \nu_t(x,da) \in \mathcal{P}(A)$--, but just an element of 
${\mathcal P}(A)$.

Stability properties of optimal solutions for a class of stochastic, closed-loop, mean-field optimal control problems have recently drawn significant attention, see Briani and Cardaliaguet \cite{BrianiCardaliaguet},
Cardaliaguet and Souganidis \cite{cardaliaguet-souganidis:2}
and Cardaliaguet et al. \cite{cjms2023}. Part of the analysis relies on a classical phenomenon in  optimal control theory and calculus of variations: the so-called \textit{Jacobi necessary optimality condition} (see Cannarsa and Sinestrari \cite{cannarsa} Chapter 6, and Remark 6.3.7 therein for the terminology). Generally speaking, this condition asserts that, for some optimal control problems with linear dynamics and strictly convex Hamiltonians, there is no \textit{conjugate}  point along optimal trajectories. In our framework, this is exactly the statement that, given an optimal trajectory $(t,\gamma^*_t)_{t \in [t_0,T]}$ starting from $(t_0,\gamma_0)$, then $(t_1, \gamma_{t_1}^*)$ belongs to the set $\mathcal{O}$ of Meta-Theorem \ref{meta-thm:1} for any later time $t_1 >t_0$. 
In the setting of Problem \eqref{eq:OriginalControlPb24/06} and for general vector fields $b$ satisfying \assreg,  we do not expect to recover this property. However, 
  one of the main contributions of our work is to prove that
 this is indeed the case when $b$ satisfies \assdis (see the next sub-section for some details).

To gain further intuition about the difference with the closed-loop regime addressed in \cite{BrianiCardaliaguet, cardaliaguet-souganidis:2, cjms2023}, it is helpful to revisit the corresponding Lagrangian 
interpretation. In the latter, the controlled trajectory $(\gamma_t)_{0 \le t \le T}$
describes
the statistical (or macroscopic) evolution of a continuum of agents. At the microscopic (or individual) level, 
agents evolve according to ODEs (or SDEs when the control problem is stochastic) that are driven by a common velocity field but are initialized from possibly different initial conditions. 
This picture remains true in our framework: the velocity field is 
$b(x,\nu_t)$ and the initial conditions are statistically distributed according to $\gamma_0$. 
What changes here is that two agents, in the same continuum but at different individual locations, play  the control $\nu_t$ at time $t$ in exactly  the same manner. This is consistent from a machine learning perspective and contrasts with the standard rule in \textit{closed-loop} mean-field control, where the control applied by each agent explicitly depends on its own state. In the terminology of optimal control our framework corresponds to the \textit{open-loop} case, see e.g. \cite{BONNET2022282}. The reader may have noticed another subtlety in our 
control problem: the `variable $y$' in 
\eqref{eq:deftotalcost}
is not impacted by the dynamics 
\eqref{eq:Continuitygammanu}. From a microscopic 
point of view, it says that an agent has two main features, in ${\mathbb R}^{d_1}$ and $\R^{d_2}$: whilst the first one evolves according to the ODE driven by the field $(t,x) \mapsto b(x,\nu_t)$, the second remains constant with time.

Control problems of the same kind as  Problem \eqref{eq:OriginalControlPb24/06} (with possibly different types of regularizations) have already been studied in the machine learning literature as an idealized model for the training phase of deep neural networks.  To the best of our knowledge, 
this model goes back to the pioneering works
of
E et al. 
\cite{EHanLi2019} (including the derivation of the Pontryagin maximum principle). The use of entropic penalization can be found in Hu et al, \cite{hu2019meanfieldlangevinsystemoptimal}  
Jabir et al. 
\cite{jabir2021meanfieldneuralodesrelaxed}. We also refer to Lu et al. \cite{lu2020mean}, Isobe \cite{isobe2023convergence}
Bonnet et al. \cite{BonnetCiprianiFornasierHuang}, Barboni et al. \cite{barboni2024understandingtraininginfinitelydeep}
and Ding et al. 
\cite{ding2022overparameterization}
for a more recent works in the same vein. In essence, what sets our approach apart from other works involving entropic regularization is that none of our results requires the penalization intensity $\epsilon$ to be large. 
 
Establishing a Polyak–Lojasiewicz inequality (PL inequality) is regarded as a key step in the analysis of gradient descent for neural networks. In the context of mean-field Deep ResNets, the following results have been established. In the work \cite{isobe2023convergence}, a PL inequality is shown
under a strong enough moment (instead of entropic) regularization.  The authors  of \cite{barboni2024understandingtraininginfinitelydeep} establish a PL inequality without regularization, but degenerating as the number of features (assumed finite) used during training increases. This latter condition can be interpreted as a structural smallness assumption. 
The convergence of the gradient descent (which, in our case, could be derived using Meta-Theorem \ref{meta-thm:2}, as previously explained, the descent taking the same form as in \cite{jabir2021meanfieldneuralodesrelaxed}) and 
of the generalization error (i.e., the mean field limit) have also been studied in~\cite{jabir2021meanfieldneuralodesrelaxed,BonnetCiprianiFornasierHuang}, in the presence of a penalization—by entropy in the former, and by moments in the latter—both assumed to be sufficiently strong.
We emphasize again that, by contrast, our approach imposes no minimal threshold on the intensity $\epsilon$ (of the entropic penalty). 
The interested reader may find further recent developments on PL inequalities in (among others) the following works: in the context of ODE control, Gassiat and Suciu \cite{gassiat2024gradientflowcontrolspace} establish a randomized PL inequality for gradient descent under highly oscillatory initialization; furthermore, Monmarché and Reygner \cite{monmarché2024localconvergencerateswasserstein} investigate the long-time behavior of (uncontrolled) mean-field diffusions using local PL inequalities.

In these models, the passage from discrete to continuous time raises particularly subtle questions, and different limiting regimes may arise depending on the scaling of the network parameters with depth. A substantial part of the literature focuses on diffusive scalings, for which the infinite-depth limit is described by a stochastic differential equation; see, for instance, \cite{hayou2023on, ICLR2024_f2ce452a, JMLR:v26:22-0664, cont2023asymptoticanalysisdeepresidual}. By contrast, the scaling considered in the present work corresponds to an ODE regime, which is closer in spirit to \cite{ICLR2024_b29a95e7, ding2022overparameterization, chizat2026hiddenwidthdeepresnets}. We do not attempt here to provide a systematic analysis of the discrete-to-continuous transition, although we discuss in Subsection \ref{subse:limitationsandpossiblextensions - 15/09} how our framework relates to such large-depth limits and, in particular, to the regime studied in \cite{chizat2026hiddenwidthdeepresnets}. We thank the anonymous referee for pointing out several of these references.

Another meaningful aspect of the problem, which we leave for future investigation, concerns the \textit{expressivity} of the network: given an initial distribution $\gamma_0$, what is the minimal loss achievable in the associated control problem? This question is related to the controllability of the system. In the context of deep neural networks—albeit in more or less idealized settings—, this issue has been addressed, among others, by Agrachev and Sarychev~\cite{AgrachevSarychev}, Cuchiero et al.\cite{Cruchieroetal}, Li et al.\cite{LiLinShen2023}, and Ruiz-Balet and Zuazua~\cite{Ruiz-Balet-Zuazua}. Although we do not address the controllability properties of our system, it is worth stressing that
 none of our results depend on the \textit{efficiency} of the network. In other words, it might very well be that the labels (represented by the variable 
 $Y_0$ in \eqref{eq:Lagrangian:X0,XT,Y0}) are far from any function of the features (represented by the variable $X_0$ in \eqref{eq:Lagrangian:X0,XT,Y0}). Equivalently, the initial distribution $\gamma_0$ might be \textit{far} from any measure of 
 \textit{diagonal} form $\mu_0
 \circ (i_d,F)^{-1}$ for some $\mu_0 \in \mathcal{P}_3(\R^d)$ describing the law of $X_0$ and some measurable function $F : \R^{d_1} \rightarrow \R^{d_2}$. In this framework,  the optimal cost (i.e., the optimal loss)
is not small but our result remains relevant. This is in contrast with some known convergence results, see \cite{barboni2024understandingtraininginfinitelydeep,jabir2021meanfieldneuralodesrelaxed}.

\subsection{Method of Proof}


\subsubsection{First order condition and Jacobi principle.}

Much of our analysis revolves around first and second order optimality conditions for the control problem \eqref{eq:OriginalControlPb24/06}. The first order conditions state that an optimal control $\bd{\nu}$ necessarily takes the form   
\begin{equation} 
\nu_t(a) \propto  \exp \Bigl( - \ell(a) - \frac{1}{\epsilon} \int_{\R^{d_1} \times \R^{d_2}}  b(x,a) \cdot \nabla_x u_t(x,y)  d\gamma_t(x,y) \Bigr), \quad t \in [t_0,T],
\label{eq:18/06/2025-14:44}
\end{equation} 
where "$\propto$" means "proportional to" and  $(\bd{\gamma} , \bd{u}) = (\gamma_t,u_t)_{t \in [t_0,T]}$ solves the forward-backward system
\begin{equation}
\label{eq:forwardbarckward03/05}
\left \{
\begin{array}{ll}
\displaystyle -\partial_t u_t(x,y) -  b\bigl(x,\nu_t\bigr)  \cdot \nabla_x u_t(x,y) = 0,
&\quad (t,x,y)
\in [t_0,T] \times \R^{d_1} \times \R^{d_2},
\\
\displaystyle
\qquad 
 u_T(x,y) = L(x,y), &\quad (x,y) \in \R^{d_1} \times \R^{d_2};
\\
\displaystyle \partial_t \gamma_t
+ \div_x \bigl(  b(x,\nu_t) \gamma_t  \bigr) = 0, &\quad
 \textrm{\rm in}
 \ [t_0,T] \times  \R^{d_1} \times \R^{d_2}, 
\\
\displaystyle \qquad \gamma_{t_0} = \gamma_0
&\quad
 \textrm{\rm in} 
 \ 
\R^{d_1} \times \R^{d_2}.
\end{array}
\right.
\end{equation}

If this system had a unique solution, it would characterize the (necessarily unique) optimal solution to the control problem. However, as previously mentioned, such uniqueness is not to be expected for arbitrary initial conditions $(t_0, \gamma_0)$, due to the lack of any global convexity condition ensuring the uniqueness of critical points in the control problem. From a technical perspective, establishing uniqueness for the system satisfied by $(\boldsymbol{\nu}, \boldsymbol{\gamma}, \boldsymbol{u})$ is delicate, owing to the forward--backward structure of the continuity and transport equations: the equation for $\boldsymbol{\gamma}$ is equipped with an initial condition, while the equation for $\boldsymbol{u}$ is subject to a terminal condition.

Meta-Theorem \ref{meta-thm:1} then relies on the aforementioned Jacobi  principle: for any initial condition $(t_0,\gamma_0)$, any (between the possibly many) optimal solution $\bd{\nu}^* = (\nu^*)_{t \in [t_0,T]}$ for $J ( (t_0,\gamma_0), \cdot) $ with associated optimal trajectory $\bd{\gamma}^* = (\gamma^*_t)_{t \in [t_0,T]}$, and any later time $t_1 >t_0$, $(\nu_t^*)_{ t \in [t_1,T]}$ is the unique minimizer of  $J ( (t_1, \gamma_{t_1}^*), \cdot )$ and, furthermore, this solution is \textit{stable}. In particular, for $t_1$ close to $t_0$, we have found an initial condition $(t_1, \gamma^*_{t_1})$, nearby $(t_0,\gamma_0)$, for which the control problem has a unique (and stable) optimal solution. To show that $(\nu_t^*)_{t \in [t_1,T]}$ is the only optimal solution, we prove that optimal solutions for $J ( (t_0,\gamma_0), \cdot )$ cannot \textit{bifurcate}, in the sense that any two optimal solutions $(\nu^{*,i}_t)_{t\in [t_0,T]}$, $i=1,2$, that coincide at the initial time  must 
remain identical over the entire time horizon.
 This implies the uniqueness of optimal controls for $J ( (t_1, \gamma^*_{t_1}), \cdot )$ since, otherwise, by dynamic programming, we could construct two solutions for $(t_0,\gamma_0)$  that coincide up to time $t_1$ but bifurcate afterward.

Importantly, the fact that optimal controls cannot bifurcate relies in a subtle way on the discriminating property \assdis. To explain this, let us assume for simplicity, that the state dynamic and the terminal cost do not depend on the $y$-variable. In this case, the controlled dynamic is simply a trajectory in $\mathcal{P}(\R^{d_1})$.  Consider two optimal solutions $(\nu_t^{*,i})_{t \in [t_0,T]}$, $i=1,2$, for
the initial condition $(t_0,\gamma_0)$, such that $\nu_{t_0}^{*,1} = \nu_{t_0}^{*,2}$.
If we forget about the normalizing constant in \eqref{eq:18/06/2025-14:44}, the first order conditions imply that 
\begin{equation} 
\int_{\R^{d_1}} b(x,a) \cdot \nabla_x u_{t_0}^1(x) d\gamma_{0}(x) = \int_{\R^{d_1}} b(x,a) \cdot \nabla_x u_{t_0}^2 (x) d\gamma_0(x), \quad \forall a \in A,
\label{eq:tobeproperlyreformulated}
\end{equation}
where $\bd{u}^1$ and $\bd{u}^2$ are the solutions to the backward transport equation in \eqref{eq:forwardbarckward03/05}  associated to $\bd{\nu}^{*,1}$ and $\bd{\nu}^{*,2}$ respectively. This is where  \assdis $ $ comes into play. After properly reformulating \eqref{eq:tobeproperlyreformulated} as an equality between expectations of random variables, it tells us that $\nabla_xu^1_{t_0}$ and $ \nabla_xu^2_{t_0}$ must coincide on the support of $\gamma_0$. Therefore we end up with two solutions $(\bd{\nu}^{*,i}, \bd{\gamma}^{*,i}, \bd{u}^{*,i})$, $i=1,2$ to the system of optimality  conditions \eqref{eq:18/06/2025-14:44}-\eqref{eq:forwardbarckward03/05} such that $\gamma_{t_0}^1 = \gamma_{t_0}^2 = \gamma_0$ and $\nabla_x u_{t_0}^1 = \nabla_x u^2_{t_0}$ in the support of $\gamma_0$. 
This transforms 
\eqref{eq:forwardbarckward03/05} into a forward-forward system, from which we can infer 
that 
$(\bd{\nu}^1, \bd{\gamma}^1,\bd{u}^1) =(\bd{\nu}^2, \bd{\gamma}^2,\bd{u}^2) $.

The optimality conditions \eqref{eq:18/06/2025-14:44}-\eqref{eq:forwardbarckward03/05} are stated in Theorem \eqref{thm:OCThm27/06} of Section \ref{se:3} and proved in Section \ref{sec:ExistenceandFOC}. We emphasize that similar forms of the Pontryagin maximum principle were already known see e.g. \cite{BonnetCiprianiFornasierHuang,hu2019meanfieldlangevinsystemoptimal, jabir2021meanfieldneuralodesrelaxed}. The proof of the Jacobi condition and the rigorous version of Meta-Theorem \ref{meta-thm:1} are given in Section \ref{sec:TheUniversalAppox}.

\subsubsection{Stable solutions} 
\label{subsubsec:stable}
Before we explain the proof of the local Polyak--Lojasiewicz condition of  Meta-Theorem \ref{meta-thm:2} we need to clarify what we mean by \textit{stable} solution. Let  $\bd{\nu}^*$ be an optimal solution for the initial condition $(t_0,\gamma_0)$ with associated curve and multiplier $(\bd{\gamma}^*,\bd{u}^*)$. For a sequence of admissible controls $(\bd\nu^n)_{n \in \mathbb{N}}$ with associated curves and multipliers $(\bd{\gamma}^{n},\bd{u}^n)_{n \in \mathbb{N}}$ satisfying the optimality conditions \eqref{eq:18/06/2025-14:44}-\eqref{eq:forwardbarckward03/05}, we can define the integrated relative entropy between the measures $\bd{\nu}^n$ and $\bd{\nu}^*$:
\begin{equation} 
\lambda^2_n :=  \int_{t_0}^T \mathcal{E} \bigl( \nu_t^n | \nu_t^* \bigr) dt,
\label{eq:defnlambdanintro13/01}
\end{equation}
where, for all $t \in [t_0,T]$, $\mathcal{E}(\nu^n_t|\nu_t^*)$ is the usual relative entropy between two probability measures, see \eqref{eq:def:relative:entropy}.
Let us assume that $\lambda_n >0$ for all $n \in \mathbb{N}$ and $\lim_{n \rightarrow +\infty} \lambda_n=0$. Then, by means of a compactness argument, we can show that weak limit points $(\bd{\eta}, \bd{\rho}, \bd{v})$ of $ \lambda_n^{-1} ( \bd{\nu}^n - \bd{\nu}, \bd{\gamma}^n - \bd{\gamma}, \bd{u}^n - \bd{u})$ solve the linearized equations
\begin{equation}
    {\eta}_t(a) = -   \frac{\nu^*_t(a)}{\epsilon} \Bigl[  \bigl \langle b(\cdot,a) \cdot \nabla_x u_t^* ;\rho_t \bigr \rangle  + \int_{\R^{d_1} \times \R^{d_2}}b(x,a) \cdot \nabla_x v_t(x,y)  d\gamma^*_t \bigr (x,y) -c_t\Bigr] \quad \mbox{in } [t_0,T] \times A, 
\label{eq:linSysEtaIntro}
\end{equation}
and
\begin{equation}
    \left \{
    \begin{array}{ll}
\displaystyle -\partial_t  v_t  -  b(x,\nu^*_t) \cdot \nabla_x v_t =  b(x,\eta_t)  \cdot \nabla_x u^*_t &\textrm{\rm in } [t_0,T] \times \R^{d_1} \times 
\R^{d_2}, 
\\
\qquad v_T = 0 \quad 
\textrm{\rm in } \R^{d_1} \times 
\R^{d_2},
\\
\displaystyle \partial_t  \rho_t  + \div_x \bigl(  b(x,\nu^*_t)  \rho_t \bigr) = - \div_x \bigl(  b(x,\eta_t )  \gamma^*_t \bigr) &\textrm{\rm in } (t_0,T) \times \R^{d_1} \times \R^{d_2},
\\
\qquad \rho_{t_0} =0,
    \end{array}
    \right.
\label{eq:LinSysRhoVIntro}
\end{equation}
where $c_t$, in \eqref{eq:linSysEtaIntro}, is a normalizing constant to ensure that $\eta_t$ integrates to $0$. In the equations above, $\rho_t$ lies in the dual of a space of differentiable functions with some growth at infinity and $\langle \cdot ; \cdot\rangle$ denotes the corresponding duality bracket.  A  stable solution $\bd{\nu}^*$ is precisely a solution to the control problem \eqref{eq:OriginalControlPb24/06} for which $(0,0,0)$ is the only solution to the system \eqref{eq:linSysEtaIntro}--\eqref{eq:LinSysRhoVIntro}. 
Of course, stability is related to second-order optimality conditions. Indeed, given a minimizer $\bd{\nu}^*$ for $J \bigl( (t_0,\gamma_0), \cdot)$, we always have 
\begin{equation} 
\frac{d^2}{dh^2} \Big|_{h= 0} J \bigl( (t_0,\gamma_0), \bd{\nu}^* + h \bd{\eta} \bigr) \geq 0 
\label{eq:inequalityforequality18/06}
\end{equation}
for all admissible perturbations $\bd{\eta}$. For a fixed $\bd{\eta}$, we can show that  equality holds in \eqref{eq:inequalityforequality18/06}  if and only if there is $(\bd \rho, \bd v)$ such that $(\bd{\eta}, \bd{\rho}, \bd{v})$ solves the linearized system \eqref{eq:linSysEtaIntro}-\eqref{eq:LinSysRhoVIntro}. Stable solutions are then precisely those for which 
$$ \frac{d^2}{dh^2} \Big|_{h= 0} J \bigl( (t_0,\gamma_0), \bd{\nu}^* + h \bd{\eta} \bigr) > 0 $$
for any non-trivial (ie non-zero) admissible perturbation $\bd{\eta}$.

We introduce stable solutions in Section \ref{se:3}, where we properly define $\mathcal{O}$ as the set of initial conditions $(t_0,\gamma_0)$ for which there is a unique and stable global minimizer. In the same section, we state the second-order optimality conditions associated to Problem \eqref{eq:OriginalControlPb24/06} in Theorem \ref{prop:SecondOrderConditions3Avril}. The proof of these conditions is postponed to  Section \ref{sec:SOC} after we give an exhaustive analysis of the linearized equations appearing in \eqref{eq:LinSysRhoVIntro} in Section \ref{sec:LinearizedEquations}.

\subsubsection{PL inequality.}

Let us now provide a sketch of proof for Meta-Theorem \ref{meta-thm:2} when the compact set $\mathcal{K}$ is just taken as a single element $\left \{ (t_0,\gamma_0) \right \} \subset \mathcal{O}$. The detailed proof is given in Section \ref{se:5}. We argue by contradiction and assume that we can find a positive sequence $(c_n)_{n \in \mathbb{N}}$ converging to $0$ together with a sequence of admissible controls $(\bd{\nu}^n)_{ n \in \mathbb{N}}$ such that 
\begin{equation} 
\lim_{ n \rightarrow +\infty} \lambda_n^2 =0, \quad \mbox{ and} \quad \mathcal{I} ((t_0,\gamma_0), \bd{\nu}^n) < c_n \Bigl( J \bigl((t_0,\gamma_0), \bd{\nu}^n \bigr) - J \bigl( (t_0,\gamma_0), \bd{\nu}^* \bigr) \Bigr),  
\label{eq:towardcontradictionintro13/01}
\end{equation}
where we used the notations introduced in \eqref{eq:defnlambdanintro13/01} and \eqref{eq:intro:gradient:descent} for $\lambda_n$ and $\mathcal{I}$ respectively.

Notice 
 by non-negativity of ${\mathcal I}$
that the second equation in \eqref{eq:towardcontradictionintro13/01} implies that 
${\boldsymbol \nu}^n \neq \boldsymbol \nu^*$ and therefore $\lambda_n >0$, which
makes it possible to let
\begin{equation} 
\eta_t^n := \frac{\nu_t^n - \nu^*_t}{\lambda_n}, \quad \rho_t^n := \frac{\gamma^n_t - \gamma^*_t}{\lambda_n}, \quad v_t^n := \frac{u_t^n - u^*_t}{\lambda_n}, \quad t \in [t_0,T],  
\label{eq:normalizednewvariablesintro}
\end{equation}
where $(\bd{\gamma}^n,\bd{u}^n)$ and $(\bd{u}^*, \bd{\gamma}^*)$ are the solutions to \eqref{eq:forwardbarckward03/05} associated to $\bd{\nu}^n$ and $\bd{\nu}^*$ (the optimal control for 
 the initial condition
$(t_0,\gamma_0)$) respectively.
By making the difference between  the equations satisfied by $(\bd{\gamma}^n,\bd{u}^n)$ and $(\bd{\gamma}^*, \bd{u}^*)$, we get 
\begin{equation}
\left \{
\begin{array}{ll}
\displaystyle -\partial_t {v}_t^n -  b\bigl(x,{\nu}_t^n\bigr)  \cdot \nabla_x {v}_t^n = b (x, \eta_t^n) \cdot \nabla_x u^*_t
\ &\textrm{\rm in}
\  [t_0,T] \times \R^{d_1} \times \R^{d_2},
\\
\displaystyle
\qquad 
 {v}_T^n(x,y) = 0 \ &\textrm{\rm in} \ \R^{d_1} \times \R^{d_2};
\\
\displaystyle \partial_t {\rho}_t^n
+ \div_x \bigl(  b(x,{\nu}_t^n)  {\rho}_t^n  \bigr) = 
 - \div_x \bigl( b(x,\eta_t^n) \gamma^*_t \bigr) \
 & \textrm{\rm in}
 \ [t_0,T] \times  \R^{d_1} \times \R^{d_2}, 
\\
\displaystyle \qquad  {\rho}_{t_0}^n = 0
\ &\textrm{\rm in} 
 \ 
\R^{d_1} \times \R^{d_2}.
\end{array}
\right.
\label{eq:systemrhonvn29MaiIntro}
\end{equation}

The following two observations are in order.
 On the one hand, it is not too difficult to see that the right-hand side of the second equation in \eqref{eq:towardcontradictionintro13/01} is (at most) of order $\lambda_n^2$ (see Lemma \ref{lem:expensioncost02/03}). That is, there exists $C>0$ independent of $n$ such that
\begin{equation}
J \bigl( (t_0,\gamma_0),{\boldsymbol \nu}^n \bigr) - J \bigl( ( t_0, \gamma_0), \boldsymbol \nu^* \bigr) \leq C \lambda_n^2.
\label{eq:02/07-20:23:intro}
\end{equation}
On the other hand, thanks to an explicit formula for the functionnal $\mathcal{I}$ given in Section \ref{se:5} and to the equation satisfied by $\bd\nu^*$ we can rewrite $\mathcal{I}\bigl( (t_0,\gamma_0),{\boldsymbol \nu}^n \bigr)$ as
\begin{equation} 
\begin{split}
&\mathcal{I} \bigl( (t_0,\gamma_0),{\boldsymbol \nu}^n \bigr)
\\
&= \int_{t_0}^T \int_A \Bigl| \epsilon \nabla_a \log \frac{\nu_t^n}{\nu^*_t}(a) + \lambda_n \nabla_a \int_{\R^{d_1}\times \R^{d_2}}b(x,a) \cdot d (\nabla_x v_t^n \gamma_t^n + \nabla_x u_t^* \rho_t^n)(x,y) \Bigr|^2 d\nu_t^n(a)dt,
\label{eq:rewritingFischerintro}
\end{split}
\end{equation}
and get, by triangular inequality
\begin{align*}
 \mathcal{I} \bigl( (t_0,\gamma_0) ,{\boldsymbol \nu}^n \bigr) & \geq \frac{1}{2} \int_{t_0}^T  \int_A \Bigl| \epsilon \nabla_a \log \frac{\nu_t^n}{\nu^*_t} (a) \Bigr|^2 d\nu_t^n(a) dt  \\
 &-\lambda_n^2 \int_{t_0}^T \int_A \Bigl| \nabla_a \int_{\R^{d_1} \times \R^{d_2}} b(x,a) \cdot d(\nabla_x v_t^n \gamma_t^n + \nabla_x u^*_t \rho_t^n)(x,y) \Bigr|^2 d \nu^n_t(a)dt.
\end{align*}
Using log-Sobolev inequality, which we show in Lemma \ref{lem:LSI} is satisfied by $\nu_t^*$,
we can handle the first term on the right-hand side to obtain 
\begin{equation}
\begin{split}
&\mathcal{I} \bigl( (t_0,\gamma_0),\bd\nu^n \bigr) 
\\
&\hspace{15pt} \geq C\lambda_n^2 \Bigl( 1 -\int_{t_0}^T \int_A \bigl| \nabla_a \int_{\R^{d_1} \times \R^{d_2}} b(x,a) \cdot d(\nabla_x v_t^n \gamma_t^n + \nabla_x u^*_t \rho_t^n)(x,y) \bigr|^2 d \nu^n_t(a)dt \Bigr). 
\label{eq:02/07/20:03:intro}
\end{split}
\end{equation}
The result will follow if we  justify that $(\bd v^n,\bd\rho^n)_{n \in \mathbb{N}}$ vanishes to $(0,0)$ (in an appropriate sense) as $n \rightarrow +\infty$. Indeed, in this case, we can hope that the integral term in the right-hand side of \eqref{eq:02/07/20:03:intro} vanishes with $n$ large, thanks to \eqref{eq:02/07-20:23:intro} and \eqref{eq:towardcontradictionintro13/01},
 which would contradict the fact that $c_n$ tends to $0$ and obtain 
the desired contradiction. Therefore, the next step is to
prove that $(\bd\eta^n, \bd\rho^n, \bd v^n)_{n \in \mathbb{N}}$ converges toward a solution to the linearized system \eqref{eq:linSysEtaIntro}--\eqref{eq:LinSysRhoVIntro} and conclude by stability of $\bd\nu^*$ that the limit 
 is necessarily  $(0,0,0)$. While it is not too hard to see from \eqref{eq:systemrhonvn29MaiIntro} that $(\bd\rho^n,\bd v^n)_{n \in \mathbb{N}}$ converges toward a solution to \eqref{eq:LinSysRhoVIntro} (see Proposition \ref{prop:preliminaryworkprop1-28/02}), it is more difficult to prove that $\bd \eta^n$ converges  toward a solution to \eqref{eq:linSysEtaIntro}. Indeed, the challenge is that we have no equation for $\bd \nu^n - \bd{\nu}^*$. Our idea is first to justify from the second equation of \eqref{eq:towardcontradictionintro13/01} together with \eqref{eq:02/07-20:23:intro} that 
\begin{equation}
\label{eq:02/07-20:26:intro}
    \lim_{n \rightarrow +\infty} \frac{1}{\lambda_n^2} \mathcal{I} \bigl( (t_0,\gamma_0), \bd\nu^n \bigr) =0,
\end{equation}
and then 
to deduce that the right-side of \eqref{eq:rewritingFischerintro}, normalized by $\lambda_n^2$, tends to $0$. By passing to the limit inside the terms on the right-hand side of \eqref{eq:rewritingFischerintro}, we would be led to
\begin{align} 
&0= \int_{t_0}^T \int_A \Bigl| \epsilon \nabla_a \frac{\eta_t}{\nu^*_t}(a) + \nabla_a \int_{\R^{d_1} \times \R^{d_2}} b(x,a) \cdot  \nabla_x v_t(x,y) d\gamma^*_t(x,y) + \nabla_a \bigl \langle b(\cdot,a) \cdot \nabla_x u_t^*; \rho_t \bigr \rangle \Bigr|^2 d\nu^*_t(a)dt.  \notag
\end{align}
Therefore, 
$\bd\eta$ would solve \eqref{eq:linSysEtaIntro} and then
$(\bd \eta, \bd \rho , \bd v)$ would solve the linearized system
\eqref{eq:linSysEtaIntro}--\eqref{eq:LinSysRhoVIntro}. This would complete the proof. Unfortunately, passing to the limit in \eqref{eq:rewritingFischerintro} (after normalizing by $\lambda_n^2$) is not so straightforward, and we have to proceed a little bit  differently to conclude (see Section \ref{se:5} for the details).

\subsection{Limitations of the Methods and Possible Extensions}

\label{subse:limitationsandpossiblextensions - 15/09}

All our main results rely, in one way or another, on the entropic regularization. Among other things, the resulting Gibbs form of the optimal control, see \eqref{eq:18/06/2025-14:44}, is essential to prove the \textit{injectivity} property used to derive the Jacobi condition and therefore to obtain Meta-Theorem \ref{meta-thm:1}. We also heavily rely on this penalization to
establish the PL 
inequality, the proof of Meta-Theorem \ref{meta-thm:2} stemming from a perturbative argument for a log-Sobolev inequality. We also stress that all the quantitative results of the paper depend implicitly on the regularization parameter $\epsilon$,
 though certainly not uniformly as $\epsilon \to 0^+$. Moreover, the set $\mathcal{O}$ of \textit{good} initial conditions also depends on this parameter. A perspective of research would be to understand the limit $\epsilon \rightarrow 0^+$ at various places in the argument. This could be done while possibly keeping a moment penalization. We also mention that the control problem exhibits some intriguing features when there is no regularization, see 
Lu et al. \cite{lu2020mean}: in that case, any local minimum --whose existence is not clear in the absence of penalization-- is global. This makes the regime $\epsilon \rightarrow 0^+$ all the more interesting. In the same vein, it is worth noting that the interplay between entropic penalization and local PL inequalities is also an intriguing question. On the one hand, our approach relies heavily on the log-Sobolev inequality satisfied by the minimizers, which requires $\epsilon>0$. On the other hand, the presence of an entropic penalization is not, in general and by itself, sufficient to ensure that local minimizers satisfy a PL inequality. Indeed, examples are discussed in \cite{monmarché2024localconvergencerateswasserstein} of entropy-penalized functionals defined on the space of probability measures that admit local (even global) minimizers which do not satisfy a PL inequality. In our framework, it is the Jacobi condition, which is intimately connected to the universal approximation property, that guarantees the stability of the minimizer.

As we already mentioned, Meta-Theorem \ref{meta-thm:1} is especially relevant to understand the convergence of the gradient descent associated to the control problem \eqref{eq:OriginalControlPb24/06}. The convergence was obtained in a similar setting by \cite{jabir2021meanfieldneuralodesrelaxed} and  \cite{barboni2024understandingtraininginfinitelydeep} but under certain additional structural  conditions detailed in 
Subsection \ref{subse:comparisonlit}. As explained following the statement of Meta-Theorem \ref{meta-thm:2}, the descent is anticipated to converge when  initialized in a neighborhood of the optimal control $\bd{\nu}^*$, provided the initial condition $(t_0,\gamma_0)$ belongs to the set $\mathcal{O}$ introduced in Meta-Theorem \ref{meta-thm:1}. In light of the PL inequality, the costs are expected to decrease at an exponential rate along the descent. Actually, the recent results of \cite{monmarché2024localconvergencerateswasserstein}, concerning a static optimization problem over the space of probability measures, suggest that exponential convergence of the control trajectories could also be established. We leave the verification of these conjectures for future research.

We do not address generalization bounds, that is, the improvement in accuracy achieved by training the network with larger feature samples in the training set. This is very much connected to the convergence of the optimal value and of the optimal distribution of parameters when the initial data consists of an empirical measure of the form
$ N^{-1} \sum_{i=1}^N \delta_{(X_0^{i},Y_0^{i})}$, 
where $(X_0^{i},Y_0^{i})_{1 \leq i \leq N}$ are independent random variables sampled from some $\gamma_0 \in \mathcal{P}(\R^{d_1} \times \R^{d_2})$, and $N \to \infty$.
 In the mean-field optimal control literature, this corresponds to looking for mean-field limits and this is sometimes referred to as the \textit{convergence problem}.  When the initial condition $(t_0,\gamma_0)$ belongs to the set $\mathcal{O}$, we expect to obtain sharper rates of convergence for the minimal values. Intuitively, the argument would be as follows. When the \textit{value function}
$U$, obtained by mapping the initial data to the minimal cost, is sufficiently smooth 
 on the whole domain $[0,T] \times {\mathcal P}(\R^{d_1} \times \R^{d_2})$, the \textit{weak-error} approach of \cite{CST} should yield error bounds of the form
\begin{equation} 
\label{eq:weak-error-bounds-01/07}
\biggl|  \E \biggl[ U \biggl(t_0, \frac1N \sum_{i=1}^N \delta_{(X_0^{i},Y_0^{i})} \biggr) \biggr] - U(t_0,\gamma_0)  \biggr| \leq \frac{C_U}{N},
\end{equation}
where the constant $C_{U}$ would depend on $U$ and, in particular, on global regularity properties of the second-order linear derivative $\delta^2 U/\delta \gamma^2$. By contrast, the Lipschitz regularity of $U$
alone (with respect to the measure argument, regarded as an element of ${\mathcal P}(\R^{d_1} \times {\mathbb R}^{d_2})$ equipped with some distance $\textrm{\rm d}$)
can \textit{only} lead to error bounds of the form 
\begin{equation} 
\label{eq:GlobalLipIntro-14/09}
\E \biggl[  \biggl|   U\biggl(t_0,
\frac1{N}
\sum_{i=1}^N \delta_{(X_0^{i},Y_0^{i})}
\biggr) - U(t_0,\gamma_0)       \biggr| \biggr]  \leq \Lip_{\textrm{\rm d}}(U) \times \E  \biggl[  \textrm{\rm d} \biggl(
\frac1{N}\sum_{i=1}^N \delta_{(X_0^{i},Y_0^{i})}, \gamma_0 \biggr) \biggr],  
\end{equation}
where $\Lip_{\textrm{\rm d}}(U)$ denotes the Lipschitz constant of $U$ with respect to $\textrm{\rm d}$. And, the right-hand side can at best be
of order $N^{-1/2}$ if the distance $\textrm{\rm d}$ is sufficiently weak, see for instance \cite{ddj2023}. 
In our setting, we expect estimates of the form \eqref{eq:GlobalLipIntro-14/09} to hold, leading to global error bounds of order $N^{-1/2}$.  On the other hand we do not expect $U$ to be smooth enough to obtain global error bounds of the form \eqref{eq:weak-error-bounds-01/07}. In fact $U$ is not even expected to be differentiable on the whole domain $[0,T] \times \mathcal{P}(\R^{d_1} \times \R^{d_2})$. 
However, we expect it to enjoy additional regularity on the set $\mathcal{O}$.  Indeed, $U$ can be shown to solve an infinite-dimensional first-order Hamilton-Jacobi over $[0,T] \times \mathcal{P}(\R^{d_1} \times \R^{d_2})$. Then, the set $\mathcal{O}$ is precisely the corresponding  \textit{region of strong regularity}, in the terminology of Fleming, see e.g. \cite{fleming1986asymptotic}. As shown in \cite{cardaliaguet-souganidis:2} and \cite{cjms2023} (for a class of stochastic, closed-loop, mean-field optimal control problems), the local second-order regularity of $U$  on $\mathcal{O}$ can be established
by carefully analyzing the stability of the  system of optimality conditions. We can then hope to get, locally in 
$\mathcal{O}$, tight error bounds similar to \eqref{eq:weak-error-bounds-01/07}. More precisely we can hope that, for every compact subset $\mathcal{K}$ of $ \mathcal{O}$ (for an appropriate topology)   there is a constant $C_{\mathcal{K},U}$ such that \eqref{eq:weak-error-bounds-01/07} holds with $C_U$ replaced by $C_{\mathcal{K},U}$ uniformly over $\mathcal{K}$. We also expect that the optimal controls $\bd{\nu}^N$ for $(t_0,
{N}^{-1} \sum_{i=1}^N \delta_{(X_0^{i},Y_0^{i})})$ converge to the (unique and stable) optimal control $\bd{\nu}$ for $(t_0,\gamma_0)$.
We plan to address these questions in a future contribution.

Another natural question, already alluded to above, is whether the local stability properties developed in this work can be used to revisit the approximation of the continuous mean-field model by neural networks of finite, but large, depth and width.

Actually, finite-depth and finite-width approximations are often studied simultaneously. For instance, a recent contribution by Chizat \cite{chizat2026hiddenwidthdeepresnets} analyzes the convergence, as the depth tends to infinity, uniformly with respect to the width, of the training dynamics of finite ResNets toward those of a Neural Mean ODE. More precisely, the comparison is carried out after a fixed number of gradient-descent steps, with error estimates that are uniform in the hidden width.
Extending such results to our setting is a promising direction, although it would require some care. Indeed, because of the entropic regularization, the corresponding particle training dynamics naturally involve an additional diffusive component and should therefore be described by a noisy, Langevin-type gradient dynamics rather than by the deterministic gradient descent considered in \cite{chizat2026hiddenwidthdeepresnets}. It would thus be particularly interesting to establish a finite-depth approximation result, uniform with respect to the width, for these stochastic training dynamics. One may further ask whether, within the stability regions identified in the present work, the resulting error estimates could be made uniform, or at least better controlled, with respect to the number of training iterations, in addition to being uniform in the width.

In this regard, it is worth noting that, in \cite{chizat2026hiddenwidthdeepresnets}, the neurons are initialized independently not only within each layer but also across layers. This is an important difference from \cite{ding2022overparameterization}, where the passage from the finite-depth, finite-width regime to the mean-field limit is studied as both depth and width tend to infinity under a more general initialization scheme. There, independence is imposed across neurons, whereas the parameters associated with a given neuron index across different layers may be coupled through a common stochastic process in the depth variable. Thus, nontrivial correlations across layers are allowed at initialization.
For this class of initializations, one may ask questions analogous to those raised above for independently initialized layers. In particular, since the approximation results in \cite{ding2022overparameterization} are not uniform along the training dynamics, it is natural to ask whether, in our setting, such estimates could become uniform in time within our stability regions. If so, the resulting statement would be closely related to the results of \cite{jabir2021meanfieldneuralodesrelaxed}, in which the authors obtain uniform-in-time bounds of this type in a regime where the strength of the entropic regularization is sufficiently large.

Independently of the training dynamics, a related question, more directly connected to our variational analysis but addressed neither in \cite{chizat2026hiddenwidthdeepresnets} nor in \cite{ding2022overparameterization}, is to compare optimal trajectories of the continuous control problem with approximately optimal trajectories of finite-depth ResNets. The strong time regularity of the optimal controls induced by the entropic penalization should be particularly useful in addressing this question, as it may allow for quantitative control of the error introduced by the discretization in depth.

For clarity, let us also stress that a direct adaptation of our arguments to the discrete-time setting appears difficult. Indeed, the \textit{Jacobi condition} established in this paper, which ultimately leads to the local Polyak--Lojasiewicz inequality, is intrinsically tied to the continuous-time formulation and does not seem to admit a straightforward discrete-time counterpart. It may therefore be more promising to exploit the stability properties of the continuous problem in order to control its finite-depth approximation.

Let us finally comment on the local nature of our PL inequality. This stands in contrast with the recent result of Gassiat and Suciu \cite{gassiat2024gradientflowcontrolspace}, who establish global convergence estimates for a class of control problems related to deep neural networks. The mechanism underlying their result is, however, quite different from the one developed here: under suitable Hörmander-type assumptions, the roughness of the initialization of the descent induces a global non-degeneracy property, whereas our PL inequality stems from the local non-degeneracy of a given minimizer, as encoded by the Jacobi condition. It would be interesting to investigate whether an appropriate form of roughness or randomization of the initialization could yield a similar global mechanism in our mean-field setting.
On the other hand, as already pointed out, the coexistence of several local minimizers arises naturally in entropy-regularized optimization problems over spaces of probability measures. This is illustrated by the recent work of Monmarché and Reygner \cite{monmarché2024localconvergencerateswasserstein}, where local PL-type inequalities are established in neighborhoods of individual stationary solutions. This suggests that, in the absence of additional structural assumptions, a local stability theory such as the one developed here may provide the natural framework for entropy-regularized problems.

\color{black}

\subsection{Organization of the Paper}
The paper is organized as follows. Section 
\ref{se:3}
reviews the first- and second-order conditions associated with problem 
\eqref{eq:OriginalControlPb24/06}--\eqref{eq:deftotalcost}. This, in particular, allows us to specify the functional spaces in which the various equations are posed. We also provide a complete definition of the notion of stable solution. However, the proofs of the results are postponed to the second part of the paper. This structure allows the reader to directly access the proofs of the main results in Sections \ref{sec:preparatorywork}, \ref{sec:TheUniversalAppox}, and 
\ref{se:5}.
Section \ref{sec:preparatorywork}  formalizes the compactness argument presented in paragraph \ref{subsubsec:stable}. The proof of Meta-Theorem \ref{meta-thm:1} is given in Section \ref{sec:TheUniversalAppox} (including a discussion on the role of the discriminating property), with Theorem \ref{thm:main:stability} providing a more precise version of the result. The proof of Meta-Theorem \ref{meta-thm:2} is given in Section \ref{se:5}, where the PL inequality is stated in Theorem \ref{prop:main:local:polyak:lojasiewicz}. Sections \ref{sec:ExistenceandFOC}, \ref{sec:LinearizedEquations}, and \ref{sec:SOC} return to the material introduced in Section \ref{se:3}. Section \ref{sec:ExistenceandFOC} focuses primarily on the first-order conditions for the control problem \eqref{eq:OriginalControlPb24/06}--\eqref{eq:deftotalcost}. The analysis of second-order conditions is split in two parts: the linearized equations, which serve as a preliminary step, are studied in Section \ref{sec:LinearizedEquations}, while the second-order conditions themselves are derived in Section \ref{sec:SOC}.
In Section \ref{sec:ProofsofSection3}, we elaborate on some of the auxiliary compactness arguments used in Section \ref{sec:preparatorywork}. Finally, a number of auxiliary results are stated in Appendix  \ref{sec:AppendixA}.

\subsection{Notation}

Throughout the text we use the following set of notations.

\subsubsection*{Spaces of regular functions and their duals}  
For some integers $d_1,d_2 \geq 1$ and some $ k \in \mathbb{N}$, we denote by $\mathcal{C}^k(\R^{d_1} \times \R^{d_2})$ the space of $k$-times continuously differentiable functions on $\R^{d_1} \times \R^{d_2}$. Variables in $\R^{d_1} \times \R^{d_2}$ will be denoted by $(x,y)$. The gradient and Hessian of $\varphi: \R^{d_1} \times \R^{d_2} \rightarrow \R$ are denoted by $\nabla \varphi$ and $\nabla^2 \varphi$, while the partial gradient and Hessian with respect to the $x$ variable are denoted by $\nabla_x \varphi$ and $\nabla^2_x \varphi$. For $p \geq 1$ and $k \geq 0$, we denote by $\mathcal{C}^k_{p}(\R^{d_1} \times \R^{d_2})$ the subset of $\mathcal{C}^k(\R^{d_1} \times \R^{d_2})$ 
consisting of functions whose growth, as well as that of their derivatives up to order $k$, is at most polynomial of order $p$, endowed with the norm
$$ \norm{ \varphi}_{\mathcal{C}^k_p(\R^{d_1} \times \R^{d_2})} := \sup_{(x,y) \in \R^{d_1} \times \R^{d_2}}  \sum_{ | ( \bd{\alpha}, \bd{\beta})| \leq k}   \frac{ \Bigl| \partial^{|(\bd{\alpha}, \bd{\beta})|}_{ x_1^{\alpha_1} \cdots x_{d_1}^{\alpha_{d_1}} y_1^{\beta_1} \cdots y_{d_2}^{\beta_{d_2}}   }  \varphi (x,y) \Bigr| }{1+|x|^p + |y|^p},  $$
where $(\bd{\alpha}, \bd{\beta}) = (\alpha_1, \dots,\alpha_{d_1},\beta_1,\dots,\beta_{d_2}) \in \mathbb{N}^{d_1 +d_2} $, and $|(\bd{\alpha}, \bd{\beta})| = \alpha_1 +  \cdots +\alpha_{d_1} + \beta_1+ \cdots +\beta_{d_2}$.
It will be convenient to distinguish between the growth of $\varphi$ itself and that of its derivatives: for $p,q \geq 1$ and $ k \geq 1$, we define $\mathcal{C}^k_{p,q}(\R^{d_1} \times \R^{d_2})$ as the set of functions such that $\varphi$ belongs to $\mathcal{C}^0_p(\R^{d_1} \times \R^{d_2})$ and  $\nabla \varphi$ belongs to $\mathcal{C}^{k-1}_q(\R^{d_1} \times \R^{d_2})$, endowed with the norm
$$ \norm{ \varphi}_{\mathcal{C}^k_{p,q}(\R^{d_1} \times \R^{d_2})} := \norm{ \varphi}_{\mathcal{C}^0_p(\R^{d_1} \times \R^{d_2})} + \norm{ \nabla \varphi}_{\mathcal{C}_q^{k-1}(\R^{d_1} \times \R^{d_2})}.$$

The space of $k$-times continuously differentiable functions on $\R^{d_1} \times \R^{d_2}$ with bounded derivatives is denoted by $\mathcal{C}^k_b(\R^{d_1} \times \R^{d_2})$ (not to be confused with $\mathcal{C}^k_0(\R^{d_1} \times \R^{d_2})$ defined later on) and it is endowed with the norm 
$$\norm{ \varphi}_{\mathcal{C}^k_{b}(\R^{d_1} \times \R^{d_2})} := \sup_{(x,y) \in \R^{d_1} \times \R^{d_2}}  \sum_{ | ( \bd{\alpha}, \bd{\beta})| \leq k}    \Bigl| \partial^{|(\bd{\alpha}, \bd{\beta})|}_{ x_1^{\alpha_1} \cdots x_{d_1}^{\alpha_{d_1}} y_1^{\beta_1} \cdots y_{d_2}^{\beta_{d_2}}   }  \varphi (x,y) \Bigr|.  $$
We also use the notation $\mathcal{C}^k_{1,b}(\R^{d_1} \times \R^{d_2})$ for
the space of functions of linear growth with bounded derivatives of order $1 \leq j \leq k$. The associated norm is denoted by $\norm{ \cdot }_{\mathcal{C}^k_{1,b}}$.
The subset of $\mathcal{C}^k_b(\R^{d_1} \times \R^{d_2})$ consisting of functions that, together with all their partial derivatives up to order $k$, vanish at infinity is denoted by $\mathcal{C}^k_0(\R^{d_1} \times \R^{d_2})$. It is endowed with the norm $\norm{ \cdot }_{\mathcal{C}^k_0(\R^{d_1} \times \R^{d_2})} := \norm{ \cdot }_{\mathcal{C}^k_b(\R^{d_1} \times \R^{d_2})}. $ 
The topological dual spaces are denoted with a ``$*$'' as in $(\mathcal{C}_b(\R^{d_1} \times \R^{d_2}))^*$, $(\mathcal{C}^k_{p,q}(\R^{d_1} \times \R^{d_2}))^*$, etc... They are implicitly endowed with the dual norms and the corresponding duality brackets are generaly denoted by $\langle \cdot  ;  \cdot\rangle$.

When the underlying space $\R^{d_1} \times \R^{d_2}$ is clear from context, we omit it in notations such as $\mathcal{C}^k_p$, $\mathcal{C}^k_{p,q}$, $(\mathcal{C}^k_p)^*$, $(\mathcal{C}^k_{p,q})^*$... When $k=0$, we omit the superscript and simply write $\mathcal{C}_b$, $\mathcal{C}_0$, $\mathcal{C}_p$...

The control space $ \R^{d'}$ is denoted by $A$, and  similarly, we define $\mathcal{C}^k(A)$, $\mathcal{C}^k_0(A)$, $\mathcal{C}^k_b(A)$, $\mathcal{C}^k_p(A)$ and $\mathcal{C}^k_{p,q}(A)$, along with their corresponding norms and dual spaces.

The following lemma is used regularly throughout the paper. Its understanding at this stage is not essential. The proof is left as an exercise to the reader and relies essentially on the definitions of the various spaces and duality brackets introduced above. In this respect, the statement itself illustrates the repeated use of these notations throughout the rest of the paper. Thanks to \assreg (for the vector field $b: \R^{d_1} \times A \rightarrow \R^{d_1}$)
and with $\vee$
being the usual notation for the maximum of two real numbers, we have
\begin{lem}
\label{lem:convenientexpress}
For $k \geq 0$ and $p \geq q \geq 0$ with $k+ q \leq 4$, there exists a constant $C_b$, depending on the vector field $b$ and $k,p$ and $q$, such that, for $\rho^{i} \in (\mathcal{C}^k_p)^*$ and $ \varphi^{i} : \R^{d_1} \times \R^{d_2} \rightarrow \R$ with $\nabla_x \varphi^{i} \in \mathcal{C}^k_{p-q}$, for $i=1,2$, and with the notation $f^{i}(a) := \langle b(\cdot,a) \cdot \nabla_x \varphi^{i}; \rho^{i} \rangle$ for all $a \in A$, it holds
\begin{equation} 
\norm{f^{i}}_{\mathcal{C}^q_{k+1}(A)} \leq C_b \norm{ \nabla_x \varphi^{i}}_{\mathcal{C}^k_{p-q}} \norm{ \rho^{i}}_{(\mathcal{C}^k_p)^*}. 
\label{eq:firstconvenientexpress02/03}
\end{equation}
Now, for $k_1,k_2,q \geq 0$ and $p_1,p_2 \geq q$ with $k_1 \vee k_2 +q \leq 4$, there exists a constant $C_b >0$, depending on $b$ and $k_1,k_2,q,p_1,p_2$, such that, for  $\nabla_x \varphi^{i} \in \mathcal{C}^{k_2}_{p_2-q} \cap \mathcal{C}^{k_1}_{p_1-q}$ and $\rho^{i} \in (\mathcal{C}^{k_1}_{p_1})^* \cap (\mathcal{C}^{k_2}_{p_2})^*$, for $i=1,2$,
\begin{equation} 
\begin{split}
&\norm{f^{2} - f^{1}}_{\mathcal{C}^q_{k_1 \vee k_2+1}(A)} 
\\
&\hspace{15pt} \leq C_b \Bigl( \norm{ \nabla_x\varphi^2}_{\mathcal{C}^{k_2}_{p_2-q}} + \norm{\rho^1}_{(\mathcal{C}^{k_1}_{p_1})^*} \Bigr) \Bigl( \norm{ \rho^2-\rho^1}_{(\mathcal{C}^{k_2}_{p_2})^*} + \norm{ \nabla_x \varphi^2 - \nabla_x \varphi^1}_{\mathcal{C}^{k_1}_{p_1-q}} \Bigr).  \end{split}
\label{eq:secondconvenientexpress02/02}
\end{equation}
\end{lem}

\subsubsection*{Measure theory} 
We identify $(\mathcal{C}_0(A))^*$ with the set of finite Radon measures on $A$, denoted by $\mathcal{M}(A)$. The subset of positive Radon measures is denoted by $\mathcal{M}^+(A)$, and the subspace of finite Radon measures $\nu \in \mathcal{M}(A)$ such that $(1+|a|^k) \nu$ belongs to $\mathcal{M}(A)$ 
is denoted by $\mathcal{M}_{(1+|\cdot|^k)} (A)$ and endowed with the norm 
$$ \norm{ \nu}_{\mathcal{M}_{(1+|\cdot|^k)}} := \sup_{\varphi \in \mathcal{C}(A), |\varphi(a)| \leq (1+|a|^k) } \int_{A} \varphi(a) d\nu(a) = \int_{A} (1+|a|^k) d|\nu|(a),$$
where $\vert \nu \vert$
is the total variation of the (finite) measure 
$\nu$. 

For $p \geq 1$, we denote by $\mathcal{P}_p(A)$ the space of Borel probability measures on $A$ with finite moment of order $p$ (so that 
${\mathcal P}_p(A) \subset {\mathcal M}_{(1+ \vert \cdot \vert^p)}(A)$), endowed with the Monge-Kantorovich-Rubinstein distance of order $p$
$$ d_p(\mu, \nu)^p := \inf_{ \pi \in \Gamma(\mu,\nu) } \int_{\R^n} |x-y|^p d\pi(x,y),$$
where $\Gamma(\mu,\nu) \subset \mathcal{P}_p(A^2)$ is the subset of couplings between $\mu$ and $\nu$. 
As usual, the space of probability measures on $A$ (without any integrability condition) is denoted by ${\mathcal P}(A)$. 
We define similarly $\mathcal{P}_p(\R^{d_1} \times \R^{d_2})$ (and $\mathcal{P}(\R^{d_1} \times \R^{d_2})$).

For two probability measures $\nu, \tilde{\nu}$ on some Polish space $E$, we define the relative entropy of $\nu$ with respect to $\tilde{\nu}$ by
\begin{equation}
\label{eq:def:relative:entropy}
\mathcal{E} \bigl( \nu | \tilde{\nu} \bigr) := \left \{ 
\begin{array}{ll}
\displaystyle \int_E \log  \frac{d \nu}{d \tilde{\nu}} d\nu & \mbox{ if } \nu \mbox{ is absolutely continuous with respect to } \tilde{\nu}, \\
 +\infty & \mbox{ otherwise,} 
 \end{array}
\right.
\end{equation}
where $\frac{d \nu}{d\tilde{\nu}}$ denotes the Radon-Nikodym derivative of $\nu$ with respect to $\tilde{\nu}$. When a probability measure $\nu \in \mathcal{P}(A)$ is absolutely continuous with respect to the Lebesgue measure over $A$, we denote its density by the same symbol $\nu$ and  we write $\nu(a)da$ for $d\nu(a)$. 

In the same spirit of simplifying notation, we will often write $b(x, \nu)$ instead of the full expression $\int_A b(x,a)\, d\nu(a)$.

\subsubsection*{Families indexed by a time parameter.}

For some time interval $[t_0,T]$ and some metric space $(E,d_E)$, we denote by $\mathcal{C}([t_0,T], E)$ the space of continuous curves from $[t_0,T]$ to $E$ endowed with the metric 
$$ \sup_{t \in [t_0,T]} d_E(e^1_t, e^2_t), \quad \quad \mbox{for} \hspace{5pt} \bd{e}^1 = (e^1_t)_{t \in [t_0,T]}, \ \bd{e}^2 = (e^2_t)_{t \in [t_0,T]} \in \mathcal{C}([t_0,T], E).$$
Throughout, we use $\bd{bold}$ characters to denote family of functions parametrized by the time parameter $t$. We also define $\mathcal{C}^{\alpha}([t_0,T],E)$, for $\alpha \in (0,1)$, the subset of $\alpha$-Hölder continuous curves. We denote by $\mathcal{L}^{\infty}([t_0,T], E)$ the set of bounded functions from $[t_0,T]$ to $E$ (without any measurability condition).

\subsubsection*{Additional functional spaces.}

Throughout the text we introduce several additional functional spaces. For some $t_0 \in [0,T]$, the spaces of controls and control perturbations $\mathcal{A}(t_0), \mathcal{D}(t_0)$ and $\mathcal{A}^{\ell}(t_0)$ are given respectively in Definitions \ref{def:admissible:control}, \ref{defn:D(t_0)} and \ref{def:admissible:perturbation}. Similarly, the space $\mathcal{R}(t_0)$ of solutions to the linearized continuity equation in \eqref{eq:LinSysRhoVIntro} is given in Definition \ref{def:R(t_0)}.

\section{Optimality Conditions and Stable Solutions}
\label{se:3}

In this section, we introduce the first  and second-order necessary conditions 
satisfied by any minimizer 
of 
\eqref{eq:OriginalControlPb24/06}
and then define properly the notion of 
stable minimizer. Even though these 
preliminary 
results play a fundamental role in the rest of the paper, we feel it sufficient to postpone their proofs 
to Sections \ref{sec:ExistenceandFOC}, \ref{sec:LinearizedEquations} and \ref{sec:SOC}.

\subsection{Admissible Controls and Existence of a Minimizer}
\label{subse:3.1}
We begin by clarifying the definition of the control problem 
\eqref{eq:OriginalControlPb24/06}--\eqref{eq:deftotalcost}
and in particular the choice of the admissible controls. 
We consider controls ${\boldsymbol \nu}$ as trajectories taking values in ${\mathcal P}(A)$, with finite mean relative entropy with respect to the Gibbs density associated with the potential $\ell$: 

\begin{defn}
\label{def:admissible:control}
For a fixed initial time $t_0 \in [0,T]$, 
${\mathcal A}(t_0)$ is defined as the collection of 
measurable mappings ${\boldsymbol \nu} : t \in [t_0,T] \mapsto \nu_t \in {\mathcal P}(A)$ such that 
\begin{equation}
\label{eq:definitionadmissiblecontrols}
\int_{t_0}^T \mathcal{E} \bigl( \nu_t | \nu^\infty \bigr) dt < +\infty, 
\end{equation}
where $\nu^\infty$ is the log-concave probability defined in \eqref{eq:defnuinfty03/07}.
\end{defn}

The notion of measurability used for the mapping ${\boldsymbol \nu}$ is explained in Remark \ref{rmk:measurability}. 
Moreover, 
observe that the entropy penalization can be rewritten as the sum of a, somehow standard, moment penalization and the entropy with respect to the Lebesgue measure
\begin{equation}
\label{eq:Lepsilon:reformulation}
\mathcal{E}(\nu|\nu^{\infty}) = \int_A \ell(a) d\nu(a) + \int_A \log \nu(a) d\nu(a) - \log \int_A e^{-\ell(a)}da.
\end{equation}
In fact, we will prove in Lemma \ref{lem:cost:L4:moments}
that there 
exists a constant $C>0$ such that, for any ${\boldsymbol \nu} \in {\mathcal A}(t_0)$, 
\begin{equation} 
\int_{t_0}^T \int_A |a|^4 d\nu_t(a)dt + \sup_{t_0 \leq t_1 < t_2 \leq T} 
\biggl\{ \frac{1}{\sqrt{t_2-t_1}} \int_{t_1}^{t_2} \int_A |a|^2 d\nu_t(a)dt
\biggr\} \leq 
C
\biggl( 
1+
\int_{t_0}^T {\mathcal E}\bigl( \nu_t \vert 
\nu^\infty \bigr) dt\biggr).
\label{eq:boundfromPinsker16Sept:sec:3}
\end{equation}
This motivates the introduction of the following bigger vector space that will contain any linear combinations of elements of $\mathcal{A}(t_0)$. This will become especially relevant to state the second-order optimality conditions of Subsection \ref{subse:SOC04/03} and to define the stable solutions in  Subsection \ref{sec:subseStableSol}.

\begin{defn}
\label{defn:D(t_0)}
    We call ${\mathcal D}(t_0)$ the set of measurable maps $\bd{\nu} : t \in [t_0,T] \mapsto 
    \nu_t \in \mathcal{M}(A)$ (where $\mathcal{M}(A)$ is the set of finite measures over $A$) such that
\begin{equation} 
\norm{ \bd{\nu} }_{\mathcal{D}(t_0)} := \int_{t_0}^T \int_A(1+|a|^4) d|\nu_t|(a)dt + \sup_{t_2 > t_1 \in [t_0,T]} \frac{1}{\sqrt{t_2-t_1}} \int_{t_1}^{t_2} \int_A (1+|a|^2) d|\nu_t|(a)dt <+\infty. 
\label{eq:definitionadmissiblecontrols:aux}
\end{equation}
\end{defn}

\begin{rmk}
\label{rmk:measurability}
By contrast with elements of $\mathcal{A}(t_0)$, elements of $\mathcal{D}(t_0)$ are neither required to be positive nor to have unit mass for all $t$.
Measurability of the mapping ${\boldsymbol \nu}$
is understood in the sense that the map 
$t \in [t_0,T] \mapsto \int_A f d \nu_t$
is measurable for every $f \in {\mathcal C}_0(A)$. Using the representation
of the positive part
\begin{equation*}
\nu_t^+(E) = 
\sup_{f \in {\mathcal C}_0(A)} \Bigl\{ \int_A f d\nu_t ; \quad 0\leq f \leq  {\mathrm 1}_E \Bigr\}, \quad E 
\in {\mathcal B}(A), 
\end{equation*}
together with the separability of ${\mathcal C}_0(A)$, this implies that, for any  Borel subset $E$ of $A$, the mapping $t \in [t_0,T] \mapsto \nu_t^+(E)$ is measurable, and similarly for $\nu_t^-$. The converse is true in the sense that any mapping ${\bd \nu}$, such that $t \in [t_0,T] \mapsto \nu_t^{\pm}(E)$
is measurable 
for any  Borel subset $E$ of $A$, is
measurable in the former sense. 
When the arrival space of ${\boldsymbol \nu}$ is
restricted to ${\mathcal P}(A)$ (as in the case in Definition \eqref{def:admissible:control}), 
measurability can be viewed either by equipping 
${\mathcal P}(A)$ with
the Borel $\sigma$-field induced by the weak convergence topology (i.e., when probability measures are tested with respect to continuous and bounded functions on $A$)
or by requiring 
$t \in [t_0,T] \mapsto \nu_t(E)$ to be measurable for any $E \in {\mathcal B}(A)$. 
Both are equivalent, 
see \cite[Proposition 7.25]{BertsekasShreve}.

Back to the setting of Definition \ref{defn:D(t_0)},
the mapping $t \in [t_0,T] \mapsto \vert \nu_t \vert = \nu_t^+ + \nu_t^- \in {\mathcal M}(A)$ is  measurable. This guarantees that 
the norm in  
\eqref{eq:definitionadmissiblecontrols:aux} is well-defined. Equivalently, 
${\boldsymbol \nu}$ can be regarded as a finite measure on 
$[t_0,T] \times A$ defined by
\begin{equation*}
{\boldsymbol \nu}(I \times B) = \int_I \nu_t(B) dt,
\end{equation*}
for any two Borel subsets $I$ and $B$ of $[t_0,T]$ and $A$ respectively. 
The measure ${\boldsymbol \nu}$ has the Lebesgue measure as first marginal on $[t_0,T]$
and 
the collection 
$(\nu_t)_{t_0 \le t \le T}$ can be regarded as the kernel 
resulting from the disintegration of ${\boldsymbol \nu}$
with respect to ${\rm Leb}_{[t_0,T]}$. 
We sometimes write 
${\boldsymbol \nu} = {\rm Leb}_{[t_0,T]} \otimes (\nu_t)_{t_0 \le t \le T}$, 
or $d\nu(t,a) = d\nu_t(a) dt$.
In the framework of Definition \ref{def:admissible:control}, 
this point of view permits to 
regard 
$\bd \nu$ as an element of 
$(T-t_0) \cdot {\mathcal P}([t_0,T] \times A)$
and then to 
 interpret
\eqref{eq:definitionadmissiblecontrols} as a relative entropy on the wider space 
$[t_0,T] \times A$. Indeed, by \cite[Corollary 2.7]{BudhirajaDupuisbook} (with the definition of the relative entropy being extended in a trivial way to positive measures with the same mass), 
\begin{equation}
\label{eq:interpretation:wider:space}
\int_{t_0}^T {\mathcal E}\bigl( \nu_t \vert 
\nu^\infty \bigr) dt
= {\mathcal E}\bigl( {\boldsymbol \nu} \vert {\rm Leb}_{[t_0,T]}
\times \nu^\infty \bigr), 
\end{equation}
which provides another interpretation of the cost $J ((t_0,\gamma_0), \cdot  )$
as the sum of the original regression cost and a time-space entropic penalty.

Lastly, note for future reference that, up to a redefinition of the mapping $t \mapsto \nu_t$ on a Lebesgue-negligible subset of $[t_0, T]$, we may always assume that $\nu_t$ belongs to $\mathcal{M}_{1+|a|^4}(A)$ for all $t \in [t_0, T]$ --that is, the measure $(1 + |a|^4)\nu_t$ is finite on $A$.
\label{rmk:defnD(t_0)}
\end{rmk}

For any element $\bd{\nu} \in \mathcal{D}(t_0)$ (so, in particular, for any admissible control in $\mathcal{A}(t_0)$), it follows from \assreg \ that the velocity field $(t,x) \mapsto b(x,\nu_t)$ in 
the continuity equation 
\eqref{eq:Continuitygammanu} is well-defined and three times continuously 
differentiable in the variable $x$, with 
\begin{equation} 
 \int_{t_0}^T \norm{ b(\cdot , \nu_t)}_{\mathcal{C}^3_b} dt + \sup_{t_1 < t_2 \in [t_0,T]} \frac{1}{\sqrt{t_2-t_1}} \int_{t_1}^{t_2} \norm{b(\cdot,\nu_t)}_{\mathcal{C}^1_b}dt \leq C \norm{\bd{\nu}}_{\mathcal{D}(t_0)}, 
\label{eq:estimateb29/01}
\end{equation}
where we recall the convenient notation $b(x,\nu) = \int_A b(x,a)d\nu(a)$ for any $\nu \in \mathcal{M}(A)$.
Estimate \eqref{eq:estimateb29/01} makes it possible to solve  the continuity equation \eqref{eq:Continuitygammanu}. For $(t_0,\gamma_0) \in [0,T] \times \mathcal{P}_2(\R^{d_1} \times \R^{d_2})$ and $\bd{\nu} \in \mathcal{D}(t_0)$, we say that $\bd{\gamma} \in \mathcal{C}([t_0,T], \mathcal{P}_2(\R^{d_1} \times \R^{d_2}))$ is a solution to \eqref{eq:Continuitygammanu} if, for all test function $\varphi \in \mathcal{C}^1_b([t_0,T] \times \R^{d_1} \times \R^{d_2})$ (i.e., the function and its first time and space derivatives are continuous and bounded) and all $t_1\in [t_0,T]$, 
\begin{equation}
\label{eq:CE:def}
\begin{split}
\int_{\R^{d_1} \times \R^{d_2}} \varphi_{t_1}(x,y) d\gamma_{t_1}(x,y) &= \int_{\R^{d_1} \times \R^{d_2}} \varphi_{t_0}(x,y) d\gamma_0(x,y) \\
&+ \int_{t_0}^{t_1} \int_{\R^{d_1} \times \R^{d_2}} \bigl \{ \partial_t \varphi_t(x,y) + b(x, \nu_t) \cdot \nabla_x \varphi_t(x,y) \bigr \} d\gamma_t(x,y)dt. 
\end{split}
\end{equation}

Thanks to the regularity of the vector field driving the equation, well-posedness of the continuity equation is standard. The solution is obtained by pushing the initial condition $\gamma_0$ along the flow of the ODE \eqref{eq:ODE:intro}. 
In Appendix \ref{subse:A:1}, we gather a sequence of statements leading to the following result:

\begin{prop}
\label{lem:compactnessContinuityEquation24/06}
    Take ${\boldsymbol \nu}=(\nu_t)_{t_0 \le t \le T} \in \mathcal{D}(t_0)$. Then, for any given $\gamma_0 \in \mathcal{P}_2(\R^{d_1} \times \R^{d_2})$, there is a unique distributional solution ${\boldsymbol \gamma} \in \mathcal{C}([t_0,T], \mathcal{P}_2(\R^{d_1} \times \R^{d_2}))$ to the continuity equation
\begin{equation}
\label{eq:lem:compactnessContinuityEquation24/06:statement}
 \partial_t \gamma_t + \div_x( b(x,\nu_t) \gamma_t) = 0 \quad \mbox{ in } (t_0,T) \times \R^{d_1} \times \R^{d_2}, \quad \gamma(t_0)= \gamma_0.
 \end{equation}
It is given by 
$(\gamma_t = (X_t^{t_0,\cdot}, i_d) \# \gamma_0)_{t_0 \le t \le T}$
where $(X_t^{t_0,\cdot})_{t_0 \le t \le T}$
is the flow of 
\eqref{eq:ODE:intro}, solution to 
$$ \dot{X}^{t_0,x}_t = b(X_t^{t_0,x}, \nu_t) \quad t \in [t_0,T], \quad X_{t_0}^{t_0,x} = x. $$
Moreover, there exists a non-decreasing function $\Lambda : {\mathbb R}_+ \rightarrow {\mathbb R}_+$, independent of $(t_0,\gamma_0)$ and 
${\boldsymbol \nu}$, such that, for each $p \in [1,3]$ such that $\gamma_0$ belongs to $\mathcal{P}_p(\R^{d_1} \times \R^{d_2})$, 
 \begin{equation}
\label{eq:lem:compactnessContinuityEquation24/06:statement:48}
\sup_{t \in [t_0,T]}\frac{\int_{\R^{d_1} \times \R^{d_2}} \bigl( |x|^2 + |y|^2 \bigr)^{p/2}  d\gamma_t(x,y)}{1 + \int_{\R^{d_1} \times \R^{d_2}} \bigl( |x|^2 + |y|^2 \bigr)^{p/2}  d\gamma_0(x,y)}  + \sup_{t \neq s \in [t_0,T]}  \frac{d_p(\gamma_t,\gamma_s)}{\sqrt{|t-s|}} \leq \Lambda
 \bigl( \norm{\bd{\nu}}_{\mathcal{D}(t_0)} \bigr).
 \end{equation}
 \end{prop}
By an approximation argument, we can prove that, whenever $\gamma_0$ belongs to $\mathcal{P}_p(\R^{d_1} \times \R^{d_2})$ for some $p \geq 2$, then equation \eqref{eq:CE:def} is satisfied for any test function $\varphi \in \mathcal{C}^1([t_0,T] \times \R^{d_1} \times \R^{d_2})$ such that $\varphi, \partial_t \varphi$ and $\nabla_x \varphi$ have growth of order $p$ in $(x,y)$ at infinity.

Notice that, since $\gamma_0$ lies in $\mathcal{P}_2(\R^{d_1} \times \R^{d_2})$ then, for all $\bd{\nu} \in \mathcal{A}(t_0)$,  $\gamma_T$ belongs to 
${\mathcal P}_2({\mathbb R}^{d_1} 
\times {\mathbb R}^{d_2})$ which guarantees in the end that the total cost $J((t_0,\gamma_0),{\boldsymbol \nu})$ is finite.

The control problem 
\eqref{eq:OriginalControlPb24/06}--\eqref{eq:deftotalcost} 
is studied in an exhaustive manner in
Section \ref{sec:ExistenceandFOC}.
Results can be summarized as follows: 

\begin{prop}
Let $(t_0,\gamma_0) \in [0,T] \times {\mathcal P}_2({\mathbb R}^{d_1} 
\times {\mathbb R}^{d_2})$. 
Then, for any $M \in {\mathbb R}$, the sub-level set 
$\{ {\boldsymbol \nu} \in {\mathcal A}(t_0) : J((t_0,\gamma_0),{\boldsymbol \nu})
\leq M \}$ is relatively compact for the weak topology on 
$(T-t_0) \cdot {\mathcal P}([t_0,T] \times A)$
and the functional ${\boldsymbol \nu} \in {\mathcal A}(t_0) \mapsto 
J((t_0,\gamma_0),{\boldsymbol \nu})$ is lower-semicontinuous for 
the weak topology (cf. Remark \ref{rmk:measurability} for the definition of the latter). In particular, the optimal control problem \eqref{eq:OriginalControlPb24/06}--\eqref{eq:deftotalcost} has at least one solution (i.e., one minimizer) in 
${\mathcal A}(t_0)$. 
\end{prop}

\subsection{First-Order Condition and Consequences}

We now address the first-order condition associated with the control problem 
\eqref{eq:OriginalControlPb24/06}--\eqref{eq:deftotalcost}.
We start with the following definition: For $t_0\in [0,T]$ and ${\boldsymbol \nu} \in {\mathcal D}(t_0)$, we call equation adjoint to ${\boldsymbol \gamma}$ (at least, when ${\bd \nu}$ is the control driving $\bd \gamma$) 
    the (backward) transport
    equation
\begin{equation}
\label{eq:adjoint:equation}
\left \{
\begin{array}{ll}
\displaystyle -\partial_t  {u}_t(x,y) -  b\bigl(x, {\nu}_t\bigr)  \cdot \nabla_x  {u}_t(x,y) = 0,
&\quad (t,x,y)
\in [t_0,T] \times \R^{d_1} \times \R^{d_2},
\\
\displaystyle
\qquad 
 {u}_T(x,y) = L(x,y), &\quad (x,y) \in \R^{d_1} \times \R^{d_2}.
\end{array}
\right.
\end{equation}
Solvability of equation 
\eqref{eq:adjoint:equation} is
addressed
in the space 
$\mathcal{C}([t_0,T], \mathcal{C}^1_{2,1})$. We say that ${\boldsymbol u}=(u_t)_{t_0 \le t \le T} \in \mathcal{C}([t_0,T], \mathcal{C}^1_{2,1})$ is a solution to the transport equation if, for all $t \in [t_0,T]$ and all $(x,y) \in \R^{d_1} \times \R^{d_2}$,
$$ u_{t}(x,y) = L(x,y) + \int_{t}^{T} b(x,\nu_s) \cdot \nabla_x u_s(x,y)ds.$$

Similarly to Proposition \ref{lem:compactnessContinuityEquation24/06}, the following result
is presented in 
Appendix \ref{subse:A:1}, see Proposition \ref{prop:SolutionBackxard16SeptAppendix}:
\begin{prop}
\label{prop:SolutionBackxard16Sept}
    Take $t_0 \in [0,T]$ and ${\boldsymbol \nu} \in \mathcal{D}(t_0)$. 
    Then, there exists a unique solution ${\boldsymbol u}=(u_t)_{t_0 \le t \le T} \in \mathcal{C}([t_0,T], \mathcal{C}^1_{2,1})$ to the transport equation \eqref{eq:adjoint:equation}. Moreover, for every $t \in [t_0,T]$, 
    the mapping $(x,y) \mapsto u_t(x,y)$ belongs to $\mathcal{C}^3(\R^{d_1} \times \R^{d_2})$ and there exists a non-decreasing function $\Lambda : {\mathbb R}_+ \rightarrow 
    {\mathbb R}$, independent of $t_0$ and ${\boldsymbol \nu}$, such that 
    $$ \sup_{t \in [t_0,T]} \norm{u_t}_{\mathcal{C}^3_{2,1}} + \sup_{ t_2 > t_1 \in [t_0,T]} \frac{\norm{u_{t_2}-u_{t_1}}_{\mathcal{C}^1_{1}}}{\sqrt{t_2-t_1}} \leq \Lambda \bigl(\norm{\bd{\nu}}_{\mathcal{D}(t_0)} \bigr).$$
\end{prop}
By 
\eqref{eq:boundfromPinsker16Sept:sec:3}, the   above bound remains true for $\bd{\nu} \in \mathcal{A}(t_0)$ if one replaces the right-hand side by 
$\Lambda ( \int_{t_0}^T
    \mathcal{E}  ( \nu_t |  \nu^\infty )dt)$
    (for a possibly different choice of the function $\Lambda$).

\begin{rmk}
    Notice that the functional spaces in which we solve the equations for $\bd{u}$ and $\bd{\gamma}$ are significantly different. This is due to the different divergence/non-divergence structures of these equations, which call for different techniques. This is also related to the regularity of the boundary data for these two equations: $\gamma_0$ is only a probability measure, so $\gamma_t$, the solution at time $t$ will have no more spatial regularity than $\gamma_0$ (and will not even have a density if $\gamma_0$ does not have one), while $u_t$ has the same spatial regularity as the terminal condition $L$ which is a (relatively) smooth function. On the other hand, both equations are intimately connected to the flow of the underlying ODE since 
     $$ \gamma_t = \gamma_0 \circ (X_t^{t_0, \cdot}, i_d)^{-1} , \quad u_t(x,y) = L(X_T^{t,x},y). $$
In particular, if $\gamma_0$ has a density, then so does $\gamma_t$, and it is given by $ \gamma_t(x,y) = \gamma_0( \chi_t^x,y) | \det \nabla_x X_t^{t_0, \chi_t^x}|^{-1} $, where $\chi$ is the inverse of the flow: $\chi_t^x = (X_t^{t_0,\cdot})^{-1}(x).$
\end{rmk}

The following 
result provides the form 
of the 
forward-backward
system satisfied at optimality. 
A similar result can be found in 
\cite{hu2019meanfieldlangevinsystemoptimal,jabir2021meanfieldneuralodesrelaxed}:

\begin{thm}
\label{thm:OCThm27/06}
    Let $(t_0,\gamma_0) \in [0,T] \times \mathcal{P}_3(\R^{d_1} \times \R^{d_2})$.
    Then, 
    any minimizer $ \boldsymbol{\nu}^* \in {\mathcal A}(t_0)$
    to 
    $J( (t_0, \gamma_0) , \cdot )$ admits a jointly continuous density (still denoted $(t,a) \mapsto \nu^*_t(a)$) given, for all $(t,a) \in [t_0,T] \times A$, by
\begin{equation} 
\nu^*_t(a) = \frac{1}{z^*_t} \exp \Bigl( - \ell(a) - \frac{1}{\epsilon} \int_{\R^{d_1} \times \R^{d_2}}  b(x,a) \cdot \nabla_x u^*_t(x,y)  d\gamma^*_t(x,y) \Bigr),
\label{eq:NOCfornut2Avril}
\end{equation}    
where $z^*_t$ is a normalizing constant, $\boldsymbol{\gamma}^* = (\gamma^*_t)_{t \in [t_0,T]}$ 
is the solution to the continuity equation 
\eqref{eq:Continuitygammanu} starting from $(t_0,\gamma_0)$
and driven by 
$\boldsymbol \nu^*$, 
and $\boldsymbol{u}^* = (u^*_t)_{t \in [t_0,T]}$
solves the transport equation 
\eqref{eq:adjoint:equation}
driven by 
$\boldsymbol \nu^*$. 
 
 In particular, the pair  $(\boldsymbol{u}^*,\boldsymbol{\gamma}^*)$ solves the forward-backward system 
\begin{equation}
\label{eq:NOCforutgammat2Avril}
\left \{
\begin{array}{ll}
\displaystyle -\partial_t u^*_t(x,y) -  b\bigl(x,\nu^*_t\bigr)  \cdot \nabla_x u^*_t(x,y) = 0,
&\quad (t,x,y)
\in [t_0,T] \times \R^{d_1} \times \R^{d_2},
\\
\displaystyle
\qquad 
 u^*_T(x,y) = L(x,y), &\quad (x,y) \in \R^{d_1} \times \R^{d_2};
\\
\displaystyle \partial_t \gamma^*_t
+ \div_x \bigl(  b(x,\nu^*_t) \gamma^*_t  \bigr) = 0, &\quad
 \textrm{\rm in}
 \ [t_0,T] \times  \R^{d_1} \times \R^{d_2}, 
\\
\displaystyle \qquad \gamma^*_{t_0} = \gamma_0
&\quad
 \textrm{\rm in} 
 \ 
\R^{d_1} \times \R^{d_2}.
\end{array}
\right.
\end{equation}
\end{thm}

\color{black}

\begin{rmk}
\label{rmk:3.2}
Theorem \ref{thm:OCThm27/06} is proven in 
Subsection \ref{subse:7:2}. At this stage, 
the following remarks are in order: 
\begin{enumerate}[(i)]
\item Although the optimization problem is taken over the space of measures $\mathcal{A}(t_0)$ we will often identify an optimal solution $\boldsymbol{\nu}^*$ with its continuous density given by \eqref{eq:NOCfornut2Avril}.
\item
Notice that $z^*_t$
in \eqref{eq:NOCfornut2Avril}
is a normalization constant that 
guarantees that $\nu^*_t(A)=1$, namely 
$z^*_t$
is given, for all $t \in [t_0,T]$, by
    \begin{equation}
    \label{eq:F:formule:pour:tilde:z}
    z^*_t = \int_{A} 
    \exp \Bigl(  - \ell(a) - \frac{1}{\epsilon} \int_{\R^{d_1} \times \R^{d_2}}b(x,a) \cdot \nabla_x u^*_t(x,y)d \gamma^*_t(x,y)  \Bigr)  da.
    \end{equation}
By (i) in \assreg \, and because $ \gamma^*_t$ belongs to ${\mathcal P}_2({\mathbb R}^{d_1} 
\times {\mathbb R}^{d_2})$
and 
$\nabla_x  u^*_t$ has linear growth in $(x,y)$, 
the integral inside the exponential is finite
and less than $C(1+\vert a \vert)$
for a constant $C$ independent of $t$. 
\item 
The probability measure $\nu^*_t$
is defined as a Gibbs measure (which we may compare to 
$\nu^\infty$ in \eqref{eq:definitionadmissiblecontrols}). Equivalently, 
it is the unique 
density solving the stationary Fokker-Planck equation
\begin{equation} 
\displaystyle - \epsilon \Delta_a \nu_t(a) - \div_a \Bigl[ \Bigl( \epsilon \nabla_a \ell(a) + \nabla_a \Bigl( \int_{\R^{d_1} \times \R^{d_2}} b(x,a) \cdot \nabla_x u^*_t(x,y) d \gamma^*_t(x,y) \Bigr) \Bigr) \nu_t(a)  \Bigr] =0, \ a \in A. 
\label{eq:stationnaryFPe24Sept}
\end{equation}
\item
The reader will easily recognize that 
$\boldsymbol \gamma^*=( \gamma^*_t)_{t_0 \le t \leq T}$
in 
\eqref{thm:OCThm27/06}
is the optimal curve associated 
with the optimal control $ {\boldsymbol \nu}^*$. The function
$\boldsymbol u^*=( u^*_t)_{t_0 \le t \le T}$ is referred to as the adjoint trajectory or as the multiplier. 
\end{enumerate}
\end{rmk}

From the study of the system \eqref{eq:NOCfornut2Avril}-
\eqref{eq:NOCforutgammat2Avril}, we are able to establish further regularity properties 
of the optimal control
$\boldsymbol \nu^*$ and the   multiplier $\boldsymbol u^*$. In particular optimal solutions $\bd{\nu}^*$ are continuous in time (by contrast with generic elements of $\mathcal{A}(t_0)$).
The following result is proven in Subsection 
\ref{sec:additionalregandfirststability}:

\begin{prop}
    \label{prop:regularityfromOC}
Let $(t_0,\gamma_0) \in [0,T] \times {\mathcal P}_2({\mathbb R}^{d_1} \times 
{\mathbb R}^{d_2})$ and assume that $({\boldsymbol \nu}, {\boldsymbol \gamma}, {\boldsymbol u})$ is a solution to the system \eqref{eq:NOCfornut2Avril}--\eqref{eq:NOCforutgammat2Avril}. Then, $u$ has the following regularity 
\begin{align} 
 &t \mapsto u_t \in \mathcal{C}^2_2 \mbox{ is  bounded and } t \mapsto u_t \in \mathcal{C}^1_2 \mbox{ is continuous},
\label{eq:testfn1foru}
\\
&t \mapsto \partial_t u , \nabla_x u \in \mathcal{C}^2_1 \mbox{ are bounded and }  t \mapsto \partial_t u_t , \nabla_x u_t \in \mathcal{C}^1_2 \mbox{ are continuous. }  
\label{eq:testfn2foru}
\end{align}
 Moreover, there exists a non-decreasing function $\Lambda : {\mathbb R}_+
 \rightarrow [1,+\infty)$,  independent of $(t_0,\gamma_0)$ and 
 $({\boldsymbol \nu}, {\boldsymbol \gamma}, {\boldsymbol u})$, such that, for all $(t,a)$ in $[t_0,T] \times A$, we have the exponential bounds
\begin{equation}
\label{eq:bound:nu(a):exp(-ell(a))}
 \Lambda^{-1}_{\bd{\nu},\gamma_0}
 e^{-2\ell(a)}  \leq \nu_t(a) \leq \Lambda_{\bd \nu,\gamma_0}  e^{-\ell(a)/2},
\end{equation}
as well as the time regularity
\begin{equation} 
\sup_{t \in [t_0,T]} \mathcal{E} \bigl(\nu_t | \nu^{\infty} \bigr) + \sup_{t_1 \neq t_2 \in [t_0,T]} \frac{\bigl\| e^{3 \ell/4}  ( \nu_{t_2} - \nu_{t_1} ) \bigr \|_{L^{\infty}}}{|t_2-t_1|} +  \sup_{t_1 \neq t_2 \in [t_0,T]} \frac{\mathcal{E} \bigl( \nu_{t_2}| \nu_{t_1} \bigr)}{|t_2- t_1|} \leq \Lambda_{\bd{\nu}, \gamma_0}, 
\label{eq:estimaterelativeentropy25/12}
\end{equation}
where we wrote for simplicity
$$ \Lambda_{\bd{\nu}, \gamma_0} :=\Lambda\biggl(
\int_{t_0}^T \int_A |a|^4 d\nu_t(a)dt + \int_{\R^{d_1} \times \R^{d_2}} (|x|+|y|)d\gamma_0(x,y) \biggr).$$ 
\end{prop}

\begin{rmk}
    The proofs leading to our main results will rely on some compactness arguments, requiring careful tracking of how  the quantitative estimates depend on the initial point $(t_0,\gamma_0)$ and on the control $\bd{\nu}$. This is the purpose of these additional functions $\Lambda$ appearing throughout the text. Since all the 
    estimates
    used in these compactness arguments will be uniform in 
    the parameter $t_0$, we consistently omit it in notations such as $\Lambda_{\boldsymbol{\nu}, \gamma_0}$.
\end{rmk}

We recall 
from \eqref{eq:boundfromPinsker16Sept:sec:3}
that 
$\int_{t_0}^T \int_A |a|^4 d\nu_t(a)dt$ can be bounded by
$C( 1+\int_{t_0}^T {\mathcal E}( \nu_t \vert 
\nu^\infty) dt)$.
Notice also that we have not put 
a \textit{star} on 
$({\boldsymbol \nu}, {\boldsymbol \gamma}, {\boldsymbol u})$ because, at this stage, 
${\boldsymbol \nu}$ may 
not be a minimizer of 
$J((t_0,\gamma_0),\cdot)$. We show in Lemma \ref{lem:7.4-04/03} that, 
when $\bd{\nu}^*$ is optimal is  for $J ( (t_0,\gamma_0), \cdot )$,   $\int_{t_0}^T \mathcal{E}(\nu^*_t|\nu^{\infty})dt$ can be bounded by $C(1+\int_{\R^{d_1} \times \R^{d_2}} (|x|^2 + |y|^2)d\gamma_0(x,y))$ for some $C>0$ independent of $(t_0,\gamma_0)$. Then, the non-decreasing function $\Lambda_{\bd{\nu}, \gamma_0}$ provided by Proposition \ref{prop:regularityfromOC} can be replaced by a non-decreasing function of $\int_{\R^{d_1} \times \R^{d_2}} (|x|^2 + |y|^2) d\gamma_0(x,y)$ only.

From Theorem \ref{thm:OCThm27/06}, we can also deduce that the measure $\nu_t^*$ satisfies,  for all $t$ in $[t_0,T]$, some functional inequalities. More generally, for any element $\bd{\nu}$ of $\mathcal{A}(t_0)$, let us define the  associated element $\Gamma[\bd\nu]$ in $\mathcal{A}(t_0)$ by, for all $t$ in $[t_0,T]$, 
\begin{equation}
\Gamma_t[\bd{\nu}] \propto \exp \Bigl( -\ell(a) - \frac{1}{\epsilon} \int_{\R^{d_1} \times \R^{d_2}} b(x,a) \cdot \nabla_x u_t(x,y)  d\gamma_t(x,y) \Bigr), 
\label{eq:defnGammaNut09/01}
\end{equation}
where $(\bd{\gamma},\bd{u})$ is the solution to 
\begin{equation}
\left \{
\begin{array}{ll}
\displaystyle -\partial_t u_t(x,y) -  b\bigl(x,\nu_t\bigr)  \cdot \nabla_x u_t(x,y) = 0,
&\quad (t,x,y)
\in [t_0,T] \times \R^{d_1} \times \R^{d_2},
\\
\displaystyle
\qquad 
 u_T(x,y) = L(x,y), &\quad (x,y) \in \R^{d_1} \times \R^{d_2};
\\
\displaystyle \partial_t \gamma_t
+ \div_x \bigl(  b(x,\nu_t) \gamma_t  \bigr) = 0, &\quad
 \textrm{\rm in}
 \ [t_0,T] \times  \R^{d_1} \times \R^{d_2}, 
\\
\displaystyle \qquad \gamma_{t_0} = \gamma_0
&\quad
 \textrm{\rm in} 
 \ 
\R^{d_1} \times \R^{d_2}.
\end{array}
\right.
\label{eq:curvemulti28/02}
\end{equation}
Because the two equations are completely decoupled, there is no difficulty in solving each of them, see Propositions \ref{lem:compactnessContinuityEquation24/06} and \ref{prop:SolutionBackxard16Sept}. The next result, proven in Section \ref{sec:additionalregandfirststability}, states that $\Gamma_t[\bd{\nu}]$ satisfies a log-Sobolev inequality. Notice that the result applies in particular to any solution of \eqref{eq:NOCfornut2Avril}--\eqref{eq:NOCforutgammat2Avril} since, in this case, $ \Gamma [\bd{\nu}]$ is equal to  $\bd{\nu}.$

\begin{lem}
Take $(t_0,\gamma_0) \in [0,T] \times \mathcal{P}_3(\R^{d_1} \times \R^{d_2})$
and $ \bd{\nu} \in \mathcal{A}(t_0)$. Then, there exists a constant
$C_{\rm LSI}(\Gamma[\bd \nu]) >0 $ such that, for all $t \in [t_0,T]$ and every smooth enough $f : A \rightarrow \R^+$ with $\int_A f(a) d \Gamma_t[\bd\nu](a) =1$, it holds
\begin{equation} 
\int_A f(a) \log \bigl(f(a)\bigr)d \Gamma_t[ \bd\nu](a) \leq C_{\rm LSI}(\Gamma[\bd \nu]) \int_A | \nabla_a \log f(a)|^2 f(a) d\Gamma_t[ \bd\nu](a).
\label{eq:defLSI}
\end{equation}
Moreover, there exists a non-decreasing function $\Lambda : \R_+ \rightarrow \R_+$ independent from $(t_0,\gamma_0)$ and $\bd{\nu}$ such that 
$$ C_{\rm LSI} ( \Gamma[ \bd \nu]) \leq \Lambda \Bigl( \int_{\R^{d_1} \times \R^{d_2}} \bigl( |x|^2 +|y|^2 \bigr)^{3/2} d \gamma_0(x,y) + \int_{t_0}^T \int_A |a|^4 d\nu_t(a) dt \Bigr).  $$
\label{lem:LSI}
\end{lem}

\begin{rmk}
\label{rmk:thirdmoments}
It is important in the statement of Lemma \ref{lem:LSI} that the initial measure has a bounded third-order moment. In short, our strategy to obtain the log-Sobolev inequality relies on a perturbation argument, deriving it from the inequality satisfied by the reference measure \(\nu^\infty \propto e^{-\ell}\). Due to the linear growth of \(b(x,a)\) in \(a\), simpler proofs based on the Holley-Stroock lemma do not apply.
In particular, we need to control the Hessian of the mapping
$a \mapsto \int_{\mathbb{R}^{d_1} \times \mathbb{R}^{d_2}} b(x,a) \cdot \nabla_x u_t(x,y) d\gamma_t(x,y)$
which ultimately requires ensuring that the term
$\int_{\mathbb{R}^{d_1} \times \mathbb{R}^{d_2}} (1 + |x|^2) |\nabla_x L(x,y)| \, d\gamma_0(x,y)$
is bounded.
Because the terminal cost \(L\) is allowed to have quadratic growth, we thus require \(\gamma_0\) to belong to \(\mathcal{P}_3(\mathbb{R}^{d_1} \times \mathbb{R}^{d_2})\).
\end{rmk}

In the same vein, we use the exponential bound \eqref{eq:bound:nu(a):exp(-ell(a))} to prove, in Section \ref{sec:additionalregandfirststability}, the following inequality, which relies on a generalization of Pinsker's inequality, see \cite{BolleyVillani}.

\begin{lem}
    Take $(t_0,\gamma_0) \in [0,T] \times \mathcal{P}_2(\R^{d_1} \times \R^{d_2})$. Then, there exists a non-decreasing function $\Lambda : \R_+ \rightarrow \R_+$ such that, for any minimizer  $\bd{\nu}^*$ of $J \bigl( (t_0,\gamma_0), \cdot \bigr)$ and any $\bd{\nu} \in \mathcal{A}(t_0)$,
    $$ \norm{ \bd{\nu} - \bd{\nu}^*}_{\mathcal{D}(t_0)} \leq \Lambda_{\gamma_0}  \biggl( \sqrt{ \int_{t_0}^T \mathcal{E}(\nu_t|\nu_t^*) dt} + \int_{t_0}^T \mathcal{E}(\nu_t|\nu_t^*) dt \biggr),$$
where $\Lambda_{\gamma_0}$ is short-hand notation for 
$$  \Lambda_{\gamma_0} = \Lambda \Bigl( \int_{\R^{d_1} \times \R^{d_2}} (|x|^2 + |y|^2) d\gamma_0(x,y) \Bigr).$$
\label{lem:PinskerSecFOC28/02}
\end{lem}

Another consequence of Theorem \ref{thm:OCThm27/06} is that we can derive some stability properties for the minimal value in the control problem as a function of the initial point $(t_0,\gamma_0) \in [0,T] \times \mathcal{P}_2(\R^{d_1} \times \R^{d_2})$. If we define the value function $U: [0,T] \times \mathcal{P}_2(\R^{d_1} \times \R^{d_2}) \rightarrow \R$ 
by
\begin{equation}
\label{eq:value function:U}
U(t_0,\gamma_0) := \inf_{ \bd{\nu} \in \mathcal{A}(t_0) } J \bigl( (t_0,\gamma_0), \bd{\nu} \bigr),
\end{equation}
then we have the following result, also proven in Section \ref{sec:additionalregandfirststability}.
\begin{prop}
    $U$ is locally Lipschitz continuous over $[0,T] \times \mathcal{P}_2(\R^{d_1} \times \R^{d_2})$. More precisely, there is a non decreasing function $\Lambda : \R_+ \rightarrow \R_+$ such that, for any $(t_1,\gamma_1), (t_2,\gamma_2) \in [0,T] \times \mathcal{P}_2(\R^{d_1} \times \R^{d_2})$, it holds
    $$ \bigl| U(t_2,\gamma_2) - U(t_1,\gamma_1) \bigr| \leq \sup_{i=1,2} \Lambda \biggl( \int_{\R^{d_1} \times \R^{d_2}} ( |x|^2 + |y|^2 ) d\gamma_{i}(x,y) \biggr) \bigl( |t_2- t_1| + d_2( \gamma_1,\gamma_2) \bigr).$$
\label{prop:LipRegValueFuncion25Nov}
\end{prop} 
The estimates from Proposition \ref{prop:regularityfromOC} together with Proposition \ref{prop:LipRegValueFuncion25Nov}  also make it 
possible to obtain the following stability 
result by means of a compactness argument
(whose proof is given in Subsection 
\ref{sec:additionalregandfirststability}). 

\begin{lem}
\label{lem:convergenceifuniqueness}
 Take $(t_0,\gamma_0) \in [0,T] \times \mathcal{P}_{3}(\R^{d_1} \times \R^{d_2})$ such that
    $J((t_0, \gamma_0),\cdot)$
    has a unique minimizer 
    $\boldsymbol \nu^*$. Denote by $\boldsymbol \gamma^*$ and $\boldsymbol u^*$ the corresponding curve and multiplier. Assume that $(t^n_0, \gamma^n_0)_{ n\geq 1}$ 
    converges in $[0,T] \times \mathcal{P}_2(\R^{d_1} \times \R^{d_2})$ toward $(t_0,\gamma_0)$ and take, for all $n \geq 1$, a minimizer $\boldsymbol \nu^{*,n}$ of $J((t_0^n , \gamma_0^n),\cdot )$ with corresponding curve and multiplier $\boldsymbol \gamma^{*,n}$ and $\boldsymbol u^{*,n}$. Then $(\boldsymbol \nu^{*,n}, \boldsymbol \gamma^{*,n},\boldsymbol u^{*,n})_{ n \geq 1}$ converges toward $(\boldsymbol \nu^{*}, \boldsymbol \gamma^*, \boldsymbol u^*)$ in the following sense (with the standard notation $t \vee t'$  for $\max(t,t')$):
$$\lim_{n \rightarrow \infty} \biggl[  \sup_{t \in [t^n_0 \vee t_0,T]} \Bigl \{ \norm{ e^{\ell/2}(\nu^{*,n}_t - \nu^{*}_t) }_{L^{\infty}} +\mathcal{E} \bigl (\nu_t^{*,n} | \nu^*_t \bigr) + \bigl\| u_t^{*,n} - u^{*}_t \bigr\|_{\mathcal{C}_1^2}  +  d_2\bigl( \gamma_t^{*,n}, \gamma^*_t\bigr) \Bigr \} \biggr]  =0.$$
\end{lem}

\subsection{Lagrangian Representation}
\label{subse:3:2}
We now introduce an alternative formulation of the optimality conditions
stated in Theorem 
\ref{thm:OCThm27/06}, but in terms of 
a forward-backward system of two Ordinary Differential Equations (ODEs). This representation is easily derived from Proposition  
\ref{lem:compactnessContinuityEquation24/06}
and Lemma \ref{lem:ApproxBODE21Sept}:
\begin{prop}
\label{prop:FBODErepresentation}
Let $(t_0,\gamma_0) \in [0,T] \times {\mathcal P}_2({\mathbb R}^{d_1} \times {\mathbb R}^{d_2})$
and   consider also a probability space $(\Omega, \mathcal{F}, \mathbb{P})$ equipped with a pair  $(X_0, Y_0)$
    of random variables with values in 
    ${\mathbb R}^{d_1} \times {\mathbb R}^{d_2}$
    such that ${\mathbb P} \circ (X_0,Y_0)^{-1}= \gamma_0$. Then, 
    any triple 
    $({\boldsymbol \nu}, {\boldsymbol \gamma}, {\boldsymbol u})$ solving the system \eqref{eq:NOCfornut2Avril}--\eqref{eq:NOCforutgammat2Avril}
can be represented by means of 
the forward-backward system of ODEs
\begin{equation}
\label{eq:OC-ODEformulation}
    \left \{
    \begin{array}{ll}
    \dot{X}_t = b(X_t,\nu_t), \ &t \in [t_0,T];
 \qquad X_{t_0} =X_0;
    \\
    \dot{Z}_t = - \nabla_x b(X_t,\nu_t) Z_t, \  &t \in [t_0,T]; \qquad Z_T = \nabla_x L(X_T, Y_0);
    \end{array}
    \right.
\end{equation}
with $\E$ denoting the expectation under $\mathbb{P}$, in the sense that  
\begin{equation} 
\forall (t,a) \in [t_0,T] \times A, \quad
\nu_t(a) = \frac{1}{z_t} 
\exp \Bigl(  -\ell(a) - \frac{1}{\epsilon} \E \bigl[ b(X_t,a) \cdot Z_t \bigr] \Bigr), 
\label{eq:OC-ODEtildenu}
\end{equation}
and
\begin{equation}
\label{eq:representation:FB:F}
\begin{split}
&\forall t \in [t_0,T],
\quad
{\gamma}_t = {\mathbb P} \circ (X_t,Y_0)^{-1}, 
\\
&{\mathbb P}\bigl(\bigl\{ \forall t \in [t_0,T], \quad Z_t = \nabla_x{u}_t(X_t,Y_0) \bigr\} \bigr) = 1. 
\end{split}
\end{equation}
\end{prop}

Notice that \eqref{eq:OC-ODEtildenu}--\eqref{eq:OC-ODEformulation} read as the optimality conditions when the cost functional is formulated over controlled ODEs instead of 
controlled continuity equations. In this case, the problem \eqref{eq:OriginalControlPb24/06} is replaced by
$$ \inf_{\nu \in \mathcal{A}(t_0)} 
\biggl\{ \E \bigl[L(X_T,Y_0) \bigr] + \epsilon \int_{t_0}^T \mathcal{E}(\nu_t|\nu^{\infty})dt \Bigr \}, $$
where $X_t$ solves the ODE
$$\dot{X}_t = b(X_t,\nu_t), \quad t \in [t_0,T]; \quad {\mathbb P} \circ (X_0,Y_0)^{-1} = \gamma_0.$$

\subsection{Second-Order Optimality Conditions}

\label{subse:SOC04/03}

We now address second-order optimality conditions, the analysis of which requires an appropriate set of control perturbations:

\begin{defn}
\label{def:admissible:perturbation}
For an 
initial time $t_0 \in [0,T]$, we let $\mathcal{A}^{\ell}(t_0)$ (the superscript $\ell$ standing for \textit{linearized}) be the subspace of $\mathcal{D}(t_0)$ consisting of elements $\bd{\eta}$ such that, for all $t \in [t_0,T]$, $\eta_t(A) = 0$.
\end{defn}

The key idea in the analysis of the second-order conditions is to associate, with any optimal control, a linearized problem. For an initial condition $(t_0,\gamma_0) \in [0,T] \times {\mathcal P}_3({\mathbb R}^{d_1} \times {\mathbb R}^{d_2})$, an element $\bd{\nu}$ of $\mathcal{D}(t_0)$  with $(\boldsymbol \gamma, \boldsymbol u)$ as associated curve and multiplier, and another element  ${\boldsymbol \eta} \in {\mathcal D}(t_0)$, we consider the advection equation 
\begin{equation}
\left \{
\begin{array}{ll}
\displaystyle \partial_t \rho_t  + \div_x \bigl(  b(x,\nu_t) \rho_t \bigr) =- \div_x \bigl(  b(x,\eta_t) \gamma_t \bigr) &\quad \mbox{in } (t_0,T) \times \R^{d_1} \times\R^{d_2},
\\
\displaystyle \rho_{t_0}=0, &\quad \mbox{in } \R^{d_1} \times \R^{d_2}.
\end{array}
\right.
\label{eq:rhotProp3.2}
\end{equation}
Equation \eqref{eq:rhotProp3.2} is obtained by differentiating with respect to $\lambda$ the solution $\bd{\gamma}^{\lambda}$ to the continuity equation \eqref{eq:Continuitygammanu} with control $\bd{\nu} + \lambda \bd{\eta}$.
To make sense of equation \eqref{eq:rhotProp3.2}, we will use test functions $\varphi : [t_0,T] \times \R^{d_1} \times \R^{d_2} \rightarrow \R$ satisfying
\begin{align} 
&t \mapsto \varphi_t \in \mathcal{C}^2_2 \mbox{ is bounded and } r \mapsto \varphi_t \in \mathcal{C}^1_2 \mbox{ is continuous},
\label{eq:testfn1-09/03}
\\ 
&t\mapsto \partial_t \varphi_t , \nabla_x \varphi_t \in \mathcal{C}^2_1 \mbox{ are bounded and } t \mapsto \partial_t \varphi_t, \nabla_x \varphi_t \in \mathcal{C}^1_2 \mbox{ are continuous. }
\label{eq:testfn2-09/03}
\end{align} 
We show in Lemma \ref{lem:BackThenLemA.17} that, for any such test function $\bd{\varphi}$ and any $\bd{\rho} \in \mathcal{C}([t_0,T], (\mathcal{C}^2_2)^*)  $ satisfying $\sup_{t \in [t_0,T]} \sup_{\varphi \in \mathcal{C}^2_2, \norm{ \varphi}_{\mathcal{C}^1_2} \leq 1} \langle \varphi; \rho_t \rangle <+\infty$, the maps $t \mapsto \langle \varphi_t ; \rho_t \rangle$ and $t \mapsto \langle \partial_t \varphi_t ; \rho_t \rangle$ are continuous. Moreover, by Proposition \ref{prop:measurable+fubini15/01}, the map $ t\mapsto \langle b(\cdot,\nu_t) \cdot \nabla_x \varphi_t ; \rho_t \rangle$ is integrable and coincides with $t \mapsto \int_A \langle b(\cdot,a) \cdot \nabla_x \varphi_t; \rho_t \rangle d\nu_t(a) $. 

Then, we say that 
\begin{equation}
\bd{\rho} \in \mathcal{C}([t_0,T], (\mathcal{C}^2_2)^*) \quad \mbox{ with } \quad \sup_{t \in [t_0,T]} \sup_{\varphi \in \mathcal{C}^2_2, \norm{ \varphi}_{\mathcal{C}^1_2} \leq 1} \langle \varphi; \rho_t \rangle < +\infty, 
\label{eq:spaceforrho12/01}
\end{equation}
is solution to \eqref{eq:rhotProp3.2} if, for every test function $\varphi : [t_0,T] \times \R^{d_1} \times \R^{d_2} \rightarrow \R$ satisfying \eqref{eq:testfn1-09/03}--\eqref{eq:testfn2-09/03}, it holds, for any $t_1 \in [t_0,T]$, 
\begin{equation} 
\langle \varphi_{t_1}; \rho_{t_1} \rangle = \int_{t_0}^{t_1} \langle \partial_t\varphi_t + b(\cdot, \nu_t) \cdot \nabla_x \varphi_t; \rho_t \rangle dt + \int_{t_0}^{t_1} \int_{\R^{d_1} \times \R^{d_2}} b(x,\eta_t) \cdot \nabla_x \varphi_t(x,y)d \gamma_t(x,y) dt. 
\label{eq:weakeqrho12/01Section2}
\end{equation}
Solvability of this equation is studied in the smaller space
${\mathcal R}(t_0)$ defined below:
\begin{defn}
\label{def:R(t_0)}
    For $t_0 \in [0,T]$, we define $\mathcal{R}(t_0)$ as the subset of $ \mathcal{C}([t_0,T], ( \mathcal{C}^2_2)^*)$ consisting of elements $\bd{\rho}$ such that 
 $$ \norm{ \bd{\rho}}_{\mathcal{R}(t_0)} := \sup_{t \in [t_0,T]}\sup_{ \phi \in \mathcal{C}^2_2, \norm{\phi}_{\mathcal{C}^1_3} \leq 1} | \langle \phi, \rho_t \rangle | + \sup_{t_1 < t_2 \in [t_0,T]} \frac{\norm{ \rho_{t_2} - \rho_{t_1}}_{(\mathcal{C}^2_2)^*}}{\sqrt{t_2-t_1}} < +\infty.$$   
\end{defn}    

The following result is proven in Section \ref{sec:LinearizedEquations}.

\begin{prop}
\label{prop:differentiatinggammalambda09/02}
Let $(t_0, \gamma_0) \in [0,T] \times \mathcal{P}_3(\R^{d_1} \times \R^{d_2})$ and $\bd{\nu}, \bd{\eta} \in \mathcal{D}(t_0)$. Let $\bd{\gamma}$ be the solution to the continuity equation \eqref{eq:lem:compactnessContinuityEquation24/06:statement} starting from $(t_0,\gamma_0)$ with control $\bd{\nu}$. Then, there is a unique solution $\bd{\rho} \in \mathcal{R}(t_0)$ to the linearized equation \eqref{eq:rhotProp3.2}. It is given by 
\begin{equation} 
\frac{d}{d\lambda} \Big|_{\lambda =0} \bd{\gamma}^{\lambda} = \bd{\rho} \quad \mbox{\rm in } \mathcal{C}([t_0,T], (\mathcal{C}^2_2)^*), 
\label{eq:derivative19/02}
\end{equation}
where $\bd{\gamma}^{\lambda}$ is the solution to the continuity equation starting from $(t_0,\gamma_0)$ with control $\bd{\nu} + \lambda \bd{\eta}$. It satisfies the estimate 
\begin{equation}
\label{eq:bound:for:rho:useful}
\norm{\bd{\rho}}_{\mathcal{R}(t_0)} \leq  \Lambda \biggl( \norm{ \bd{\nu}}_{\mathcal{D}(t_0)}  + \int_{\R^{d_1} \times \R^{d_2}} (|x|^2 + |y|^2)^{3/2} d\gamma_0(x,y) \biggr) \norm{\bd\eta}_{\mathcal{D}(t_0)}, 
\end{equation}
for some non-decreasing function $\Lambda : \R^+ \rightarrow \R^+$ independent from $(t_0,\gamma_0, \bd{\nu}, \bd{\eta})$. Moreover, for all $t_1 \in [t_0,T]$, $\rho_{t_1}$ extends uniquely to $\mathcal{C}^1_2$ and it is given, for any $\phi \in \mathcal{C}^1_2$ by 
\begin{equation} 
\langle \phi; \rho_{t_1} \rangle = \frac{d}{d\lambda} \Big|_{\lambda = 0} \int_{\R^{d_1} \times \R^{d_2}} \phi(x,y) d\gamma_{t_1}^{\lambda}(x,y) =  \int_{t_0}^{t_1} \int_{\R^{d_1} \times \R^{d_2}} b(x,\eta_t) \cdot \nabla_x \varphi_t(x,y) d\gamma_t(x,y) dt, 
\label{eq:derivativeinlambda31/01}
\end{equation}
where $\varphi : [t_0,t_1] \times \R^{d_1} \times \R^{d_2} \rightarrow \R$ is the solution to 
$$ - \partial_t \varphi_t - b(x, \nu_t) \cdot \nabla_x \varphi_t = 0 \quad \mbox{\rm in } [t_0,t_1] \times \R^{d_1} \times \R^{d_2}; \quad \varphi_{t_1} = \phi \quad \mbox{\rm in } \R^{d_1} \times \R^{d_2}.$$
\end{prop}

\begin{rmk}
In the proposition above, the meaning of condition \eqref{eq:derivative19/02} is
$$ \lim_{\lambda \rightarrow 0} \sup_{t \in [t_0,T]} \norm{ \frac{\gamma_t^{\lambda} - \gamma_t}{\lambda} - \rho_t }_{(\mathcal{C}^2_2)^*} =0.  $$
\end{rmk}

If $\bd{\nu}^*$ is an optimal control for $J ( (t_0,\gamma_0), \cdot )$ and $\bd{\eta}$ is a perturbation in $\mathcal{A}^{\ell}(t_0)$, we can now define the  
auxiliary cost functional 
\begin{align} 
\notag &\mathcal{J} \bigl( (t_0,\gamma_0), \bd{\nu}^* , \bd{\eta} \bigr)
\\
 &:=
\left \{
\begin{array}{ll}
\displaystyle \epsilon \int_{t_0}^T  \int_{A} \frac{|\eta_t(a)|^2}{\nu^*_t(a)} da  dt 
+ 2 \int_{t_0}^T  \Bigl \langle b( \cdot,\eta_t) \cdot \nabla_x u^*_t ; \rho_t \Bigr \rangle dt & \mbox{if }  \bd{\eta} \in L^2 \bigl( (\boldsymbol{\nu}^*)^{-1} \bigr),
\\
+ \infty & \mbox{otherwise, }
\label{eq:minformathcalJ1Avril}
\end{array}
\right.
\end{align}
where, in the first line, 
$\bd \rho$
solves
\eqref{eq:rhotProp3.2}
and 
 the condition 
 $\bd{\eta} \in L^2  ( (\boldsymbol{\nu}^*)^{-1})$ means that $\boldsymbol{\eta}$ is absolutely continuous with respect to the Lebesgue measure on  $[t_0,T] \times A$ with the Radon-Nikodym derivative
$(t,a) \mapsto \eta_t(a)$ belonging to $L^2((\boldsymbol{ \nu}^*)^{-1})$. Notice that, for any $\bd{\eta} \in \mathcal{A}^\ell(t_0)$, the second term in the  right-hand side of \eqref{eq:minformathcalJ1Avril} is finite, see Proposition \ref{prop:measurable+fubini15/01}, and therefore the total cost $\mathcal{J}$ is defined without ambiguity with values in  $\R \cup \{+\infty \}$.
The relation of the functional  $\mathcal{J}$ to the control problem will become apparent in Section \ref{sec:SOC}. In brief, for any perturbation $\boldsymbol{\eta} \in \mathcal{A}^\ell(t_0)$ such that $\bd{\nu} + \lambda \bd{\eta}^*$ belongs to $\mathcal{A}(t_0)$ for $\lambda$ small enough, we have
$$ \frac{d^2}{d\lambda^2} \Big|_{\lambda = 0} J \bigl( (t_0 , \gamma_{0}), \boldsymbol{\nu}^* + \lambda \bd\eta \bigr) =  \mathcal{J} \bigl( (t_0, \gamma_0), \boldsymbol{\nu}^*, \boldsymbol{\eta} \bigr).$$
Optimality conditions for minimizers of \eqref{eq:minformathcalJ1Avril} will involve an adjoint variable,  solution to the linearized backward transport equation
\begin{equation}
    \left \{
    \begin{array}{ll}
    -\partial_t v_t(x,y) - b(x, \nu_t^*) \cdot \nabla_x v_t(x,y) = b(x,\eta_t) \cdot \nabla_x u_t^*(x,y) \quad \mbox{ in } [t_0,T] \times \R^{d_1} \times \R^{d_2}, \\
    v_T(x,y) = 0 \quad \mbox{ in } \R^{d_1} \times \R^{d_2}.
    \end{array}
    \right.
\label{eq:linearizedeqvt04/03}
\end{equation}
We say that  $v \in \mathcal{C}([t_0,T], \mathcal{C}^1_{1})$ is a solution to \eqref{eq:linearizedeqvt04/03} if, for all $t_1 < t_2 \in [t_0,T]$ and all $(x,y) \in \R^{d_1} \times \R^{d_2}$,
$$ v_{t_1}(x,y) - v_{t_2}(x,y) = \int_{t_1}^{t_2} \bigl \{ b(x, \nu^*_t) \cdot \nabla_xv_t(x,y) + b(x, \eta_t) \cdot \nabla_x u^*_t(x,y)\bigr \}dt.  $$
Section \ref{sec:LinTransEq} is devoted to the analysis of Equation \eqref{eq:linearizedeqvt04/03}, where we state in particular the next result:
\begin{prop}
Take $t_0 \in [0,T]$, $\bd{\nu}^* \in \mathcal{A}(t_0)$ with associated solution $\bd{u}^*$ to the backward transport equation \eqref{eq:adjoint:equation}. Take  $\bd{\eta} \in \mathcal{D}(t_0)$.  Then, there is a unique solution to the linearized transport equation \eqref{eq:linearizedeqvt04/03}. For all $t \in [t_0,T]$, $v_t$ belongs to $\mathcal{C}^2$ and we have the estimate
    $$ \sup_{t \in [t_0,T]} \norm{ v_t }_{\mathcal{C}^{2}_{1}} + \sup_{t_1 < t_2 \in [t_0,T]}  \frac{\norm{v_{t_2} - v_{t_1} }_{\mathcal{C}^1_{1}}}{\sqrt{t_2-t_1}} \leq \Lambda \bigl(  \norm{\bd{\nu}^*}_{\mathcal{D}(t_0)} \bigr)  \norm{ \bd{\eta}}_{\mathcal{D}(t_0)},   $$
for some non-decreasing function $\Lambda : \R_+ \rightarrow \R_+$ independent from $(t_0, \bd{\nu}^*, \bd{\eta})$. Moreover, for all $t \in [t_0,T] $ and all $(x,y) \in \R^{d_1} \times \R^{d_2}$, the solution is given by 
\begin{equation} 
v_t(x,y) = \int_{t}^T b \bigl( X_s^{*,t,x},\eta_s) \cdot \nabla_x u_s \bigl( X_s^{*,t,x},y \bigr) ds, 
\label{eq:representationformulavt}
\end{equation}
where $(X_s^{t,x})_{s \in [t,T]}$ is the flow of the ODE, solution to 
$$ \dot{X}^{*,t,x}_s = b(X_s^{*,t,x},\nu^*_s) \quad s \in  [t,T], \quad X_t^{*,t,x} = x. $$
\label{prop:SolnVt19/02}
\end{prop}

Within this framework, the second-order optimality conditions, whose proof is given in Section \ref{sec:SOC}, take the following form:
\begin{thm} Given 
$(t_0, \gamma_0) \in [0,T] \times \mathcal{P}_3(\R^{d_1} \times \R^{d_2})$, 
let $(\boldsymbol \nu^*, \boldsymbol \gamma^*, \boldsymbol u^*)$ be a minimizer of $J ( (t_0,\gamma_0),\cdot )$ (completed with its optimal curve and 
    multiplier) and $t_1 \in [t_0,T)$. Then, for any  ${\boldsymbol \eta} \in \mathcal{A}^\ell(t_1)$ with associated solution ${\boldsymbol \rho} \in \mathcal{R}(t_1)$ to the linearized continuity equation  \eqref{eq:rhotProp3.2}, 
$$ \mathcal{J} \bigl( (t_1, \gamma_{t_1}), \bd{\nu}^*,{\boldsymbol \eta} \bigr) \geq 0.$$
Moreover,  $\mathcal{J} ( (t_1, \gamma_{t_1}) , \boldsymbol \nu^*,{\boldsymbol \eta} ) =0$ if and only if there exists ${\boldsymbol v} \in \mathcal{C}([t_1,T], \mathcal{C}^1_{1})$ such that 
\begin{equation}
    {\eta}_t(a) = -   \frac{\nu^*_t(a)}{\epsilon} \Bigl[  \bigl \langle  b(\cdot,a)  \cdot \nabla_x u^*_t ;  \rho_t \bigr \rangle + \int_{\R^{d_1} \times \R^{d_2}} b(x,a) \cdot \nabla_x v_t (x,y) d \gamma^*_t  (x,y) -c_t\Bigr] 
\label{eq:nutProp3.2}
\end{equation}
for almost all 
$(t,a) \in [t_1,T] \times A$, where $c_t$ is a normalizing constant to ensure that $\eta_t$ integrates to $0$, 
and $({\boldsymbol v},{\boldsymbol \rho})$ solves the linearized system
\begin{equation}
    \left \{
    \begin{array}{ll}
\displaystyle -\partial_t  v_t  -  b(x,\nu^*_t) \cdot \nabla_x v_t =  b(x,\eta_t)  \cdot \nabla_x u^*_t &\textrm{\rm in } [t_1,T] \times \R^{d_1} \times 
\R^{d_2}, 
\\
\qquad v_T = 0 \quad 
\textrm{\rm in } \R^{d_1} \times 
\R^{d_2};
\\
\displaystyle \partial_t  \rho_t  + \div_x \bigl(  b(x,\nu^*_t)  \rho_t \bigr) = - \div_x \bigl(  b(x,\eta_t )  \gamma^*_t \bigr) &\textrm{\rm in } (t_1,T) \times \R^{d_1} \times \R^{d_2},
\\
\qquad \rho_{t_1} =0.
    \end{array}
    \right.
\label{eq:vtrhotProp3.2}
\end{equation}
\label{prop:SecondOrderConditions3Avril}
\end{thm}

\begin{rmk}

(i) Notice that ${\boldsymbol \eta}$ solves
\begin{equation}
\label{rmk:formulact}
\begin{split}
-\epsilon \Delta_a &\eta_t - \div_a \Bigl[ \Bigl( \epsilon \nabla_a \ell + \nabla_a \int_{\R^{d_1} \times \R^{d_2}} b(x,a) \cdot \nabla_x u^*_t(x,y) d \gamma^*_t(x,y) \Bigr)  \eta_t  \Bigr] 
    \\
&= \div_a \Bigl[ \Bigl( \nabla_a \Bigl[ \bigl \langle b(\cdot,a) \cdot \nabla_x u_t^*; \rho_t \bigr \rangle + \int_{\R^{d_1} \times \R^{d_2}} b(x,a) \cdot  \nabla_x v_t(x,y) d \gamma^*_t (x,y) \Bigr]  \Bigr) \nu^*_t \Bigr], 
\end{split}
\end{equation}
in $A$ for all $t \in [t_0,T]$, which is the linearized version of \eqref{eq:stationnaryFPe24Sept}.
Indeed, computing the derivative of the log and then inserting the expression for
$\eta_t$, one has 
\begin{align*} 
&\epsilon \nabla_a \log \frac{\eta_t}{\nu^*_t}(a) = - \frac{\nu^*_t(a)}{\eta_t(a)} \nabla_a \Bigl[ \bigl \langle b(\cdot,a) \cdot \nabla_x u_t^*;\rho_t \bigr \rangle +  \int_{\R^{d_1} \times \R^{d_2}} b(x,a) \cdot \nabla_x v_t(x,y) d\gamma^*_t(x,y) \Bigr]. 
\end{align*}
Moreover, using the expression for $\nu_t^*$, one can rewrite the left-hand side in the form
\begin{align*} 
& \epsilon \nabla_a \log \frac{\eta_t}{\nu^*_t}(a) 
=\epsilon \nabla_a \log \eta_t(a) - \epsilon \nabla_a \log \nu^*_t(a) 
\\
&= \epsilon \nabla_a \log \eta_t(a) + \epsilon \nabla_a \ell(a) + \nabla_a \int_{\R^{d_1} \times \R^{d_2}} b(x,a) \cdot \nabla_x u^*_t(x,y) d\gamma^*_t(x,y). 
\end{align*}
Identifying the right-hand sides of the last two displays, 
multiplying both sides by $\eta_t$ and taking the divergence in $a$, we conclude that $\bd{\eta}$ solves \eqref{rmk:formulact}.

(ii) In Lemma \ref{lem:differentiabilitydualitybracket31/01} we prove that, for $t \in [t_0,T]$,  $a \mapsto \langle b(\cdot,a) \cdot \nabla_x u_t^*; \rho_t \rangle $ is continuously differentiable and its gradient with respect to $a$  is given by $\langle \nabla_a[ b(\cdot,a)\cdot \nabla_x u_t^*]; \rho_t \rangle$.

(iii) We also notice that the constant $c_t$ is given by
    $$c_t = \bigl \langle  b(\cdot,\nu^*_t)  \cdot \nabla_x u^*_t ;  \rho_t \bigr \rangle + \int_{\R^{d_1} \times \R^{d_2}} b(x,\nu^*_t) \cdot \nabla_x v_t (x,y) d \gamma^*_t  (x,y).$$
\end{rmk}

Similarly to Proposition \ref{prop:regularityfromOC}, we can use the second order conditions to infer more regularity on $\boldsymbol{(\eta, \rho, v)}$. The proof is also given in Section \ref{sec:SOC}.

\begin{prop}
\label{prop:reg:v,rho}
   Assume that $(\bd{\eta}, \bd{v}, \bd{\rho})$ is a solution to \eqref{eq:nutProp3.2}--\eqref{eq:vtrhotProp3.2} around a solution $(\bd{\nu}, \bd{\gamma},\bd{u})$ of \eqref{eq:NOCfornut2Avril}-\eqref{eq:NOCforutgammat2Avril}. Then, $t \mapsto \eta_t \in \mathcal{M}_{1+|a|^3}(A)$ is continuous and $\bd{v}$ and $\nabla_x \bd{v}$ are jointly (in $(t,x,y)$) $\mathcal{C}^1$.
\end{prop}

\subsection{Stability conditions}

\label{sec:subseStableSol}
Based on these second-order conditions, it is natural to introduce the notion of stable solutions characterized by the unique solvability of the linearized system. The interpretation is that, for a stable solution $\boldsymbol \nu^*$ for $J  ( (t_0,\gamma_0), \cdot  )$, we have 
$$ \frac{d^2}{d\lambda^2}\Big|_{\lambda =0} J \bigl ( (t_0,\gamma_0), \bd\nu^* + \lambda \boldsymbol{\eta} \bigr) >0$$
for any non trivial (i.e., $\boldsymbol{\eta} \neq 0$) perturbation $\boldsymbol{\eta} \in \mathcal{A}^\ell(t_0)$ such that $\bd{\nu}^* + \lambda \bd{\eta}$ belongs to $\mathcal{A}(t_0)$ for small enough $\lambda$.

\begin{defn}
    For $\gamma_0 \in \mathcal{P}_3(\R^{d_1} \times \R^{d_2})$ and $t_0 \in [0,T)$, 
    we say that $\boldsymbol \nu^* = (\nu^*_t)_{t_0 \leq t \leq T}$ is a stable minimizer of $J((t_0,\gamma_0),\cdot)$ if $(0,0,0)$ is the only solution to the linearized system \eqref{eq:nutProp3.2}--\eqref{eq:vtrhotProp3.2} (starting from $t_0$, around the solution $(\bd{\nu}^*, \bd{\gamma}^*, \bd{u}^*)$ solution to \eqref{eq:NOCfornut2Avril}--\eqref{eq:NOCforutgammat2Avril}), with uniqueness being understood 
    among the triples $({\boldsymbol \eta},{\boldsymbol \rho},{\boldsymbol v})$
    in 
    $\mathcal{A}^\ell(t_0) \times \mathcal{R}(t_0) \times \mathcal{C}([t_0,T] , \mathcal{C}^1_{1})$.
\label{defn:stablesolutions}
\end{defn}

We can now define properly the set 
${\mathcal O}$ introduced in Subsection 
\ref{subse:2.4}, on which all our main
results are constructed. In words,
${\mathcal O}$
is the set of pairs $(t_0,\gamma_0)$ 
of initial time $t_0$ and initial distribution of  features 
$\gamma_0$ for which there is a unique stable optimal solution, namely
\begin{equation} 
\mathcal{O} := \Bigl \{ (t_0,\gamma_0) \in [0,T] \times \mathcal{P}_{3}(\R^{d_1} \times \R^{d_2});  \mbox{  there is a unique stable minimum for } J\bigl((t_0,\gamma_0),\cdot\bigr) \Bigr \}.
\label{eq:elfamososetO}
\end{equation}
We emphasise that by `unique stable', we mean that $J((t_0,\gamma_0),\cdot)$ admits a unique solution and that this solution is stable. At this stage, it is not clear that the set $\mathcal{O}$ is not empty (which property will be proven in Section \ref{sec:TheUniversalAppox}).

\section{Perturbation Analysis around Minimizers}

\label{sec:preparatorywork}

The goal of this Section is to provide quantitative and qualitative stability  properties of the cost $J \bigl( (t_0,\gamma_0), \bd{\nu})$ as well as the couple $(\bd{\gamma}, \bd{u})$ solution to \eqref{eq:curvemulti28/02} and the measure $\Gamma[\bd{\nu}]$ solution to \eqref{eq:defnGammaNut09/01}  with respect to perturbations of the control $\bd{\nu}$ around  an optimal control $\bd{\nu}^*$. The results will be necessary to  prove that the set $\mathcal{O}$ is open in Section \ref{sec:TheUniversalAppox} and to prove the PL inequality in Section \ref{se:5}. The technical proof of the compactness argument of Proposition \ref{prop:preliminaryworkprop1-28/02} and the subsequential convergence statements of Propositions \ref{prop:convergencektn03/03} and  \ref{prop:cvgcesumofIs03/03} are postponed to Section \ref{sec:ProofsofSection3}.

Throughout this section we consider the following situation. We take $(t_0,\gamma_0) \in [0,T] \times \mathcal{P}_3(\R^{d_1} \times \R^{d_2})$ and $\bd{\nu}^*$ an optimal control for $J \bigl( (t_0,\gamma_0), \cdot \bigr)$ with associated optimal trajectory and multiplier $ (\bd{\gamma}^*, \bd{u}^*)$. We also consider a sequence of tuples $(t_0^n,\gamma_0^n, \bd{\nu}^{*,n}, \bd{\nu}^n)_{n \in \mathbb{N}}$ satisfying the following properties.

\noindent {\bf Property (${\mathcal Q_0}$).}
\begin{enumerate}[(i)]
\item The sequence $(t_0^n,\gamma_0^n)_{n \in \mathbb{N}}$ converges to $(t_0,\gamma_0)$ in $[0,T] \times \mathcal{P}_3(\R^{d_1} \times \R^{d_2})$;
\item For each $n \in {\mathbb N}$, $\bd{\nu}^{*,n}$ is an optimal solution for $J \bigl( (t_0^n,\gamma_0^n), \cdot \bigr)$ with associated trajectory and multiplier $(\bd{\gamma}^{*,n}, \bd{u}^{*,n})$;
\item The sequence $(\bd{\nu}^{*,n}, \bd{\gamma}^{*,n}, \bd{u}^{*,n})_{n \in \mathbb{N}}$ converges to $(\bd{\nu}^*,\bd{\gamma}^*, \bd{v}^*)$ in the sense of Lemma \ref{lem:convergenceifuniqueness};
\item For each $n \in \mathbb{N}$, $\bd{\nu}^n$ is an element of $\mathcal{A}(t_0^n)$ with associated curve and multiplier $(\bd{\gamma}^n, \bd{u}^n)$ solution to \eqref{eq:curvemulti28/02}  starting from $(t_0^n,\gamma_0^n)$ and \textbf{distinct} from $\bd{\nu}^{*,n}$;
\item The following convergence holds
$$ \lim_{n \rightarrow +\infty} \int_{t_0^n}^T \mathcal{E}(\nu_t^n| \nu_t^{*,n})dt =0. $$ 
\end{enumerate}

\subsection{ Quantitative Properties}
Thanks to items $i)$ and $ii)$, by Proposition \ref{prop:regularityfromOC}, Proposition \ref{prop:SolutionBackxard16Sept}
and Lemma \ref{lem:compactnessContinuityEquation24/06} together with the growth assumption on $\ell$,
we obtain 
\begin{lem}
\label{lem:Q2-28/02}
In the setting of {\bf Property $(\mathcal{Q}_0)$}, 
There exist constants $C,c>0$ such that, for any $n \in \mathbb{N}$,
\begin{align}
\notag \sup_{t \in [t_0^n,T]} \int_A e^{c(1+|a|^4)} &d\nu_t^{*,n}(a) +  \norm{ \bd{\nu}^{*,n}}_{\mathcal{D}(t_0^n)} \\
&+ \sup_{t \in [t_0^n,T]} \norm{ \nabla_x u_t^{*,n}}_{\mathcal{C}^2_1} + \sup_{t \in [t_0^n,T]} \int_{\R^{d_1} \times \R^{d_2}} (|x|^2 +|y|^2)^{3/2} d\gamma_t^{*,n}(x,y) \leq C.
\label{eq:uniformestimatemiddleoftheproof}
\end{align}
\end{lem}

Define $(\lambda_n , \bd{\eta}^n,\bd{\rho}^n, \bd{v}^n)_{n \in \mathbb{N}} $ by
\begin{equation} 
\lambda_n^2 := \int_{t_0^n}^T \mathcal{E}(\nu_t^n|\nu_t^{*,n})dt, \quad \bd{\eta}^n := \frac{\bd{\nu}^n - \bd{\nu}^{*,n}}{\lambda_n}, \quad  \bd{\rho}^n := \frac{\bd{\gamma}^n - \bd{\gamma}^{*,n}}{\lambda_n}, \quad \bd{v}^n := \frac{\bd{u}^n - \bd{u}^{*,n}}{\lambda_n}. 
\label{eq:normalizedvariables03/03}
\end{equation} 
Notice that, since $\bd{\nu}^n \neq \bd{\nu}^{*,n}$, $\lambda_n >0$ and the quantities above are well-defined.

It is convenient to introduce the auxiliary variable $\bd{k}^n : [t_0^n,T] \times  A \rightarrow \R$ defined, for all $n \in \mathbb{N}$ and all $(t,a) \in [t_0^n,T] \times A$ by 
\begin{equation}
    k_t^n(a) :=  \frac{1}{\epsilon} \int_{\R^{d_1} \times \R^{d_2}} b(x,a) \cdot d(\nabla_x v_t^n \gamma_t^n + \nabla_x u^{*,n}_t \rho_t^n)(x,y).
\label{eq:ktn28/02}
\end{equation}
Notice that $k_t^n $ is nothing but the normalized difference between the arguments appearing in the exponentials defining $\bd{\nu}^{*,n}$ and $\Gamma[\bd{\nu}^{n}]$, ie,
$$ k_t^n (a) = \frac{1}{\lambda_n}\Bigl[  \frac{1}{\epsilon}\int_{\R^{d_1} \times \R^{d_2}} b(x,a) \cdot \nabla_xu_t^n(x,y) d\gamma_t^n(x,y) - \frac{1}{\epsilon} \int_{\R^{d_1} \times \R^{d_2}} b(x,a) \cdot \nabla_xu_t^{*,n}(x,y) d\gamma_t^{*,n}(x,y) \Bigr]. $$
We can express the cost of $\bd{\nu}^n$ as well as the Gibbs measure $\Gamma_t[\bd \nu^n]$ defined in \eqref{eq:defnGammaNut09/01} with respect to these new variables and we have
\begin{lem}
    In the setting of {\bf Property $(\mathcal{Q}_0)$} and with the variables introduced in  \eqref{eq:normalizedvariables03/03}, it holds, for all $n \in \mathbb{N}$,
\begin{equation}
 J \bigl( (t_0^n,\gamma_0^n),{\boldsymbol \nu}^n \bigr) - J \bigl( (t_0^n,\gamma_0^n), \boldsymbol \nu^{*,n} \bigr) = \epsilon \lambda_n^2  + \lambda_n^2 \int_{t_0^n}^T \int_{\R^{d_1} \times \R^{d_2}} b(x,\eta_t^n) \cdot \nabla_x u^{*,n}_t (x,y) d\rho_t^n(x,y)dt,  \label{eq:rewritingdeltaJsSe3}
\end{equation}
and, for all $n \in \mathbb{N}$ and all $t \in [t^n_0,T]$,
\begin{equation}
\Gamma_t[\bd{\nu}^n] \propto \nu_t^{*,n}  e^{ -\lambda_n k_t^n}.
\label{eq:rewritinggammanutn28/02}
\end{equation}
\label{lem:rewriting02/02}
\end{lem}

\begin{proof}[Proof of Lemma \ref{lem:rewriting02/02}]
\textit{Step 1. Expression for the costs.} We start with the proof of \eqref{eq:rewritingdeltaJsSe3}. We first use the elementary relation 
$$ a \log \frac{a}{c} - b \log \frac{b}{c} = a \log \frac{a}{b} + (a-b) \log \frac{b}{c} \quad \mbox{for all }  a,b,c \in \R^*_+  $$
to obtain
\begin{align}
\notag     J \bigl( (t_0^n,\gamma_0^n),{\boldsymbol \nu}^n \bigr) &- J \bigl( (t_0^n,\gamma_0^n), \boldsymbol \nu^{*,n} \bigr) = \epsilon \int_{t_0^n}^T \int_{A} \log \frac{\nu_t^n(a)}{\nu^{*,n}_t(a)} d\nu_t^n(a)dt \\
 &+ \epsilon \int_{t_0^n}^T \int_A\log \frac{\nu_t^{*,n}(a) }{\nu^{\infty}(a)}d(\nu_t^n - \nu_t^{*,n})(a) dt + \int_{\R^{d_1} \times \R^{d_2}} L(x,y) d(\gamma_T^n - \gamma_T^{*,n})(x,y). \label{eq:27juilletNumber0-02/03}
\end{align}
On the one hand, the explicit expression
\eqref{eq:NOCfornut2Avril}
 for $\nu_t^{*,n}$, for $t \in [t_0^n,T]$, leads to
\begin{equation}
 \epsilon \log \frac{\nu_t^{*,n}(a)}{\nu^{\infty}(a)}  = - \int_{\R^{d_1} \times \R^{d_2}} b(x, a) \cdot \nabla_x u_t^{*,n}(x,y) d \gamma_t^{*,n}(x,y) + c_t^n, \quad a \in A,  
\label{eq:27juilletNumber1-02/03}
\end{equation}
where $c_t^n$ is independent from $a$. On the other hand, using the equations satisfied by $\boldsymbol u^{*,n}$, $\boldsymbol \gamma^{*,n}$ and ${\boldsymbol \gamma}^n$, 
see 
\eqref{eq:curvemulti28/02}, we have, by Lemma \ref{lem:duality:1,2:F},
\begin{equation} 
\int_{\R^{d_1} \times \R^{d_2}} L(x,y) d(\gamma_T^n - \gamma_T^{*,n})(x,y) =  \int_{t_0^n}^T \int_{\R^{d_1} \times \R^{d_2}} b(x,\nu_t^n-\nu_t^{*,n}) \cdot \nabla_x u_t^{*,n}(x,y) d\gamma_t^n(x,y) dt. 
\label{eq:27juilletNumber2-02/03}
\end{equation}

Inserting \eqref{eq:27juilletNumber1-02/03} and \eqref{eq:27juilletNumber2-02/03} in \eqref{eq:27juilletNumber0-02/03} 
and recalling 
that $\nu^n_t(A) = \nu^{*,n}_t(A)=1$ for  every 
$t \in [t_0^n,T]$, we obtain
\begin{align} 
\notag J \bigl( (t_0^n,\gamma_0^n),{\boldsymbol \nu}^n \bigr) - J \bigl( (t_0^n,\gamma_0^n), &\boldsymbol \nu^{*,n} \bigr) = \epsilon \int_{t_0^n}^T \int_{A} \log \frac{\nu_t^n(a)}{\nu^{*,n}_t(a)} d\nu_t^n(a)dt \\
& + \int_{t_0^n}^T \int_{\R^{d_1} \times \R^{d_2}} b(x,\nu^n_t - \nu^{*,n}_t) \cdot \nabla_x u^{*,n}_t (x,y) d(\gamma^n_t - \gamma^{*,n}_t)(x,y)dt.  \label{eq:rewritingdeltaJs}
\end{align}
Recalling the definition of $\lambda_n, \bd{\eta}^n$ and $\bd{\rho}^n$ we obtain \eqref{eq:rewritingdeltaJsSe3}.

\textit{Step 2. Expression for $\Gamma[\bd{\nu}^n]$.} We go on with the expression for the probability measure $\Gamma_t[\bd{\nu}^n]$ defined in \eqref{eq:defnGammaNut09/01} for some $t$ in $[t_0^n,T]$. Recalling the explicit expression \eqref{eq:NOCfornut2Avril} for $\nu_t^{*,n}$ we know that there are two constants $z_t^n, z_t^{*,n}$ such that
$$ \Gamma_t[\bd{\nu}^n](a) = \frac{1}{z_t^n} \exp \Bigl(-\ell(a) -\frac{1}{\epsilon} \int_{\R^{d_1} \times \R^{d_2}} b(x,a) \cdot \nabla_x u_t^n (x,y) d\gamma_t^n(x,y) \Bigr),  $$
$$ \nu_t^{*,n}(a) = \frac{1}{z_t^{*,n}} \Bigl(-\ell(a) -\frac{1}{\epsilon} \int_{\R^{d_1} \times \R^{d_2}} b(x,a) \cdot \nabla_x u_t^{*,n} (x,y) d\gamma_t^{*,n}(x,y) \Bigr),    $$
and, as a consequence, combining these two expressions
\begin{align*}
    \Gamma_t[\bd{\nu}^n](a) &= \frac{z_t^{*,n}}{z_t^n} \nu_t^{*,n}(a) \exp \Bigl( -\frac{1}{\epsilon} \int_{\R^{d_1} \times \R^{d_2}} b(x,a) \cdot d \bigl( \nabla_x u_t^n \gamma_t^n - \nabla_x u_t^{*,n} \gamma_t^{*,n} \bigr) (x,y) \Bigr). 
\end{align*}
Rewriting 
$$ \nabla_x u_t^n \gamma_t^n - \nabla_x u_t^{*,n} \gamma_t^{*,n} = (\nabla_x u_t^n - \nabla_x u_t^{*,n}) \gamma_t^n + \nabla_x u_t^{*,n} (\gamma_t^n - \gamma_t^{*,n}) $$
and recalling that $\rho_t^n = \lambda_n^{-1}( \gamma_t^n - \gamma_t^{*,n}) $, $v_t^n = \lambda_n^{-1}(u_t^n - u_t^{*,n})$ we obtain
$$ \Gamma_t[\bd{\nu}^n] (a) = \frac{z_t^{*,n}}{z_t^n} \nu_t^{*,n}(a) e^{-\lambda_n k_t^n(a)}, $$
with $k_t^n $ defined by  \eqref{eq:ktn28/02}.
\end{proof}

By making the difference between the system
\eqref{eq:curvemulti28/02}
(which is satisfied by 
$({\boldsymbol \gamma}^n,{\boldsymbol u}^n)$)
and 
the system \eqref{eq:NOCforutgammat2Avril}
(which is satisfied by
$(\boldsymbol \gamma^{*,n},\boldsymbol u^{*,n})$)
and then by dividing by $\lambda_n$, we get 
\begin{equation}
\left \{
\begin{array}{ll}
\displaystyle -\partial_t {v}_t^n(x,y) -  b\bigl(x,{\nu}_t^n\bigr)  \cdot \nabla_x {v}_t^n(x,y) = b (x, \eta_t^n) \cdot \nabla_x u^{*,n}_t(x,y)
\ &\textrm{\rm in}
\  [t_0^n,T] \times \R^{d_1} \times \R^{d_2},
\\
\displaystyle
\qquad 
 {v}_T^n(x,y) = 0 \ &\textrm{\rm in} \ \R^{d_1} \times \R^{d_2};
\\
\displaystyle \partial_t {\rho}_t^n
+ \div_x \bigl(  b(x,{\nu}_t^n)  {\rho}_t^n  \bigr) = 
 - \div_x \bigl( b(x,\eta_t^n) \gamma^{*,n}_t \bigr) \
 & \textrm{\rm in}
 \ (t_0^n,T) \times  \R^{d_1} \times \R^{d_2}, 
\\
\displaystyle \qquad  {\rho}_{t_0^n}^n = 0
\ &\textrm{\rm in} 
 \ 
\R^{d_1} \times \R^{d_2}.
\end{array}
\right.
\label{eq:systemrhonvn29Mai-28/02}
\end{equation}
Invoking Lemma  \ref{lem:PinskerSecFOC28/02} and recalling the definition of $\lambda_n$ we can control $\bd{\nu}^n$ in $\mathcal{D}(t_0^n)$. Then, applying Lemma \ref{lem:backthenA.21-09/02} for $\bd{\rho}^n$ and Lemma \ref{lem:diffulambda24/02:13:23} for $\bd{v}^n$   
together with Lemma \ref{lem:Q2-28/02}, we get the following uniform estimate:
\begin{lem}
In the setting of {\bf Property $(\mathcal{Q}_0)$}, there is $C >0$ such that, for all  $n \in \mathbb{N}$,
    $$ \norm{ \bd{\eta}^n}_{\mathcal{D}(t_0^n)} + \norm{ \bd{\rho}^n}_{\mathcal{R}(t_0^n)} + \sup_{t \in [t_0^n,T]} \norm{ v_t^n}_{\mathcal{C}^2_1} + \sup_{t_1 <t_2 \in [t_0^n,T]} \frac{\norm{v_{t_2}^n - v_{t_1}^n}_{\mathcal{C}^1_1}}{\sqrt{t_2-t_1}} + \sup_{t \in [t_0^n,T]} \norm{ k_t^n}_{\mathcal{C}^1_2(A)} \leq C.  $$
\label{lem:timetogotoGallia02/03}
\end{lem}

\begin{proof}[Proof of Lemma \ref{lem:timetogotoGallia02/03}]
    By Lemma \ref{lem:PinskerSecFOC28/02} and recalling the definition of $\lambda_n$ we can control $\bd{\eta}^n$ in $\mathcal{D}(t_0^n)$. The estimates on $\bd{\rho}^n$ and $\bd{v}^n$ then follow from Lemmas \ref{lem:backthenA.21-09/02} and \ref{lem:diffulambda24/02:13:23} respectively together  with Lemma \ref{lem:Q2-28/02} which guarantees that $\norm{\bd{\nu}^{*,n}}_{\mathcal{D}(t_0^n)}$ and $\int_{\R^{d_1} \times \R^{d_2}} (|x|^2 + |y|^2)^{3/2}d\gamma_0^n(x,y)$ are bounded independently from $n \in \mathbb{N}$. The estimate on $\bd{k}^n$ then follows from the regularity assumptions on $b$ through the formula \eqref{eq:secondconvenientexpress02/02} with $q=1$, $k_1=0$, $p_1=3$, $p_2=3$, $k_2=1$ since it gives
    $$ \norm{k_t^n}_{\mathcal{C}^1_2(A)}  \leq C_b \bigl( \norm{\nabla_x u_t^{*,n} }_{\mathcal{C}^1_2} + \int_{\R^{d_1} \times \R^{d_2}} (|x|^2+|y|^2)^{3/2} d\gamma_t^n(x,y) \bigr) \bigl( \norm{ \rho_t^n}_{(\mathcal{C}^1_3)^*} + \norm{ \nabla_x v_t^n}_{\mathcal{C}^0_3} \bigr). 
    $$
Notice that the term involving $\bd{\gamma^n}$ is bounded independently from $n$ by Lemma \ref{lem:compactnessContinuityEquation24/06}.
\end{proof}

A direct consequence is 
\begin{lem}
    In the setting of {\bf Property $(\mathcal{Q}_0)$}, we have the following expansion of the costs, as $ n \rightarrow +\infty$
    $$ J\bigl( (t_0^n,\gamma_0^n) , \bd{\nu}^n \bigr) = J \bigl( (t_0^n,\gamma_0^n), \bd{\nu}^{*,n} \bigr) + O(\lambda_n^2).  $$
\label{lem:expensioncost02/03}
\end{lem}

\begin{proof}[Proof of Lemma \ref{lem:expensioncost02/03}]
    We use the expression of the cost from Lemma \ref{lem:rewriting02/02} and then formula \eqref{eq:firstconvenientexpress02/03} with $q=0,k=1,p=3$ which gives, after an application of Fubini's theorem,
\begin{align*} 
\bigl| \int_{\R^{d_1} \times \R^{d_2}} b(x,\eta_t^n) \cdot \nabla_x u_t^{*,n}(x,y) d\rho_t^n(x,y) \Bigr| & \leq \int_{A} \Bigl| \int_{\R^{d_1} \times \R^{d_2}} b(x,a) \cdot \nabla_x u_t^{*,n}(x,y) d\rho_t^n(x,y) \Bigr| d|\eta_t^n|(a) \\
&\leq C_b \norm{ \nabla_x u_t^{*,n}}_{\mathcal{C}^1_3} \norm{ \rho^n_t}_{(\mathcal{C}^1_3)^*} \int_A (1+|a|^2) d|\eta_t^n|(a).
\end{align*}
Integrating in time and recalling the estimates from Lemma \ref{lem:timetogotoGallia02/03} and Lemma \ref{lem:Q2-28/02} we obtain the result.
\end{proof}

\subsection{Weak Limits}

Thanks to the estimates of Lemma \ref{lem:timetogotoGallia02/03} we can also  find weak limit points for the sequence $(\bd{\eta}^n, \bd{\rho}^n,\bd{v}^n)_{n \in \mathbb{N}}$ and  pass to the limit in the system \eqref{eq:systemrhonvn29Mai-28/02}. The proof is given in Section \ref{sec:ProofsofSection3}.
\begin{prop} 
\label{prop:preliminaryworkprop1-28/02}
In the setting of {\bf Property $(\mathcal{Q}_0)$}, there is $(\bd{\eta}, \bd{\rho}, \bd{v}) \in \mathcal{A}^l(t_0) \times \mathcal{R}(t_0) \times \mathcal{C}([t_0,T], \mathcal{C}^1_1)$  such that $(\bd{\eta}^n, \bd{\rho}^n, \bd{v}^n)_{n \in \mathbb{N}}$ converges, up to a sub-sequence, to $(\bd{\eta}, \bd{\rho}, \bd{v})$ in the following sense:
\begin{enumerate}
\item[i)] \textbf{(convergence of $(\bd{\eta}^n)_{n \in \mathbb{N}}$)} 
 For any $[t_1,t_2] \subset (t_0,T]$ and for any $f \in \mathcal{C}([t_1,t_2] \times A)$ 
satisfying $|f_t(a)| \leq C(1+|a|^{3+\delta})$ for some  $(C,\delta) \in \R^+ \times (0,1)$ and all  $ (t,a ) \in [t_1,t_2] \times A$, 
\begin{equation}  
\lim_{ n \rightarrow +\infty} \int_{t_1}^{t_2} \int_A f_t(a) d(\eta_t^n - \eta_t)(a)dt =0;
\label{eq:testfnnutn03/03}
\end{equation}
\item[ii)] \textbf{(convergence of $(\bd{\rho}^n,\bd{v}^n)_{n \in \mathbb{N}}$)} 
\begin{equation}
\lim_{n \rightarrow +\infty} \Bigl \{ \sup_{t \in [t_0 \vee t_0^n,T]} \norm{ \rho_t^n - \rho_t}_{(\mathcal{C}_{2}^2)^*} + \sup_{t \in [t_0 \vee t_0^n,T] } \norm{v_t^n -v_t}_{\mathcal{C}^1_2} \Bigr \} =0.
\label{eq:cvgcerhonvn04/03}
\end{equation}
\end{enumerate}
Moreover, for any such limit point, $(\bd{\rho}, \bd{v})$ is solution to 
\begin{equation}
\label{eq:v:rho:proof:Meta2.2-02/03}
    \left \{
    \begin{array}{ll}
    -\partial_t v_t - b(x,\nu^*_t) \cdot \nabla_x v_t = b (x, \eta_t) \cdot \nabla_x u^*_t \quad &\mbox{\rm in }  [t_0,T] \times \R^{d_1} \times \R^{d_2}, 
    \\
     \qquad v_T = 0 \quad &\mbox{\rm in } \R^{d_1} \times \R^{d_2};
    \\
     \partial_t \rho_t + \div_x \bigl( b(x,\nu^*_t) \rho_t \bigr) = - \div_x ( b(x,\eta_t) \gamma^*_t ) \quad &\mbox{\rm in }  (t_0,T) \times \R^{d_1} \times \R^{d_2}, 
     \\
     \qquad \rho_{t_0} = 0 \quad &\mbox{\rm in } \R^{d_1} \times \R^{d_2}.
    \end{array}
    \right.
\end{equation}
\end{prop}

The convergence of $(\bd{\nu}^n, \bd{\rho}^n,\bd{v}^n)_{n \in \mathbb{N}}$ translates into a convergence for $(\bd{k}^n)_{n \in \mathbb{N}}$ defined through \eqref{eq:ktn28/02}. For a given limit point $(\bd{\eta}, \bd{\rho}, \bd{v})$ we define, for all $(t,a) \in [t_0,T] \times A$ 
\begin{equation} 
k_t(a) := \frac{1}{\epsilon}\int_{\R^{d_1} \times \R^{d_2}} b(x,a) \cdot \nabla_x v_t(x,y) d\gamma^*_t(x,y)  + \frac{1}{\epsilon}\langle b(\cdot,a) \cdot \nabla_x u^*_t; \rho_t \rangle. 
\label{eq:defnkt03/03}
\end{equation}
The next convergence statement for $(\bd{k}^n)_{n \in \mathbb{N}}$ together with the result of Proposition \ref{prop:cvgcesumofIs03/03} follow from Proposition \ref{prop:preliminaryworkprop1-28/02}, the uniform estimates from Lemma \ref{lem:Q2-28/02} and the growth and regularity assumptions on the vector field $b$. The detailed proofs are given in Section \ref{sec:ProofsofSection3}.
\begin{prop}
\label{prop:convergencektn03/03}
In the setting of {\bf Property $(\mathcal{Q}_0)$} and with the limit point $(\bd{\eta}, \bd{\rho}, \bd{v})$ given by Proposition \ref{prop:preliminaryworkprop1-28/02},  the map $t \mapsto k_t \in \mathcal{C}^1_2(A)$ is bounded and $\bd{k}^n$ converges to $\bd{k}$ in the following ways:
$$ \lim_{n \rightarrow +\infty} \sup_{t \in [t_0^n \vee t_0,T]} \norm{ k_t^n -k_t}_{\mathcal{C}^0_3(A)} = \lim_{n \rightarrow +\infty} \sup_{t \in [t_0^n \vee t_0,T]} \int_A |k_t^n(a) - k_t(a) | d\nu_t^{*,n}(a)  =0. $$
If the limit point $(\bd{\eta}, \bd{\rho}, \bd{v})$ is the triple $(0,0,0)$, then the convergence can be strengthened to 
$$ \lim_{n \rightarrow +\infty} \sup_{t \in [t_0^n \vee t_0,T] } \norm{k^n_t }_{\mathcal{C}^1_3(A)} = \lim_{n \rightarrow +\infty} \int_{t_0^n}^T \int_A | \nabla_a k_t^n(a) |^2 d\nu_t^n(a)dt =0. $$
\end{prop}

As a consequence, and recalling the expression \eqref{eq:rewritinggammanutn28/02} for $\Gamma_t[\bd{\nu}^n]$, we get
\begin{prop}
\label{prop:cvgcesumofIs03/03}
In the setting of {\bf Property $(\mathcal{Q}_0)$} and for the limit point $(\bd{\eta}, \bd{\rho}, \bd{v})$ given by Proposition \ref{prop:preliminaryworkprop1-28/02}, we have 
\begin{equation} 
\lim_{ n \rightarrow +\infty} \int_{t_0^n \vee t_0}^T \Bigl( \int_A \Bigl|  \frac{\Gamma_t[\bd{\nu}^n](a) - \nu_t^{*,n}(a)}{\lambda_n} + \nu_t^*(a) \bigl( k_t(a) - \int_A k_t(a') d\nu_t^*(a') \bigr) \Bigr|da \Bigr)^2 dt =0,  
\label{eq:04/05-22:10}
\end{equation}
where $\bd{k}$ is defined in \eqref{eq:defnkt03/03}.
\end{prop}

In particular, when $(\bd{\nu}^n,\bd{\gamma}^n,\bd{u}^n)$ is a solution to \eqref{eq:NOCfornut2Avril}-\eqref{eq:NOCforutgammat2Avril} we can replace $\Gamma[\bd{\nu}^n]$ by $\bd{\nu}^n$ in \eqref{eq:04/05-22:10} and we deduce the following result.
\begin{prop}
In the setting of {\bf Property $(\mathcal{Q}_0)$}, if we assume as well that the triple $(\bd{\nu}^n,\bd{\gamma}^n, \bd{u}^n)$ is solution to \eqref{eq:NOCfornut2Avril}-\eqref{eq:NOCforutgammat2Avril} for all $n \in \mathbb{N}$, then any weak limit point $(\bd{\eta}, \bd{\rho}, \bd{v})$ of $(\bd{\eta}^n,\bd{\rho}^n, \bd{v}^n)_{n \in \mathbb{N}}$ in the sense of Proposition \ref{prop:preliminaryworkprop1-28/02} solves the linearized system \eqref{eq:nutProp3.2}-\eqref{eq:vtrhotProp3.2}.
\label{prop:preliminayworkProp2-28/02}
\end{prop}

\section{Discriminating Property and the Jacobi Condition}
\label{sec:TheUniversalAppox}

This section is dedicated to the 
proof of 
Meta-Theorem 
\ref{meta-thm:1}, a more precise and rigorous version of which is given below. Recall that stable solutions were introduced in Definition \ref{defn:stablesolutions} and the set $\mathcal{O}$ in \eqref{eq:elfamososetO}.

\begin{thm}
\label{thm:main:stability}
    Let $(t_0,\gamma_0) \in [0,T) \times \mathcal{P}_3(\R^{d_1} \times \R^{d_2})$ and  $\boldsymbol \nu^*=(\nu^*_t)_{t_0 \le t \le T}$ 
    be a (non-necessarily unique) minimizer of $J ((t_0,\gamma_0),\cdot )$ with 
    corresponding curve 
    $\boldsymbol \gamma^* =
    (\gamma^*_t)_{t_0 \leq t \leq T}$. Then, for all $t_1 \in (t_0,T]$,  
    $(t_1, \gamma^*_{t_1}) \in \mathcal{O}$.
        Moreover, the set $\mathcal{O}$ is open and dense in $[0,T] \times \mathcal{P}_3(\R^{d_1} \times \R^{d_2})$.
\end{thm}

The first part of Theorem
\ref{thm:main:stability}
can be referred to as a \textit{Jacobi condition}, or \textit{Jacobi no conjugate point optimality condition}, (see \cite{cannarsa} Chapter 6 for similar results in finite dimension and Remark 6.3.7 therein for the terminology). The proof is
given in Subsection
\ref{subse:4:2}, except for the (more expected) proof of the openness of the set $\mathcal{O}$, which is given in Section \ref{sec:Oisopen04/03}. The main ingredients are the dynamic programming principle and 
the following two 
Propositions \ref{prop:uniquenessNOC} 
and \ref{prop:uniquenessLinSyst}, which 
follow themselves from \assdis.

\begin{prop}
Let $(t_0,\gamma_0) \in [0,T) \times {\mathcal P}_3({\mathbb R}^{d_1} \times {\mathbb R}^{d_2})$. 
    Assume that $({\boldsymbol \nu}^1,{\boldsymbol \gamma}^1,{\boldsymbol u}^1)=(\nu^1_t, \gamma^1_t, u_t^1)_{t_0 \leq t \leq T}$ and $({\boldsymbol \nu}^2,{\boldsymbol \gamma}^2,{\boldsymbol u}^2)=(\nu^2_t, \gamma^2_t, u_t^2)_{t_0 \le t \le T}$ are two solutions of the system \eqref{eq:NOCfornut2Avril}--\eqref{eq:NOCforutgammat2Avril} with $(t_0,\gamma_0)$ as initial condition. Then,
    $$ \nu^1_{t_0} = \nu^2_{t_0} \Rightarrow 
    \Bigl( \forall t \in [t_0,T], \quad \nu^1_t = \nu^2_t \Bigr).$$
    In particular, when $\nu_{t_0}^1 = \nu_{t_0}^2$, then $(\gamma_t^1,u_t^1) = (\gamma_t^2,u_t^2)$ for all $t \geq t_0$.
    \label{prop:uniquenessNOC}
\end{prop}

In words, the result above can be described as a \textit{no bifurcation} property: if two optimal controls coincide at the initial time they must coincide throughout the whole time horizon. We have an analogous result for the linearized system:

\begin{prop}
    Assume that $(\eta_t, \rho_t,v_t)_{t \geq t_0}$ is a solution to the linearized system \eqref{eq:nutProp3.2}-\eqref{eq:vtrhotProp3.2} for some given triple $(\nu^*_t, \gamma^*_t, u^*_t)_{t \geq t_0}$ solution to \eqref{eq:NOCfornut2Avril}-\eqref{eq:NOCforutgammat2Avril} . Then,
    $$ \eta_{t_0} = 0 \Rightarrow \Bigl( \forall t \in [t_0,T], \quad \eta_t = 0 \Bigr).$$
    In particular, when $\eta_{t_0} = 0$, then $(\rho_t, v_t) = (0,0)$ for all $t \geq t_0$.
    \label{prop:uniquenessLinSyst}
\end{prop}

Subsections \ref{subse:4:2}, \ref{subse:4:uniquenessNOC} and \ref{subse:4:uniquenessLinSyst} are dedicated to the proofs of Theorem \ref{thm:main:stability}, Proposition \ref{prop:uniquenessNOC} and Proposition \ref{prop:uniquenessLinSyst} respectively. In Subsection \ref{sec:Oisopen04/03} we prove that $\mathcal{O}$ is open and that stable solutions are \textit{isolated}. In Subsection \ref{subse:DPandUAT} we explain the link between the discriminating property imposed on the vector field $b$ and the universal approximation property of typical activation functions in machine learning.

\subsection{Proof of Theorem 
\ref{thm:main:stability}}
\label{subse:4:2}

Taking for granted Propositions 
\ref{prop:uniquenessNOC}
and 
\ref{prop:uniquenessLinSyst}, we prove 
Theorem 
\ref{thm:main:stability}.

\begin{proof}
    \textit{Step 1.} Take $(t_0,\gamma_0) \in [0,T) \times \mathcal{P}_3(\R^{d_1} \times \R^{d_2})$ and an arbitrary minimizer $\boldsymbol \nu^* = (\nu^*_t)_{t_0 \le t \le T}$ of $J((t_0,\gamma_0),\cdot)$ with corresponding curve $\boldsymbol \gamma^* = (\gamma^*_t)_{t_0 \le t \le T}$. 
    We claim that for any $t_1 \in (t_0,T]$, 
    $(\nu^*_t)_{t_1 \le t \le T}$ is the unique minimizer of $J((t_1, \gamma^*_{t_1}),\cdot).$
By dynamic programming, it is indeed 
 a minimizer of 
$J ((t_1, \gamma^*_{t_1}),\cdot )$. 
Uniqueness is shown as follows. 

Let ${\boldsymbol \nu}=(\nu_t)_{t_1 \le t \le T}$ be an (arbitrary) optimal control for $J((t_1, \gamma^*_{t_1}),\cdot)$ with 
corresponding curve ${\boldsymbol \gamma}=(\gamma_t)_{t_1 \le t \le T}$. We extend $(\nu_t , \gamma_t)_{t_1 \le t \le T}$ to $[t_0,t_1)$ by setting
\begin{equation} 
(\bar{\nu}_t, \bar{\gamma}_t) :=
\left \{
\begin{array}{ll}
(\nu^*_t, \gamma^*_t) &\mbox{ if } t \in [t_0,t_1), 
\\
(\nu_t , \gamma_t) &\mbox{ if } t \in [t_1,T].
\end{array}
\right.
\end{equation}
Since $\gamma_{t_1} = \gamma^*_{t_1}$, 
we can easily recombine the two 
continuity equations 
(i.e., the second equation in \eqref{eq:NOCforutgammat2Avril})
solved by ${\boldsymbol \gamma}$
and $\boldsymbol \gamma^*$
respectively and then deduce that 
$\bar{\boldsymbol \gamma} = (\bar{\gamma}_t)_{t_0 \le t \le T}$ solves the continuity equation associated with the control 
$\bar{\boldsymbol \nu} = (\bar{\nu}_t)_{t_0 \leq t \le T}$. Moreover, by dynamic programming, $\bar{\boldsymbol \nu} =
(\bar{\nu}_t)_{t_0 \le t \le T}$ is an optimal control for $J((t_0,\gamma_0),\cdot)$. 
By construction, $\bar{\nu}_{t_0} = \nu^*_{t_0}$. Therefore, we deduce from Proposition \ref{prop:uniquenessNOC} that $\bar{\nu}_t = \nu^*_t $ for all $t \in [ t_0,T]$. In particular, we have $\nu_t = \nu^*_t$ for all $t \in [ t_1,T]$, which proves the claim.
\vskip 4pt

\textit{Step 2.} Take $(t_0,\gamma_0)$
and $(\boldsymbol \nu^*,\boldsymbol \gamma^*)$
as in the first step. 
We claim that for any $t_1 \in (t_0,T]$,  $(\nu^*_t)_{t_1 \le t \le T}$ is a stable solution of $J((t_1, \gamma^*_{t_1}),\cdot).$

Indeed, call 
$\boldsymbol u^*=( u^*_t)_{t_0 \le t \le T}$
the second component of the 
solution
 $(\gamma^*_t, u^*_t)_{t_0 \leq t \le T}$ to the system \eqref{eq:NOCforutgammat2Avril} starting from $(t_0,\gamma_0)$ and associated with $\boldsymbol \nu^* = (\nu^*_t)_{t_0 \le t \le T}$. Take now a solution $({\boldsymbol \eta},{\boldsymbol \rho}, {\boldsymbol v})=(\eta_t, \rho_t, v_t)_{t_1 \leq t \le T}$ to the linearized system \eqref{eq:nutProp3.2}-- \eqref{eq:vtrhotProp3.2} starting from $t_1$, around the triple 
 $(\boldsymbol \nu^*,\boldsymbol \gamma^*,\boldsymbol u^*)$. 
By Theorem  \ref{prop:SecondOrderConditions3Avril}, $({\boldsymbol \eta}, {\boldsymbol \rho})$ is a minimum of $\mathcal{J}  ( (t_1, \gamma^*_{t_1}), \boldsymbol \nu^*, \cdot )$ and the corresponding value is $0$. We extend $({\boldsymbol \eta},{\boldsymbol \rho})$ to $[t_0,T]$ by setting 
\begin{equation} 
(\bar{\eta}_t, \bar{\rho}_t) :=
\left \{
\begin{array}{ll}
(0, 0) &\mbox{ if } t \in [t_0,t_1), 
\\
(\eta_t , \rho_t) &\mbox{ if } t \in [t_1,T].
\end{array}
\right.
\end{equation}
We easily check that
$$ \mathcal{J} \bigl( (t_0,\gamma_0), \boldsymbol \nu^*, \bar{\boldsymbol  \eta}  \bigr) = 0 +  \mathcal{J} \bigl( (t_1, \gamma_{t_1}^*), \boldsymbol  \nu^*, {\boldsymbol \eta} \bigr) =0. $$
By Theorem  \ref{prop:SecondOrderConditions3Avril} again, we deduce that $(\bar{\eta}_t, \bar{\rho}_t)_{t_0 \le t \le T}$ is a minimum of $\mathcal{J}  (t_0, (\boldsymbol \nu^*, \boldsymbol \gamma^*, \boldsymbol u^*), \cdot  )$. In particular, there is a multiplier $\bar{\boldsymbol v}=(\bar{v}_t)_{ t_0 \leq t \le T}$ such that the triple $(\bar{\eta}_t, \bar{\rho}_t, \bar{v}_t)_{t_0 \le t  \le T}$ solves the linearized system \eqref{eq:nutProp3.2}--\eqref{eq:vtrhotProp3.2} starting from $t_0$, around the solution $(\boldsymbol \nu^*, \boldsymbol \gamma^*, \boldsymbol u^*)$. By construction, $\bar{\eta}_{t_0}=0$. Therefore, we can use Proposition \ref{prop:uniquenessLinSyst} to infer that $(\bar{\eta}_t, \bar{\rho}_t, \bar{u}_t) = (0, 0,0)$ for all $t \in [t_0,T]$. In particular $(\eta_t, \rho_t) = (0, 0)$ for all $t \in [t_1,T]$, from which we easily deduce that $(\eta_t, \rho_t, v_t) = (0, 0,0)$ for all $t \in  [t_1,T]$.
This shows that $(t_1, \gamma^*_{t_1}) \in {\mathcal O}.$
\vskip 4pt

\textit{Step 3.}
As a consequence, we can prove that $\mathcal{O}$ is dense. 
    Take indeed $(t_0,\gamma_0) \in [0,T) \times \mathcal{P}_3(\R^{d_1} \times \R^{d_2})$
    and an optimal control $(\nu^*_t)_{t_0 \le t \le T}$ for $J((t_0,\gamma_0),\cdot)$ with associated curve $(\gamma^*_t)_{t_0 \leq t \le T}$. By Steps 1 and 2, we know that $(t, \gamma^*_t)$ belongs to $\mathcal{O}$ for any $t \in (t_0,T]$. Since $t \mapsto \gamma^*_t \in \mathcal{P}_3(\R^{d_1} \times \R^{d_2})$ is continuous, we easily get the result. The proof that $\mathcal{O}$ is open is the object of Subsection \ref{sec:Oisopen04/03}. 
\end{proof}

\subsection{Proof of Proposition 
\ref{prop:uniquenessNOC}}
\label{subse:4:uniquenessNOC}

\begin{proof}
Following the statement, we consider  
 $(t_0,\gamma_0) \in [0,T) \times {\mathcal P}_3({\mathbb R}^{d_1} \times {\mathbb R}^{d_2})$
 and $({\boldsymbol \nu}^1,{\boldsymbol \gamma}^1,{\boldsymbol u}^1)$ and $({\boldsymbol \nu}^2,{\boldsymbol \gamma}^2,{\boldsymbol u}^2)$ two solutions of  \eqref{eq:NOCfornut2Avril}--\eqref{eq:NOCforutgammat2Avril} with $(t_0,\gamma_0)$ as initial condition. 
 \vskip 4pt

    \textit{Step 1. ODE representation.} 
Following Proposition \ref{prop:FBODErepresentation}, we represent $( \bd{\gamma}^{i}, \bd{u}^{i})$ for $i=1,2$ by means of a system of ordinary differential equations with random initial conditions.  On some probability space $(\Omega, \mathcal{F}, \mathbb{P})$, 
    consider a pair $(X_0,Y_0)$
    of ${\mathbb R}^{d_1} \times \R^{d_2}$-valued random variables with ${\mathbb P} \circ (X_0,Y_0)^{-1} = \gamma_0$. For $i =1,2$, 
    we call 
    ${\boldsymbol X}^i$
    the solution to the ODE 
    \begin{equation*} 
    \dot{X}_t^i = b(X_t^i,\nu_t^i), \quad 
    t \in [t_0,T]; \quad X_{t_0}^i = X_0,
    \end{equation*}
    which is indeed uniquely solvable thanks
    to \assreg $ $
    and to Proposition 
    \ref{prop:regularityfromOC}. We then let 
    \begin{equation}
Z_t^i = \nabla_x u_t^i(X_t^i, Y_0), \quad 
t \in [t_0,T]. 
\label{eq:Zti:representation:nabla_xuti}
    \end{equation}
    By 
    differentiating in $x$ 
    the first equation 
    in 
    \eqref{eq:NOCfornut2Avril}
    and then
    expanding (in time) the right-hand side
    in \eqref{eq:Zti:representation:nabla_xuti} (which is possible thanks to the regularity of ${\boldsymbol u}^i$ 
    ensured by Proposition \ref{prop:regularityfromOC}), we deduce that ${\boldsymbol Z}^i=(Z_t^i)_{t_0 \le t \le T}$ solves the 
    (uniquely solvable) 
    backward linear ODE 
    \begin{equation*}
\dot{Z}_t^i = - \nabla_x b(X_t^i,\nu_t^i) 
Z_t^i, \quad t \in [t_0,T]; \quad Z_T^i = \nabla_x L(X_T^i, Y_0). 
\end{equation*} 
Equivalently, $({\boldsymbol X}^{i},{\boldsymbol Z}^{i})$ is the solution to
the forward-backward ODE system:
\begin{equation}
    \left \{
    \begin{array}{ll}
    \dot{X}_t^{i} = b(X_t^{i},\nu_t^{i}), &\quad  t \in [t_0,T], 
    \quad X_{t_0}^i = X_0;
    \\
    \dot{Z}_t^{i} = - \nabla_x b(X_t^{i},\nu_t^{i})Z_t^{i}, 
    &\quad t \in [t_0,T], 
    \quad  Z_T^{i} = \nabla_x L(X_T^{i},Y_0).
    \end{array}
    \right.
\label{eq:ODE:representation:se:4}
\end{equation}
We easily check that, for all $t\in [t_0,T]$,  ${\mathbb P} \circ 
( X_t^{i} , Y_0 )^{-1} = \gamma_t^{i}$.
\vskip 4pt

\textit{Step 2. Conditional expectation of the backward component upon $X_0$.} For $i=1,2$, ${\boldsymbol X}^i$ solves an ODE with  a Lipschitz continuous velocity field. Therefore, for all $t \in [t_0,T]$ we can find two homeomorphisms $F_t^{i} : \R^{d_1} \rightarrow \R^{d_1}$, $i=1,2$ such that
$$ X_t^{i} = F_t^{i} \bigl( X_0 \bigr).$$
In particular, the $\sigma$-fields 
$(\sigma(X_t^i))_{i=1,2}$ generated by 
$X_t^1$ and $X_t^2$ respectively are each 
equal to $\sigma(X_0)$.

We then introduce, for all $t \in [t_0,T]$, the notation
\begin{equation} 
\hat{Z}_t^{i} = \E \bigl[ Z_t^{i} |X_0 \bigr]. 
\label{eq:notation-conditionalexpectation}
\end{equation}
Notice that, for each $t \in [t_0,T]$, 
$\hat{Z}_t^i$ is just uniquely defined 
${\mathbb P}$-almost surely. 
However, using the time continuity 
of ${\boldsymbol Z}^i$, it is standard to have a version of all these conditional expectations 
such that ${\mathbb P}$-almost surely, $t \in [t_0,T] \mapsto \hat{Z}_t^i$
is continuous.  

Since $X_t^i$ is $\sigma(X_0)$-measurable for each $t \in [t_0,T]$, we have
$$ \E \bigl[ \nabla_x b(X_t^{i},\nu_t^{i}) Z_t^{i} | X_0 \bigr] =    \nabla_x b(X_t^{i},\nu_t^{i}) \E \bigl[ Z_t^{i} | X_0 \bigr]=
  \nabla_x b(X_t^{i},\nu_t^{i})  \hat{Z}_t^{i}. $$
We deduce that, for $i=1,2$, $\hat{\boldsymbol Z}^{i} = (\hat{Z}^{i}_t)_{ t_0 \leq t \leq T}$ solves
$$ \tfrac{d}{dt} \hat{Z}_t^{i} = - \nabla_x b( X_t^{i},\nu_t^{i}) \hat{Z}_t^{i}, \quad t \in [t_0,T].  $$

\vskip 4pt

\textit{Step 3. Injectivity property.} 
Recalling formula
\eqref{eq:NOCfornut2Avril}, we can rewrite $\nu_{t}^{i}$, for $i=1,2$ as 
$$ \nu_{t}^{i}(a) = \frac1{z_t^i}   \exp \Bigl(  -\ell(a) - \frac{1}{\epsilon} \E \bigl[ b(X_t^{i},a) \cdot Z_t^{i} \bigr] \Bigr), \quad a \in A. $$
The above identity is true for 
any $t \in [t_0,T]$. In particular, we can choose $t=t_0$. 
By assumption, we already have 
$\nu_{t_0}^1 = \nu_{t_0}^2$.
Then, taking $a=0$
(which is possible since $\nu_{t_0}^i$ is a 
smooth function of $a$, see Proposition 
 \ref{prop:regularityfromOC})
 and using the fact that 
$b(x,0) = 0$ for all $x$
(see item (i) in 
\assreg), we obtain 
 $z_{t_0}^1 = z_{t_0}^2$. 
We
deduce that 
$$\forall a \in A, \quad \E \bigl[ b(X_0,a) \cdot Z_{t_0}^{1} \bigr] = \E \bigl[ b(X_0,a) \cdot Z_{t_0}^{2} \bigr].  $$
Recalling the notation \eqref{eq:notation-conditionalexpectation}, we deduce from 
\assdis
 that, ${\mathbb P}$-almost surely, 
$ \hat{Z}_{t_0}^1 = \hat{Z}_{t_0}^2$. 
\vskip 4pt

\textit{Step 4 - Stability argument and conclusion.} 
The point is to propagate the 
identity 
$ \hat{Z}_{t_0}^1 = \hat{Z}_{t_0}^2$
to any time $t \in (t_0,T]$. 
First, 
we address the regularity 
of 
the coefficients $b$ and $\nabla_x b$
(in 
\eqref{eq:ODE:representation:se:4})
with respect to the measure argument. 
We claim that there exists $C >0$ such that, for all $x \in \R^d$ and $t \in [t_0,T]$,
\begin{equation} 
|b(x,\nu_t^2) - b(x,\nu_t^1) | + |\nabla_x b(x,\nu_t^2) - \nabla_x b(x,\nu_t^1)| \leq C \bigl( \E \bigl[ | X_t^2 - X_t^1 |^2 \bigr]^{1/2} + \E \bigl[ |\hat{Z}_t^2 - \hat{Z}_t^1 | \bigr] \bigr).  
\label{eq:needssomeestimates}
\end{equation}
In a nutshell, the inequality 
\eqref{eq:needssomeestimates}
follows from the integrability properties of 
${\boldsymbol \nu}^1$
and 
${\boldsymbol \nu}^2$
stated in 
Proposition \ref{prop:regularityfromOC} and the assumptions on $b$
stated in \assreg. 
We provide a sketch of the proof 
at the end of the paragraph.

Thanks to the equations satisfied by ${\boldsymbol X}^1$ and ${\boldsymbol X}^2$, to 
\eqref{eq:needssomeestimates}
and to the Lipschitz property of 
$b$ in $x$ (locally in $a$, the gradient in 
$x$ growing at most at a quadratic rate in 
$a$), we then find a constant $C>0$ such that, $\mathbb{P}-$almost-surely,
\begin{equation} 
|X_t^2 -X_t^1| \leq C \int_{t_0}^t \bigl( \E \bigl[ | X_s^2 - X_s^1 |^2 \bigr]^{1/2} + \E \bigl[ |\hat{Z}_s^2 - \hat{Z}_s^1 | \bigr] \bigr) ds, \quad t \in [t_0,T]. 
\label{eq:totakeexpectations22/02}
\end{equation}
Using the identity $\hat{Z}^1_{t_0} = \hat{Z}^2_{t_0}$, the equations for $\hat{\boldsymbol Z}^1$ and $\hat{\boldsymbol Z}^2$ (seen forward in time) lead to
$$ |\hat{Z}_t^2 - \hat{Z}_t^1| \leq C \int_{t_0}^t\bigl( |X_s^2 - X_s^1| + \E \bigl[ |X_s^2 - X_s^1|^2 \bigr]^{1/2} + \E \bigr[ |\hat{Z}_s^2 - \hat{Z}_s^1 |\bigr] \bigr) |\hat{Z}^2_s| ds, \quad t \in [t_0,T], $$
and so, by Cauchy-Schwarz inequality
\begin{equation} 
\E \bigl[|\hat{Z}_t^2 - \hat{Z}_t^1| \bigr] \leq C \int_{t_0}^t\bigl( \E \bigl[ |X_s^2 - X_s^1|^2 \bigr]^{1/2} + \E \bigr[ |\hat{Z}_s^2 - \hat{Z}_s^1 |\bigr] \bigr) \E \bigl[ |\hat{Z}^2_s|^2 \bigr]^{1/2} ds, \quad t \in [t_0,T]. 
\label{eq:alsoGronwall22/02}
\end{equation}
We make the following observation: since $(\nabla_x b(X_t^i,\nu_t^i))_{t_0 \le t \le T}$
is bounded, there exists
a constant $C>0$ such that $\sup_{t_0 \le t \le T} \vert Z_t^i \vert \leq C 
\vert Z_T^i \vert = C \vert \nabla_x L(X_T^i,Y_0) \vert$ for $i \in \{1,2\}$.
Recalling that 
$({\boldsymbol X}^i)_{i=1,2}$
are square-integrable (since $({\boldsymbol \gamma}^i)_{i=1,2}$
take values in 
${\mathcal P}_3({\mathbb R}^{d_1} \times {\mathbb R}^{d_2})$), 
we deduce that $  {\mathbb E} [ \sup_{t_0 \le t \le T} \vert Z_t^i \vert^2]
\leq C$ and, by Jensen's inequality for the conditional expectation
\begin{equation}
\sup_{t \in [t_0,T]} \E [|\hat{Z}^{i}_t |^2] \leq  C.
\label{eq:fromJensen27/04}
\end{equation} 
Therefore,
$$ \E \bigl[|\hat{Z}_t^2 - \hat{Z}_t^1| \bigr] \leq C \int_{t_0}^t\bigl( \E \bigl[ |X_s^2 - X_s^1|^2 \bigr]^{1/2} + \E \bigr[ |\hat{Z}_s^2 - \hat{Z}_s^1 |\bigr] \bigr)  ds, \quad t \in [t_0,T]. $$
Since \eqref{eq:totakeexpectations22/02} is true $\mathbb{P}-$almost-surely, it also holds in $L^2(\mathbb{P})$ and we deduce that 
\begin{equation} 
\E \bigl[|X_t^2 -X_t^1|^2 \bigr]^{1/2} \leq C \int_{t_0}^t \bigl( \E \bigl[ | X_s^2 - X_s^1 |^2 \bigr]^{1/2} + \E \bigl[ |\hat{Z}_s^2 - \hat{Z}_s^1 | \bigr] \bigr) ds, \quad t \in [t_0,T]. 
\label{eq:tobeGronwalled22/02}
\end{equation}
Combining \eqref{eq:tobeGronwalled22/02} and \eqref{eq:alsoGronwall22/02} and then using Grönwall's lemma, we obtain
$$\forall t \in [t_0,T], \quad \E \bigl[ |X_t^1 -X_t^2|^2 \bigr]^{1/2} = \E \bigl[| \hat{Z}_t^1 - \hat{Z}_t^2 |] =0.  $$
Returning to the equations satisfied by 
${\boldsymbol \nu}^1$ and 
${\boldsymbol \nu}^2$, we deduce in particular that $\nu_t^1 = \nu_t^2$ for all $t \in [t_0,T]$. This easily implies that $(\gamma_t^{1}, u_t^1) = (\gamma_t^2,u_t^2)$ for all $t \in [t_0,T]$.

We now explain the proof of 
\eqref{eq:needssomeestimates}. 
By item (i) in 
\assreg, 
\begin{equation*}
\vert b(x,\nu_t^1) - b(x,\nu_t^2) 
\vert + | \nabla_x b(x,\nu_t^1) - \nabla_x b(x,\nu_t^2)| \leq C   \int_{{\mathbb R}^d} (1+ \vert a\vert^2) |\nu_t^1(a) - \nu_t^2(a)|  da, \quad t \in [t_0,T], \ x \in {\mathbb R}^d. 
\end{equation*}
It then remains to study the difference 
$\nu_t^1(a)-\nu_t^2(a)$, which relies itself on the difference 
${\mathbb E}[b(X_t^1,a) \cdot Z_t^1 
- b(X_t^2,a) \cdot Z_t^2]={\mathbb E}[b(X_t^1,a) \cdot \hat Z_t^1 
- b(X_t^2,a) \cdot \hat Z_t^2]$.  Then, using once again item (i) in \assreg,
we obtain 
\begin{equation*} 
\begin{split}
\bigl\vert {\mathbb E}\bigl[b(X_t^1,a) \cdot \hat{Z}_t^1 
- b(X_t^2,a) \cdot \hat{Z}_t^2\bigr] \bigr\vert &\leq 
C (1+\vert a \vert^2) {\mathbb E}\Bigl[ \bigl( 1 + \vert \hat{Z}_t^1 \vert \bigr) 
\vert X_t^1 - X_t^2 \vert +
\vert \hat{Z}_t^1 - \hat{Z}_t^2\vert \Bigr]
\\
&\leq C (1+\vert a \vert^2) \Bigl( {\mathbb E} \bigl[ 
\vert X_t^1 - X_t^2 \vert^2 \bigr]^{1/2} + \E \bigl[
\vert \hat{Z}_t^1 - \hat{Z}_t^2\vert \bigr] \Bigr),
\end{split} 
\end{equation*}
with the second line following from Cauchy-Schwarz inequality and \eqref{eq:fromJensen27/04}. 
Using the fact that $\ell$ grows 
at least as $\vert a \vert^4$ together with the bound 
$\vert {\mathbb E}[b(X_t^i,a) \cdot Z_t^i] \vert \leq C ( 1 + \vert a \vert) $, we can easily insert the latter bound in the exponential 
writing of $\nu_t^i$, see
\eqref{eq:NOCfornut2Avril}, 
with a similar argument for the normalizing constants $z_t^i$.
\end{proof}

\subsection{Proof of Proposition of \ref{prop:uniquenessLinSyst}}
\label{subse:4:uniquenessLinSyst}

The proof of Proposition \ref{prop:uniquenessLinSyst} is very similar to the one of Proposition \ref{prop:uniquenessNOC}. 

\begin{proof} 
    Let 
$(t_0,\gamma_0) \in [0,T] \times {\mathcal P}_3({\mathbb R}^{d_1} \times {\mathbb R}^{d_2})$
    and
    $(\boldsymbol \nu^*,\boldsymbol u^*,\boldsymbol \gamma^*)=(\nu^*_t, \gamma^*_t, u^*_t)_{t_0 \leq t \leq T}$
    be a minimizer of $J ((t_0,\gamma_0),\cdot )$ (and its corresponding curve and multiplier). 
\vskip 4pt
    
    \textit{Step 1. ODE representation.}
    Following the proof 
    of Proposition 
    \ref{prop:uniquenessNOC}, we introduce the solution 
    $\boldsymbol X^*=( X^*_t)_{t_0 \le t \le T}$ of the ODE 
    $$ \frac{d}{dt} X^*_t = b\bigl(X^*_t,\nu^*_t\bigr), \quad t \in [t_0,T], \quad X^*_{t_0} = X_0;$$
where $(X_0,Y_0)$ is an ${\mathbb R}^{d_1} 
\times {\mathbb R}^{d_2}$-valued random variable, constructed on some 
probability space $(\Omega, \mathcal{F}, \mathbb{P})$, 
with ${\mathbb P} \circ (X_0,Y_0)^{-1} = \gamma_0$. 
Moreover, letting $\boldsymbol Z^* =( Z^*_t := \nabla_x u^*_t(X^*_t,Y_0))_{t_0 \le t \le T}$, we know that
$$\frac{d}{dt} 
Z^*_t = -\nabla_x b\bigl( X^*_t,\nu^*_t\bigr) Z^*_t, \quad
t \in [t_0,T]; \quad Z^*_T = \nabla_x L \bigl(X^*_T,Y_0\bigr).$$

We then introduce ${\boldsymbol \zeta}=(\zeta_t := \nabla_x v_t(X^*_t,Y_0))_{t_0 \le t \le T}$. Differentiating in $x$ the first equation in 
\eqref{eq:vtrhotProp3.2} (which is possible thanks to Proposition 
\ref{prop:reg:v,rho}), we obtain 
\begin{equation}
\label{eq:zeta:t:diff}
\begin{split}
\dot{\zeta}_t
&= \nabla_x \partial_t v_t\bigl( X^*_t,Y_0\bigr) + \nabla_{xx}^2 v_t\bigl( X^*_t,Y_0 \bigr) b\bigl( X^*_t, \nu^*_t\bigr) 
\\
&= - \nabla_x b \bigl( X^*_t, \nu^*_t\bigr) \nabla_{x} v_t\bigl( X^*_t,Y_0 \bigr) - \nabla_x b \bigl(  X^*_t, \eta_t \bigr) \nabla_x  u^*_t \bigl( X^*_t,Y_0 \bigr)
-\nabla_{xx}^2 u^*_t\bigl( X^*_t,Y_0 \bigr) b\bigl( X^*_t,\eta_t\bigr)
\\
&= - \nabla_x b \bigl( X^*_t, \nu^*_t\bigr) \zeta_t - \nabla_x b \bigl(  X^*_t, \eta_t \bigr)  Z^*_t 
-\nabla_{xx}^2 u^*_t\bigl( X^*_t,Y_0 \bigr) b\bigl( X^*_t,\eta_t\bigr).
\end{split}
\end{equation}

\textit{Step 2. Conditioning upon $X_0$.}
 Following again the proof 
    of Proposition 
    \ref{prop:uniquenessNOC}, we define the 
   (time-continuous) collection of 
   conditional expectations: 
$$ \hat{\zeta}_t := \E \bigl[ \zeta_t | X_0 \bigr], 
\quad t \in [t_0,T].$$
Taking conditional expectation in the equation
\eqref{eq:zeta:t:diff}
satisfied by ${\boldsymbol \zeta}$, using the fact that $ X^*_t$ is 
$\sigma(X_0)$-measurable for any $t \in [t_0,T]$ and recalling \eqref{eq:discriminating26/05}, we get, for all $t \in [t_0,T]$
\begin{equation}
\begin{split}
&\displaystyle \frac{d}{dt} \hat{\zeta}_t = - \nabla_x b \bigl(X^*_t, \nu^*_t\bigr) \hat{\zeta}_t  - \nabla_x b\bigl( X^*_t,\eta_t \bigr) \E \bigl[ Z^*_t |X_0 \bigr]  - \E \bigl[ \nabla_{xx}^2 u^*_t\bigl(X^*_t,Y_0\bigr) | X_0] b\bigl(X^*_t , \eta_t \bigr).
\end{split}
\label{eq:equationhatZ26/06}
\end{equation}

\textit{Step 3. Using the discriminating property.}
Thanks to the explicit expression 
\eqref{eq:nutProp3.2}
for $\eta_t$, the condition $\eta_{t_0} = 0$
(combined with the already known condition $\rho_{t_0}=0$, see the last line
in \eqref{eq:vtrhotProp3.2}) leads to
$$ \int_{\R^{d_1} \times \R^{d_2}} b(x,\nu^*_{t_0} - \delta_a) \cdot \nabla_x v_{t_0}(x,y) d \gamma^*_{t_0}(x,y) = 0, \quad  a \in A, $$
which can be rewritten as
$$ \E \bigl[ b(X_0,\nu^*_{t_0} - \delta_a) \cdot {\zeta}_{t_0} \bigr] = 0, \quad   a \in A. $$
Taking first $a=0$ (and using the assumption $b(x,0)=0$), we see that $\E [ b(X_0,\nu^*_{t_0}) \cdot \zeta_{t_0} ] =0$ and then returning to the identity for a general $a$, we get, by linearity of $b(x,\nu)$ with respect to $\nu$, 
$$\E \bigl[ b(X_0,a) \cdot \zeta_{t_0} \bigr] =0, \quad a \in A.$$
By \assdis, we obtain, ${\mathbb P}$-almost surely, 
\begin{equation} 
 \hat{\zeta}_{t_0} =   \E \bigl[ \zeta_{t_0}|X_0\bigr] = 0.  
\label{eq:discriminating26/05}
\end{equation}
\vskip 4pt

\textit{Step 4. Stability argument.}
Notice 
from 
\eqref{eq:nutProp3.2}
that for any $t \in [t_0,T]$, $\eta_t$ can be rewritten as
\begin{equation} 
\begin{split}
\eta_t(a) 
&= \frac{\nu^*_t(a)}
{\epsilon} 
 \E \Bigl[ \Bigl(b\bigl(X^*_t,  \nu^*_t\bigr) - b\bigl(X^*_t, a \bigr) \Bigr) \cdot \hat{\zeta}_t \bigr] + \frac{\nu^*_t(a)}
{\epsilon} 
\bigl \langle \bigl( b(\cdot,\nu^*_t) - b(\cdot,a) \bigr) \cdot \nabla_x u^*_t; \rho_t \bigr \rangle.  
\end{split}
\label{eq:numoinstildenu26/06}
\end{equation}
Using Proposition \ref{prop:regularityfromOC} together with the integrability of $\nu_t^*$ we infer that 
$$ \int_A (1+|a|^2) d|\eta_t|(a) \leq C \Bigl( \E \bigl[ |\hat{\zeta}_t  \vert \bigr] + \norm{ \rho_t}_{(\mathcal{C}^1_3)^*} \Bigr).   $$
In particular using \assreg we deduce the upper bound
$$|b(x,\eta_t )| + |\nabla_x b(x, \eta_t )| \leq C\Bigl( \E \bigl[ |\hat{\zeta}_t  \vert \bigr] + \norm{ \rho_t}_{(\mathcal{C}^1_3)^*} \Bigr),$$
for some $C>0$ and for all $(t,x) \in [t_0,T] \times \R^{d_1}$. Recalling the definition of $\bd{Z}^*$, the estimate  $\sup_{ \in [t_0,T]} \norm{ u^*_t}_{\mathcal{C}^2_{2,1}} < +\infty$ from Proposition \ref{prop:SolutionBackxard16Sept}  and the integrability of $\gamma_t \in \mathcal{P}_3(\R^{d_1} \times \R^{d_2})$ we easily show that  
$$ \sup_{t \in[t_0,T]} \E \bigl[ |Z_t^*| \bigr] + \E  \bigl[ | \nabla_{xx}^2 u_t^*(X_t^*,Y_0)| \bigr] < +\infty, $$
and then, using the equation \eqref{eq:equationhatZ26/06}, seen forward in time from the initial condition $\hat{\zeta}_{t_0}=0$ and applying Grönwall's Lemma we find that 
$$ \E \bigl[ | \hat{\zeta}_t| \bigr] \leq  C \int_{t_0}^t \Bigl( \E \bigl[ |\hat{\zeta}_s  \vert \bigr] + \norm{ \rho_s}_{(\mathcal{C}^1_3)^*} \Bigr) ds.  $$
However, using the explicit formula for $\bd{\rho}$ from Proposition \ref{prop:differentiatinggammalambda09/02} we easily obtain 
$$ \norm{\rho_t}_{(\mathcal{C}^1_3)^*} \leq C \int_{t_0}^t \int_A (1+|a|)d|\eta_s|(a) ds.$$
Combined together this leads to 
$$ \E \bigl[ | \hat{\zeta}_t | \bigr] + \norm{\rho_t}_{(\mathcal{C}^1_3)^*} \leq C \int_{t_0}^t \Bigl( \E \bigl[ |\hat{\zeta}_s  \vert \bigr] + \norm{ \rho_s}_{(\mathcal{C}^1_3)^*} \Bigr) ds.$$
We deduce from Grönwall's Lemma again that 
$$ \E \bigl[ | \hat{\zeta}_t | \bigr] = \norm{\rho_t}_{(\mathcal{C}^1_3)^*} = 0, \quad  t \in [t_0,T].$$
Getting back to the equation \eqref{eq:numoinstildenu26/06} we deduce that $\eta_t= 0$ for all $t \in [t_0,T]$. 
Then, by 
\eqref{eq:vtrhotProp3.2}, we get $(\rho_t,v_t) = 0$ for all $t\in [t_0,T]$. 
\end{proof}

\subsection{Topological Properties of the Set $\mathcal{O}$.}

\label{sec:Oisopen04/03}

We now prove some further topological properties for the set $\mathcal{O}$.

\begin{prop}
    The set $\mathcal{O}$ is open in $[0,T] \times \mathcal{P}_3(\R^{d_1} \times \R^{d_2})$.
\label{prop:Oisopen}
\end{prop}

\begin{proof}
    Toward a contradiction, suppose that there exist $(t_0,\gamma_0) \in \mathcal{O}$ and a sequence $(t_0^n, \gamma^n_0)_{ n \geq 1} \notin \mathcal{O}$ converging in $[0,T] \times \mathcal{P}_3(\R^{d_1} \times \R^{d_2})$ to $(t_0,\gamma_0)$. Up to subsequences, there are two possibilities:
\begin{itemize}
    \item For all $n \geq 1$, there are two distinct minima of $J  ( (t_0^n, \gamma_0^n), \cdot  )$.
    \item For all $n \geq 1$, there is a unique minimum for $J  ((t_0^n, \gamma_0^n), \cdot  )$ but it is not stable. 
\end{itemize}
\vskip 4pt

\textit{Case 1.} 
For all $n \geq 1$, let $ \bd{\nu}^{*,n}$ and $\bd{\nu}^n$ be
two distinct minima of $J ( (t_0^n,\gamma_0^n) ,\cdot )$ with associated curve and multiplier $(\bd{\gamma}^{*,n}, \bd{u}^{*,n})$ and $(\bd{\gamma}^n,\bd{u}^n)$ respectively. We define the integrated relative entropy $\lambda_n^2 := \int_{t_0^n}^T \mathcal{E} ( \nu_t^n | \nu_t^{*,n})dt $ and the new variables $(\bd{\eta}^n,\bd{\rho}^n,\bd{v}^n) := \lambda_n^{-1} ( \bd{\nu}^n -\bd{\nu}^{*,n}, \bd{\gamma}^n - \bd{\gamma}^{*,n}, \bd{u}^n - \bd{u}^{*,n})$. By Lemma \ref{lem:relativeentropyasasqureddistance28/02}, there exists a constant $C >0$ independent of $n \in \mathbb{N}$ such that 
\begin{equation} 
\lambda_n^2 \leq C \sup_{t \in [t_0^n,T]} \norm{ \gamma^{*,n}_t - \gamma_t^n}^2_{(\mathcal{C}^2_2)^*}. 
\label{eq:lambdanbyrhon28/02}
\end{equation}
In particular, for the same $C>0$ as above and for all $n \in \mathbb{N}$,
\begin{equation} 
\sup_{t \in [t_0^n,T]} \norm{ \rho_t^n}_{(\mathcal{C}^2_2)^*} \geq 1/\sqrt{C}. 
\label{eq:lowerboundrhon28/02}
\end{equation}
By Lemma~\ref{lem:convergenceifuniqueness} (recalling that $(t_0,\gamma_0) \in \mathcal{O}$), we know that the two sequences $(\bd{\nu}^{*,n}, \bd{\gamma}^{*,n}, \bd{u}^{*,n})_{n \in \mathbb{N}}$ and $(\bd{\nu}^{n}, \bd{\gamma}^{n}, \bd{u}^{n})_{n \in \mathbb{N}}$ converge to $(\bd{\nu}^*, \bd{\gamma}^*, \bd{u}^*)$, in the same sense as specified in Lemma~\ref{lem:convergenceifuniqueness}. In particular, by \eqref{eq:lambdanbyrhon28/02}, $\lambda_n \rightarrow 0$ as $n \rightarrow +\infty$. We are precisely in the framework of Section \ref{sec:preparatorywork} (see {\bf Property (${\mathcal Q_0}$)}) and we can apply the results of Propositions \ref{prop:preliminaryworkprop1-28/02} and \ref{prop:preliminayworkProp2-28/02} to find a triple $(\bd{\eta}, \bd{\rho}, \bd{v}) \in \mathcal{A}^l(t_0) \times \mathcal{R}(t_0) \times \mathcal{C}([t_0,T], \mathcal{C}^1_1)$ solution to the linearized system \eqref{eq:nutProp3.2}-\eqref{eq:vtrhotProp3.2} such that $(\bd{\eta}^n, \bd{\rho}^n,\bd{v}^n)$ converges, up to a subsequence, to $(\bd{\eta}, \bd{\rho}, \bd{v})$ in the sense of Proposition \ref{prop:preliminaryworkprop1-28/02}. Since $(t_0,\gamma_0)$ belongs to $\mathcal{O}$, the limit triple $(\bd{\nu}, \bd{\rho}, \bd{v})$ is necessarily equal to $(0,0,0)$. In particular, by Proposition \ref{prop:preliminaryworkprop1-28/02} again, $\lim_{n \rightarrow +\infty} \norm{ \rho_t^n}_{(\mathcal{C}^2_2)^*} =0$. This is in contradiction with \eqref{eq:lowerboundrhon28/02}. 
\vskip 4pt

\textit{Case 2.} We now address the case where, for all $n \in \mathbb{N}$, the functional $J((t_0^n, \gamma_0^n), \cdot)$ admits a unique minimizer $\bd{\nu}^{*,n} := (\nu^{*,n}_t)_{t \geq t_0^n}$, which is not stable.
 This means that we can find $(\bd{\eta}^n, \bd{\rho}^n, \bd{v}^n):=(\eta_t^n, \rho_t^n,v_t^n)_{ t \in  [t_0^n,T]}$ a non-trivial solution to the linearized system around $(\bd{\nu}^{*,n}, \bd{\gamma}^{*,n}, \bd{u}^{*,n})$ where $ \bd{\gamma}^{*,n} := (\gamma_t^{*,n})_{t \geq t_0^n}$ is the trajectory associated to $\bd{\nu}^{*,n}$ and $\bd{u}^{*,n} := (u_t^{*,n})_{t \geq t_0^n}$ the associated multiplier. By a stability argument similar to Lemma \ref{lem:backthenitwasLemma6.2}, $\bd{\rho}^n$ cannot be identically $0$ (otherwise we would have $(\bd{\eta}^n ,\bd{\rho}^n,\bd{v}^n) = (0,0,0)$ ). We look at the system satisfied by 
$$ \frac{\eta_t^n}{\lambda_n}, \quad \frac{\rho_t^n}{\lambda_n}, \quad \frac{v_t^n}{\lambda_n}, \quad  t \in [t_0^n,T], \quad \quad  \lambda_n := \sup_{t \in [t_0^n,T]} \norm{ \rho_t^n}_{(\mathcal{C}_2^2)^*} >0,$$
and then proceed similarly as in the proof of the first step to obtain a contradiction by taking $n \rightarrow +\infty$ and getting a non-trivial solution to the linearized system.
\end{proof}

The next result shows that optimal solutions $\bd{\nu}^*$ for $J((t_0,\gamma_0),\cdot)$ are isolated when $(t_0,\gamma_0) \in \mathcal{O}$: there is a neighborhood of $\bd{\nu}^*$ with no other critical point. 

\begin{prop}
Take $(t_0,\gamma_0) \in \mathcal{O}$ with optimal solution $\bd{\nu}^*$ and associated curve and multiplier $(\bd{\gamma}^*, \bd{v}^*)$.  For some $r>0$, there is no other solution $(\bd{\nu}, \bd{\gamma}, \bd{u})$ to \eqref{eq:NOCfornut2Avril} - \eqref{eq:NOCforutgammat2Avril} with $\int_{t_0}^T \mathcal{E} \bigl( \nu_t | \nu^*_t \bigr) dt \leq r^2.$   
\label{prop:isolated}
\end{prop}

\begin{proof}
    Otherwise we can find a sequence $(\bd{\nu}^n, \bd{\gamma}^n, \bd{u}^n) \in \mathcal{A}(t_0) \times \mathcal{C}([t_0,T], \mathcal{P}_3(\R^{d_1} \times \R^{d_2})) \times \mathcal{C}([t_0,T], \mathcal{C}^1_2)$ solution to \eqref{eq:NOCfornut2Avril}-\eqref{eq:NOCforutgammat2Avril} but distinct from $(\bd{\nu}^*, \bd{\gamma}^*, \bd{u}^*)$  such that
    $$\lim_{n \rightarrow +\infty} \Bigl \{ \lambda_n^2 :=  \int_{t_0}^T \mathcal{E} \bigl( \nu_t^n | \nu^*_t \bigr) dt \Bigr \}=0, $$
and $\lambda_n >0$ for all $n \in \mathbb{N}$. We are precisely in the framework of Section \ref{sec:preparatorywork} with $(t_0^n, \gamma_0^n) = (t_0,\gamma_0)$ and $\bd{\nu}^{*,n} = \bd{\nu}^*$ for all $n \in \mathbb{N}$ therein. The rest of the proof is identical to the first case in the proof of Proposition \ref{prop:Oisopen}.
\end{proof}

\subsection{Discriminating Property and the Universal Approximation Theorem}

\label{subse:DPandUAT}

In this subsection we explain the link between the discriminating property \assdis $ $ and the universal approximation theorem within the prototypical example \ref{ex:protoex}. First we state a useful, but strictly equivalent form of \assdis.

\vskip 4pt

\hypertarget{ass:dis2}{\noindent \textbf{Discriminating Property - Equivalent Formulation.}}
For any probability measure 
$\pi$ on ${\mathbb R}^{d_1} \times \R^{d_1}$ such that 
$\int_{{\mathbb R}^{d_1} \times 
{\mathbb R}^{d_1}} \vert z\vert d \pi(x,z) 
< \infty$, the following 
implication holds true:
\begin{equation*} 
\biggl( \forall a \in A, \quad\int_{{\mathbb R}^{d_1} \times {\mathbb R}^{d_1}}
\Bigl[
b(x,a) \cdot z 
\Bigr] d \pi(x,z) 
=0
\biggr) 
\Rightarrow 
\biggl( 
{\rm for} \  
\pi_{\rm x}\textrm{\rm -a.e.} \  
x \in {\mathbb R}^{d_1}, 
\quad 
\int_{{\mathbb R}^{d_1}}
z \pi(x,dz) =0
\biggr),
\end{equation*} 
where $\pi_{\rm x}$ denotes the first marginal law of 
$\pi$ on ${\mathbb R}^{d_1}$ (i.e., 
$\pi_{\rm x} := \pi \circ ((x,z) \mapsto x)^{-1}$)
and $x \in {\mathbb R}^{d_1} \mapsto \pi(x,\cdot)
\in {\mathcal P}({\mathbb R}^{d_1})$ is a measurable mapping obtained by disintegrating 
$\pi$ with respect to the first marginal (i.e., 
for any two Borel subsets $E_{\rm x}$ and $E_{\rm z}$ of 
${\mathbb R}^{d_1}$, $\pi(E_{\rm x} \times E_{\rm z}) = \int_{E_{\rm x}} \pi(x,E_{\rm z}) d\pi_{\rm x}(x)$). 
\vskip  4pt

The following statement shows that, within 
the prototypical example \ref{ex:protoex}, 
$b$ satisfies the discriminating property 
if the activation function $\sigma$ satisfies the conclusion of the universal approximation theorem
 (notice that $a_0$ below is a scalar whilst we took it as 
a vector in 
\eqref{eq:prototype:example}):

\begin{lem}
\label{lem:universal:to:discriminating}
Within the prototypical example \ref{ex:protoex}, assume that 
the closure, for the supremum norm over 
${\mathbb R}^{d_1}$, of the linear span of the 
set $\{x \mapsto a_0 \sigma(a_1 \cdot x + a_2), \ (a_0,a_1,a_2) \in {\mathbb R} \times {\mathbb R}^{d_1} 
\times {\mathbb R}
\}$ contains the set ${\mathcal C}_0({\mathbb R}^{d_1})$
of continuous functions on ${\mathbb R}^{d_1}$ vanishing at $\infty$, then 
$b$ in 
\eqref{eq:prototype:example} satisfies \assdis. 
\end{lem}

In words, the assumption of Lemma 
\ref{lem:universal:to:discriminating} may be formulated as follows: for any function $f \in  \mathcal{C}_0(\R^{d_1})$ and any $\epsilon >0$, 
there exist
an integer
$m \geq 1$
and a tuple $(a^j:=(a_0^{j},a_1^{j},a_2^{j}))_{ 1 \leq j \leq m} \in ({\mathbb R} \times 
{\mathbb R}^{d_1} \times {\mathbb R})^m$ such that 
    $$ \sup_{ x \in \R^{d_1}} \biggl| f(x) - \sum_{j=1}^m a_0^{j} \sigma (a_1^{j} \cdot x + a_2^{j}) \biggr| \leq \epsilon.$$

Compared to the standard formulation of the universal approximation result, the above approximation property is slightly unusual because the state space (over which the supremum norm is taken) is non-compact. However, most of the classical activation functions $\sigma$ satisfy the above statement, see for instance 
Itô
\cite{ITO1992105}. In particular, all our assumptions are satisfied if $\sigma$ is the hyperbolic tangent or the logistic function.

\begin{proof}
Take a probability measure 
$\pi$ on ${\mathbb R}^{d_1} \times  {\R^{d_1}}$ such that 
$\int_{{\mathbb R}^{d_1} \times 
{\mathbb R}^{d_1}} \vert z\vert d \pi(x,z) 
< \infty$. Assume that 
\begin{equation} 
\forall a \in A, \quad \int_{{\mathbb R}^{d_1} \times {\mathbb R}^{d_1}} \Bigl[ b(x,a) \cdot z \Bigr]
d \pi(x,z) =0. 
\label{eq:7juillet20:19}
\end{equation}
We first prove that, for any continuous function 
$f$ from ${\mathbb R}^{d_1}$ to ${\mathbb R}^{d_1}$ vanishing at infinity,
$\int_{{\mathbb R}^{d_1} \times {\mathbb R}^{d_1}}
[ f(x) \cdot z] d\pi(x,z)=0$. 
To do so, it suffices to prove that, 
for any $\varepsilon >0$, 
\begin{equation} 
\label{eq:7juillet20:19:F}
\biggl\vert \int_{{\mathbb R}^{d_1} \times {\mathbb R}^{d_1}}
\bigl[ f(x) \cdot z \bigr] d\pi(x,z) \biggr\vert \leq \varepsilon. 
\end{equation}
By the universal approximation property,
we know that, for any coordinate 
$i \in \{1,\cdots,d_1\}$, 
there exist
an integer $m_i \geq 1$
and a tuple 
$(a_0^{i,j},a_1^{i,j},a_2^{i,j})_{1 \leq j \leq m_i} \in ({\mathbb R} \times {\mathbb R}^{d_1} \times {\mathbb R})^{m_i}$ such that 
$$ \sup_{ x \in \R^{d_1}} \biggl| \sum_{j=1}^{m_i} a_0^{i,j} \sigma( a_1^{i,j} \cdot x + a_2^{i,j})  - f^i(x)\biggr| \leq \varepsilon,$$
where $f^i$ denotes the $i$th coordinate of $f$. Writing $e^i$ for the $i$th vector 
of the canonical basis of ${\mathbb R}^{d_1}$, the above can be reformulated as 
\begin{equation*}
\begin{split}
&\sup_{ x \in \R^{d_1}} \biggl| \sum_{i=1}^{d_1} \sum_{j=1}^{m_i} b\biggl( x,\bigl( a_0^{i,j} e^i, a_1^{i,j}, a_2^{i,j} \bigr) 
\biggr)- f(x)\biggr|
 = \sup_{ x \in \R^{d_1}} \biggl| \sum_{i=1}^{d_1} \sum_{j=1}^{m_i} a_0^{i,j} \sigma( a_1^{i,j} \cdot x + a_2^{i,j}) e^i  - f(x)\biggr| \leq c \varepsilon,
\end{split}
\end{equation*}
for a constant $c$ only depending on $d_1$. 
And therefore, \eqref{eq:7juillet20:19} leads to
\eqref{eq:7juillet20:19:F}. As announced, we deduce that, for any $f\in \mathcal{C}_0(\R^{d_1},\R^{d_1})$, $\int_{{\mathbb R}^{d_1} \times {\mathbb R}^{d_1}} [f(x) \cdot z] d \pi(x,z) = 0$. 
Since $\int_{{\mathbb R}^{d_1} \times {\mathbb R}^{d_1}} \vert z \vert d \pi(x,z) < \infty$, the latter integral is well-defined and, in fact, by a standard approximation argument, the identity 
is true for any bounded and measurable function 
$f$ from ${\mathbb R}^{d_1}$ into itself. In the end, we have shown that, for any 
such $f$, 
\begin{equation*}
\int_{{\mathbb R}^{d_1}} \biggl[ f(x) \cdot \biggl( \int_{{\mathbb R}^{d_1}} z \pi(x,dz)
\biggr) \biggr] d\pi_{\rm x}(x) = 0. 
\end{equation*}
which shows that, for $\pi_{\rm x}$-almost every $x \in {\mathbb R}^{d_1}$, $\int_{{\mathbb R}^{d_1}} z \pi(x,dz) =0$.
\end{proof}

\section{Local Polyak--Lojasiewicz Condition}
\label{se:5}

The main purpose of this section is 
to state and prove a rigorous version of Meta-Theorem \ref{meta-thm:2}. For $(t_0,\gamma_0) \in [0,T] \times \mathcal{P}_2(\R^{d_1} \times \R^{d_2})$, a control $\bd{\nu} \in \mathcal{A}(t_0)$ and associated pair $(\bd{\gamma}, \bd{u})$ solution to the forward-backward system \eqref{eq:curvemulti28/02}, we associate the following functional:
\begin{equation} 
\begin{split}
&\mathcal{I} \bigl( (t_0,\gamma_0),
{\boldsymbol \nu}\bigr) 
\\
&:= \int_{t_0}^T \int_A \Bigl| \epsilon \nabla_a \log \nu_t(a) + \epsilon \nabla_a \ell(a) + \nabla_a \int_{\R^{d_1} \times \R^{d_2}} b(x,a) \cdot \nabla_x u_t(x,y) d \gamma_t(x,y) \Bigr |^2 d\nu_t(a)dt. 
\label{eq:defmathcalI}
\end{split}
\end{equation}
It is implicitly understood the left-hand side is equal to $+\infty$ if, 
for $t$ in a non-null Borel subset of $[t_0,T]$, $\nu_t$ is not absolutely continuous with respect to the Lebesgue measure or if $\nu_t$ is absolutely continuous but the root $\sqrt{\nu_t}$ does not belong to $H^1(A)$ (the subset of $L^2(A)$ with weak derivative in $L^2(A)$). There is  another interpretation 
of ${\mathcal I}$, which is very useful in the proof
of Theorem 
\ref{prop:main:local:polyak:lojasiewicz}.
For ${\boldsymbol \nu} \in 
{\mathcal A}(t_0)$
we recall the notation $\Gamma[\bd{\nu}] \in \mathcal{A}(t_0)$ introduced in \eqref{eq:defnGammaNut09/01}:

\begin{equation}
\Gamma_t[\bd\nu](a) :=  \frac1{z_t} 
\exp \biggl( -\ell(a) - \frac{1}{\epsilon}  \int_{\R^{d_1} \times \R^{d_2}} b(x,a) \cdot \nabla_x u_t(x,y) d\gamma_t(x,y) \biggr), \quad a \in A, 
\quad t \in [t_0,T],
\label{eq:nutinfty}
\end{equation}
where $z_t$
is a normalization constant 
(similar to
$z^*_t$ in 
Remark 
\ref{rmk:3.2}).
In particular, 
taking the logarithm in 
\eqref{eq:nutinfty},
we can rewrite $\mathcal{I}$ as a Fischer information:
\begin{equation} 
\mathcal{I} \bigl( (t_0,\gamma_0),
{\boldsymbol \nu} \bigr) = \epsilon^2 \int_{t_0}^T \int_A \Bigl| \nabla_a \log \frac{\nu_t}{\Gamma_t[\bd\nu]}(a) \Bigr|^2 d\nu_t(a)dt.
\label{eq:IasaFischerInfo03/03}
\end{equation}
We also recall, see Lemma \ref{lem:LSI}, that $\Gamma[\bd{\nu}]$ satisfies a log-Sobolev inequality.

At this stage, we notice that, for any $(t_0,\gamma_0) \in [0,T] \times \mathcal{P}_3(\R^{d_1} \times \R^{d_2})$, $\mathcal{I}((t_0,\gamma_0),{\boldsymbol \nu}) = 0$ if and only if $({\boldsymbol \nu}, {\boldsymbol \gamma},{\boldsymbol u})$ solves the system of optimality conditions \eqref{eq:NOCfornut2Avril}-\eqref{eq:NOCforutgammat2Avril}.
The result established in this section is to push the latter observation further. 
In words, 
the rigorous version of Meta-Theorem 
\ref{meta-thm:2}
stipulates that, for ${\boldsymbol \nu}$ in the neighborhood of a  stable minimizer 
$\boldsymbol \nu^*$ (in the sense of Definition 
\ref{defn:stablesolutions}), 
the functional ${\mathcal I}((t_0,\gamma_0),{\boldsymbol \nu})$ grows at least like the difference between
$J((t_0,\gamma_0),{\boldsymbol \nu})$
and 
$J((t_0,\gamma_0),\boldsymbol \nu^*)$.
This principle is referred to as a local Polyak--Lojasiewicz condition.
It 
takes the following form:

\begin{thm}
\label{prop:main:local:polyak:lojasiewicz}
    For every compact subset $\mathcal{K}$ of $\mathcal{O}$, there exist $r,c >0$ such that, for all $(t_0,\gamma_0) \in \mathcal{K}$ with associated stable solution $\boldsymbol \nu^* = (\nu^*_t)_{t_0 \le t \le T}$ and  ${\boldsymbol \nu} \in \mathcal{A}(t_0)$, it holds
    $$ \int_{t_0}^T \mathcal{E} \bigl( \nu_t | \nu^*_t \bigr) dt \leq r^2 \quad \Longrightarrow \quad \mathcal{I} \bigl( (t_0,\gamma_0) ,{\boldsymbol \nu} \bigr) \geq c \Bigl( J \bigl( (t_0,\gamma_0) ,{\boldsymbol \nu} \bigr) - J \bigl( (t_0, \gamma_0) , \boldsymbol \nu^* \bigr) \Bigr). $$
\end{thm}

The rest of this subsection is devoted to the proof of Theorem \ref{prop:main:local:polyak:lojasiewicz} which is divided in a series of steps.

\textit{Step 1. Contradicting the statement.}  
Generally speaking, we argue by contradiction.
The first point is thus to notice that if the conclusion of Theorem \ref{prop:main:local:polyak:lojasiewicz}
does not hold, then there exists 
a sequence of tuples $(r_n,c_n,t_0^n,\gamma_0^n,\boldsymbol \nu^{*,n})_{n \in {\mathbb N}}$ satisfying the 
following 
\vskip 4pt

\noindent {\bf Property (${\mathcal Q}$).}
\begin{enumerate}[(i)]
\item The sequences $(r_n)_{n \in {\mathbb N}}$ and $(c_n)_{n \in {\mathbb N}}$ are positive valued sequences converging to $0$ as $n  \rightarrow +\infty$;
\item For each $n \in {\mathbb N}$, 
 $(t_0^n, \gamma_0^n)$ belongs to  $\mathcal{K}$ (and therefore to $\mathcal{O}$); as such, $J((t_0^n,\gamma_0^n),\cdot)$ 
 has a unique stable solution denoted $\boldsymbol \nu^{*,n}=(\nu_t^{*,n})_{t_0^n \leq t \leq T}$;
 \item For each $n \in {\mathbb N}$, there exists ${\boldsymbol \nu}^n \in \mathcal{A}(t_0^n)$ such that
\begin{equation} 
\int_{t_0^n}^T \mathcal{E} \bigl(\nu_t^n | \nu_t^{*,n} \bigr) dt \leq r_n^2 \quad \mbox{ and } \quad  \mathcal{I} \bigl( (t_0^n,\gamma_0^n),{\boldsymbol \nu}^n \bigr) < c_n \Bigl( J \bigl( (t_0^n,\gamma_0^n), {\boldsymbol \nu}^n \bigr) - J \bigl( (t_0^n, \gamma_0^n), \boldsymbol \nu^{*,n} \bigr) \Bigr);
\label{eq:19Mai15:07}
\end{equation}
\item The sequence $(t_0^n,\gamma_0^n)_{ n \in {\mathbb N}}$ converges in $[0,T] \times \mathcal{P}_3(\R^{d_1} \times \R^{d_2})$
toward some $(t_0,\gamma_0)$ in ${\mathcal O}$.
\end{enumerate}

Notice that item (iv) is somewhat for free since $\mathcal{K}$ is a
compact subset of ${\mathcal O}$. Notice in particular that  $(t_0,\gamma_0)$ is assumed to be in ${\mathcal O}$.
\vskip 4pt

Because $(t_0,\gamma_0)$ belongs to ${\mathcal O}$, 
we can apply Lemma \ref{lem:convergenceifuniqueness}
and deduce

\begin{lem}
\label{lem:Q1} 
Under \textbf{\bf Property (${\mathcal Q}$)}, 
denote by $\boldsymbol \gamma^{*,n}$ and 
$\boldsymbol u^{*,n}$ the curve and multiplier associated with 
each $\boldsymbol \nu^{*,n}$. Then, the sequence $(\boldsymbol \nu^{*,n},\boldsymbol \gamma^{*,n},\boldsymbol u^{*,n})_{n \in {\mathbb N}}$
converges strongly (i.e., in the same sense as in the statement of
Lemma \ref{lem:convergenceifuniqueness}) toward $(\boldsymbol \nu^*,\boldsymbol \gamma^*,
\boldsymbol u^*)=
(\nu^*_t, \gamma^*_t, u^*_t)_{t_0 \le t \leq T}$, 
with $\boldsymbol \nu^*$ denoting 
the (unique) minimizer of $J ( (t_0,\gamma_0), \cdot )$, 
and $\boldsymbol \gamma^*$ and $\boldsymbol u^*$
denoting the corresponding curve and multiplier. 
\end{lem}

So far, we have just used items (i), (ii) and (iv) in 
\textbf{Property (${\mathcal Q}$)}. We now make use of 
item (iii). For a given $n \in \mathbb{N}$, we call $({\boldsymbol \gamma}^n,
{\boldsymbol u}^n)$ the solution to \eqref{eq:curvemulti28/02} associated with ${\boldsymbol \nu}^n$. Also, we introduce the integrated relative entropy
\begin{equation}
\lambda_n^2:= \int_{t_0^n}^T \mathcal{E} \bigl( \nu_t^n | \nu_t^{*,n} \bigr)   dt 
\label{eq:defnlambdan}.
\end{equation}
Notice that the second equation in \eqref{eq:19Mai15:07} implies that 
${\boldsymbol \nu}^n \neq \boldsymbol \nu^{*,n}$ and therefore $\lambda_n >0$. On the other hand, the first equation in \eqref{eq:19Mai15:07} together with the fact that the sequence $(r_n)_{n \in \mathbb{N}}$ converges to $0$ shows that $(\lambda_n)_{n \in \mathbb{N}}$ converges to $0$ as well. Together with Lemma \ref{lem:Q1} this shows that we are precisely in the setting of Section \ref{sec:preparatorywork}. In particular, {\bf Property $(\mathcal{Q}_0)$} therein holds. 

The first step toward a contradiction is the following observation.
\begin{lem}
\label{lem:step:2}
Under \textbf{\bf Property (${\mathcal Q}$)}
and with the notation \eqref{eq:defnlambdan} for $\lambda_n$,
we have
\begin{equation}
\lim_{ n \rightarrow +\infty} \frac{1}{\lambda_n^2} \mathcal{I}\bigl((t_0^n, \gamma_0^n),{\boldsymbol \nu}^n\bigr) = 0.
\label{eq:29MAI12:17}
\end{equation}
\end{lem}

\begin{proof}
    By Lemma \ref{lem:expensioncost02/03} we know that 
    $$ J \bigl( (t_0^n,\gamma_0^n), \bd{\nu}^n \bigr) - J \bigl( (t_0^n,\gamma_0^n), \bd{\nu}^{*,n} \bigr) = O(\lambda_n^2) $$
as $n \rightarrow +\infty.$ Dividing the second equation in \eqref{eq:19Mai15:07} by $\lambda_n^2$ and recalling that $c_n \rightarrow 0$ as $n \rightarrow +\infty$ we get the result.
\end{proof}
Since $\lambda_n >0$ for all $n \in \mathbb{N}$ we can introduce the normalized variables
\begin{equation} 
\eta_t^n := \frac{\nu_t^n - \nu^{*,n}_t}{\lambda_n}, \quad \rho_t^n := \frac{\gamma^n_t - \gamma^{*,n}_t}{\lambda_n}, \quad v_t^n := \frac{u_t^n - u^{*,n}_t}{\lambda_n}, \quad t \in [t_0^n,T].  
\label{eq:normalizednewvariables}
\end{equation}
We can now apply Proposition \ref{prop:preliminaryworkprop1-28/02} in Section \ref{sec:preparatorywork} to find $(\bd{\eta}, \bd{\rho}, \bd{v}) \in \mathcal{A}^l(t_0) \times \mathcal{R}(t_0) \times \mathcal{C}([t_0,T], \mathcal{C}^1_1)$ solution to 
\begin{equation}
    \left \{
    \begin{array}{ll}
    -\partial_t v_t - b(x,\nu^*_t) \cdot \nabla_x v_t = b (x, \eta_t) \cdot \nabla_x u^*_t \quad &\mbox{\rm in }  [t_0,T] \times \R^{d_1} \times \R^{d_2}, 
    \\
     \qquad v_T = 0 \quad &\mbox{\rm in } \R^{d_1} \times \R^{d_2};
    \\
     \partial_t \rho_t + \div_x \bigl( b(x,\nu^*_t) \rho_t \bigr) = - \div_x ( b(x,\eta_t) \gamma^*_t ) \quad &\mbox{\rm in }  (t_0,T) \times \R^{d_1} \times \R^{d_2}, 
     \\
     \qquad \rho_{t_0} = 0 \quad &\mbox{\rm in } \R^{d_1} \times \R^{d_2};
    \end{array}
    \right.
\label{eq:v:rho:proof:Meta2.2}
\end{equation}
such that $(\bd{\eta}^n, \bd{\rho}^n, \bd{v}^n)_{n \in \mathbb{N}}$ converges, up to a sub-sequence, to $(\bd{\eta}, \bd{\rho}, \bd{v})$ in the sense of Proposition \ref{prop:preliminaryworkprop1-28/02}.

\textit{Step 2. Proving that $({\boldsymbol \eta}, {\boldsymbol \rho},{\boldsymbol v}) = (0,0,0)$.}

\begin{prop}
\label{prop:step:3}
Under \textbf{\bf Property (${\mathcal Q}$)}
and with the notations introduced in \textit{Step 1}, 
it holds that 
$({\boldsymbol \eta},{\boldsymbol  \rho},{\boldsymbol v}) = (0,0,0)$.
\end{prop}

\begin{proof}
We are going to show that $(\bd{\eta}, \bd{\rho}, \bd{v})$ is solution to the linearized system \eqref{eq:nutProp3.2}-\eqref{eq:vtrhotProp3.2} and conclude by stability of $\bd{\nu}^*$ since stable solutions are precisely those for which $(0,0,0)$ is the only solution to \eqref{eq:nutProp3.2}-\eqref{eq:vtrhotProp3.2}.  Following 
\eqref{eq:IasaFischerInfo03/03}, we let, for all $n \in {\mathbb N}$ and
 $t \in [t_0^n,T]$,
\begin{equation}
\label{eq:expression:nut:n,infinity}
 \Gamma_t[ \bd\nu^n](a) = \frac1{z_t^{n}}  \exp 
\biggl( -\ell(a) - \frac{1}{\epsilon} \int_{\R^{d_1} \times \R^{d_2}} b(x,a) \cdot \nabla_xu_t^n(x,y) d \gamma_t^n(x,y) \biggr), \quad a \in A,
\end{equation}
where $z_t^{n}$ is a normalization constant. 
By Lemma \ref{lem:timetogotoGallia02/03}, $\norm{\bd{\nu}^n}_{\mathcal{D}(t_0^n)} = \norm{ \bd{\nu}^{*,n} + \lambda_n \bd{\eta}^n}_{\mathcal{D}(t_0^n)}$ is bounded independently from $n \in \mathbb{N}$. By the same Lemma, $\int_{\R^{d_1} \times \R^{d_2}} (|x|^2 + |y|^2)^{3/2}d\gamma_0^n(x,y)$ is also bounded independently from $n \in \mathbb{N}$. Therefore we can apply Lemma \ref{lem:LSI} and deduce that $\Gamma_t[\bd{\nu}^n]$ satisfies a log-Sobolev inequality with constant independent from $t \in [t_0^n,T]$ and $n \in \mathbb{N}$. Thanks to \eqref{eq:29MAI12:17} 
in the statement of 
Lemma 
\ref{lem:step:2}
and to the log-Sobolev inequality 
(see Lemma 
\ref{lem:LSI}, with an additional approximation argument 
allowing us to choose $f\equiv \nu_t^n / \Gamma_t[\bd{\nu}^n]$), 
we deduce that
$$ \lim_{ n \rightarrow +\infty} \frac{1}{\lambda_n^2} \int_{t_0^n}^T \int_A \log \frac{\nu_t^n(a)}{\Gamma_t[\bd\nu^n] (a)} d\nu_t^n(a)dt = \lim_{n \rightarrow +\infty} \frac{1}{\lambda_n^2} \mathcal{I} \bigl( (t_0^n,\gamma_0^n), \bd{\nu}^n \bigr) =  0.  $$
Using Pinsker's inequality (\cite[(5.2.2)]{BakryGentilLedoux}), this leads to
\begin{equation} 
\lim_{n \rightarrow +\infty} \frac{1}{\lambda_n^2} \int_{t_0^n}^T \biggl( \int_A \bigl| \nu_t^n (a) - \Gamma_t[\bd\nu^n](a) \bigr|da \biggr)^2dt =0.
\label{eq:Pinsker30Mai}
\end{equation}
We now return back 
to the definition of $\bd{\eta}^n$ in
\eqref{eq:normalizednewvariables}, from which we obtain the decomposition:
$$ \eta_t^n(a) = \frac{1}{\lambda_n} \Bigl( \nu_t^n(a) - \Gamma_t[\bd\nu^n](a) \Bigr) + \frac{1}{\lambda_n}\Bigl(  \Gamma_t[\bd\nu^n](a) -\nu_t^{*,n}(a) \Bigr), \quad (t,a) \in [t_0^n,T] \times A.  $$
Introducing the variable
$$ k_t(a) := \frac{1}{\epsilon} \int_{\R^{d_1} \times \R^{d_2}} b(x,a) \cdot \nabla_x v_t(x,y) d\gamma^*_t(x,y)  + \frac{1}{\epsilon} \langle b(\cdot,a) \cdot \nabla_x u^*_t; \rho_t \rangle, \quad (t,a) \in [t_0,T] \times A $$
and the constant $c_t = \int_A k_t(a) d\nu_t^*(a)$ we deduce
\begin{equation}
\label{eq:somecomputation30Mai}
\begin{split}
\int_{t_0^n \vee t_0}^T \biggl( \int_A \Bigl| &\eta_t^n(a) + \nu_t^*(a) \bigl(k_t(a) - c_t \bigr)  \Bigr| da \biggr)^2dt  \leq \frac{2}{\lambda_n^2} \int_{t_0^n \vee t_0}^T  \biggl( \int_A \bigl| \nu_t^n (a) - \Gamma_t[ \bd \nu^n](a) \bigr|da \biggr)^2dt 
    \\
&\hspace{15pt} + 2 \biggl\{ \int_{t_0^n \vee t_0}^T \biggl( \int_A \Bigl| \frac{1}{\lambda_n} \bigl(\Gamma_t[\bd\nu^n](a) - \nu_t^{*,n}(a) \bigr) +\nu^*_t(a) \bigl(k_t(a) -c_t \bigr) \Bigr|da \biggr)^2dt \biggr\}.
\end{split}
\end{equation} 
The first term in the right-hand side is handled by \eqref{eq:Pinsker30Mai} and the second term by Proposition \ref{prop:cvgcesumofIs03/03}
and we get 
\begin{equation}
\lim_{n \rightarrow +\infty}
\int_{t_0^n \vee t_0}^T 
\biggl( \int_A
\Bigl\vert 
\eta_t^n(a) + \nu_t^*(a) (k_t(a) -c_t) \Bigr\vert da \biggr)^2 dt =0. 
\label{eq:abovedisplay25/02}
\end{equation}
 We now apply Proposition \ref{prop:preliminaryworkprop1-28/02} and specifically the convergence it provides for 
 $(\bd{\eta}^n)_{n \in \mathbb{N}}$. We choose  in \eqref{eq:testfnnutn03/03}  therein a smooth test function $f$  with compact support included in 
$(t_0,T] \times A$. 
We deduce from \eqref{eq:abovedisplay25/02} that 
\begin{equation*}
\int_{t_0}^T 
\int_A \varphi_t(a) \eta_t(a) 
da dt = - \int_{t_0}^T \int_A \varphi_t(a) (k_t(a) - c_t) \nu_t^*(a) dadt
\end{equation*}
which gives, for almost every 
$(t,a) \in [t_0,T] \times A$, 
$$ \eta_t(a) = -\nu_t^*(a) ( k_t(a) -c_t).$$
Together with 
\eqref{eq:v:rho:proof:Meta2.2}, 
this means that $({\boldsymbol \eta},{\boldsymbol \rho},{\boldsymbol  v})$ is a solution to the linearized system \eqref{eq:nutProp3.2}-\eqref{eq:vtrhotProp3.2} and therefore -since $(t_0,\gamma_0)$  belongs to $\mathcal{O}$- we have $({\boldsymbol \eta},{\boldsymbol  \rho},{\boldsymbol v}) =(0, 0,0)$. 

\end{proof}

\textit{Step 3. Conclusion.}
We now complete the proof of 
Theorem \ref{prop:main:local:polyak:lojasiewicz}. We use the latter statement to establish:
\begin{prop}
There is no sequence 
$(r_n,c_n,t_0^n,\gamma_0^n,\boldsymbol \nu^{*,n})_{n \in {\mathbb N}}$
satisfying 
\textbf{\bf Property (${\mathcal Q}$)}, i.e., 
\textbf{\bf Property (${\mathcal Q}$)} is empty. 
In particular, Theorem 
\ref{prop:main:local:polyak:lojasiewicz} holds true. 
\end{prop}

\begin{proof}
Under \textbf{\bf Property (${\mathcal Q}$)}, with the notations introduced in  
 \textit{Step 2} we write $\mathcal{I}$ as a Fischer information, as in \eqref{eq:IasaFischerInfo03/03} 
\begin{equation*}
\begin{split}
\mathcal{I} \bigl( (t_0,\gamma_0),
{\boldsymbol \nu}^n \bigr) = \epsilon^2 \int_{t_0^n}^T \int_A \Bigl| \nabla_a \log \frac{\nu_t^n}{\Gamma_t[\bd \nu^n]}(a) \Bigr|^2 d\nu_t^n(a)dt,
\end{split}
\end{equation*}
which we rewrite
$$\mathcal{I} \bigl( (t_0,\gamma_0),
{\boldsymbol \nu}^n \bigr) = \epsilon^2 \int_{t_0^n}^T \int_A \Bigl| \nabla_a \log \frac{\nu_t^n}{ \nu_t^{*,n}}(a)
-
\nabla_a \log \frac{ \Gamma_t[ \bd \nu^n]}{\nu_t^{*,n}}
(a)
\Bigr|^2 d\nu_t^n(a)dt. $$
We recall from \eqref{eq:rewritinggammanutn28/02} in Lemma \ref{lem:rewriting02/02} that the probability measure $\Gamma_t[\bd{\nu}^n]$ can be rewritten as 
$$ \Gamma_t[\bd{\nu}^n] \propto \nu_t^{*,n} \exp \Bigl( -\frac{\lambda_n}{\epsilon} \int_{\R^{d_1} \times \R^{d_2}} b(x,\cdot) \cdot d(\nabla_x u_t^{*,n} \rho_t^n + \nabla_x v_t^n\gamma_t^n)(x,y) \Bigr),  $$
and then, by Young's inequality 
\begin{equation}
\label{eq:30Mai:15:24}
\begin{split}
    & \epsilon^2 \int_{t_0^n}^T  \int_A \Bigl|  \nabla_a \log \frac{\nu_t^n}{\nu^{*,n}_t} (a) \Bigr|^2 d\nu_t^n(a) dt \leq 2 \mathcal{I}(t_0^n,\gamma_0^n,{\boldsymbol \nu}^n) 
    \\
 &\hspace{15pt} + 2 \lambda_n^2 \int_{t_0^n}^T \int_A \Bigl| \nabla_a \int_{\R^{d_1} \times \R^{d_2}} b(x,a) \cdot d(\nabla_x v_t^n \gamma_t^n + \nabla_x u^{*,n}_t \rho_t^n)(x,y) \Bigr|^2 d \nu^n_t(a) dt.
\end{split}
\end{equation}
From Lemma \ref{lem:step:2}, the second part of Proposition \ref{prop:convergencektn03/03} and display \eqref{eq:30Mai:15:24}, we conclude that 
$$ \lim_{n \rightarrow +\infty} \frac{1}{\lambda_n^2} \int_{t_0^n}^T  \int_A  \Bigl|  \nabla_a \log \frac{\nu_t^n}{\nu^{*,n}_t} (a) \Bigr|^2 d\nu_t^n(a) dt = 0.$$
However, by log-Sobolev inequality for $\boldsymbol \nu^{*,n}$ (see Lemma 
\ref{lem:LSI}), 
$$ \lambda_n^2 = \int_{t_0^n}^T \int_A \log \frac{\nu_t^n(a)}{\nu^{*,n}_t(a)} d\nu_t^n(a)dt \leq C \int_{t_0^n}^T  \int_A \Bigl| \epsilon \nabla_a \log \frac{\nu_t^n}{\nu^{*,n}_t} (a) \Bigr|^2 d\nu_t^n(a) dt, $$
for some $C >0$ independent from $n \in \mathbb{N}$. Dividing by $\lambda^2_n$ and letting $n \rightarrow +\infty$ we obtain
$$ 1 \leq \lim_{n \rightarrow +\infty} \frac{C}{\lambda_n^2} \int_{t_0^n}^T  \int_A  \Bigl|  \nabla_a \log \frac{\nu_t^n}{\nu^{*,n}_t} (a) \Bigr|^2 d\nu_t^n(a) dt = 0, $$
which is the desired contradiction.
\end{proof}

\color{black}

\section{Existence of Optimal Controls and First Order Conditions}

\label{sec:ExistenceandFOC}

In this section, we address the optimal control problem \eqref{eq:OriginalControlPb24/06}.
We first establish the existence of a minimizer
in Subsection 
\ref{subse:7:1}. 
In Subsection 
\ref{subse:7:2}, we give a rigorous proof of the first order  condition. Further properties of the first order system are established in 
Subsection 
\ref{sec:additionalregandfirststability}.

\subsection{Existence of Optimal Solutions}
\label{subse:7:1}

{Here our objective is to establish the following statement:}
\begin{prop}
\label{prop:existenceprop}
For any $(t_0,\gamma_0) \in [0,T] \times {\mathcal P}_2({\mathbb R}^{d_1} \times {\mathbb R}^{d_2})$, 
the minimization problem \eqref{eq:OriginalControlPb24/06} admits at least one solution.
\end{prop}

The proof relies on several preliminary technical lemmas. In particular we will need the properties of solutions to the continuity equation already presented in Lemma \ref{lem:compactnessContinuityEquation24/06}.  

We state the following result (already quoted in \eqref{eq:boundfromPinsker16Sept:sec:3}), which provides a bound for the fourth moment of ${\boldsymbol \nu} \in {\mathcal A}(t_0)$ in terms of the cost $J((t_0,\gamma_0),{\boldsymbol \nu})$. 
(We recall that $\nu^\infty$, which appears in the statement below, 
has been introduced in \eqref{eq:defnuinfty03/07}.)

\begin{lem}
\label{lem:cost:L4:moments}
There exists a constant $C>0$ such that, for any $t_0 \in [0,T]$ and ${\boldsymbol \nu} \in \mathcal{A}(t_0)$,
\begin{equation} 
\int_{t_0}^T \int_A |a|^4 d\nu_t(a)dt + \sup_{t_0 \leq t_1 < t_2 \leq T} 
\biggl\{ \frac{1}{\sqrt{t_2-t_1}} \int_{t_1}^{t_2} \int_A |a|^2 d\nu_t(a)dt
\biggr\} \leq C \biggl( 1+ 
\int_{t_0}^T {\mathcal E}(\nu_t \vert \nu^\infty) dt \biggr).
\label{eq:boundfromPinsker16Sept}
\end{equation}
As a consequence, there exist two constants $c,C$, with $c>0$, such that, for any $t_0 \in [0,T]$ and any 
${\boldsymbol \nu} \in {\mathcal A}(t_0)$, 
\begin{equation}
\label{eq:lem:cost:L4:moments:00}
J\bigl( (t_0,\gamma_0) , {\boldsymbol \nu}\bigr) \geq - C + c \int_{t_0}^T 
\int_A \vert a \vert^4 d \nu_t(a) dt. 
\end{equation}
In particular, the right-hand side is (uniformly) bounded
on sub-level sets 
of $J((t_0,\gamma_0),\cdot)$.
\end{lem}

\begin{proof}
We start with the following observation. Letting $g_{d'}(a) := (2\pi)^{-d'/2} \exp( - |a|^2/2)$ for $a \in A$ (with $d'$ denoting the dimension of $A$)
and using the non-negativity of the relative entropy between two probability measures, 
we deduce that there exists a constant $C$ such that, for any $\nu \in \mathcal{P}_2(A)$,
\begin{equation*}
\begin{split}
\mathcal{E}(\nu \vert \nu^\infty) 
= \mathcal{E}(\nu \vert g^{d'})
+ \int_A \log \biggl( \frac{g_d'(a)}{\nu^\infty(a)}
\biggr) d \nu(a) 
&\geq \int_A \log \biggl( \frac{g_d'(a)}{\nu^\infty(a)}
\biggr) d \nu(a)
\\
&\geq - C + \int_A  \Bigl( \ell(a)- \frac{a^2}2
\Bigr) d \nu(a).
\end{split}
\end{equation*}

Take now $t_0 \in [0,T]$
and ${\boldsymbol \nu} \in {\mathcal A}(t_0)$. 
Apply the above inequality with $\nu=\nu_t$. 
By 
the coercivity condition (ii) in \assreg \
and then
Young's inequality, we deduce that (for a possibly new value of the constant $C$)
\begin{equation*}
\int_{t_0}^T \int_A |a|^4 d\nu_t(a)dt  \leq C \biggl( 1+ 
\int_{t_0}^T {\mathcal E}(\nu_t \vert \nu^\infty) dt \biggr).
\end{equation*}
And then, by Cauchy-Schwarz inequality, 
we have, for any $t_1,t_2 \in [t_0,T]$
with $t_1 < t_2$, 
\begin{equation*}
\begin{split}
 \frac{1}{\sqrt{t_2-t_1}} \int_{t_1}^{t_2} \int_A |a|^2 d\nu_t(a)dt
&\leq 
\biggl\{ \int_{t_1}^{t_2} \int_A |a|^4 d\nu_t(a)dt
\biggr\}^{1/2}
 \leq C^{1/2} \biggl( 1+ 
\int_{t_0}^T {\mathcal E}(\nu_t \vert \nu^\infty) dt \biggr)^{1/2},
\end{split}
\end{equation*}
from which 
\eqref{eq:boundfromPinsker16Sept}
easily follows. 

We turn to the proof of 
\eqref{eq:lem:cost:L4:moments:00}. 
Recall 
\eqref{eq:deftotalcost}
for the definition of 
$J$. Since the function $L$ in the definition of $J$ is lower bounded, we deduce that there exists a constant 
$C$ (independent of $(t_0,\gamma_0)$ and ${\boldsymbol \nu}$) 
such that 
\begin{equation}
\label{eq:lem:cost:L4:moments:1}
- C +
\epsilon \int_{t_0}^T {\mathcal E}(\nu_t \vert \nu^\infty) 
dt 
 \leq 
J \bigl( (t_0,\gamma_0),{\boldsymbol \nu} \bigr).
\end{equation}

Using 
\eqref{eq:boundfromPinsker16Sept}, 
we complete the proof. 
\end{proof}

The next statement provides a very useful compactness and continuity result for the controls and their related trajectories and multipliers.

\begin{lem}
\label{lem:cost:convergence}
Let $(t_0,\gamma_0) \in [0,T] \times {\mathcal P_p}({\mathbb R}^{d_1} \times {\mathbb R}^{d_2})$ for some $p \geq 2$ and 
$({\boldsymbol \nu}^n)_{n \geq 1}$ be 
a sequence of elements of
${\mathcal A}(t_0)$ satisfying 
\begin{equation} 
\sup_{ n \geq 1} 
J\bigl( (t_0,\gamma_0),{\boldsymbol \nu}^n \bigr) < 
+ \infty.
\label{eq:CoercivityEstimate24/06:bis}
\end{equation}
Then, ($\bd{\nu}^n)_{ n \geq 1}$ admits some weak limit points and any weak limit 
${\boldsymbol \nu}$
belongs to ${\mathcal A}(t_0)$. 
Moreover, denoting by 
$({\boldsymbol \gamma}^n ,\bd{u}^n)_{n \geq 1}$  the curves and multipliers corresponding 
to $({\boldsymbol \nu}^n)_{n \geq 1}$, see equations
\eqref{eq:lem:compactnessContinuityEquation24/06:statement}  and \eqref{eq:adjoint:equation},
and denoting 
by $({\boldsymbol \gamma} , \bd{u})$ the curve and multiplier
corresponding to 
${\boldsymbol \nu}$, the following three properties hold true: 
\begin{enumerate}[(i)]
\item The sequence
$({\boldsymbol \gamma}^n)_{n \geq 1}$ is bounded in 
${\mathcal C}^{1/2}([t_0,T],{\mathcal P}_p({\mathbb R}^{d_1} \times {\mathbb R}^{d_2}))$ and the sequence $(\boldsymbol{u}^n)_{n \geq 1}$ is bounded in $\mathcal{C}^{1/2} \bigl( [t_0,T], \mathcal{C}^1_{2,1} \bigr) \cap \mathcal{L}^{\infty}([t_0,T], \mathcal{C}^3_{2,1})$;
In particular, the sequence 
$({\boldsymbol \gamma}^n)_{n \geq 1}$ is relatively compact in 
${\mathcal C}^{1/2}([t_0,T],{\mathcal P}_{p-\delta}({\mathbb R}^{d_1} \times {\mathbb R}^{d_2}))$ for any $\delta \in (0,p-1)$, and
the sequence 
$(\boldsymbol{u}^n)_{n \geq 1}$
and the sequence of its spatial derivatives are relatively compact for the
uniform topology on compact subsets of $[0,T] \times {\mathbb R}^{d_1} \times {\mathbb R}^{d_2}$;
\item Along any sub-sequence $(\varphi(n))_{n \geq 1}$ with $\varphi : \mathbb{N}^* \rightarrow \mathbb{N}^*$ strictly increasing, such that 
$({\boldsymbol \nu}^{\varphi(n)})_{n \geq 1}$
converges (in the weak sense) to 
${\boldsymbol \nu}$, 
$({\boldsymbol \gamma}^{\varphi(n)})_{n \geq 1}$ converges in 
${\mathcal C}([t_0,T],{\mathcal P}_{p - \delta}({\mathbb R}^{d_1} \times {\mathbb R}^{d_2}))$ to ${\boldsymbol \gamma}$ for any $\delta \in (0,p-1)$ and $(\bd{u}^{\varphi(n)})_{n \geq 1}$ converges to $\bd{u}$ in the following sense: for any closed ball $B$ in $\R^{d_1} \times \R^{d_2}$, it holds 
$$ \lim_{n \rightarrow +\infty} \sup_{t \in [t_0,T]} \norm{ u^{\varphi(n)}_t - u_t}_{\mathcal{C}^1(B)} = 0.  $$
\item $J((t_0,\gamma_0),{\boldsymbol \nu}) \leq \liminf_{n \rightarrow 
\infty}
J((t_0,\gamma_0),{\boldsymbol \nu}^n)$.
\end{enumerate}
\end{lem}
\begin{proof}
{ \ }
\vskip 4pt 

\textit{Step 1.}
For each $n \geq 1$, 
we insert the bound 
\eqref{eq:CoercivityEstimate24/06:bis}
in the inequality 
\eqref{eq:lem:cost:L4:moments:1}. 
Recalling the formula 
\eqref{eq:interpretation:wider:space}, we deduce that 
\begin{equation*}
\sup_{n \geq 1} {\mathcal E} \bigl( {\boldsymbol \nu}^n
\vert {\rm Leb}_{[t_0,T]} \times 
\nu^\infty \bigr) < \infty. 
\end{equation*}
Therefore, by 
\cite[Lemma 2.4]{BudhirajaDupuisbook},
the sequence
$({\boldsymbol \nu}^n)_{n \geq 1}$ (regarded up to a normalization by $T-t_0$ as a sequence of probability measures on $[0,T] \times A$) has a weakly converging sub-sequence, still denoted by $({\boldsymbol \nu}^n)_{n \geq 1}$, with limit point ${\boldsymbol \nu}$. 
Testing the convergence against $a$-independent test functions, we find that the time marginal of ${\boldsymbol \nu}$ is the {Lebesgue measure 
and we denote by 
$t \in [t_0,T] \mapsto  \nu_t \in {\mathcal P}(A)$
the disintegration of ${\boldsymbol \nu}$, i.e. 
$d {\boldsymbol \nu}(t,a) = d {\nu}_t(a) dt$ (by normalizing 
${\boldsymbol \nu}$ by $T-t_0$, we are reduced to the disintegration of 
a probability measure, from which we deduce 
the 
$(\nu_t)_{t_0 \le t \le T}$'s
are indeed probability measures)}.
By lower-semicontinuity of the entropy, see again 
\cite[Lemma 2.4]{BudhirajaDupuisbook}, we deduce that 
${\mathcal E} ( {\boldsymbol \nu}
\vert {\rm Leb}_{[t_0,T]} \times 
\nu^\infty)
< \infty$. 
This shows
that ${\boldsymbol \nu}$ belongs to 
${\mathcal A}(t_0)$. 
\vskip 4pt

\textit{Step 2.} 
Using the bound \eqref{eq:CoercivityEstimate24/06:bis} together with Lemma
\ref{lem:cost:L4:moments}, we
deduce that 
\begin{equation}
\label{eq:proof:lemma7:3:a4:bound}
\sup_{n \geq 1} \biggl \{  \int_{t_0}^T \int_A \vert a \vert^4 d \nu^n_t(a) dt + \sup_{t_1 < t_2 \in [t_0,T]} \frac{1}{\sqrt{t_2 - t_1}}\int_{t_0}^{t_2} \int_A (1+|a|^2) d\nu_t^n(a) dt \biggr \} < + \infty. 
\end{equation}
Property \textit{(i)} is then a consequence of Proposition \ref{lem:compactnessContinuityEquation24/06} and Proposition \ref{prop:SolutionBackxard16Sept}.
\vskip 4pt

\textit{Step 3.} We now prove item \textit{(ii)} in the statement. We start with the analysis of the limit points of the sequence $(\bd{\gamma}^n)_{n \geq 1}$ (whose existence is guaranteed by item \textit{(i)} in the statement). With the implicit convention that 
the sub-sequence $(\varphi(n))_{n \geq 1}$
in item \textit{(ii)} of the statement
is taken as the identity, i.e. $(\varphi(n)=n)_{n \geq 1}$, 
we are going to show that, for any $\delta \in (0,p-1)$, $(\bd\gamma^n)_{n\geq 1}$ converges in $\mathcal{C}([t_0,T], \mathcal{P}_{p-\delta}(\R^{d_1} \times \R^{d_2}))$ toward the unique weak solution ${\boldsymbol \gamma}$ of 
\begin{equation} 
\partial_t  {\gamma}_t + \div_x (  b(x,  {\nu}_t)  {\gamma}_t ) = 0, \quad  {\gamma}(t_0) = \gamma_0.
\label{eq:fpeexistenceproof}
\end{equation}
 We take $\bar{\boldsymbol \gamma} \in \mathcal{C}([t_0,T], \mathcal{P}_p(\R^{d_1} \times \R^{d_2}))$ as a limit point
  and for the ease of notation, do not relabel the subsequence here again. 

In order to prove that 
$\bar{\boldsymbol \gamma}$ satisfies 
\eqref{eq:fpeexistenceproof}, we
 fix $\varphi \in \mathcal{C}^{\infty}_c([t_0,T] \times \R^{d_1} \times \R^{d_2})$. For every $n \geq 1$, it holds
\begin{equation}
\label{eq:I_1n-I2n}
\begin{split}
 \int_{\R^{d_1} \times \R^{d_2}} \varphi_{t_0}(x,y)d\gamma_0(x,y) 
&= \int_{\R^{d_1} \times \R^{d_2}} \varphi_T(x,y) d\gamma^n_T(x,y)
-
 \int_{t_0}^T \int_{\R^{d_1}\times \R^{d_2}}   \partial_t \varphi_t(x,y)d\gamma^n_t(x)dt
\\
&\hspace{15pt}
 - \int_{t_0}^T \int_{\R^{d_1}\times \R^{d_2}}    b(x,\nu_t^n) \cdot \nabla_x \varphi_t(x,y)  
 d\gamma_t^n(x)dt
 \\
 &=: I_1^n - I_2^n.
 \phantom{\biggl(}
\end{split}
\end{equation}
\color{black}
By regularity of $\varphi$, we easily have that
\begin{align} 
\label{eq:lim:I_1n}
\lim_{ n \rightarrow +\infty} I_1^n = \int_{\R^{d_1} \times \R^{d_2}} \varphi_T(x,y) d\bar{\gamma}_T(x,y) - \int_{t_0}^T \int_{\R^{d_1} \times \R^{d_2}}  \partial_t \varphi_t(x,y)  d\bar{\gamma}_t(x,y)dt. 
\end{align}
We now handle $I_2^n$, which we rewrite in the form 
\begin{equation*}
I_2^n = \int_{t_0}^T \int_A \biggl[ \int_{{\mathbb R}^{d_1} \times 
{\mathbb R}^{d_2}}
b(x,a) \cdot \nabla_x \varphi_t(x,y) d \gamma_t^n (x,y) 
\biggr] d \nu_t^n(a) dt
=: \int_{t_0}^T \int_A f^n(t,a) d \nu_t^n(a) dt,
\end{equation*}
with an obvious definition for $f^n$. By item (i) in Assumption (\textbf{Regularity}), 
there exists a constant $C$ such that, 
for any $n \geq 1$, 
$\vert f^n(t,a)\vert \leq C(1+\vert a \vert)$.
Moreover,
because $\sup_{t \in [t_0,T]}
d_{p-\delta}(\gamma_t^n,\bar{\gamma}_t) \rightarrow 0$
as $n \rightarrow + \infty$, it holds, 
for any $(t,a) \in [t_0,T] \times A$, 
\begin{equation*}
\lim_{n \rightarrow + \infty}
f^n(t,a) = \int_{\R^{d_1}\times \R^{d_2}}
b(x,a) \cdot \nabla_x \varphi_t(x,y) d \bar{\gamma}_t(x,y) =: f(t,a), 
\end{equation*}
uniformly on compact subsets of 
$[0,T] \times A$. By combining the weak convergence of $\bd{\nu}^n$ toward $\bd{\nu}$ together with 
the $L^4$-bound 
\eqref{eq:proof:lemma7:3:a4:bound}, we easily deduce that 
\begin{equation*}
\begin{split}
\lim_{n \rightarrow \infty}
I^n_2 = \lim_{n \rightarrow \infty}
\int_{[t_0,T] \times A} 
f^n(t,a) 
d {\boldsymbol \nu}^n(t,a) 
&=
\int_{[t_0,T] \times A} 
f(t,a) 
d {\boldsymbol \nu}(t,a)
\\
&= \int_{t_0}^T \int_A \biggl[ \int_{\R^{d_1}\times \R^{d_2}}
b(x,a) \cdot \nabla_x \varphi_t(x,y) d \bar{\gamma}_t (x,y) 
\biggr] d \nu_t(a) dt. 
\end{split}
\end{equation*}
Inserting 
\eqref{eq:lim:I_1n} and the above display in 
\eqref{eq:I_1n-I2n},  
we deduce that
\begin{align*}
\int_{\R^{d_1}\times \R^{d_2}} \varphi_T(x,y) d \bar{\gamma}_T(x,y) &= \int_{\R^{d_1}\times \R^{d_2}}\varphi_{t_0}(x,y)d\gamma_0(x,y) \\
&+ \int_{t_0}^T \int_{\R^{d_1}\times \R^{d_2}}\bigl[ \partial_t \varphi_t(x,y) +  b(x,\nu_t) \cdot \nabla_x \varphi_t(x,y) \bigr] d \bar{\gamma}_t(x,y)dt.
\end{align*}
By the uniqueness result established in Proposition 
\ref{lem:compactnessContinuityEquation24/06}, 
 $\bar{\boldsymbol \gamma}$ coincides with the (unique) solution ${\boldsymbol \gamma}$ of the equation 
\eqref{eq:lem:compactnessContinuityEquation24/06:statement}. Moreover, the whole sequence $({\boldsymbol \gamma}_n)_{n \geq 1}$ converges to ${\boldsymbol \gamma}$
in 
${\mathcal C}([t_0,T],{\mathcal P}_{p-\delta}({\mathbb R}^{d_1} \times {\mathbb R}^{d_2}))$.
\vskip 4pt

\textit{Step 4.} We go on with the analysis of the limit points of $(\bd{u}^n)_{n \geq 1}$ (whose existence is guaranteed by item \textit{(i)} in the statement). We consider $\bar{u} \in \mathcal{C}^{1/2}([t_0,T], \mathcal{C}^1_{2,1})$ such that, along a sub-sequence $(\varphi(n))_{ n \geq 1}$ as in item \textit{(ii)} of the statement, 
for all closed ball $B$ of $\R^{d_1} \times \R^{d_2}$,
\begin{equation} 
\lim_{n \rightarrow +\infty} \sup_{t \in [t_0,T]} \norm{ u_t^{\varphi(n)} -\bar{u}_t}_{\mathcal{C}^1(B)} =0.
\label{eq:27Nov15:01}
\end{equation}
For all $(t,x,y) \in [t_0,T] \times \R^{d_1} \times \R^{d_2}$ and all $n \geq 1$ we have
$$ u_t^{\varphi(n)}(x,y) = L(x,y) + \int_t^T b(x,\nu_s^{\varphi(n)}) \cdot \nabla_x u_s^{\varphi(n)}(x,y) ds.$$
Using the weak convergence of $(\bd{\nu}^{\varphi(n)})_{n \geq 1}$ to $\bd{\nu}$,
together with the convergence \eqref{eq:27Nov15:01} and the uniform bound  \eqref{eq:proof:lemma7:3:a4:bound}, 
we deduce that, for all $(t,x,y) \in [t_0,T] \times \R^{d_1} \times \R^{d_2}$,
$$ \bar{u}_t(x,y) = L(x,y) + \int_t^T b(x,\nu_s) \cdot \nabla_x \bar{u}_s(x,y) ds,$$
and, 
by the uniqueness result established in Proposition 
\ref{prop:SolutionBackxard16Sept}, we conclude that the whole sequence $(\bd{u}^{\varphi(n)})_{n \geq 1}$ converges to the unique
solution of the transport equation
$$ -\partial_t u_t - b(x,\nu_t) \cdot \nabla_x u _t = 0 \quad \mbox{ in } [t_0,T] \times \R^{d_1} \times \R^{d_2}, \quad u_T = L \quad \mbox{in } \R^{d_1} \times \R^{d_2}. $$
This completes the proof of item \textit{(ii)} in the statement.
\vskip 4pt

\textit{Step 5.}
We now establish item \textit{(iii)}. By Step 1 (weak compactness of the sequence 
$({\boldsymbol \nu}^n)_{n \geq 1}$), we can assume without any loss of generality that
$({\boldsymbol \nu}^n)_{n \geq 1}$
converges to ${\boldsymbol \nu}$.
By lower boundedness and continuity of $L$, and by convergence in $d_{p-\delta}$ of $(\gamma_T^n)_{n \geq 1}$ toward $ {\gamma}_T$
(which follows from \textit{Step 3)}, we deduce that 
$$\int_{\R^{d_1} \times \R^{d_2}} L(x,y)d {\gamma}_T(x,y) \leq \liminf_{n \rightarrow +\infty}  \int_{\R^{d_1} \times \R^{d_2}} L(x,y) d\gamma_T^n(x,y).$$
Moreover, by \cite[Lemma 2.4]{BudhirajaDupuisbook}
\begin{equation*}
{\mathcal E}\bigl( {\boldsymbol \nu} \vert \textrm{\rm Leb}_{[t_0,T]}
\times \nu^\infty \bigr) 
\leq 
\liminf_{n \rightarrow \infty}
{\mathcal E}\bigl( {\boldsymbol \nu}^n\vert \textrm{\rm Leb}_{[t_0,T]}
\times \nu^\infty \bigr).
\end{equation*}
Using the formula 
\eqref{eq:interpretation:wider:space} in order to rewrite the two costs below in terms of the above two entropies, we deduce that 
\begin{equation*}
J\bigl( (t_0,\gamma_0),{\boldsymbol \nu} \bigr) 
\leq 
\liminf_{n \rightarrow \infty}
J\bigl( (t_0,\gamma_0),{\boldsymbol \nu}^n \bigr). 
\end{equation*} 
This completes the proof. 
\end{proof} 
\color{black}

We can easily deduce the existence of optimal solutions from the result above.

\begin{proof}[Proof of Theorem \ref{thm:OCThm27/06}]
We take a minimizing sequence $({\boldsymbol \nu}^n)_{n \geq 1}$ of $J((t_0,\nu_0),\cdot)$. 
By Lemma \ref{lem:cost:L4:moments}, we can find ${\boldsymbol \nu}$
in ${\mathcal A}(t_0)$ such that 
$J((t_0,\nu_0),{\boldsymbol \nu}) \leq \liminf_{n \rightarrow \infty}
J( (t_0,\gamma_0),{\boldsymbol \nu}^n)$. This suffices to conclude. 
\end{proof}

\subsection{Optimality Conditions: Proof of Theorem \ref{thm:OCThm27/06}}
\label{subse:7:2}

We go on with proof of the first order optimality conditions. The strategy is to \textit{linearize} the terminal cost around the optimal solution in order to bring ourselves back to the minimization of a (convex)  relative entropy for which we now the unique minimizer.

\begin{proof}[Proof of Theorem \ref{thm:OCThm27/06}]
    Let $\boldsymbol \nu^*=(\nu^*_t)_{t_0 \le t \le T}$ be an optimal control 
and     
    $\boldsymbol \gamma^* = (\gamma^*_t)_{t_0 \le t \le T}$ be the corresponding optimal trajectory. For another 
    ${\boldsymbol \nu}= (\nu_t)_{t_0 \le t \le T} \in \mathcal{A}(t_0)$ and some $\lambda \in [0,1]$, we let ${\boldsymbol \nu}^{\lambda} := (1-\lambda) \boldsymbol \nu^* + \lambda {\boldsymbol \nu}$ and, then, 
    we call ${\boldsymbol \gamma}^{\lambda} = (\gamma_t^{\lambda})_{t_0 \le t \le T}$ the solution to the continuity equation
    \eqref{eq:lem:compactnessContinuityEquation24/06:statement} driven by ${\boldsymbol \nu}^{\lambda}$.
    By optimality of $\boldsymbol \nu^*$, we have, for all $\lambda \in [0,1],$
\begin{equation}
\label{eq:tildenubetterlambdanu24/06}
J \bigl( (t_0,\gamma_0), \bd{\nu}^* \bigr) \leq J \bigl( (t_0, \gamma_0), \bd{\nu}^{\lambda} \bigr).
\end{equation}
By (linear) convexity of the function ${\boldsymbol \nu} \in {\mathcal A}(t_0) \mapsto \int_{t_0}^T \mathcal{E} \bigl(\nu_t | \nu^{\infty})dt$ (recall that $\nu^{\infty}$ has been introduced in \eqref{eq:defnuinfty03/07}),
we deduce from \eqref{eq:tildenubetterlambdanu24/06} that
\begin{align*}
    \epsilon \int_{t_0}^T\mathcal{E}\bigl (\nu^*_t | \nu^{\infty}) dt & \leq \epsilon \int_{t_0}^T \mathcal{E} (\nu_t^{\lambda} | \nu^{\infty})dt + \int_{\R^{d_1} \times \R^{d_2}} L(x,y) d ( \gamma_T^{\lambda} - \gamma^*_T)(x,y) \\
    & \leq (1-\lambda) \epsilon \int_{t_0}^T \mathcal{E} \bigl( \nu^*_t | \nu^{\infty}) dt + \lambda \epsilon \int_{t_0}^T \mathcal{E} \bigl( \nu_t |\nu^{\infty} \bigr)dt + \int_{\R^{d_1} \times \R^{d_2}} L(x,y) d ( \gamma_T^{\lambda} - \gamma^*_T)(x,y).
\end{align*}
Subtracting $ \epsilon \int_{t_0}^T\mathcal{E} \bigl (\nu^*_t | \nu^{\infty} \bigr)dt$ from both sides and dividing by $\lambda$, we get
\begin{equation*}
\epsilon \int_{t_0}^T \mathcal{E} \bigl( \nu^*_t |\nu^{\infty} \bigr) dt \leq \epsilon \int_{t_0}^T \mathcal{E} \bigl( \nu_t |\nu^{\infty} \bigr) dt + \frac{1}{\lambda} \int_{\R^{d_1} \times \R^{d_2}} L(x,y)d (\gamma_T^{\lambda} - \gamma^*_T)(x,y). 
\end{equation*}
Letting $\lambda \rightarrow 0^+$ we obtain, thanks to \eqref{eq:derivativeinlambda31/01} in Proposition \ref{prop:differentiatinggammalambda09/02},
$$ \epsilon \int_{t_0}^T \mathcal{E} \bigl (\nu^*_t | \nu^{\infty} \bigr)dt \leq \epsilon \int_{t_0}^T \mathcal{E} \bigl(\nu_t |\nu^{\infty} \bigr)dt + \int_{t_0}^T \int_{\R^{d_1} \times \R^{d_2}} b(x,\nu_t - \nu^*_t) \cdot \nabla_x u^*_t(x,y) d \gamma^*_t(x,y)  dt, $$
where $\boldsymbol u^*$ is the solution to the transport equation \eqref{eq:adjoint:equation} driven by $\boldsymbol \nu^*$. 
Equivalently, this can be rewritten as 
\begin{align} 
\label{eq:proof:optimality:condition:1:F}
\epsilon \int_{t_0}^T \mathcal{E} \bigl (\nu^*_t | \nu^{\infty} \bigr)dt &+ \int_{t_0}^T \int_{\R^{d_1} \times \R^{d_2}} b(x, \nu^*_t) \cdot \nabla_x u^*_t(x,y) d \gamma^*_t(x,y)  dt \\
\notag &\leq \epsilon \int_{t_0}^T \mathcal{E} \bigl (\nu_t |\nu^{\infty} \bigr)dt + \int_{t_0}^T \int_{\R^{d_1} \times \R^{d_2}} b(x,\nu_t) \cdot \nabla_x u^*_t(x,y) d \gamma^*_t(x,y)  dt. 
\end{align}
Following the statement, we now introduce, for each $t \in [t_0,T]$, the 
probability density (over $A$):
\begin{equation*}
\nu^{*,\infty}_t(a) : = \frac{1}{z^{*,\infty}_t}\exp 
\biggl( -   \ell(a) -  \frac{1}{\epsilon}\int_{\R^{d_1} \times \R^{d_2}}b(x,a) \cdot \nabla_x u^*_t(x,y)d \gamma^*_t(x,y)  \biggr), 
\quad a \in A, 
\end{equation*}
where $z_t^{*,\infty}$ is the normalization constant
\begin{equation*}
z^{*,\infty}_t := \int_{A} 
\exp \biggl( -   \ell(a) -  \frac{1}{\epsilon}\int_{\R^{d_1} \times \R^{d_2}}b(x,a) \cdot \nabla_x u^*_t(x,y)d \gamma^*_t(x,y)  \biggr) da. 
\end{equation*}
By (i) in \assreg \, and because $\gamma^*_t$ belongs to ${\mathcal P}_2({\mathbb R}^{d_1} 
\times {\mathbb R}^{d_2})$
and 
$\nabla_x  u^*_t$ has linear growth (see Proposition 
\ref{prop:SolutionBackxard16Sept}), 
the integral inside the exponential appearing in the definition of 
$\nu_t^{*,\infty}$ is finite
and less than $C(1+\vert a \vert)$
for a constant $C$ independent of $t$. 
In particular, one can easily prove that
\begin{equation*}
\int_{t_0}^T {\mathcal E}\bigl( \nu_t^{*,\infty} \vert 
\nu^\infty \bigr) dt < + \infty. 
\end{equation*}
From the latter definition, we deduce that 
$\boldsymbol \nu^{*,\infty} :=
{\rm Leb}_{[t_0,T]} \otimes (\nu^{*,\infty}_t)_{t_0 \le t \le T}$
belongs to ${\mathcal A}(t_0)$ (see Definition \ref{def:admissible:control}). We then rewrite the right-hand side of 
\eqref{eq:proof:optimality:condition:1:F} in the form 
\begin{equation}
\begin{split}
& \epsilon \int_{t_0}^T \mathcal{E}(\nu_t |\nu^{\infty})dt + \int_{t_0}^T \int_{A} \int_{\R^{d_1} \times \R^{d_2}} b(x,a) \cdot \nabla_x u^*_t(x,y) d \gamma^*_t(x,y) d\nu_t(a) dt 
\\
&= \epsilon \int_{t_0}^T \mathcal{E} \bigl( \nu_t | \nu_t^{*,\infty} \bigr) dt + \epsilon \int_{t_0}^T \log \frac{z^{\infty}}{z_t^{*,\infty}} dt,
\end{split}
\end{equation}
and, similarly, for the left-hand side of \eqref{eq:proof:optimality:condition:1:F}
\begin{equation}
\begin{split}
& \epsilon \int_{t_0}^T \mathcal{E}(\nu^*_t |\nu^{\infty})dt + \int_{t_0}^T \int_{A} \int_{\R^{d_1} \times \R^{d_2}} b(x,a) \cdot \nabla_x u^*_t(x,y) d \gamma^*_t(x,y) d\nu^*_t(a) dt 
\\
&= \epsilon \int_{t_0}^T \mathcal{E} \bigl( \nu^*_t | \nu_t^{*,\infty} \bigr) dt + \epsilon \int_{t_0}^T \log \frac{z^{\infty}}{z_t^{*,\infty}} dt,
\end{split}
\end{equation} 
where $z^{\infty} = \int_A e^{-\ell(a)}da$ is the normalizing constant for $\nu^{\infty}$.
We deduce from \eqref{eq:proof:optimality:condition:1:F} 
that $\boldsymbol \nu^*$ is a minimizer of the function 
${\boldsymbol \nu} \in 
{\mathcal A}(t_0)
\mapsto 
\epsilon \int_{t_0}^T \mathcal{E} ( \nu_t | \nu_t^{*,\infty}  ) dt
= \epsilon {\mathcal E}( {\boldsymbol \nu} \vert 
\boldsymbol \nu^{*,\infty})$.
However,
by strict convexity of the relative entropy, the latter function has a unique minimizer, which is 
$\boldsymbol \nu^{*,\infty}$. 
This proves that 
${\boldsymbol \nu}^* =\boldsymbol \nu^{*,\infty}$, and we deduce that, for almost every 
$t \in [t_0,T]$, $\nu^*_t$ is equal to $\nu_t^{*,\infty}$. 
Collecting the equations for $\boldsymbol \nu^*$, $\boldsymbol \gamma^*$ and $\boldsymbol u^*$, we get the system 
\eqref{eq:NOCfornut2Avril}--\eqref{eq:NOCforutgammat2Avril}. 

It remains to justify that the density $(t,a) \mapsto \nu_t^{*,\infty}(a)$ is jointly continuous. This follows from the explicit formula \eqref{eq:NOCfornut2Avril}, 
assumptions (i) and (ii) in \assreg \, together with the estimates on $\boldsymbol \gamma^*$ and $\boldsymbol u^*$ proven in 
Lemma 
\ref{lem:compactnessContinuityEquation24/06}
and
Proposition \ref{prop:SolutionBackxard16Sept}.
\end{proof}

\subsection{Additional Regularity and First Stability Results}

\label{sec:additionalregandfirststability}

We first prove that solutions of the system of optimality conditions satisfy further regularity properties.

\begin{proof}[Proof of Proposition \ref{prop:regularityfromOC}]
 { \ }

\textit{Step 1.}
Throughout the proof, $\Lambda_{\bd{\nu}}, \Lambda_{\gamma_0,\bd \nu} : \R_+ \rightarrow [1,+\infty)$ denote two non-decreasing functions, respectively of $\int_{t_0}^T \int_A|a|^4d\nu_t(a)dt$ and of $\int_{t_0}^T \int_A|a|^4d\nu_t(a)dt + \int_{\R^{d_1} \times \R^{d_2}} (|x| + |y|)d\gamma_0(x,y)$ and are allowed to change from line to line. Below, the last two arguments are not explicitly written as inputs of the two functions but are implicitly understood. Therefore, we simply write $\Lambda_{\bd{\nu}}$ and $\Lambda_{\gamma_0, \bd{\nu}}$.
By Proposition 
\ref{prop:SolutionBackxard16Sept} 
,  we know that ${\boldsymbol u}$ satisfies 
\begin{equation}
\sup_{t \in [t_0,T]} \norm{\nabla_x u_t}_{\mathcal{C}^0_1} \leq \Lambda_{\boldsymbol \nu },
\label{eq:nablaxut29sept}
\end{equation}
while Proposition \ref{lem:compactnessContinuityEquation24/06} gives
$$\sup_{t \in [t_0,T]} \int_{\R^{d_1} \times \R^{d_2}} (|x| + |y|) d\gamma_t(x,y) \leq \Lambda_{\bd{\nu}, \gamma_0}. $$
Thanks to the representation formulas
\eqref{eq:NOCfornut2Avril}
  and   \eqref{eq:F:formule:pour:tilde:z}
  (with $\boldsymbol \nu^*$ replaced by 
  ${\boldsymbol \nu}$)
and to the assumptions (i) and (ii) (on $b$ and $\ell$) in 
\assreg, 
we deduce that, for all $(t,a) \in [t_0,T] \times A$, 
\begin{equation}
\label{eq:loweruppernut29Sept}
\bigl[ \Lambda_{\boldsymbol \nu, \gamma_0}
    \bigr]^{-1}
e^{-4\ell(a)/3} \leq \nu_t(a) \leq  \Lambda_{\boldsymbol \nu, \gamma_0} e^{-3\ell(a)/4}. 
\end{equation}
This proves the first claim in the statement. 
\vskip 4pt

\textit{Step 2.} In order to prove the second claim, we consider a pair of random variables $(X_0,Y_0) \sim \gamma_0$ supported on some probability space $(\Omega, \mathcal{F}, \mathbb{P})$. We call $(X_t)_{t_0 \le t \le T}$ the solution to
the ODE
$$\dot{X}_t = b(X_t, \nu_t),
\quad t \in [t_0,T]; 
\quad X_{t_0}=X_0.$$
By Proposition \ref{lem:compactnessContinuityEquation24/06}, we know that $\gamma_t = \mathbb{P} \circ (X_t,Y_0)^{-1}$,  
for all $t \in [t_0,T]$. 
Letting $Z_t := \nabla_x u_t(X_t,Y_0)$ for $t \in [t_0,T]$,
the representation formulas 
\eqref{eq:NOCfornut2Avril}
  and   \eqref{eq:F:formule:pour:tilde:z}
  (for
  ${\boldsymbol \nu}$)
  can be rewritten in the form 
\begin{equation} 
\label{eq:probarepresentationnut}
\begin{split}
&\nu_t(a) = \frac{1}{z_t} \exp 
\biggl( - \ell(a) -  \frac{1}{\epsilon}\E \bigl[ b(X_t,a) \cdot Z_t   \bigr] \biggr), \quad (t,a) \in [t_0,T] \times A, 
\\
&z_t = \int_A  \exp 
\biggl(  - \ell(a) -  \frac{1}{\epsilon}\E \bigl[ b(X_t,a) \cdot Z_t   \bigr] \biggr) da, \quad t \in [t_0,T].
\end{split}
\end{equation}
By Lemma \ref{lem:ApproxBODE21Sept} 
, $(Z_t)_{t \in [t_0,T]}$ solves the backward ODE
\begin{equation*}
\dot{Z}_t = - \nabla_xb(X_t, \nu_t)Z_t, \quad t \in [t_0,T]; \quad Z_T = \nabla_x L(X_T,Y_0). 
\end{equation*}
Thanks to \eqref{eq:loweruppernut29Sept} and to the 
 growth assumptions on $b$ and $\nabla_x b$, see 
(i) in \assreg, we get, $\mathbb{P}-$almost-surely
and for all $t \in [t_0,T]$,
$$ |b(X_t, \nu_t)| +| \nabla_x b(X_t, \nu_t)| \leq C \int_A (1+|a|^2) d\nu_t(a) \leq \Lambda_{\boldsymbol \nu, \gamma_0}.$$
Getting back to the definition of $(Z_t)_{t \in [t_0,T]}$, together with the estimate \eqref{eq:nablaxut29sept} we get   
\begin{align*}
|Z_t| &\leq \Lambda_{\boldsymbol \nu}(1+|X_t| + |Y_0|) \leq \Lambda_{\bd \nu} (1+|X_0| + |Y_0|)
\end{align*}
where we used that 
$$ |X_t - X_0| = \biggl| \int_{t_0}^t b(X_s, \nu_s)ds \biggr| \leq C \int_{t_0}^T \int_A |a|d \nu_s(a)ds \leq \Lambda_{\bd {\nu}}.  $$
We deduce that, 
    $\mathbb{P}-$almost-surely and for all $(t_1,t_2) \in [t_0,T]$, 
$$ |X_{t_2} - X_{t_1}| + |Z_{t_2} - Z_{t_1} | \leq  
\Lambda_{\boldsymbol \nu, \gamma_0} (1+|X_0| + |Y_0|)
|t_2 - t_1|. $$
Using the growth assumptions on $b$ and $\nabla_x b$ again, this leads to 
$$ \bigl| \E \bigl[ b(X_{t_2},a) \cdot Z_{t_2}\bigr] - \E \bigl[ b(X_{t_1},a) \cdot Z_{t_1}\bigr] \bigr| \leq   
 \Lambda_{\boldsymbol \nu, \gamma_0} (1+|a|^2 )  |t_2 -t_1|, \quad  (t_1,t_2,a) \in [t_0,T]^2 \times A. $$
Using 
\eqref{eq:loweruppernut29Sept}, 
we also have 
$z_t \geq \Lambda_{\boldsymbol \nu, \gamma_0}^{-1}$
for all $t \in [t_0,T]$. Moreover, thanks to 
the explicit formulas in \eqref{eq:probarepresentationnut}, 
we also deduce 
\begin{equation*}
\begin{split}
\bigl\vert z_{t_2} - z_{t_1} \bigr\vert &\leq 
\int_A \exp \Bigl( - \ell(a)
\Bigr) \biggl\vert 
\exp 
\Bigl(   - \frac{1}{\epsilon} {\mathbb E}
\bigl[ b(X_{t_2},a) \cdot Z_{t_2}   \bigr] \Bigr) 
- 
\exp 
\Bigl(   - \frac{1}{\epsilon} {\mathbb E}
\bigl[ b(X_{t_1},a) \cdot Z_{t_1}   \bigr] \Bigr) 
\biggr\vert 
da
\\
&\leq \frac{1}{\epsilon} \int_A \exp \Bigl( - \ell(a)
+ \Lambda_{\bd{\nu}, \gamma_0} ( 1 +\vert a \vert) 
\Bigr) \Bigl\vert {\mathbb E}
\bigl[ b(X_{t_2},a) \cdot Z_{t_2}   \bigr]   - {\mathbb E}
\bigl[ b(X_{t_1},a) \cdot Z_{t_1}   \bigr]  \Bigr\vert 
da
\\
&\leq \Lambda_{\boldsymbol \nu, \gamma_0}
\vert t_2 - t_1 \vert
\int_A \bigl( 1 + \vert a \vert^2 \bigr) \exp \Bigl( - \ell (a)
+ \Lambda_{\bd{\nu},\gamma_0} ( 1 +\vert a \vert) 
\Bigr) da \\
&\leq \Lambda_{\boldsymbol \nu,\gamma_0} 
\vert t_2 - t_1 \vert. 
\end{split}
\end{equation*}
Similarly, 
\begin{equation*}
\begin{split}
\bigl\vert z_{t_2} \nu_{t_2}(a)- z_{t_1} \nu_{t_1}(a)
\bigr\vert &\leq  
\exp \Bigl( -  \ell(a)
\Bigr) \biggl\vert 
\exp 
\Bigl(   - \frac{1}{\epsilon} {\mathbb E}
\bigl[ b(X_{t_2},a) \cdot Z_{t_2}   \bigr] \Bigr) 
-
\exp 
\Bigl(   - \frac{1}{\epsilon} {\mathbb E}
\bigl[ b(X_{t_1},a) \cdot Z_{t_1}   \bigr] \Bigr) 
\biggr\vert 
\\
&\leq \Lambda_{\boldsymbol \nu, \gamma_0}
\vert t_2 - t_1 \vert 
  \bigl( 1 + \vert a \vert^2 \bigr) \exp \Bigl( -\ell(a)
+ \Lambda_{\bd\nu,\gamma_0} ( 1 +\vert a \vert) 
\Bigr). 
\end{split}
\end{equation*}
Combining the last two displays with the two bounds 
\eqref{eq:loweruppernut29Sept}
and $z_{t_1},z_{t_2} \geq [\Lambda_{\boldsymbol \nu , \gamma_0}]^{-1}$, we deduce that 
\begin{equation*}
\bigl\vert  \nu_{t_2}(a)- \nu_{t_1}(a)
\bigr\vert \leq \Lambda_{\boldsymbol \nu, \gamma_0} \vert t_2 - t_1 \vert 
e^{- 3 \ell(a)/4},
\end{equation*}
i.e., 
\begin{equation}
\label{eq:F:nu_t2-nut1}
\bigl\| e^{3 \ell/4}  \bigl[ \nu_{t_2} - \nu_{t_1} \bigr]
\bigr\|_{L^\infty} \leq \Lambda_{\boldsymbol \nu, \gamma_0} \vert t_2 - t_1 \vert. 
\end{equation}
In particular, 
\begin{equation}
\label{eq:proof:entropy:t1:t2:F}
\begin{split}
{\mathcal E}\bigl(\nu_{t_2} \vert \nu_{t_1} \bigr) 
&= \int_A  \ln 
\biggl(  1+ \frac{\nu_{t_2}(a)-\nu_{t_1}(a)}{\nu_{t_1}(a)} \biggr) d \nu_{t_2}(a) 
\\
&\leq \int_A  \frac{\vert \nu_{t_2}(a)-\nu_{t_1}(a)\vert}{\nu_{t_1}(a) }  d \nu_{t_2}(a)
 \leq \Lambda_{\boldsymbol \nu, \gamma_0} \vert t_2 - t_1\vert \int_A 
\frac{\nu_{t_2}(a)}{\nu_{t_1}(a)} e^{- 3 \ell(a)/4}  da. 
\end{split}
\end{equation}
By 
\eqref{eq:loweruppernut29Sept}, we complete the proof
of the penultimate display in the statement. We proceed similarly to justify that 
$$ \sup_{t \in [t_0,T]} \mathcal{E} \bigl( \nu_t | \nu^{\infty} \bigr) \leq \Lambda_{\bd{\nu}, \gamma_0} $$
and complete the proof of \eqref{eq:estimaterelativeentropy25/12}.
\vskip 4pt

\textit{Step 3.}
By 
Lemma \ref{lem:backwardtransportsmoothcontrol23/01} with $l=3$, $p=2$ and $q=1$
(the continuity of $t \mapsto \nu_t$ here following from the second step right above), we get the regularity properties of 
$u$, as stated in \eqref{eq:testfn1foru}--\eqref{eq:testfn2foru}.
\end{proof}

We go on with the proof of the log-Sobolev inequality satisfied by the probability measures $(\Gamma_t[\bd{\nu}])_{t \in [t_0,T]}$ when $\bd{\nu}$ is an element of $\mathcal{A}(t_0)$.
\begin{proof}[Proof of Lemma \ref{lem:LSI}]
For $R >0$, let $\varphi^R : A \rightarrow \R$ be a cutoff function satisfying 
    \begin{equation}
    \label{eq:conditions:varphiR}
    \left\{
    \begin{array}{ll}
|\varphi^R(a)| + |\nabla \varphi^R(a)| + |\nabla^2 \varphi^R(a)| \leq c, &\quad {\rm for} \ a \in A,
\\
\varphi^R(a) = 0, &\quad {\rm if} \ |a| \geq R+1, 
\\
\varphi^R(a) = 1, &\quad {\rm if} \  |a| \leq R,
\end{array}
\right.
\end{equation}
for some $c>0$ independent of $R >0$. Let us rewrite
the argument inside the exponential 
factor in 
\eqref{eq:defnGammaNut09/01}
as
\begin{equation} 
\ell(a) + \frac{1}{\epsilon} \int_{\R^{d_1} \times \R^{d_2}} b(x,a) \cdot \nabla_x u_t(x,y)  d\gamma_t(x,y) = \ell_t^1(a) + \ell_t^2(a),
\label{eq:decompositionL}
\end{equation}
 with
 \begin{equation*}
 \left\{
 \begin{array}{ll}
 \displaystyle \ell_t^1(a) :=\ell(a) + \frac{1}{\epsilon} \biggl(\int_{\R^{d_1} \times \R^{d_2}} b(x,a) \cdot \nabla_x u_t(x,y)  d\gamma_t(x,y) \biggr) \bigl(1-\varphi^R(a) \bigr), &\quad a \in A, 
 \\
\displaystyle \ell_t^2(a) := \frac{1}{\epsilon}\biggl(\int_{\R^{d_1} \times \R^{d_2}} b(x,a) \cdot \nabla_x u_t(x,y) d\gamma_t(x,y) \biggr) \varphi^R(a), &\quad a \in A.  
\end{array}
\right.
\end{equation*}
The aim is to show that, for a suitable choice of the function 
$\varphi^R$ (subject to the constraint \eqref{eq:conditions:varphiR}), the term $\ell^1$ is strongly convex 
in $a$ uniformly with respect to the parameter
$t$. To do so, we recall from Assumption 
\textbf{(Regularity)} that  $|\nabla^2_{aa} b(x,a)| \leq C(1+|a|)(1+|x|^2)$, for $x \in {\mathbb R}^{d_1}$ and $a \in A$. This says that 
\begin{equation*}
\Bigl\vert
  \nabla_{aa}^2 \bigl( \ell_t^1 - \ell \bigr) (a) 
\Bigr\vert \leq 
\left\{ 
\begin{array}{ll}
\displaystyle
C (1+|a|) \|\nabla_x u_t \|_{\mathcal{C}^0_1} \int_{{\mathbb R}^{d_1} \times
{\mathbb R}^{d_2}} 
\bigl( 1 + \vert x \vert^2 \bigr) (1+|x|+|y|)
d \gamma_t(x,y), &\quad \textrm{\rm if} \ \vert a \vert > R, 
\\
\displaystyle
0 , &\quad \textrm{\rm if} \ \vert a \vert \leq R.
\end{array}
\right.
\end{equation*} 
In particular, thanks to Propositions
\ref{lem:compactnessContinuityEquation24/06} with $p=3$
and \ref{prop:SolutionBackxard16Sept}, there exists a non decreasing function $\Lambda_{\boldsymbol{\nu},\gamma_0} : \R^+ \rightarrow \R^+$ such that 
\begin{equation*}
\Bigl\vert
  \nabla_{aa}^2 \bigl( \ell_t^1 - \ell\bigr) (a) 
\Bigr\vert \leq 
\left\{ 
\begin{array}{ll}
\displaystyle
(1+|a|) \Lambda_{\bd{\nu}, \gamma_0}, &\quad \textrm{\rm if} \ \vert a \vert > R, 
\\
\displaystyle
0 , &\quad \textrm{\rm if} \ \vert a \vert \leq R.
\end{array}
\right.
\end{equation*} 
where, following 
the notation used in the proof of 
Proposition 
\ref{prop:regularityfromOC}, 
$\Lambda_{\bd\nu,\gamma_0}$ is a shorthand notation for
$$ \Lambda  \biggl( \int_{t_0}^T \int_A |a|^4 d\nu_t(a)dt + \int_{\R^{d_1} \times \R^{d_2}} (|x|^2 + |y|^2)^{3/2}d\gamma_0(x,y) \biggr).$$
We now use the fact that $\nabla^2_{aa} 
\ell$ grows at least quadratically fast, see  \eqref{eq:convexityassumptionL}. 
Therefore, we can find $R:=R_{\bd{\nu},\gamma_0}>0$, depending  
on $(\bd{\nu}, \gamma_0)$ only through an upper bound for $\Lambda_{\bd{\nu}, \gamma_0}$ such that, choosing this $R$ in the definition of $\ell^1$, we have 
$$ \nabla^2_{aa} \ell_t^1(a) \geq \frac{1}{2} \nabla_{aa}^2\ell(a), \quad \forall a \in A, \quad \forall t \in [t_0,T],  $$
and, by assumption \eqref{eq:convexityassumptionL}, there is $\mu >0$ (depending only on $\ell$) such that 
$$ \nabla_{aa}^2 \ell_t^1(a) \geq \mu I_{d'}, \quad \forall a \in A, \quad \forall t \in [t_0,T]. $$
Therefore, by \cite[Corollary 5.7.2]{BakryGentilLedoux}, the probability density (on $A$) that is proportional to $\exp{ (-\ell_t^1(a))}$ satisfies a log-Sobolev inequality with some constant $C$ depending only on $\mu$. With the same choice of parameter $R$ in the definition of $\ell^2$, we get, using again Propositions
\ref{lem:compactnessContinuityEquation24/06} and  \ref{prop:SolutionBackxard16Sept},
\begin{align*} 
|\ell_t^2(a)| &\leq C(1+R_{\bd{\nu}, \gamma_0}) \norm{ \nabla_x u_t}_{\mathcal{C}^0_1} \int_{\R^{d_1} \times \R^{d_2}} (1+|x|) d\gamma_t(x,y) \\
&\leq \Lambda_{\bd{\nu},\gamma_0}(1+R_{\bd{\nu}, \gamma_0}).
\end{align*}
We can conclude from the decomposition \eqref{eq:decompositionL} and Holley-Stroock's perturbation property
(see
\cite[Proposition 5.1.6]{BakryGentilLedoux}) 
that $\Gamma_t[\bd\nu]$ satisfies a Log-Sobolev inequality with constant
$$ C_{\rm LSI}(\Gamma_t[\bd{\nu}]) \leq C e^{ \Lambda_{\bd{\nu}, \gamma_0}(1+R_{\bd \nu,\gamma_0}) } $$
for some $C$ depending only on $\mu$.
\end{proof}

We continue with the proof of Lemma \ref{lem:PinskerSecFOC28/02}. 
\begin{proof}[Proof of Lemma \ref{lem:PinskerSecFOC28/02}]
To estimate the first term in the definition of the norm $\norm{ \cdot }_{\mathcal{D}(t_0)}$, we rely on a generalization of Pinsker inequality stated in \cite[Theorem 2.1]{BolleyVillani} and recalled in Appendix \ref{sec:PinskerAppendix} for convenience. It gives, for all $\kappa>0$, $\bd{\nu} \in \mathcal{A}(t_0)$ and $t \in [t_0,T]$,
\begin{align*} 
 \int_A (1+|a|^4) & |\nu_t(a) - \nu^{*}_t(a)|da \\
 &\leq \kappa \biggl[  \frac{3}{2} +  \log \biggl( \int_A 
 \exp \Bigl\{ \frac2{\kappa} (1+|a|^4) \Bigr\} d\nu^{*}_t(a) \biggr) \biggr] \biggl[  \mathcal{E}\bigl( \nu_t | \nu_t^{*} \bigr)^{1/2} + \frac{1}{2} \mathcal{E}\bigl( \nu_t | \nu_t^{*} \bigr) \biggr].  
\end{align*}
Thanks to estimate \eqref{eq:bound:nu(a):exp(-ell(a))} and the growth assumption on $\ell$, we can choose $\kappa>0$ large enough depending only on $\int_{\R^{d_1} \times \R^{d_2}}(|x|^2+|y|^2)d\gamma_0(x,y)$, such that $\int_A \exp\{ (2/\kappa) (1+|a|^4) \} d\nu^{*}_t(a)$ is bounded independently 
of  $t \in [t_0,T]$. Integrating in time and using Cauchy-Schwarz inequality, we 
can find a non-decreasing function $\Lambda : \R^+ \rightarrow \R^+$ such that
\begin{align}
\int_{t_0}^T \int_A (1+|a|^4) |\nu_t(a) - \nu_t^{*}(a) |da dt &\leq \Lambda_{\gamma_0}  \biggl[ \int_{t_0}^T \mathcal{E}\bigl( \nu_t | \nu_t^{*} \bigr)^{1/2} dt + \int_{t_0}^T \mathcal{E}\bigl( \nu_t | \nu_t^{*} \bigr)  dt \biggr] \nonumber
\\ 
&\leq \Lambda_{\gamma_0}  \biggl[ \sqrt{T-t_0} \biggr( \int_{t_0}^T \mathcal{E}\bigl( \nu_t | \nu_t^{*} \bigr) dt \biggr)^{1/2} + \int_{t_0}^T \mathcal{E}\bigl( \nu_t | \nu_t^{*} \bigr) dt \biggr], \label{eq:Pinsker+CS}
\end{align}
where $\Lambda_{\gamma_0} := \Lambda ( \int_{\R^{d_1} \times \R^{d_2}} (|x|^2 +|y|^2) d\gamma_0(x,y) ).$

Finally, for all $t \in [t_0,T]$, we can use the second item in the generalized Pinsker inequality proven in \cite[Theorem 2.1]{BolleyVillani} 
--see Appendix \ref{sec:PinskerAppendix}-- to justify that, for every $c>0$,  
\begin{align} 
\notag \int_A (1+|a|^2) | \nu_t(a) - \nu_t^{*}(a)|dadt &= \frac{2}{c} \norm{ c(1+|a|^2) (\nu_t - \nu_t^{*})}_{TV} \\
&\leq \frac{2\sqrt{2}}{c} \biggl( 1 + \log \int_A e^{c^2(1+|a|^2)^2} d\nu^{*}_t(a) \biggr)^{1/2} \mathcal{E} \bigl( \nu_t | \nu_t^{*} \bigr)^{1/2},  
\label{eq:towardtimeregularitynutn}
\end{align}
where $\norm{\cdot}_{TV}$ is the total variation norm. In particular, recalling the exponential bound \eqref{eq:bound:nu(a):exp(-ell(a))}, we can choose $c>0$ depending only on $\int_{\R^{d_1} \times \R^{d_2}} (|x|^2 + |y|^2) d\gamma_0(x,y)$ such that the integral term in the right-hand side is bounded independently of $t \in [t_0,T]$. If we integrate \eqref{eq:towardtimeregularitynutn} in time and use Cauchy-Schwarz inequality, we find, for all $t_1,t_2 \in [t_0,T]$ with $t_1 \leq t_2$ and for possibly another value of the non-decreasing function $\Lambda$,
$$ \int_{t_1}^{t_2} \int_A (1+|a|^2) d|\nu_t - \nu_t^{*}|(a) dt \leq C \sqrt{t_2 -t_1}  \Bigl( \int_{t_1}^{t_2} \mathcal{E} (\nu_t | \nu_t^{*} ) dt \Bigr)^{1/2}.$$
Combined with \eqref{eq:Pinsker+CS}, this completes the proof of the lemma. 
\end{proof}

Before we prove Proposition \ref{prop:LipRegValueFuncion25Nov}, we first establish the following result.

\begin{lem}
    There exists $C >0$ such that, for any $(t_0,\gamma_0) \in [0,T] \times \mathcal{P}_2(\R^{d_1} \times \R^{d_2})$ and any 
    minimizer $\bd{\nu}^*$ of $J \bigl( (t_0,\gamma_0) , \cdot)$,
    $$ \int_{t_0}^T \mathcal{E}(\nu^*_t |\nu^{\infty}) dt \leq C \biggl( 1+ \int_{\R^{d_1} \times \R^{d_2}} (|x|^2 + |y|^2) d\gamma_0(x,y) \biggr).$$
\label{lem:7.4-04/03}
\end{lem}

\begin{proof}
    Take $(t_0,\gamma_0) \in [0,T] \times \mathcal{P}_2(\R^{d_1 } \times \R^{d_2})$. Since we can always see $t \in [t_0,T] \mapsto \nu^{\infty} $ as a competitor in the control problem, we have 
    $$ J \bigl( (t_0,\gamma_0), \bd{\nu}^* \bigr) \leq \int_{\R^{d_1} \times \R^{d_2}} L(x,y) d\gamma_T^{\infty}(x,y)$$
where $(\gamma^{\infty}_t)_{t \in [t_0,T]}$ is the solution to 
$$ \partial_t \gamma_t^{\infty} + \div_x (b(x,\nu^{\infty}) \gamma_t^{\infty}) = 0 \quad \mbox{ in } (t_0,T) \times \R^{d_1} \times \R^{d_2};  \quad \gamma^{\infty}_{t_0} = \gamma_0.$$
Using the fact that $L$ is at most of quadratic growth, we  deduce from
Proposition \ref{lem:compactnessContinuityEquation24/06} that
$$ J \bigl( (t_0,\gamma_0), \bd{\nu}^* \bigr) \leq C \biggl( 1 + \int_{\R^{d_1} \times \R^{d_2}} (|x|^2 + |y|^2) d\gamma_0(x,y) \biggr),$$
for some $C>0$ independent from $(t_0,\gamma_0) \in [0,T] \times \mathcal{P}_2(\R^{d_1} \times \R^{d_2})$.
\end{proof}

We also need  
the next result, which we use to compare the metrics $\norm{ \cdot }_{(\mathcal{C}^1_{2,1})^*}$ and $d_2$. 

\begin{lem}
    \label{lem:comparisonNorms}
For any two probability measures $\gamma^{1}, \gamma^{2} \in \mathcal{P}_2(\R^{d_1} \times \R^{d_2})$, we have
\begin{align*} 
\norm{ \gamma^1 - \gamma^2  }_{(\mathcal{C}^1_{2,1})^*} 
\leq  \biggl[  1+  \sqrt{2} \biggl(\int_{\R^{d_1} \times \R^{d_2}} 
\bigl(|x|^2 + |y|^2\bigr)
d\bigl(\gamma^{1} + \gamma^{2}\bigr)(x,y) \biggr)^{1/2}  \biggr] d_2\bigl(\gamma^{1}, \gamma^{2}\bigr).  
\end{align*}
\end{lem}

\begin{proof} 
    Let $\phi \in \mathcal{C}^1(\R^{d_1} \times \R^{d_2})$, with $|\nabla \phi(x,y)| \leq 1+ |x| + |y|$ for all $(x,y) \in \R^{d_1} \times \R^{d_2}$. Then, for
    any probability space 
    $(\Omega, \mathcal{F}, \mathbb{P})$  equipped with two couples of random variables $((X^{i},Y^{i}))_{i=1,2}$ such that 
    $\mathbb{P} \circ (X^{i}, Y^{i})^{-1} = \gamma^{i}$ for $i=1,2$ (with $\gamma^1$ and $\gamma^2$ as in the statement), we have
\begin{align*}
    \int_{\R^{d_1} \times \R^{d_2}} \phi(x,y) d(&\gamma^2 - \gamma^1)(x,y) = \E \Bigl[ \phi(X^2,Y^2) - \phi(X^1,Y^1) \Bigr] \\
    &\leq \E \Bigl[\bigl(1+|(X^2,Y^2)| + |(X^1,Y^1)|\bigr)\bigl(|(X^2,Y^2)-(X^1,Y^1)|\bigr) \Bigr] \\
    &\leq  \E \Bigl[(1+|(X^2,Y^2)| + |(X^1,Y^1)|)^2 \Bigr]^{1/2} \E \Bigl[ |X^2-X^1|^2 + |Y^2-Y^1|^2\Bigr]^{1/2}
\end{align*}
where we used Cauchy-Schwarz inequality at the third line. We get the result by taking the infimum over all possible couples $(X^1,Y^1)$ and $(X^2,Y^2)$ with $(X^{i},Y^{i}) \sim^{\mathbb{P}} \gamma^{i}$ for $i=1,2$ (which does impact the left-hand side), 
and then the supremum over $\phi$.
\end{proof}

\begin{proof}[Proof of Proposition \ref{prop:LipRegValueFuncion25Nov}] 
{ \ }
  
  \textit{Preliminary step.} By Lemma \ref{lem:7.4-04/03}, there exists $C>0$ such that, for any $(t_0,\gamma_0) \in [0,T] \times \mathcal{P}_2(\R^{d_1} \times \R^{d_2})$ and any minimizer $\bd{\nu}^* = (\nu_t^*)_{t \in [t_0,T]}$ 
  of $J( (t_0,\gamma_0), \cdot )$, it holds
\begin{equation} 
\int_{t_0}^T  \mathcal{E} \bigl( \nu_t^* | \nu^{\infty} \bigr) dt + \norm{ \bd{\nu}^*}_{\mathcal{D}(t_0)} \leq C \biggl(1+\int_{\R^{d_1} \times \R^{d_2}} (|x|^2+|y|^2) d\gamma_0(x,y) \biggr).
\label{eq:tobefollowed24/05}
\end{equation}
Obviously, the bound for the first term in the left-hand side directly follows from
Lemma \ref{lem:7.4-04/03}. Then, the bound for $\norm{\bd{\nu}^*}_{\mathcal{D}(t_0)}$ 
follows from \eqref{eq:boundfromPinsker16Sept} in Lemma \ref{lem:cost:L4:moments}.
\vskip 4pt
   
\textit{Main step: local Lipschitz regularity of $U$.} We first address the regularity 
of $U$ in the measure variable. Let $t_0 \in [0,T]$  
and
$\gamma^1, \gamma^2\in \mathcal{P}_2(\R^{d_1} \times \R^{d_2})$ be fixed,
and $\bd{\nu}^*$ be an optimal control for $J( (t_0,\gamma^2), \cdot )$. We call $(\gamma_t^1)_{t \in [t_0,T]}$ and 
$(\gamma_t^2)_{t \in [t_0,T]}$ the solutions to the continuity equation
\eqref{eq:lem:compactnessContinuityEquation24/06:statement} driven by the same control $\bd{\nu}^*$ but starting respectively from $(t_0,\gamma^1)$ and $(t_0,\gamma^2)$ (so that 
$\bd{\gamma}^2=(\gamma_t^2)_{t \in [t_0,T]}$ is an optimal curve for $J( (t_0,\gamma^2), \cdot )$).
Using the optimality of $\bd{\nu}^*$ for $J ( (t_0,\gamma^2), \cdot  )$, we have
$$U(t_0,\gamma^1) - U(t_0,\gamma^2) \leq \int_{\R^{d_1} \times \R^{d_2}} L(x,y) d\gamma_T^1(x,y) - \int_{\R^{d_1} \times \R^{d_2}} L(x,y)d \gamma_T^2(x,y).$$ 
In order to handle the right-hand side, 
we use a duality argument. We call
$u: [t_0,T] \times \R^{d_1} \times \R^{d_2} \rightarrow \R$ the solution to the backward equation \eqref{eq:adjoint:equation} with control $\bd{\nu}^*$. By Lemma \ref{lem:duality:1,2:F}
(with ${\boldsymbol \nu}^1$ and ${\boldsymbol \nu}^2$ being both equal to ${\boldsymbol \nu}^*$ therein,
so that the second line 
in 
\eqref{eq:duality:relation:statement}
disappears), we deduce that
\begin{align*}
& U(t_0,\gamma^1) - U(t_0,\gamma^2) 
\\
&\leq \int_{\R^{d_1} \times \R^{d_2}} u_{t_0}(x,y) d\gamma^1(x,y) -  \int_{\R^{d_1} \times \R^{d_2}} u_{t_0}(x,y) d\gamma^2(x,y) \\
&\leq \sup_{t \in [t_0,T]} \norm{ u_t}_{\mathcal{C}^1_{2,1}} \norm{\gamma^1 - \gamma^2}_{(\mathcal{C}^1_{2,1})^*} 
\phantom{\biggl)}
\\
&\leq \Lambda \biggl( \int_{\R^{d_1} \times \R^{d_2}} (|x|^2 + |y|^2) d\gamma^2(x,y) \biggr)  \biggl[ \int_{\R^{d_1} \times \R^{d_2}}(|x|^2 + |y|^2)
d \bigl( \gamma^1 + \gamma^2 \bigr) (x,y) \biggr]^{1/2} d_2 (\gamma^2,\gamma^1),
\end{align*}
for some non-decreasing function $\Lambda: \R_+ \rightarrow \R_+$, independent of $t_0$, $\gamma^1$ and $\gamma^2$. To get the last line, we used 
Proposition \ref{prop:SolutionBackxard16Sept} together with the preliminary step
in order to bound 
$\sup_{t \in [t_0,T]} \norm{ u_t}_{\mathcal{C}^1_{2,1}}$
and then Lemma \ref{lem:comparisonNorms}
in order to bound $\| \gamma^1 - \gamma^2 \|_{(\mathcal{C}^1_{2,1})^*}$.
Exchanging the roles of $\gamma^1$ and $\gamma^2$, we conclude that 
\begin{equation} 
\label{eq:lipingamma25Nov}
\begin{split}
&\bigl| U(t_0,\gamma^1) - U(t_0,\gamma^2) \bigr| \leq \Lambda \biggl(    \int_{\R^{d_1} \times \R^{d_2}}(|x|^2 + |y|^2)d \bigl( \gamma^1 + \gamma^2 \bigr) (x,y) \biggr)    d_2 (\gamma^2,\gamma^1),
\end{split}
\end{equation}
for some possibly different function $\Lambda$ independent of $(t_0,(\gamma^1,\gamma^2)) \in [0,T] \times (\mathcal{P}_2(\R^{d_1} \times \R^{d_2}))^2$. This proves the local Lipschitz property in the measure argument.

In order to address the time regularity, we proceed as follows. We fix $(t_0,\gamma_0) \in [0,T] \times \mathcal{P}_2(\R^{d_1} \times \R^{d_2}) $ as well as $t_1 \in (t_0,T]$. We take $\bd{\nu} = (\nu_t)_{t \in [t_0,T]}$ an optimal control for $(t_0,\gamma_0)$ (for simplicity, we remove the superscript $*$
in the notation 
$\bd{\nu}$ of the optimal control)
and we denote by $\bd{\gamma} = (\gamma_t)_{t\in [t_0,T]}$ the resulting curve. By dynamic programming, $(\nu_t)_{t \in [t_1,T]}$ is a minimizer of $J((t_1,\gamma_{t_1}),\cdot)$, with 
 $U(t_1, \gamma_{t_1})$
 as optimal cost.
As a consequence, we get
\begin{align}
\label{eq:Ut1gamma0:Ut_0gamma0}
U(t_1,\gamma_0) - U(t_0,\gamma_0) &= U(t_1, \gamma_0) - U(t_1, \gamma_{t_1}) - \epsilon \int_{t_0}^{t_1} \mathcal{E} \bigl( \nu_t | \nu^{\infty} \bigr)dt.
\end{align}
We handle the difference $U(t_1, \gamma_0) - U(t_1, \gamma_{t_1})$  by means of 
\eqref{eq:lipingamma25Nov}. 
To do so, 
we notice from Proposition \ref{lem:compactnessContinuityEquation24/06} and the preliminary step that
$\int_{{\mathbb R}^{d_1} \times {\mathbb R}^{d_2}} (\vert x \vert^2 + \vert y \vert^2) d \gamma_{t_1}(x,y) \leq \Lambda (  \int_{{\mathbb R}^{d_1} \times {\mathbb R}^{d_2}} (\vert x \vert^2 + \vert y \vert^2) d \gamma_{0}(x,y))$, for a non-decreasing function $\Lambda : \R_+ \rightarrow \R_+$ independent of the
choices of $t_0$, $t_1$, $\gamma_{0}$. And then, 
\eqref{eq:lipingamma25Nov} 
yields
\begin{equation*}
\bigl\vert 
U(t_1, \gamma_0) - U(t_1, \gamma_{t_1})
\bigr\vert \leq \Lambda \biggl( \int_{{\mathbb R}^{d_1} \times {\mathbb R}^{d_2}} (\vert x \vert^2 + \vert y \vert^2)  d \gamma_{t_0}(x,y) \biggr) d_2(\gamma_{0},\gamma_{t_1}).
\end{equation*} 
Inserting the above bound in the expansion 
\eqref{eq:Ut1gamma0:Ut_0gamma0}
and then invoking 
\eqref{eq:estimaterelativeentropy25/12} in Proposition \ref{prop:regularityfromOC},
we deduce that 
\begin{equation}
\bigl| U(t_1,\gamma_0) - U(t_0,\gamma_0) \bigr| \leq \Lambda_{\gamma_0} \Bigl( d_2 \bigl( \gamma_0, \gamma_{t_1} \bigr) +  |t_1- t_0| \Bigr), 
\label{eq:U:Lipschitz:in:time:proof}
\end{equation}
where the constant $\Lambda_{\gamma_0}$ is given by $\Lambda_{\gamma_0} = \Lambda( \int_{\R^{d_1} \times \R^{d_2}} (|x|^2 + |y|^2) d\gamma_0(x,y) )$
for a non-decreasing function $\Lambda : {\mathbb R}_+ \rightarrow {\mathbb R}_+$
independent of $t_0$, $t_1$ and $\gamma_0$.

In order to complete the proof, it remains to show that, in 
\eqref{eq:U:Lipschitz:in:time:proof}, 
$d_2(\gamma_0,\gamma_{t_1})$
can be bounded by $\Lambda_{\gamma_0} \vert t_1-t_0 \vert$. We thus establish in this paragraph 
that
the curve ${\boldsymbol \gamma}=(\gamma_t)_{t \in [t_0,T]}$
is Lipschitz continuous in $t$ w.r.t. to the distance $d_2$. 
Using the ODE representation of ${\boldsymbol \gamma}=(\gamma_t)_{t \in [t_0,T]}$
provided by Proposition 
\ref{prop:FBODErepresentation}, 
we see that it suffices to get a bound on the drift $(x \mapsto b(x,\nu_t))_{t \in [t_0,T]}$. 
By 
\eqref{eq:bound:nu(a):exp(-ell(a))} in
Proposition \ref{prop:regularityfromOC}, it is enough to get a bound for the term $\Lambda_{{\boldsymbol \nu},\gamma_0}$ in \eqref{eq:bound:nu(a):exp(-ell(a))}. 
The latter is an increasing function of 
$\int_{t_0}^T \int_A |a|^4 d\nu_t(a)dt$, which is bounded by 
$C(1+\int_{\R^{d_1} \times \R^{d_2}} (|x|^2 + |y|^2)d\gamma_0(x,y))$ thanks to the preliminary step, 
and of 
$ \int_{\R^{d_1} \times \R^{d_2}} (|x|^2 + |y|^2) d\gamma_0(x,y)$. 
Up to a new value of the constant $\Lambda_{\gamma_0}$ (equivalently, for 
a possibly different choice of the non-decreasing function  $\Lambda$
introduced right above), we get 
$$ d_2 (\gamma_0, \gamma_{t_1}) \leq \Lambda_{\gamma_0} |t_1 - t_0|.$$
Therefore,
by 
\eqref{eq:U:Lipschitz:in:time:proof},
$$ | U(t_1, \gamma_0) - U(t_0,\gamma_0)| \leq \Lambda_{\gamma_0} |t_1- t_0|,$$
which completes the proof of the statement.
\end{proof}

The rest of the section is dedicated to the proof of Lemma 
\ref{lem:convergenceifuniqueness}, which requires some preliminary results. We start with the following statement:

\begin{lem}
\label{lem:stability3-7/3-8}
There exists a non-decreasing function $\Lambda : {\mathbb R}_+ \rightarrow 
{\mathbb R}_+$ such that, for any pair of initial conditions $(t_0^1,\gamma_0^1),(t_0^2,\gamma_0^2) \in [0,T] \times {\mathcal P}_2({\mathbb R}^{d_1} \times {\mathbb R}^{d_2})$, and any pair of solutions $(\boldsymbol{\nu}^1, \boldsymbol{\gamma}^1, \boldsymbol{u}^1), (\boldsymbol{\nu}^2, \boldsymbol{\gamma}^2, \boldsymbol{u}^2)$ to the system \eqref{eq:NOCfornut2Avril}--\eqref{eq:NOCforutgammat2Avril} with respective initial condition $(t_0^1,\gamma_0^1)$ and $(t_0^2,\gamma_0^2)$,  
 \begin{equation*}
 \label{lem:backthenitwasLemma6.2}
 \begin{split}
  &\sup_{t \in [t_0^1 \vee t_0^2 ,T]} \Bigl \{ \bigl\| e^{3\ell/4}  (\nu^2_t - \nu^1_t) \bigr\|_{L^{\infty}}    + \mathcal{E}\bigl(\nu_t^1|\nu_t^2 \bigr) + \bigl\|  u_t^{2} - u_t^{1} \bigr\|_{\mathcal{C}_1^2} \Bigr \} 
  \\
  &\leq \sup_{i=1,2} \Lambda \biggl( \int_{\R^{d_1} \times \R^{d_2}} |x|^2 d\gamma_0^{i}(x,y)
  +
  \int_{t_0^{i}}^T \int_A |a|^4d\nu_t^{i}(a)dt \biggr) 
  \sup_{t \in [t_0^1 \vee t_0^2,T]} \| \gamma_t^{2} - \gamma_t^{1} \|_{(\mathcal{C}_1^2)^*}. 
  \end{split}
\end{equation*}

\end{lem}

\begin{proof}
Throughout the proof, the value of the
non-decreasing function $\Lambda :{\mathbb R}_+ \rightarrow [1,+\infty)$ is allowed
to change from line to line.
To simplify the notations, we just denote the quantity $\sup_{i=1,2} \Lambda  ( \int_{\R^d \times \R^d} |x|^2 d\gamma_0^{i}(x,y)
+
  \int_{t_0^{i}}^T \int_A (1+|a|^4)d\nu_t^{i}(a)dt) $ by $\Lambda_{\boldsymbol{\nu}^{1},\boldsymbol{\nu}^2,\gamma_0^1,\gamma_0^2}.$
  
Let us write $\delta u_t := u_t^{2} - u_t^{1}$ for $t \in [t_0^1 \vee t_0^2,T]$, and let 
$(X_s^{2,t,x})_{t \leq s \leq T}$ be the solution to
$$ \dot{X}_s^{2,t,x} = b(X_s^{2,t,x}, \nu_t^{2}), \quad
s \in [t,T]; \quad X_t^{2,t,x} = x.$$
By Lemma \ref{lem:duality:1,2:F} (with ${\boldsymbol \gamma}^1$
therein --not to be confused with the current choice of 
${\boldsymbol \gamma}^1$-- being understood as 
$(\delta_{(X_s^{2,t,x},y)})_{t \leq s \leq T}$ with initial condition $\delta_{(x,y)}$ at time $t$, which is 
a very specific 
example of application of Proposition \ref{lem:compactnessContinuityEquation24/06}
when the continuity equation is initalized from a delta mass), and $\bd{\varphi}^2$ in Lemma \ref{lem:duality:1,2:F} being understood as $\bd{u}^1$)
we obtain
\begin{equation*}
\begin{split}
u_t^1(x,y) &= L(X_T^{2,t,x},y) - \int_t^T b(X_s^{2,t,x}, \nu_s^{2} - \nu_s^{1}) \cdot \nabla_x u_s^{1}(X_s^{2,t,x},y)ds
\\
&= u_t^2(x,y) - \int_t^T b(X_s^{2,t,x}, \nu_s^{2} - \nu_s^{1}) \cdot \nabla_x u_s^{1}(X_s^{2,t,x},y)ds,
\end{split}
\end{equation*}
which gives
\begin{equation}
\delta u_t(x,y) = - \int_t^T b(X_s^{2,t,x}, \nu_s^{2} - \nu_s^{1}) \cdot \nabla_x u_s^{1}(X_s^{2,t,x},y)ds.
\label{eq:explicitformludeltautn6Avril}
\end{equation}
Since $ \| X_s^{2,t,\cdot} \|_{\mathcal{C}^2_{1,b}}$ and   $\|\nabla_xu_s^{1} \|_{\mathcal{C}^{2}_1}$
are bounded independently of $s$ and $n$ (see Propositions \ref{prop:ODE16Sept} and \ref{prop:SolutionBackxard16Sept}), we get, by the growth assumption on $b$ in \assreg, for all $t \in [t_0^1 \vee t_0^2,T]$,
\begin{equation} 
\| \delta u_t \|_{\mathcal{C}^2_1} \leq \Lambda_{\boldsymbol{\nu}^{1},\boldsymbol{\nu}^2,\gamma_0^1,\gamma_0^2} \int_{t}^T \int_A (1+|a|^3)|\nu_s^{2}(a) - \nu_s^{1}(a) |da ds. 
\label{eq:estimatedeltautn6Avril}
\end{equation}
Getting back to the expressions 
 for $\nu_t^{1}$ and $\nu_t^{2}$, see 
 \eqref{eq:NOCfornut2Avril},  
we observe that the difference between the two arguments in the exponential form of the Gibbs density 
can be estimated as follows: 
\begin{equation*}
\begin{split}
&\biggl\vert \int_{\R^{d_1} \times \R^{d_2}} \bigl[ b(x,a) \cdot \nabla_x {u}^1_t(x,y) \bigr] d {\gamma}_t^1(x,y) - \int_{\R^d \times \R^d} \bigl[ b(x,a) \cdot \nabla_x {u}^2_t(x,y) \bigr] d {\gamma}_t^2(x,y)
\biggr\vert
\\
&\leq \biggl\vert \int_{\R^{d_1} \times \R^{d_2}} \bigl[ b(x,a) \cdot \nabla_x {u}^1_t(x,y) \bigr] d \bigl( {\gamma}_t^1 - {\gamma}_t^2\bigr) (x,y)
\biggr\vert
\\
&\hspace{15pt} +\biggl\vert \int_{\R^{d_1} \times \R^{d_2}}  b(x,a) \cdot  \bigl[\nabla_x {u}^1_t(x,y) - \nabla_x {u}^2_t(x,y) \bigr] d {\gamma}_t^2(x,y)
\biggr\vert.
\end{split}
\end{equation*}
Using (i) in \assreg \, together with 
Proposition \ref{prop:SolutionBackxard16Sept}, we 
can bound the first term by 
$\Lambda_{\boldsymbol{\nu}^{1},\boldsymbol{\nu}^2,\gamma_0^1,\gamma_0^2}
(1 + \vert a \vert^3) \| \gamma_t^{2} - \gamma_t^{1} \|_{(\mathcal{C}_1^2)^*}$. Moreover,  
we can invoke Proposition
\ref{lem:compactnessContinuityEquation24/06}
to 
bound the second term by 
$\Lambda_{\boldsymbol{\nu}^{1},\boldsymbol{\nu}^2,\gamma_0^1,\gamma_0^2} 
( 1 + \vert a \vert) \|\nabla_x \delta u_t \|_{\mathcal{C}^0_1}$. 
Following the proof of \eqref{eq:F:nu_t2-nut1}, we obtain, for   $(t,a) \in [t_0^1 \vee t_0^2,T] \times A$,
\begin{equation}
\label{eq:tobeinjected}
   \vert \nu_t ^{2}(a) - \nu_t^{1}(a)| \leq \Lambda _{\boldsymbol{\nu}^{1},\boldsymbol{\nu}^2,\gamma_0^1,\gamma_0^2} \Bigl( \|\nabla_x \delta u_t \|_{\mathcal{C}^0_1} + \norm{\gamma_t^{2} - \gamma_t^{1} }_{(\mathcal{C}_1^2)^*} \Bigr)  \exp\Bigl( - \frac{3}{4} \ell(a)\Bigr).  
\end{equation}
Injecting \eqref{eq:estimatedeltautn6Avril} into \eqref{eq:tobeinjected}  and using Grönwall's Lemma, 
we get
$$ \sup_{t \in [t_0^1 \vee t_0^2,T]}   \bigl\| e^{ 3\ell /4}  \bigl[ \nu^2_t - \nu^1_t\bigr] \bigr\|_{L^{\infty}} \leq \Lambda_{\boldsymbol{\nu}^{1},\boldsymbol{\nu}^2,\gamma_0^1,\gamma_0^2} \sup_{t \in [t_0^1 \vee t_0^2,T]} \|\gamma_t^{2} - \gamma_t^{1}\|_{(\mathcal{C}_1^2)^*}.   $$
Using \eqref{eq:estimatedeltautn6Avril} again, we deduce that 
\begin{equation}
\sup_{t \in [t_0^1 \vee t_0^2 ,T]} \Bigl\{ \bigl\| e^{ 3 \ell/4}  \bigl[ \nu^2_t - \nu^1_t\bigr] \Bigr\|_{L^{\infty}}  +  \|  u_t^{2} - u_t^{1} \|_{\mathcal{C}_1^2}  \Bigr \} \leq \Lambda_{\boldsymbol{\nu}^1,\boldsymbol{\nu}^2,\gamma_0^1,\gamma_0^2} \sup_{t \in [t_0^1 \vee t_0^2,T]} \| \gamma_t^{2} - \gamma_t^{1} \|_{(\mathcal{C}_1^2)^*}. 
\label{eq:30Sept15:17}
\end{equation}
It remains to estimate the relative entropy term, 
but this follows from  the same argument as in \eqref{eq:proof:entropy:t1:t2:F}.
This concludes the proof of the lemma.
\end{proof}

Using log-Sobolev inequality,
we can also obtain the following estimate when 
two solutions to the system 
\eqref{eq:NOCfornut2Avril}--\eqref{eq:NOCforutgammat2Avril}
start from a common initial position $(t_0,\gamma_0)$ in $[0,T] \times \mathcal{P}_3(\R^{d_1} \times \R^{d_2}).$

\begin{lem}
\label{lem:relativeentropyasasqureddistance28/02}
There is a non-decreasing function $\Lambda:\R_+ \rightarrow \R_+$ such that, for any $(t_0,\gamma_0) \in [0,T] \times \mathcal{P}_3(\R^{d_1} \times \R^{d_2})$ and any two solutions $(\bd{\nu}^{i}, \bd{\gamma}^{i}, \bd{u}^{i})$, $i=1,2$ to the system \eqref{eq:NOCfornut2Avril}--\eqref{eq:NOCforutgammat2Avril} with $(t_0,\gamma_0)$ as common initial condition,
it holds
\begin{equation*}
\begin{split}
&\sup_{t \in [t_0,T]} \mathcal{E}(\nu_t^2 | \nu_t^1) 
\\
&\hspace{15pt} \leq \Lambda \biggl( \int_{\R^{d_1} \times \R^{d_2}} (|x|^2 + |y|^2)^{3/2} d\gamma_0(x,y) + \int_{t_0}^T \int_A |a|^4 d(\nu_t^1 + \nu_t^2)(a) \biggr) \sup_{t \in [t_0,T]} \norm{ \gamma_t^2 - \gamma_t^1}^2_{(\mathcal{C}^2_2)^*}. 
\end{split}
\end{equation*}
\end{lem}

\begin{proof}[Proof of Lemma \ref{lem:relativeentropyasasqureddistance28/02}]
    By the log-Sobolev inequality stated in Lemma \ref{lem:LSI}
    (with ${\boldsymbol \nu}$ being taken as 
    ${\boldsymbol \nu}^2$, 
    in which case $\Gamma({\boldsymbol \nu})$ is also equal to 
    ${\boldsymbol \nu}^2$, and $f$ being taken as $\nu_t^1/\nu_t^2$), there is a non-decreasing function $\Lambda :\R_+ \rightarrow \R_+$ such that, for all $t \in [t_0,T]$,
    $$  \mathcal{E}(\nu_t^1| \nu_t^2) \leq \Lambda_{\bd{\nu}^2, \gamma_0} \int_A 
    \Bigl| \nabla_a \log \frac{\nu_t^1}{\nu_t^2}(a) \Bigr|^2 d\nu_t^1(a) $$
where $\Lambda_{\bd{\nu}^2,\gamma_0}$ is a short-hand notation for  $\Lambda \bigl( \int_{\R^{d_1} \times \R^{d_2}} (|x|^2 + |y|^2)^{3/2} d\gamma_0(x,y) + \int_{t_0}^T \int_A |a|^4 d\nu^2_t(a) dt  \bigr) $. Recalling the equations satisfied by $\bd{\nu}^1$ and $\bd{\nu}^2$ (or equivalently \eqref{eq:defnGammaNut09/01}), and using the regularity assumption on $b$ this leads to 
\begin{align*}
    &\mathcal{E}(\nu_t^1|\nu_t^2) 
    \\
    &\leq \Lambda_{\bd{\nu}^2,\gamma_0} \frac{1}{\epsilon^2} \int_A \bigl| \nabla_a  \int_{\R^{d_1} \times \R^{d_2}} b(x,a) \cdot d \bigl( \nabla_x u_t^1 \gamma_t^1 - \nabla_xu_t^2 \gamma_t^2)(x,y) \bigr|^2 d\nu_t^1(a)  \\
    &\leq \frac{1}{\epsilon^2} \Lambda_{\bd{\nu}^2,\gamma_0} \biggl( \norm{ \nabla_xu_t^1 - \nabla_x u_t^2}_{\mathcal{C}^0_1} \int_{\R^{d_1} \times \R^{d_2}} (|x|^2 + |y|^2)d\gamma_t^1(x,y) + \norm{ \nabla_x u_t^2}_{\mathcal{C}^2_1} \norm{ \gamma_t^1 - \gamma_t^2}_{(\mathcal{C}^2_2)^*} \biggr)^2 
    \\
    &\hspace{15pt} \times \int_A (1+|a|^6) d\nu_t^1(a).
\end{align*}
It remains to use the previous Lemma \ref{lem:stability3-7/3-8} together with Proposition \ref{prop:regularityfromOC} to conclude.
\end{proof}

\color{black}

Using the estimates from Lemma \ref{lem:stability3-7/3-8}, we can now prove Lemma \ref{lem:convergenceifuniqueness}.

\begin{proof}[Proof of Lemma \ref{lem:convergenceifuniqueness}] 
Throughout the proof, we take for granted the notations introduced in the statement of  Lemma \ref{lem:convergenceifuniqueness}.
Following  \eqref{eq:tobefollowed24/05} in the preliminary step of the proof of Proposition \ref{prop:LipRegValueFuncion25Nov} together with the fact that $(\gamma_0^n)_{n \geq 1}$ converges toward $\gamma_0$ in $\mathcal{P}_2(\R^{d_1} \times \R^{d_2})$, we have
\begin{equation} 
\sup_{ n \geq 1} \Bigl \{ \int_{t_0^n}^T \int_A |a|^4 d\nu^{*,n}_t(a)dt + \int_{\R^{d_1} \times \R^{d_2}} (|x|^2 + |y|^2) d\gamma_0^n(x,y) \Bigr \} < +\infty. 
\label{eq:L4bound06/01/2025}
\end{equation}
By extending 
$\bd{\nu}^{*,n}$
to the entire $[t_0,T]$ if necessary, letting in this case 
$\nu^{*,n}_t = \nu^{*,n}_{t_0^n}$ for $t \in [t_0,t_0^n)$, 
we obtain, for any $n \geq 1$, an admissible control $\bar{\bd{\nu}}^n \in \mathcal{A}(t_0 \wedge t_0^n)$
(the fact that 
it satisfies 
\eqref{eq:definitionadmissiblecontrols}
follows from 
\eqref{eq:estimaterelativeentropy25/12}, applied to 
${\boldsymbol \nu}^{*,n}$). Taking $(t_0,\gamma_0)$ as an initial condition for 
each $n \geq 1$, we then let $\bar{\bd{\gamma}}^n = (\bar{\gamma}^n_t)_{t \in [t_0,T]}$ the resulting curve starting from $(t_0,\gamma_0)$ (which curve should not be confused with 
${\boldsymbol \gamma}^{*,n}$). 
\vskip 4pt

\textit{Step 1.} We first claim that $(\bar{\bd{\gamma}}^n)_{ n \geq 1}$ converges to $\bd\gamma^*$ in $\mathcal{C}([t_0,T], \mathcal{P}_1(\R^{d_1} \times \R^{d_2}))$. Our strategy is to prove that $(\bd{\bar{\nu}}^n)_{ n \geq 1}$ is a minimizing sequence for $J  ( (t_0,\gamma_0), \cdot  )$ and then to use Lemma \ref{lem:cost:convergence}. To this end, we express the cost of $\bd{\bar{\nu}}^n$ as
\begin{align}
    J \bigl( (t_0,\gamma_0), \bd{\bar{\nu}}^n \bigr) &= \epsilon \int_{t_0}^T \mathcal{E} \bigl( \bar{\nu}_t^n | \nu^{\infty} \bigr) dt + \int_{\R^{d_1} \times \R^{d_2}} L(x,y) d\bar{\gamma}_T^n(x,y) \label{eq:06/01_11:28} \\
    &= J \bigl( (t^n_0,\gamma^n_0), \bd{\nu}^{*,n} \bigr) + \epsilon \int_{t_0}^T \mathcal{E} \bigl( \bar{\nu}_t^n | \nu^{\infty} \bigr)dt - \epsilon \int_{t_0^n}^T \mathcal{E} \bigl( \nu_t^{*,n} | \nu^{\infty} \bigr) dt \notag \\
    &\hspace{15pt} + \int_{\R^{d_1} \times \R^{d_2}} L(x,y) d \bar{\gamma}_T^n(x,y) - \int_{\R^{d_1} \times \R^{d_2}} L(x,y) d\gamma_T^{*,n}(x,y). \notag
\end{align}
On the one hand, using the estimate \eqref{eq:estimaterelativeentropy25/12} in Proposition \ref{prop:regularityfromOC} 
and the fact that 
$\bar{\boldsymbol \nu}^n$ and 
${\boldsymbol \nu}^{*,n}$
coincide on $[t_0^n,T]$,
we find a constant $C >0$ that is independent of $n \geq 1$ thanks to \eqref{eq:L4bound06/01/2025}, such that
\begin{equation}
\label{eq:proof:convergence:comparison:entropies}
 \Bigl| \int_{t_0}^T \mathcal{E} \bigl( \bar{\nu}_t^n | \nu^{\infty} \bigr)dt -  \int_{t_0^n}^T \mathcal{E} \bigl( \nu_t^{*,n} | \nu^{\infty} \bigr) dt \Bigr| \leq C |t_0^n -t_0|. 
 \end{equation}
On the other hand, introducing $\bd{\bar{u}}^n : [t_0^n \wedge t_0, T] \times \R^{d_1} \times \R^{d_2} \rightarrow \R$ the solution to the backward equation associated to the control $\bd{\bar{\nu}}^n$ 
and observing that $\boldsymbol{\bar{u}}^n$ coincides with ${\boldsymbol u}^{*,n}$
on $[t_0^n \vee t_0,T]$, 
we get by the duality relation stated in Lemma \ref{lem:duality:1,2:F} (applied twice, once 
with $(t_0,\gamma_0)$ as initial condition and 
${\boldsymbol \nu}^1={\boldsymbol \nu}^2=\overline{\boldsymbol \nu}^n$
as control, and once with 
$(t_0^n,\gamma_0^n)$ as initial condition and 
${\boldsymbol \nu}^1={\boldsymbol \nu}^2={\boldsymbol \nu}^{\star,n}$
as control), 
\begin{equation}
\label{eq:proof:convergence:comparison:L}
\begin{split}
&\int_{\R^{d_1} \times \R^{d_2}} L(x,y) d \bar{\gamma}_T^n(x,y) - \int_{\R^{d_1} \times \R^{d_2}} L(x,y) d\gamma_T^{*,n}(x,y)\\
&= \int_{\R^{d_1} \times \R^{d_2}} \bar{u}_{t_0}^n(x,y) d\gamma_0(x,y) - \int_{\R^{d_1} \times \R^{d_2}} \bar{u}_{t_0^n}^n(x,y) d\gamma_0^n(x,y) \\
&= \int_{\R^{d_1} \times \R^{d_2}} \bigl[ \bar{u}_{t_0}^n(x,y) - \bar{u}^n_{t^n_0}(x,y) \bigr] d\gamma_0(x,y) + \int_{\R^{d_1} \times \R^{d_2}} \bar{u}^n_{t_0^n} (x,y) d(\gamma_0- \gamma_0^n) (x,y). 
\end{split}
\end{equation}
The estimate from Proposition \ref{prop:SolutionBackxard16Sept} gives
\begin{align*}
\int_{\R^{d_1} \times \R^{d_2}} \bigl[\bar{u}_{t_0}^n(x,y) - \bar{u}^n_{t^n_0}(x,y) \bigr] d\gamma_0(x,y) &\leq C \norm{ \bar{u}_{t_0}^n - \bar{u}_{t_0^n}^n}_{\mathcal{C}^0_1} \biggl( \int_{\R^{d_1} \times \R^{d_2}} (|x| + |y|)d\gamma_0(x,y) \biggr)  \\
&\leq C\sqrt{|t_0^n - t_0|}.
\end{align*}
Moreover, by Lemma \ref{lem:comparisonNorms}, 
we get
\begin{align*} 
&\int_{\R^{d_1} \times \R^{d_2}} \bar{u}^n_{t_0^n} (x,y) d(\gamma_0- \gamma_0^n) (x,y) \\
&\leq \bigl\| \bar{u}^n_{t_0^n} \bigr\|_{\mathcal{C}^1_{2,1}} \norm{ \gamma_0 - \gamma_0^n}_{(\mathcal{C}^1_{2,1})^*} \\
&\leq   \bigl\|  \bar{u}^n_{t_0^n} \bigr\|_{\mathcal{C}^1_{2,1}} 
\biggl[ 1+  \sqrt{2} \biggl(   \int_{\R^{d_1} \times \R^{d_2}} (|x|^2 + |y|^2) d \bigl( \gamma_0 + \gamma_0^n \bigr) (x,y) \biggr)^{1/2}  \biggr] d_2(\gamma_0, \gamma_0^n).
\end{align*}
Thanks to Proposition \ref{prop:SolutionBackxard16Sept} and 
 \eqref{eq:L4bound06/01/2025} again, 
we can bound 
$\|  \bar{u}^n_{t_0^n} \|_{\mathcal{C}^1_{2,1}} $
 by a constant independent of $n$. And then, the left-hand side is less than
$C d_2(\gamma_0,\gamma_0^n)$,
for some $C>0$ independent of $ n  \geq 1$. Thanks to the last two displays, we deduce that the left-hand side in 
\eqref{eq:proof:convergence:comparison:L}
is less than $C(\sqrt{\vert t_0^n - t_0 \vert} + d_2(\gamma_0,\gamma_0^n))$. 
Combining this bound together with  \eqref{eq:06/01_11:28} and 
\eqref{eq:proof:convergence:comparison:entropies}, and recalling the definition of the value function $U$ in 
\eqref{eq:value function:U}, we obtain
$$ J \bigl( (t_0,\gamma_0), \bd{\bar{\nu}}^n \bigr) \leq U (t_0^n, \gamma_0^n) + C \sqrt{|t_0^n-t_0|} + d_2(\gamma_0^n,\gamma_0).  $$
Using Proposition \ref{prop:LipRegValueFuncion25Nov}, we get
$$ \limsup_{n \rightarrow +\infty} J \bigl( (t_0,\gamma_0), \bd{\bar{\nu}}^n \bigr) \leq U(t_0,\gamma_0). $$
This implies that $(\bd{\bar{\nu}}^n)_{ n \geq 1}$ is a minimizing sequence for $J( (t_0,\gamma_0), \cdot ).$ By uniqueness, we deduce from the compactness argument of Lemma \ref{lem:cost:convergence}  that $(\bd{\bar{\nu}}^n)_{ n \geq 1}$ converges weakly to $\bd\nu^*$. By Lemma \ref{lem:cost:convergence} again, we have
$$ \lim_{n \rightarrow +\infty} \sup_{t \in [t_0,T]} d_2 \bigl(\bar{\gamma}_t^n , \gamma^*_t \bigr) =0.$$
\textit{Step 2.} Now we claim that 
$$ \lim_{n \rightarrow +\infty} \sup_{t \in [t_0^n \vee t_0,T] } d_2 (\bar{\gamma}_t^n, \gamma_t^{*,n}) = 0.$$
Indeed,
because the continuity 
equations for both $\bar{\boldsymbol \gamma}^n$ and 
${\boldsymbol \gamma}^{*,n}$ are driven by the same vector field after 
$t_0^n \vee t_0$
(which vector field is Lipschitz continuous uniformly in $n \geq 1$), 
it suffices, by 
an elementary stability argument, 
to compare 
$\bar{\gamma}^n_{t_0^n \vee t_0}$
and 
${\gamma}^{*,n}_{t_0^n \vee t_0}$, and thus to bound 
$d_2(\bar{\gamma}^n_{t_0^n \vee t_0},{\gamma}^{*,n}_{t_0^n \vee t_0})$.
Obviously, 
$$d_2\bigl(\bar{\gamma}^n_{t_0^n \vee t_0},{\gamma}^{*,n}_{t_0^n \vee t_0}\bigr) 
\leq 
d_2\bigl(\bar{\gamma}^n_{t_0^n \vee t_0},{\gamma}_0\bigr)
+
d_2\bigl({\gamma}_0,\gamma_0^n\bigr)
+d_2\bigl(\gamma_0^n,{\gamma}^{*,n}_{t_0^n \vee t_0}\bigr).
$$
We recall that $\gamma_0$ is the initial condition of 
$\bar{\boldsymbol \gamma}^n$ and
$\gamma_0^n$ the initial condition of 
${\boldsymbol \gamma}^{*,n}$. 
Following the last paragraph in the proof of Proposition \ref{prop:LipRegValueFuncion25Nov}, one can prove that $\bar{\boldsymbol \gamma}^n$ and ${\boldsymbol \gamma}^{*,n}$  are  time Lipschitz continuous, uniformly in $n \geq 1$. 
We deduce that
$$\sup_{t \in [t_0^n \vee t_0,T] } d_2 (\bar{\gamma}_t^n, \gamma_t^{*,n}) \leq C \bigl( |t_0^n - t_0| + d_2 ( \gamma_0^n, \gamma_0) \bigr),$$
for some $C>0$ independent of $n \in \mathbb{N}^*$. The claim follows easily.
\vskip 4pt

\textit{Conclusion.} It remains to apply Lemmas \ref{lem:backthenitwasLemma6.2} and \ref{lem:comparisonNorms} to conclude.

\end{proof}

\section{The Linearized Continuity and Transport Equations}

\label{sec:LinearizedEquations}

The goal of this section is to address the well-posedness of the linearized continuity equation \eqref{eq:rhotProp3.2} as well as the linearized transport equation \eqref{eq:linearizedeqvt04/03} and, in particular, to prove Propositions \ref{prop:differentiatinggammalambda09/02} and  \ref{prop:SolnVt19/02}. The corresponding formulation of these results are given in Propositions \ref{prop:differentiatinggammalamnda25/01},  
\ref{prop:A.22} and \ref{prop:A.24} for the linearized continuity equation, and in Proposition \ref{prop:linearized:transport:equation:representation} for the linearized transport equation. These results are essential for establishing the second-order optimality conditions of Theorem \ref{prop:SecondOrderConditions3Avril}, which is the purpose of Section \ref{sec:SOC}. 

The current section is somewhat lengthy and technical, and may be skipped on a first reading.

\subsection{The Linearized Continuity Equation}
\label{subse:linearized:equations}
The focus of this subsection is  the linearized equation \eqref{eq:rhotProp3.2}, that we rewrite here:
\begin{equation}
\left \{
\begin{array}{ll}
\displaystyle \partial_t \rho_t  + \div_x \bigl(  b(x,\nu_t) \rho_t \bigr) =- \div_x \bigl(  b(x,\eta_t) \gamma_t \bigr) &\quad \mbox{in } (t_0,T) \times \R^{d_1} \times\R^{d_2},
\\
\displaystyle \rho_{t_0}=0 &\quad \mbox{in } \R^{d_1} \times \R^{d_2}.
\end{array}
\right.
\label{eq:rhotProp3.2Append}
\end{equation}
We also recall that the notion of solution is given in Subsection \ref{subse:SOC04/03}.

\subsubsection{A preliminary measurability result}

\label{sec:AppendixFubini}

\begin{prop}
    Take $\varphi : [t_0,T] \times \R^{d_1} \times \R^{d_2} \rightarrow \R$ with $ \nabla_x \varphi \in \mathcal{C} \bigl( [t_0,T], \mathcal{C}^1_{2}) \cap \mathcal{L}^{\infty}([t_0,T], \mathcal{C}^2_{1})$
and $ \bd{\rho} \in \mathcal{C}([t_0,T], (\mathcal{C}^2_{1})^*) $ such that $ \sup_{t \in [t_0,T]} \sup_{\phi \in \mathcal{C}^2_1, \norm{ \phi}_{\mathcal{C}^1_2} \leq 1} | \langle \phi; \rho_t \rangle |$ is finite. Then,
\begin{enumerate}
    \item For all $t \in [t_0,T]$, the map $a \mapsto \langle b(\cdot,a)\cdot \nabla_x \varphi_t; \rho_t \bigr \rangle$ is measurable and, for all $\nu \in \mathcal{M}(A)$ such that $(1+|a|^2) \nu$ belongs to $\mathcal{M}(A)$, $\int_A | \langle b(\cdot,a) \cdot \nabla_x \varphi_t; \rho_t \rangle |d|\nu|(a) $ is finite and we have
    $$ \int_A \langle b(\cdot,a) \cdot \nabla_x \varphi_t; \rho_t \rangle d\nu(a) = \langle b(\cdot, \nu)\cdot \nabla_x \varphi_t; \rho_t \rangle. $$
    \item \label{lem:timemeasurabilitybracket12/01} For all $\bd{\nu} \in \mathcal{D}(t_0)$, the map $ t\mapsto \langle b(\cdot, \nu_t ) \cdot \nabla_x \varphi_t ; \rho_t \rangle $ is measurable and coincides with $ t \mapsto \int_A \langle b(\cdot, a) \cdot \nabla_x \varphi_t; \rho_t \rangle d \nu_t(a)$. Moreover, it satisfies 
\begin{equation} 
\int_{t_0}^T \bigl| \langle b(\cdot, \nu_t ) \cdot \nabla_x \varphi_t ; \rho_t \rangle \bigr| dt \leq C_b \sup_{t \in [t_0,T]} \sup_{\phi \in \mathcal{C}^2_1, \norm{ \phi}_{\mathcal{C}^1_2} \leq 1} | \langle \phi; \rho_t \rangle | \sup_{t \in [t_0,T]} \norm{ \nabla_x \varphi_t}_{\mathcal{C}^1_2} \int_{t_0}^T \int_A (1+|a|^2)d|\nu_t|(a)dt,
\label{eq:L1boundDualityBracket09/02}
\end{equation}
for a constant $C_b>0$ depending only on the vector field $b$.
\end{enumerate}
\label{prop:measurable+fubini15/01}
\end{prop}
Regarding point (2) in the statement, observe that, as a consequence of the first point, the function \( t \mapsto \langle b(\cdot, \nu_t) \cdot \nabla_x \varphi_t ; \rho_t \rangle \) is a priori well-defined only at times \( t \) where \( \nu_t \) belongs to \( \mathcal{M}_{1+|a|^2} \). By the definition of elements of \( \mathcal{D}(t_0) \), this set is measurable and has full Lebesgue measure. Therefore, it is indeed possible to find a measurable version of the function \( t \mapsto \langle b(\cdot, \nu_t) \cdot \nabla_x \varphi_t ; \rho_t \rangle \) and our result says that 
the bound in equation \eqref{eq:L1boundDualityBracket09/02} holds.

As for the proof of Proposition \ref{prop:measurable+fubini15/01}, it relies on a series of intermediate steps. We start with a suitable approximation of the vector field $b$, based on the following notations. For every $\epsilon >0$ we consider a partition of $A$ into Borel subsets $A = \sqcup_{p \geq 1} A_p(\epsilon) $ such that for every 
integer
$p \geq 1$, there is $a_p(\epsilon) \in A_p(\epsilon)$ such that $A_p(\epsilon) \subset B(a_p(\epsilon),\epsilon)$ (the ball of center $a_p(\epsilon)$ and radius $\epsilon$). For $\epsilon >0$ and $P \geq 1$, we define  
$$b^{\epsilon,P}(a,x) := \sum_{p =1}^P b(x,a_p(\epsilon)) \mathbf{1}_{A_p(\epsilon)}(a), \quad b^{\epsilon}(a,x) = \sum_{p \geq 1} b(x,a_p(\epsilon)) \mathbf{1}_{A_p(\epsilon)}(a). $$

\begin{lem}
\label{lem:approxb03/07}
    There is $C>0$, independent of $a \in A$ and $\epsilon >0$, such that
\begin{equation} 
\norm{ b^{\epsilon}(\cdot,a) - b(\cdot,a)}_{\mathcal{C}^1_{1}} \leq C \epsilon (1+|a|^2).
\label{eq:bepsC111}
\end{equation}
Moreover, for all $a \in A$, 
\begin{equation} 
\lim_{P \rightarrow +\infty} \norm{b^{\epsilon,P}(\cdot,a) - b^{\epsilon}(\cdot,a) }_{\mathcal{C}^1_b} =0. 
\label{eq:bepsPC1b}
\end{equation}
\end{lem}

\begin{proof}
    For all $x \in \R^{d_1}$ and all $a \in A$, $ b^{\epsilon}(x,a) = b(x,a_p(\epsilon))$ if $ a \in A_p(\epsilon).$ By assumption on $b$ and $\nabla_x b $ we deduce that
$$|b^{\epsilon}(x,a) - b(x,a) | +  |\nabla_x b^{\epsilon}(x,a) - \nabla_xb(x,a)| \leq C(1+|x|)(1+|a|^2) \epsilon$$
which leads to \eqref{eq:bepsC111}.
For the second estimate, we have
$$ \norm{ b^{\epsilon,P}(\cdot,a) - b^{\epsilon}(\cdot,a) }_{\mathcal{C}^1_b}  \leq \sum_{p=P+1}^{+\infty} \norm{ b(\cdot, a_p(\epsilon))}_{\mathcal{C}^1_b}\mathbf{1}_{A_p(\epsilon)}(a) \leq C \sum_{p=P+1}^{+\infty} (1+|a_p(\epsilon)|^2) \mathbf{1}_{A_p(\epsilon)}(a). $$
For every $a \in A$, there is $p_0$ such that $\mathbf{1}_{A_p(\epsilon)}(a)=0 $ whenever $p \geq p_0$ and we easily deduce that \eqref{eq:bepsPC1b} holds for every $a \in A$.
\end{proof}

Item (1) in the statement of Proposition 
\ref{prop:measurable+fubini15/01} is a consequence of the following lemma:
\begin{lem}
     Let $\varphi : \R^{d_1} \times \R^{d_2} \rightarrow \R$ be a function such that  $\nabla_x \varphi \in \mathcal{C}^2_1$ and let $\rho$ be en element of $\bigl( \mathcal{C}^2_1 \bigr)^* $  such that $\sup_{ \phi \in \mathcal{C}^2_1, \norm{ \phi}_{\mathcal{C}^1_2} \leq 1} |\langle \phi; \rho \rangle |$ is finite.   Then $a \mapsto \langle b(\cdot,a) \cdot \nabla_x \varphi; \rho \rangle$ is Borel-measurable and, for any $\nu \in \mathcal{M}(A)$ such that $(1+|a|^2)\nu$ belongs to $\mathcal{M}(A)$, $\int_A |\langle b(\cdot,a) \cdot \nabla_x \varphi; \rho \rangle| d|\nu|(a)$ is finite and it holds
\begin{equation} 
\int_A \langle b(\cdot,a) \cdot \nabla_x \varphi; \rho \rangle d\nu(a) = \langle b(\cdot,\nu) \cdot \nabla_x \varphi; \rho \rangle.  
\label{eq:Fubini14/01}
\end{equation}
\label{lem:measurabilityina31/01}
\end{lem}

\begin{proof}
    For every $a \in A$, we have
\begin{align*}
    \bigl| \langle b^{\epsilon}(\cdot,a) \cdot \nabla_x \varphi; \rho \rangle - \langle b(\cdot,a) \cdot \nabla_x \varphi; \rho \rangle \bigr| & \leq \norm{ [b^{\epsilon}(\cdot,a) - b(\cdot,a) ] \cdot \nabla_x \varphi}_{\mathcal{C}^1_{2}}  \sup_{ \phi \in \mathcal{C}^2_1, \norm{ \phi}_{\mathcal{C}^1_2} \leq 1} |\langle \phi; \rho \rangle | 
    \nonumber
    \\
    &\leq C \norm{ b^\epsilon(\cdot,a) - b(\cdot,a)}_{\mathcal{C}^1_{1}} \norm{\nabla_x \varphi}_{\mathcal{C}^1_{1}}   \sup_{ \phi \in \mathcal{C}^2_1, \norm{ \phi}_{\mathcal{C}^1_2} \leq 1} |\langle \phi; \rho \rangle | 
    \nonumber
    \\
    &\leq C \epsilon (1+|a|^2) \norm{\nabla_x \varphi}_{\mathcal{C}^1_{1}}  \sup_{ \phi \in \mathcal{C}^2_1, \norm{ \phi}_{\mathcal{C}^1_2} \leq 1} |\langle \phi; \rho \rangle |,
\end{align*}
which shows that 
\begin{equation}
\lim_{\epsilon \rightarrow 0} \sup_{a \in A} \frac{ \bigl| \langle b^{\epsilon}(\cdot,a) \cdot \nabla_x \varphi; \rho \rangle - \langle b(\cdot,a) \cdot \nabla_x \varphi; \rho \rangle \bigr|}{1+|a|^2} =0.
    \label{eq:passage:limit:epsilon}
\end{equation}
Similarly, we can prove that, for every $\epsilon >0$ and for every $a \in A$
\begin{equation}
 \lim_{P \rightarrow +\infty} \langle b^{\epsilon,P}(\cdot,a)\cdot \nabla_x \varphi; \rho \rangle = \langle b^{\epsilon}(\cdot,a)\cdot \nabla_x \varphi; \rho \rangle. 
\label{eq:passage:limit:P}
 \end{equation}
For every $\epsilon >0$ and every $P \geq 1$, we have, by linearity of the duality bracket,
\begin{equation}  
\langle b^{\epsilon,P}(\cdot,a)\cdot \nabla_x \varphi; \rho \rangle = \sum_{p=1}^P \langle b(\cdot, a_p(\epsilon))\cdot \nabla_x \varphi;\rho \rangle \mathbf{1}_{A_p(\epsilon)}(a), \quad a \in A. 
\label{eq:bracketforbepsP03/07}
\end{equation}
In particular, $a \mapsto \langle b^{\epsilon,P}(\cdot,a) \cdot \nabla_x \varphi; \rho \rangle$ is Borel measurable for all $\epsilon >0$ and $P \geq 1$ and so is the map $a \mapsto \langle b(\cdot,a) \cdot \nabla_x \varphi; \rho \rangle$ as the point-wise limit of a sequence of measurable functions. From \eqref{eq:bracketforbepsP03/07} and thanks again to the linearity of the duality bracket, we also deduce that, for all $\nu \in \mathcal{M}(A)$,
\begin{equation}
\int_A \langle b^{\epsilon,P}(\cdot,a) \cdot \nabla_x \varphi; \rho \rangle d\nu(a) = \langle b^{\epsilon,P}(\cdot,\nu) \cdot \nabla_x \varphi; \rho \rangle.
\label{eq:FubinibepsP03/07}
\end{equation}
By assumption on $b$, see in particular \eqref{eq:firstconvenientexpress02/03} in Lemma \ref{lem:convenientexpress}, there is $C>0$ such that, for all $\epsilon >0$, all $P >0$ and all $a \in A$,
$$|\langle b^{\epsilon,P}(\cdot,a) \cdot \nabla_x \varphi; \rho \rangle | \leq C(1+|a|^2) \norm{\nabla_x \varphi}_{\mathcal{C}^1_{1}} \norm{\rho}_{(\mathcal{C}^1_{1})^*}.$$
Therefore, if $\nu$ belongs to $\mathcal{M}_{1+|a|^2}(A)$, the quantity $\int_A |\langle b(\cdot,a) \cdot \nabla_x \varphi; \rho \rangle| d|\nu|(a)$ is finite and we can then use Lebesgue dominated convergence theorem to pass to the limit in the left-hand side of \eqref{eq:FubinibepsP03/07} to get
$$ \lim_{\epsilon \rightarrow 0} \lim_{P \rightarrow +\infty} \int_A \langle b^{\epsilon,P}(\cdot,a) \cdot\nabla_x \varphi; \rho \rangle d\nu(a) = \int_A \langle b(\cdot,a) \cdot\nabla_x \varphi; \rho \rangle d\nu(a). $$
We now address the limit in the right-hand side of 
\eqref{eq:FubinibepsP03/07}. By the same argument as above,
\begin{align*}
    | \langle b^{\epsilon,P}(\cdot, \nu) \cdot \nabla_x \varphi ; \rho \rangle - \langle b(\cdot, \nu) \cdot \nabla_x \varphi ; \rho  \rangle | &\leq C \norm{b^{\epsilon,P}( \cdot,\nu) - b(\cdot,\nu) }_{\mathcal{C}^1_{1}} \norm{\nabla_x\varphi}_{\mathcal{C}^1_{1}} \norm{\rho}_{(\mathcal{C}^1_{2})^*} \\
    &\leq C \int_A \norm{b^{\epsilon,P}( \cdot,a) - b(\cdot,a) }_{\mathcal{C}^1_{1}} d|\nu|(a).
\end{align*}
In order to show that the right-hand side tends to $0$, 
we use Lemma \ref{lem:approxb03/07} to infer that, for every $a \in A$, $$\lim_{\epsilon \rightarrow 0} \lim_{P \rightarrow +\infty} \norm{ b^{\epsilon,P}(\cdot,a) - b(\cdot,a)}_{\mathcal{C}^1_{1}} =0.$$ We also observe that, for a new value of the constant $C>0$ (which is independent of $\epsilon>0$ and $P \geq 1$), 
$$ \norm{ b^{\epsilon,P}(\cdot,a) - b(\cdot,a)}_{\mathcal{C}^1_{1}} \leq C(1+|a|^2).$$
We can apply once again Lebesgue dominated convergence theorem to conclude that \eqref{eq:Fubini14/01} holds true.
\end{proof}

We also need the following lemma in order to establish the measurability property stated in the first item of Proposition 
\ref{prop:measurable+fubini15/01}:

\begin{lem}
Let $\varphi: [t_0,T] \times \mathbb{R}^{d_1} \times \mathbb{R}^{d_2} \rightarrow \mathbb{R}$ be a function such that $ \nabla_x \varphi \in \mathcal{C}([t_0,T], \mathcal{C}^1_2) \cap \mathcal{L}^{\infty}([t_0,T], \mathcal{C}^2_1)$, and let $ \rho \in \mathcal{C}([t_0,T], (\mathcal{C}^2_1)^*)$ with $ \sup_{t \in [t_0,T]} \sup_{ \varphi \in \mathcal{C}_1^2, \norm{ \varphi}_{\mathcal{C}_2^1} \leq 1} \langle \varphi; \rho_t \rangle$  finite.
Then, for any $a \in A$, the map $t \mapsto \langle b(\cdot, a) \cdot \nabla_x \varphi_t; \rho_t \rangle$ is continuous. Moreover, for all $ t_1 < t_2 \in [t_0,T] $, we have the following inequality:
\begin{align*} 
&| \langle b(\cdot,a) \cdot \nabla_x \varphi_{t_2}; \rho_{t_2} \rangle - \langle b(\cdot,a) \cdot \nabla_x \varphi_{t_1}; \rho_{t_1} \rangle | \leq C(1+|a|^3) \\
& \times \Bigl( \sup_{t \in [t_0,T]} \norm{ \nabla_x \varphi_t}_{\mathcal{C}^2_{1}} + \sup_{t \in [t_0,T]} \sup_{ \varphi \in \mathcal{C}_1^2, \norm{ \varphi}_{\mathcal{C}_2^1} \leq 1} \langle \varphi; \rho_t \rangle\Bigr) \Bigl(\norm{ \nabla_x \varphi_{t_2} - \nabla_x \varphi_{t_1}}_{\mathcal{C}^1_{2}} + \norm{\rho_{t_2} - \rho_{t_1}}_{(\mathcal{C}^2_{1})^*} \Bigr).  
\end{align*}
\label{lem:BackThenLemA.17}
\end{lem}

\begin{proof}
    With the same notation as in the statement, 
    we have, for any $a \in A$,
\begin{align*}
    &| \langle b(\cdot,a) \cdot \nabla_x \varphi_{t_2}; \rho_{t_2} \rangle - \langle b(\cdot,a) \cdot \nabla_x \varphi_{t_1}; \rho_{t_1} \rangle | \\
    &\leq | \langle b(\cdot,a) \cdot \nabla_x( \varphi_{t_2} - \varphi_{t_1}); \rho_{t_2} \rangle | + |\langle b(\cdot,a) \cdot \nabla_x \varphi_{t_1}; \rho_{t_2} - \rho_{t_1} \rangle | \\
    &\leq \norm{\rho_{t_2}}_{(\mathcal{C}^1_2)^*} \norm{ b(\cdot,a) \cdot \nabla_x (\varphi_{t_2} - \varphi_{t_1}) }_{\mathcal{C}^1_2} + \norm{ b(\cdot,a) \cdot \nabla_x \varphi_{t_1}}_{\mathcal{C}^2_1} \norm{ \rho_{t_2} - \rho_{t_1}}_{(\mathcal{C}^2_1)^*}.
\end{align*}
And then, we complete the proof by noticing from \assreg \, that
\begin{equation*}
\begin{split}
&\norm{  b(\cdot,a) \cdot \nabla_x( \varphi_{t_2} - \varphi_{t_1})}_{\mathcal{C}^1_{2}} \leq C(1+|a|^2) \norm{  \nabla_x( \varphi_{t_2} - \varphi_{t_1})}_{\mathcal{C}^1_{2}},  
\\
& \norm{b(\cdot,a) \cdot \nabla_x \varphi_{t_1} }_{\mathcal{C}^2_{1}} \leq C(1+|a|^3) \norm{ \nabla_x \varphi_{t_1}}_{\mathcal{C}^2_{1}},
\end{split}
\end{equation*}
for $a \in A$. 
\end{proof}

We can finally finish the Proof of Proposition \ref{prop:measurable+fubini15/01}.

\begin{proof}[Proof of Proposition \ref{prop:measurable+fubini15/01}]
The first item in the statement follows from Lemma
\ref{lem:measurabilityina31/01}. 
It only remains to prove point \eqref{lem:timemeasurabilitybracket12/01}.
    We let $\tau :=T- t_0$ and, for all $N \geq 1$ and $1 \leq i \leq N $, $t_{i}^N := t_0 + \frac{i-1}{N} \tau$, so that $[t_0,T]$ can be decomposed into $\sqcup_{i=1}^{N-1} [t_{i}^N, t_{i+1}^N[  \sqcup [t_N^N,T].$ 

We first claim that, as a consequence of the measurability of the 
    map $t \in [t_0,T] \mapsto (1+ \vert a \vert^2)^{-1} \nu_t \in {\mathcal M}(A)$,   the map
$t \in [t_0,T] \mapsto \int_A \langle b(\cdot, a) \cdot \nabla_x \varphi_{t_i^N}; \rho_{t_i^N} \rangle d\nu_t(a) $ is measurable
for each fixed 
pair $(N,i)$, with $N\geq 1$ and $i \in \{1,\cdots,N\}$. Indeed, 
we know that 
the map 
$t \in [t_0,T] \mapsto \nu_t \in \mathcal{M}(A)$ is measurable.
Therefore, for any $R \geq 1$, the map 
$t \in [t_0,T] \mapsto \int_{\{\vert a \vert \leq R\}} \langle b(\cdot, a) \cdot \nabla_x \varphi_{t_i^N}; \rho_{t_i^N} \rangle d\nu_t(a) $ 
is measurable. 
Since $\bd{\nu}$ belongs to $\mathcal{D}(t_0)$, 
the set $\{ t \in [t_0,T] : \int_A (1+|a|^4)d\nu_t(a)
< \infty\}$ is a Borel subset of $[t_0,T]$ and has full Lebesgue measure. 
In particular, the map $t \in [t_0,T] \mapsto {\mathbf 1}_E(t) \int_{\{\vert a \vert \leq R\}} \langle b(\cdot, a) \cdot \nabla_x \varphi_{t_i^N}; \rho_{t_i^N} \rangle d\nu_t(a)$ is measurable. 
Since the function $a \in A \mapsto \langle b(\cdot, a) \cdot \nabla_x \varphi_{t_i^N}; \rho_{t_i^N} \rangle$ is measurable with quadratic growth (uniformly in $t$), we can easily let $R$ tend to $+\infty$
and deduce that the map 
$t \in [t_0,T] \mapsto {\mathbf 1}_E(t) \int_{A} \langle b(\cdot, a) \cdot \nabla_x \varphi_{t_i^N}; \rho_{t_i^N} \rangle d\nu_t(a)$ is measurable, which suffices to construct a measurable version
of $t \in [t_0,T] \mapsto \int_{A} \langle b(\cdot, a) \cdot \nabla_x \varphi_{t_i^N}; \rho_{t_i^N} \rangle d\nu_t(a)$.

For $(N,i)$ fixed as above and for $E$ also defined as above, we observe from Lemma \ref{lem:measurabilityina31/01} that $t \in [t_0,T] \mapsto 
{\mathbf 1}_E(t)
\int_A \langle b(\cdot, a) \cdot \nabla_x \varphi_{t_i^N}; \rho_{t_i^N} \rangle d\nu_t(a) $ coincides with $t \in [t_0,T] \mapsto {\mathbf 1}_E(t) \langle b(\cdot, \nu_t) \cdot \nabla_x \varphi_{t_i^N}; \rho_{t_i^N} \rangle$.
We deduce that, for all $t \in E$, 
\begin{equation}
\label{eq:19:mars:2025:1}
\begin{split}
&\sum_{i=1}^{N-1} \langle b(\cdot, \nu_t) \cdot \nabla_x \varphi_{t_i^N}; \rho_{t_i^N} \rangle \mathbf{1}_{[t_i^N,t_{i+1}^N[}(t) + \langle b(\cdot, \nu_t) \cdot \nabla_x \varphi_{t_N^N}; \rho_{t_N^N} \rangle \mathbf{1}_{[t_N^N,t_{N+1}^N]}(t) 
\\
&= \int_A  \Bigl \{ \sum_{i=1}^N \langle b(\cdot, a) \cdot \nabla_x \varphi_{t_i^N}; \rho_{t_i^N} \rangle \mathbf{1}_{[t_i^N,t_{i+1}^N[}(t) +\langle b(\cdot, a) \cdot \nabla_x \varphi_{t_N^N}; \rho_{t_N^N} \rangle \mathbf{1}_{[t_N^N,t_{N+1}^N]}(t) \Bigr \} d\nu_t(a).
\end{split}
\end{equation}
By the previous paragraph, the function 
\begin{equation*}
{\boldsymbol \beta}^N : t \mapsto {\mathbf 1}_E(t) \biggl( \sum_{i=1}^{N-1} \langle b(\cdot, \nu_t) \cdot \nabla_x \varphi_{t_i^N}; \rho_{t_i^N} \rangle \mathbf{1}_{[t_i^N,t_{i+1}^N[}(t) + \langle b(\cdot, \nu_t) \cdot \nabla_x \varphi_{t_N^N}; \rho_{t_N^N} \rangle \mathbf{1}_{[t_N^N,t_{N+1}^N]}(t) 
\biggr)
\end{equation*}
is measurable. 
The goal is then to pass to the limit as $N$ tends to $+\infty$. 
By Lemma \ref{lem:BackThenLemA.17},
we have, for all $t \in E$, 
\begin{align*} 
\bigl| {\boldsymbol \beta}^N(t) - 
\langle b(\cdot, \nu_t) \cdot \nabla_x \varphi_{t}; \rho_{t} \rangle
\bigr\vert  
&\leq C \biggl( \int_A (1+|a|^3) d\nu_t(a) \biggr) \Bigl( \sup_{s \in [t_0,T]} \norm{ \nabla_x \varphi_s}_{\mathcal{C}^2_{1}} + \sup_{s\in [t_0,T]} \norm{\rho_s}_{(\mathcal{C}^1_{2})^*} \Bigr)
\\
&\hspace{15pt} \times 
\sup_{ \vert s - r \vert \leq \tau/N} \Bigl(\norm{ \nabla_x \varphi_{r} - \nabla_x \varphi_{s}}_{\mathcal{C}^1_{2}} + \norm{\rho_{r} - \rho_{s}}_{(\mathcal{C}^2_{1})^*} \Bigr).  
\end{align*}
By definition of $E$, the right-hand side tends to $0$ as $N$ tends to $+\infty$. We deduce that 
the function ${\boldsymbol \beta} :  t \in [t_0,T] \mapsto {\mathbf 1}_E(t) \langle b(\cdot, \nu_t) \cdot \nabla_x \varphi_t; \rho_t \rangle $ is measurable as point-wise limit of measurable functions. Passing to 
the limit in 
\eqref{eq:19:mars:2025:1} (thanks again to 
Lemma \ref{lem:BackThenLemA.17}), 
we see that 
${\boldsymbol \beta}$
coincides with $t \in [t_0,T] \mapsto {\mathbf 1}_E(t) \int_A  \langle b(\cdot, a) \cdot \nabla_x \varphi_t; \rho_t \rangle d\nu_t(a)$. This proves the first part in item (2) of the statement.

Finally, for any $t \in [t_0,T]$ we have, by assumption on $b$,
$$ | \langle b(\cdot,\nu_t) \cdot \nabla_x \varphi_t; \rho_t \rangle | \leq C_b \norm{ \rho_t}_{(\mathcal{C}^1_2)^*} \norm{ \nabla_x \varphi_t}_{\mathcal{C}^1_2} \int_A(1+|a|^2) d|\nu_t|(a) $$
for some $C_b >0$ depending only on $b$. Estimate \eqref{eq:L1boundDualityBracket09/02} follows after integrating in $t$.
\end{proof}

Before ending this section we add a similar result about the differentiability of the map $a \mapsto \langle b(\cdot,a) \cdot  \nabla_x \varphi;\rho \bigr\rangle$.

\begin{lem}
    Let $\rho$ be an element of  $(\mathcal{C}^2_2)^*$ such that $\sup_{\varphi \in \mathcal{C}^2_2, \norm{ \varphi}_{\mathcal{C}^1_3 \leq 1} }\langle \varphi; \rho \rangle$ is finite and let $\varphi : \R^{d_1} \times \R^{d_2} \rightarrow \R$ be a function such that $\nabla_x \varphi$ belongs to $\mathcal{C}^2_1$. Then, the map $a \in A \mapsto \langle b(\cdot,a) \cdot \nabla_x \varphi;\rho \rangle$ is differentiable, and its gradient is given by $a \mapsto \langle \nabla_a[b(\cdot,a) \cdot \nabla_x \varphi]; \rho \rangle.$
\label{lem:differentiabilitydualitybracket31/01}
\end{lem}

\begin{proof}
    Let $f(x,y,a) := b(x,a) \cdot \nabla_x \varphi(x,y)$, for $(x,y,a) \in {\mathbb R}^{d_1} 
    \times {\mathbb R}^{d_2} \times A$. Then, by assumptions on $b$, for every $a_0 \in A$ we have, by Taylor's expansion
    $$ f(x,y,a) = f(x,y,a_0) + (a-a_0) \cdot \nabla_a f(x,y,a_0) + \epsilon(x,y,a),$$
where $$ \epsilon(x,y,a) =\biggl(\int_0^1 (1-t) \nabla^2_af (x,y, a_0 + t(a-a_0))dt \biggr)(a-a_0) \cdot (a-a_0).$$ 
Using the properties of $b$ in \assreg, 
we see that $\epsilon$
satisfies
$$ \lim_{a \rightarrow a_0} \frac{\norm{ \epsilon(\cdot,a)}_{\mathcal{C}^1_3}}{|a-a_0|} =0. $$
Observing that for all $a \in A$,
the three functions $(x,y) \mapsto \nabla_x \varphi(x,y)$, $(x,y) \mapsto b(x,a)$ and $(x,y) \mapsto \nabla_a b(x,a)$ belong to $\mathcal{C}^2_1$, we easily deduce that, for any 
$a \in A$, 
$(x,y) \mapsto f(x,y,a)$ and 
$(x,y) \mapsto \nabla_a f(x,y,a)$
belong to ${\mathcal C}^{2}_2$.
Writing 
$\epsilon(x,y,a) = f(x,y,a) - f(x,y,a_0) - (a-a_0) \cdot \nabla_a f(x,y,a_0)$, 
for $(x,y,a) \in {\mathbb R}^{d_1} \times {\mathbb R}^{d_2} \times A$, 
we deduce that, for any $a \in A$, 
$\epsilon(\cdot,a)$ belongs to $\mathcal{C}^2_2$. By linearity of the duality bracket we have
$$ \langle f(\cdot,a); \rho \rangle = \langle f(\cdot,a_0); \rho \rangle + (a-a_0) \cdot \langle \nabla_a f(\cdot,a_0) ; \rho \rangle + \langle \epsilon(\cdot,a); \rho \rangle  $$
and the remainder can be estimated by 
$$ \limsup_{a\rightarrow a_0} \frac{|\langle \epsilon(\cdot,a); \rho \rangle |}{ |a-a_0| }  \leq \sup_{ \varphi \in \mathcal{C}_2^2, \norm{ \varphi}_{\mathcal{C}^1_3} \leq 1} \langle \varphi; \rho \rangle  \times \lim_{a \rightarrow a_0} \frac{\norm{ \epsilon(\cdot,a)}_{\mathcal{C}^1_3}}{|a-a_0|} =0,$$
which completes the proof. 
\end{proof}

Now we turn to the well-posedness of the linearized equation
for $\bd\rho$, see 
\eqref{eq:rhotProp3.2Append}. We recall that the notion of solution is given in Section \ref{subse:SOC04/03}.

\subsubsection{Existence}

We continue the analysis with the following lemma on the stability 
properties of the continuity equation \eqref{eq:lem:compactnessContinuityEquation24/06:statement}:

\begin{lem}
\label{lem:backthenA.21-09/02}
    Take $(t_0,\gamma_0) \in [0,T] \times \mathcal{P}_p(\R^{d_1} \times \R^{d_2})$, for $ p \in [2,3]$,
    and $\bd{\nu}, \bd{\eta} \in \mathcal{D}(t_0)$. For any $\lambda \in [-1,1]$, let $ \bd{\nu}^{\lambda} := \bd{\nu} + \lambda \bd{\eta}$, and $\bd{\gamma}^{\lambda}$ be the solution to 
\begin{equation}
\label{eq:nu:gamma:lambda}
 \partial_t \gamma_t^{\lambda} + \div_x \bigl(b(x, \nu^{\lambda}_t) \gamma_t^{\lambda}\bigr) = 0 \quad \mbox{in } (t_0,T) \times \R^{d_1} \times \R^{d_2}, \quad  \gamma_{t_0}=\gamma_0.
\end{equation}
For $\lambda \neq 0,$ we define $\bd{ \delta \gamma}^{\lambda} := \lambda^{-1} [\bd{\gamma}^{\lambda}  - \bd{\gamma}^0].$ Then, there exists a non-decreasing function $\Lambda: \R_+ \rightarrow \R_+^*$, independent of $(t_0,\gamma_0)$, $\bd{\nu}$ and $\bd{\eta}$, such that  
\begin{align} 
&\sup_{\lambda \in [-1,1] \setminus \{0 \} } \Bigl \{ \sup_{t \in [t_0,T]} \norm{ \delta \gamma_t^{\lambda}}_{(\mathcal{C}^1_p)^*} +\sup_{t_1 \neq t_2 \in [t_0,T]} \frac{\norm{ \delta \gamma_{t_2}^{\lambda} - \delta \gamma_{t_1}^{\lambda}}_{(\mathcal{C}^2_{p})^*}}{\sqrt{|t_2 - t_1|}} \Bigr \} \leq \Lambda_{\bd{\nu} , \lambda \bd{\eta}, \gamma_0} \norm{\bd\eta}_{\mathcal{D}(t_0)}, 
\label{eq:23/01:17:18}
\\
&\sup_{ \lambda_1 \neq \lambda_2 \in [-1,1] \setminus \{0 \}} \frac{1}{|\lambda_2 - \lambda_1|} \sup_{t \in [t_0,T]}\norm{ \delta \gamma_{t}^{\lambda_2} - \delta \gamma_{t}^{\lambda_1}}_{(\mathcal{C}^2_p)^*} \leq 
\Lambda_{\bd{\nu} ,  \bd{\eta}, \gamma_0} \norm{\bd\eta}_{\mathcal{D}(t_0)},
\label{eq:23/01:17:54}
\end{align}
where, for simplicity, 
we used $\Lambda_{\bd{\nu} , \lambda \bd{\eta}, \gamma_0}$
(with $\lambda$ being equal to 1 in 
\eqref{eq:23/01:17:54})
as a short-hand notation:
\begin{equation}
\label{eq:lemma:2nd:linear:Lambda}
\Lambda_{\bd{\nu}  , \lambda \bd{\eta} ,\gamma_0} := \Lambda \biggl( \norm{ \bd{\nu}}_{\mathcal{D}(t_0)} + \vert \lambda\vert \, \norm{\bd\eta}_{\mathcal{D}(t_0)} + \int_{\R^{d_1} \times \R^{d_2}} (|x|^2 + |y|^2)^{p/2} d\gamma_0(x,y) \biggr).
\end{equation}
\end{lem}

\begin{proof}

\vskip 4pt

\textit{Step 1.} By subtracting the two equations \eqref{eq:nu:gamma:lambda} for the cases \(\lambda \neq 0\) and \(\lambda = 0\), and then dividing by \(\lambda\), we obtain that \(\bd{\delta \gamma}^{\lambda}\) satisfies the equation
\begin{equation} 
\partial_t \delta \gamma^{\lambda}_t + \div_x \bigl(b(x, \nu_t^{\lambda}) \delta \gamma_t^{\lambda}\bigr) = - \div_x \bigl( b(x,\eta_t) \gamma^{\lambda}_t\bigr) \quad \mbox{ in } (t_0,T) \times \R^{d_1} \times \R^{d_2}, \quad \delta \gamma_{t_0}^{\lambda} =0. \label{eq:23/01/16:08}
\end{equation}
By Lemma \ref{lem:A.4} in Appendix \ref{subse:A:1}, $\bd{\gamma}^{\lambda}$ and $\bd{\gamma}$ have 
bounded $p$-moments, uniformly in $t \in [t_0,T]$. Using a standard approximation argument, this shows that the equation above can be tested against any function $ \varphi : [t_0,T] \times \R^{d_1} \times \R^{d_2} \rightarrow \R$ such that $\varphi, \partial_t \varphi, \nabla_x \varphi$ are continuous 
and bounded, up to a multiplicative constant, by $(1+|x|^p + |y|^p)$, uniformly in 
$t \in [t_0,T]$.

Precisely, the test functions $\varphi$ are chosen 
according to the following duality principle. For all $n\in \mathbb{N}^*$ and $\lambda \in [-1,1]$, we call $\bd{\nu}^{\lambda,n}$ the approximation of $\bd{\nu}^{\lambda}$ given by Lemma \ref{lem:newregularizationb23/12}.
Then, for $t_1 \in (t_0,T]$ and $\phi \in \mathcal{C}^1_p$,  
we let
$\varphi^{\lambda,n} : [t_0,t_1] \times \R^{d_1} \times \R^{d_2} \rightarrow  \R$ be the solution to 
\begin{equation}
\label{eq:varphi:lambda,n}
 -\partial_t \varphi^{\lambda,n}_t -b(x, \nu_t^{\lambda,n}) \cdot \nabla_x \varphi^{\lambda,n}_t = 0 \quad \mbox{ in } [t_0,t_1] \times \R^{d_1} \times \R^{d_2}, \quad \varphi^{\lambda,n}_{t_1} = \phi \mbox{ in } \R^{d_1} \times \R^{d_2},
 \end{equation}
as given by Lemma \ref{lem:backwardtransportsmoothcontrol23/01}.

By 
combining 
Lemmas 
\ref{lem:propertiesnun11/01}
and
\ref{lem:backwardtransportsmoothcontrol23/01}, 
we can find a non-decreasing function $\Lambda$, independent of 
$t_0$, $\bd{\nu}$ and 
$\phi$, such that
\begin{equation} 
\sup_{ \lambda \in [-1,1]} \sup_{ n \in \mathbb{N}^*} \sup_{t \in [t_0,t_1]} \norm{ \varphi_t^{\lambda,n}}_{\mathcal{C}^1_p} \leq \Lambda \bigl( \norm{ \bd{\nu}}_{\mathcal{D}(t_0)} + 
\vert \lambda \vert \norm{\bd{\eta}}_{\mathcal{D}(t_0)} \bigr) \norm{ \phi}_{\mathcal{C}^1_p}.
\label{eq:23/01:17:11}
\end{equation}
\vskip 4pt

{\it Step 2.} Inserting the above bound in 
\eqref{eq:varphi:lambda,n}, and using the growth assumption on $b$ 
(as stated in \assreg), 
this implies that $ \| \partial_t \varphi_t^{\lambda,n} \|_{\mathcal{C}^p_b}$ 
is bounded. In particular, $\varphi^{\lambda,n}$ is an admissible test function for \eqref{eq:23/01/16:08}, and we have
\begin{align*} 
\int_{\R^{d_1} \times \R^{d_2}} \phi(x,y) d \bigl[\delta \gamma_{t_1}^{\lambda}\bigr](x,y) &= \int_{t_0}^{t_1} \int_{\R^{d_1} \times \R^{d_2}} \bigl \{\partial_t   \varphi_{t}^{\lambda,n} (x,y) + b(x, \nu_t^{\lambda}) \cdot \nabla_x \varphi_t^{\lambda,n}(x,y)  \bigr \} d
\bigl[ \delta \gamma_t^{\lambda} \bigr](x,y) 
\\
&\hspace{15pt}+ \int_{t_0}^{t_1} \int_{\R^{d_1} \times \R^{d_2}} b(x,\eta_t) \cdot \nabla_x \varphi_t^{\lambda,n}(x,y) d\gamma^{\lambda}_t(x,y)dy 
\\
&=\int_{t_0}^{t_1} \int_{\R^{d_1} \times \R^{d_2}} b(x, \nu_t^{\lambda} - \nu_t^{\lambda,n}) \cdot \nabla_x \varphi^{\lambda,n}_t(x,y) d \bigl[ \delta \gamma_t^{\lambda}\bigr](x,y) dt 
\\
&\hspace{15pt} + \int_{t_0}^{t_1} \int_{\R^{d_1} \times \R^{d_2}} b(x,\eta_t) \cdot \nabla_x \varphi_t^{\lambda,n}(x,y) d\gamma^{\lambda}_t(x,y)dt.
\end{align*}
On the one hand, we can easily show from Lemma \ref{lem:newregularizationb23/12} and \eqref{eq:23/01:17:11} that, for any $\lambda \in [-1,1] \setminus \{0 \}$, 
$$ \lim_{ n \rightarrow +\infty}  \int_{t_0}^{t_1} \int_{\R^{d_1} \times \R^{d_2}} b(x, \nu_t^{\lambda} - \nu_t^{\lambda,n}) \cdot \nabla_x \varphi^{\lambda,n}_t(x,y) 
d \bigl[ \delta \gamma_t^{\lambda} \bigr] (x,y) dt =0.$$
On the other hand, using \eqref{eq:23/01:17:11} again we get, for all $n \in \mathbb{N}^*$,
\begin{equation}
\label{eq:2nd:linearization:proof:conclusion:second:step}
\begin{split}
    &\int_{t_0}^{t_1} \int_{\R^{d_1} \times \R^{d_2}} b(x,\eta_t) \cdot  \nabla_x \varphi_t^{\lambda,n}(x,y) d\gamma^{\lambda}_t(x,y)dt 
    \\
    &\hspace{5pt} \leq \Lambda \bigl( \norm{ \bd{\nu}}_{\mathcal{D}(t_0)}
    + \vert \lambda \vert \, 
\norm{ \bd{\eta}}_{\mathcal{D}(t_0)} 
    \bigr) \norm{ \phi}_{\mathcal{C}^1_p} \Bigl(\sup_{t \in [t_0,t_1]} \int_{\R^{d_1} \times \R^{d_2}} (|x|^p + |y|^p) d\gamma^{\lambda}_t(x,y) \Bigr) 
    \\
&\hspace{15pt} \times    \int_{t_0}^{t_1} \int_A (1+|a|) d|\eta_t|(a) dt.
\end{split}
\end{equation}
and we conclude by Lemma \ref{lem:A.4} that the first 
bound for 
$\| \delta \gamma_t^{\lambda} \|_{(\mathcal{C}^1_p)^*}$
in 
\eqref{eq:23/01:17:18} holds. 
\vskip 4pt

\textit{Step 3.} We now establish the time regularity estimate in \eqref{eq:23/01:17:18}. We take $\lambda \in [-1,1] \setminus \{0 \}$ and $t_1 < t_2 \in [t_0,T] $. Let $\phi \in \mathcal{C}^2_{p}$. Then, using the equation \eqref{eq:23/01/16:08} for $\bd{\delta \gamma}^{\lambda}$ we get
\begin{equation}
\label{eq:deltagammat2lambda-deltagammat1lambda}
\begin{split}
    \int_{\R^{d_1} \times \R^{d_2}} \phi(x,y) d \bigl[\delta \gamma^{\lambda}_{t_2} - \delta \gamma_{t_1}^{\lambda} \bigr](x,y) &= \int_{t_1}^{t_2} \int_{\R^{d_1} \times \R^{d_2}} b(x, \nu_t^{\lambda}) \cdot \nabla_x \phi(x,y) d \bigl[\delta \gamma_t^{\lambda}\bigr](x,y) dt \\
    &\hspace{15pt}+ \int_{t_1}^{t_2} \int_{\R^{d_1} \times \R^{d_2}} b(x, \eta_t) \cdot \nabla_x \phi(x,y) d \gamma^{\lambda}_t (x,y). 
\end{split}
\end{equation}
Thanks to \eqref{eq:23/01:17:18} and the growth assumptions on $b$ and $\nabla_x b$, the first term in the right-hand side of 
\eqref{eq:deltagammat2lambda-deltagammat1lambda}
can be bounded by 
\begin{equation}
\label{eq:deltagammat2lambda-deltagammat1lambda:1}
\begin{split}
&\biggl\vert 
\int_{t_1}^{t_2} \int_{\R^{d_1} \times \R^{d_2}} b(x, \nu_t^{\lambda}) \cdot \nabla_x \phi(x,y) d \bigl[\delta \gamma_t^{\lambda}\bigr](x,y) dt
\biggr\vert
\\
&\hspace{15pt} \leq C \norm{ \phi}_{\mathcal{C}^2_{p}} \sup_{t \in [t_1,t_2]} 
\bigl\| \delta \gamma_t^{\lambda} \bigr\|_{(\mathcal{C}^1_p)^*} \int_{t_1}^{t_2} \int_A (1+|a|^2) d|\nu_t^{\lambda}|(a)dt.  
\end{split}
\end{equation}
We then observe that
$$ \int_{t_1}^{t_2} \int_A (1+|a|^2) d|\nu_t^{\lambda}|(a)dt \leq C \sqrt{t_2 -t_1} \Bigl( \norm{ \bd{\nu}
+ \lambda
\bd{\eta}}_{\mathcal{D}(t_0)} \Bigr),$$
for some constant $C >0$ independent of $t_1$, $t_2$ and $\lambda$. Similarly, the second term in 
\eqref{eq:deltagammat2lambda-deltagammat1lambda}
can be bounded by
\begin{equation}
\label{eq:deltagammat2lambda-deltagammat1lambda:2}
\begin{split}
&\biggl\vert \int_{t_1}^{t_2} \int_{\R^{d_1} \times \R^{d_2}}
b(x, \eta_t) \cdot \nabla_x \phi(x,y) d\gamma^{\lambda}_t(x,y)
\biggr\vert 
\\
&\hspace{15pt} \leq C \norm{ \phi}_{\mathcal{C}^1_{p}} \Bigl(\sup_{t \in [t_0,t_1]} \int_{\R^{d_1} \times \R^{d_2}} (|x|^2 + |y|^2)^{p/2} d\gamma^{\lambda}_t(x,y) \Bigr) \int_{t_1}^{t_2} \int_A (1+|a|)d |\eta_t|(a) dt. 
\end{split}
\end{equation}
Once again, Lemma \ref{lem:A.4} gives 
$$ \sup_{t \in [t_0,t_1]} \int_{\R^{d_1} \times \R^{d_2}} (|x|^2 + |y|^2)^{p/2} d\gamma^{\lambda}_t(x,y) \leq \Lambda_{\bd{\nu}  , \lambda \bd{\eta}, \gamma_0},$$
with 
$\Lambda_{\bd{\nu} , \lambda \bd{\eta}, \gamma_0}$
being as in the statement, see \eqref{eq:lemma:2nd:linear:Lambda}. Moreover, by definition of the norm $\norm{ \cdot }_{\mathcal{D}(t_0)}$, we have
$$ \int_{t_1}^{t_2} \int_A (1+|a|) d|\eta_t|(a) dt \leq \norm{ \bd{\eta}}_{\mathcal{D}(t_0)} \sqrt{t_2- t_1}.$$
By inserting the last two displays in 
\eqref{eq:deltagammat2lambda-deltagammat1lambda:2}
and then 
plugging
\eqref{eq:deltagammat2lambda-deltagammat1lambda:1}
and
\eqref{eq:deltagammat2lambda-deltagammat1lambda:2}
into 
\eqref{eq:deltagammat2lambda-deltagammat1lambda},
We easily deduce that \eqref{eq:23/01:17:18} holds, which completes the third step of the proof.
\vskip 4pt

\textit{Step 4.} In order to prove the last estimate \eqref{eq:23/01:17:54} in the statement, we argue once again by duality. For $\lambda_1 \neq \lambda_2 \in [-1,1] \setminus \{0 \}$, we let 
$ \xi^{\lambda_1, \lambda_2}_t := [ \delta \gamma_t^{\lambda_2} - \delta \gamma_t^{\lambda_1}]/(\lambda_2 - \lambda_1)$, and easily verify that the latter satisfies 
$$ \partial_t \xi_t^{\lambda_1, \lambda_2} + \div_x \Bigl(b(x,\nu_t^{\lambda_2}) \xi_t^{\lambda_1,\lambda_2}\Bigr) = - \div_x \Bigl(b(x,\eta_t) \bigl[ \delta \gamma_t^{\lambda_2} + \delta \gamma_t^{\lambda_1,\lambda_2} \bigr] \Bigr), \quad t \in [t_0,T],$$
with $\delta\gamma_t^{\lambda_1,\lambda_2} := [ 
\gamma_t^{\lambda_2} - \gamma_t^{\lambda_1}] / (\lambda_2 - \lambda_1).$ 

The above equation has a structure very similar to that of 
\eqref{eq:23/01/16:08}, with $\nu_t^\lambda$ being replaced by 
$\nu_t^{\lambda_2}$ and $\gamma_t^\lambda$ by 
$\delta \gamma_t^{\lambda_2} + \delta \gamma_t^{\lambda_1,\lambda_2}$. 

In order to repeat the duality argument used in the second step of 
the proof, one must first study the regularity 
properties of 
$\delta \gamma_t^{\lambda_2}$ and $\delta \gamma_t^{\lambda_1,\lambda_2}$. 
By the first part 
of \eqref{eq:23/01:17:18}, we know that 
$\delta \gamma^{\lambda_2}_t$ is bounded in $(\mathcal{C}^1_p)^*$, independently of 
$t \in [t_0,T]$ and 
$\lambda_2 \in [-1,1] \setminus \{0 \}$. In fact, we claim that
 $\delta \gamma^{\lambda_1,\lambda_2}_t$ is also bounded in $(\mathcal{C}^1_p)^*$ independently of 
 $t \in [t_0,T]$ and 
 $\lambda_1 \neq \lambda_2 \in [-1,1] \setminus \{0 \}$.
 This can also be seen as a consequence 
of the first part of 
\eqref{eq:23/01:17:18}, when 
the parameter $\lambda$ 
is shifted by  
$\lambda_1$: this amounts to replacing 
$\nu_t$ by $\nu_t^{\lambda_1}$
and 
$(\gamma_t^\lambda - \gamma_t^0)/\lambda$
by 
$(\gamma_t^{\lambda+\lambda_1} - \gamma_t^{\lambda_1})/\lambda$, with   
the latter being equal to 
$\delta \gamma_t^{\lambda_1,\lambda_2}$
when $\lambda$
is taken as $\lambda_2-\lambda_1$. 

The next step is to examine what 
\eqref{eq:2nd:linearization:proof:conclusion:second:step}
becomes 
when 
$\gamma_t^\lambda$ is replaced by 
$\delta \gamma_t^{\lambda_2} + \delta \gamma_t^{\lambda_1,\lambda_2}$. 
Since the latter is now regarded as an element of the dual space
$(\mathcal{C}^1_p)^*$, and no longer as an element of 
${\mathcal P}_p({\mathbb R}^{d_1} \times {\mathbb R}^{d_2})$, 
this forces us to work with a test function
$\phi$
in
$\mathcal{C}^2_p$
(and not only in 
$\mathcal{C}^1_p$). 
The conclusion follows in the same way as the derivation of \eqref{eq:2nd:linearization:proof:conclusion:second:step}.
\color{black}
\end{proof}

We use the previous result to get the existence of a solution to the linearized equation
\eqref{eq:rhotProp3.2Append}:

\begin{prop}
\label{prop:differentiatinggammalamnda25/01}
    Take $(t_0,\gamma_0) \in [0,T] \times \mathcal{P}_p(\R^{d_1} \times \R^{d_2})$, for $p \in [2,3]$, and $\bd{\nu}, \bd{\eta}$ in $\mathcal{D}(t_0)$. Then,
    the functions $(t \in [t_0,T] \mapsto \delta \gamma_t^{\lambda})_{\lambda \in [-1,1] \setminus \{0\}}$
    converge, as $ \lambda\rightarrow 0$ and in
    the space $\mathcal{C}([t_0,T], (\mathcal{C}^2_p)^*)$, toward some $\bd{\rho} \in \mathcal{C}([t_0,T], (\mathcal{C}^2_p)^*)$, solution to the linearized equation \eqref{eq:rhotProp3.2Append}. In particular, we have
$$ \frac{d}{d\lambda}\Big|_{\lambda=0} \bd{\gamma}^{\lambda} = \bd{\rho}, \quad \mbox{ in } \ \mathcal{C}([t_0,T], (\mathcal{C}^2_p)^* \bigr),$$
which means that 
$$ \lim_{\lambda \rightarrow 0} \sup_{t \in [t_0,T]} \norm{ \frac{\gamma_t^{\lambda} - \gamma_t}{\lambda} - \rho_t }_{(\mathcal{C}^2_p)^*} =0.  $$
Moreover, $\bd{\rho}$ satisfies the estimate
\begin{equation}
\sup_{t \in [t_0,T]}
\sup_{ \phi \in \mathcal{C}^2_p, \norm{\phi}_{\mathcal{C}^1_p} \leq 1} | \langle \phi, \rho_t \rangle |
 + \sup_{t_1 \neq t_2 \in [t_0,T]} \frac{\norm{ \rho_{t_2} - \rho_{t_1}}_{(\mathcal{C}^2_{p})^*}}{\sqrt{|t_2 - t_1|}}  \leq \Lambda_{\bd{\nu} , \gamma_0} \norm{\bd{\eta}}_{\mathcal{D}(t_0)},
\label{eq:24/01:22:16}
\end{equation}
for a non-decreasing function $\Lambda: \R_+ \rightarrow \R_+^*$, independent of $(t_0,\gamma_0)$ and $\bd{\nu}, \bd{\eta}$, and  
where, for simplicity, 
we used $\Lambda_{\bd{\nu} , \gamma_0}$ as a short-hand notation:
\begin{equation*}
\Lambda_{\bd{\nu} ,\gamma_0} := \Lambda \biggl( \norm{ \bd{\nu}}_{\mathcal{D}(t_0)}   + \int_{\R^{d_1} \times \R^{d_2}} (|x|^2 + |y|^2)^{p/2} d\gamma_0(x,y) \biggr).
\end{equation*}
\end{prop}

\begin{proof}
    With the notation and estimates of the previous proposition we know that $[-1,1] \ \setminus \{0 \} \ni \lambda \mapsto 
    \bd{\delta \gamma}^{\lambda} = (\delta \gamma_t^\lambda)_{t \in [t_0,T]} \in \mathcal{C}([t_0,T], (\mathcal{C}^2_p)^*)$ is Lipschitz continuous and therefore, by completeness of $\mathcal{C}([t_0,T], (\mathcal{C}^2_p)^*)$, extends uniquely at $\lambda =0$. We then find $\bd\rho \in \mathcal{C}([t_0,T] , (\mathcal{C}^2_p)^*)$ such that 
    $$ \lim_{\lambda \rightarrow 0} \bd{\delta \gamma}^{\lambda} = \bd{\rho}, \quad \mbox{ in } \mathcal{C}([t_0,T], (\mathcal{C}^2_p)^*). $$
Using the estimate 
\eqref{eq:23/01:17:18}
in the statement of Lemma \ref{lem:backthenA.21-09/02}
(assuming without any loss of generality that the function $\Lambda$ 
in the statement of 
Lemma \ref{lem:backthenA.21-09/02}
is right-continuous), we deduce that, for any $\varphi \in \mathcal{C}^2_p$ with $\norm{\varphi}_{\mathcal{C}^1_p} \leq 1$ we have
$$ |\langle \varphi ; \rho_t \rangle | = \lim_{\lambda \rightarrow 0} |\langle \varphi ; \delta \gamma^{\lambda}_t \rangle | \leq \liminf_{\lambda \rightarrow 0} \norm{\delta \gamma_t^{\lambda}}_{(\mathcal{C}^1_p)^*}
\leq \Lambda_{\bd{\nu},\gamma_0}
\norm{\bd{\eta}}_{{\mathcal D}(t_0)},$$
and similarly, for any $\varphi \in \mathcal{C}^2_p$ with $\norm{\varphi}_{\mathcal{C}^2_p} \leq 1$ and any $t_2 > t_1$ in $[t_0,T]$, we have
$$ |\langle \varphi; \rho_{t_2} - \rho_{t_1} \rangle | = \lim_{\lambda \rightarrow 0} |\langle \varphi; \delta \gamma^{\lambda}_{t_2} - \delta \gamma^{\lambda}_{t_1} \rangle | \leq \liminf_{ \lambda \rightarrow 0} \norm{ \delta \gamma_{t_2}^{\lambda} - \delta \gamma_{t_1}^{\lambda}}_{(\mathcal{C}^2_p)^*} \leq \Lambda_{\bd{\nu},\gamma_0}
\norm{\bd{\eta}}_{{\mathcal D}(t_0)} \sqrt{t_2-t_1},$$
where we used the same notation 
$\Lambda_{\bd{\nu} ,\gamma_0}$ as in the statement. We easily deduce that $\bd{\rho}$ belongs to $\mathcal{C}([t_0,T], (\mathcal{C}^2_p)^*)$ and $\sup_{\varphi \in \mathcal{C}^2_p, \norm{\varphi}_{\mathcal{C}^1_p} \leq 1} \langle \varphi; \rho_t \rangle $ is bounded independently from $t \in [t_0,T]$ and $\bd{\rho}$  satisfies estimate \eqref{eq:24/01:22:16}.

Let us now check that $\bd{\rho}$ solves the linearized equation. For $\lambda \in [-1,1] \setminus 
\{ 0\}$, $\bd{\delta \gamma}^{\lambda}$ solves \eqref{eq:23/01/16:08}.
In particular, for all test function $\varphi$ satisfying the conditions \eqref{eq:testfn1-09/03}-\eqref{eq:testfn2-09/03}, we deduce from 
\eqref{eq:23/01/16:08} that, for all $t_1 \in [t_0,T]$,
\begin{equation} 
\langle \varphi_{t_1}; \delta \gamma^{\lambda}_{t_1} \rangle = \int_{t_0}^{t_1} \langle \partial_t\varphi_t + b(\cdot, \nu^{\lambda}_t) \cdot \nabla_x \varphi_t; \delta \gamma^{\lambda}_t \rangle dt + \int_{t_0}^{t_1} \int_{\R^{d_1} \times \R^{d_2}} b(x,\eta_t) \cdot \nabla_x \varphi_t(x,y)d \gamma_t(x,y) dt. 
\label{eq:weakformdeltagammalambda25/01}
\end{equation}
Thanks to the convergence of $\bd{\delta \gamma}$ toward $\bd{\rho}$, and because $p \geq 2$,
we have the following two  limits
\begin{equation*}
\begin{split}
&\limsup_{\lambda \rightarrow 0} | \langle \varphi_{t_1}; \delta \gamma_{t_1}^{\lambda} \rangle - \langle \varphi_{t_1} ; \rho_{t_1} \rangle | \leq \norm{ \varphi_{t_1}}_{\mathcal{C}^2_2} \lim_{\lambda \rightarrow 0} \bigl\| \delta \gamma_{t_1}^{\lambda} - \rho_{t_1} \bigr\|_{(\mathcal{C}^2_2)^*} = 0,
\\
&\limsup_{\lambda \rightarrow 0} \sup_{t \in [t_0,T]} | \langle \partial_t \varphi_t; \delta \gamma_t^{\lambda} - \rho_t \rangle | \leq \sup_{t \in [t_0,T]}  \norm{ \partial_t \varphi_t}_{\mathcal{C}^2_2} \lim_{\lambda \rightarrow 0} \sup_{t \in [t_0,T]} \bigl\| \delta \gamma_t^{\lambda} -\rho_t
\bigr\|_{(\mathcal{C}^2_2)^*} =0. 
\end{split}
\end{equation*}
Moreover, for all $t \in [t_0,t_1]$,
\begin{align*}
    &\Bigl| \bigl\langle b(\cdot, \nu_t^{\lambda}) \cdot \nabla_x \varphi_t; \delta \gamma_t^{\lambda} \bigr\rangle - \bigl\langle b(\cdot, \nu_t) \cdot \nabla_x \varphi_t; \rho_t 
    \bigr\rangle \Bigr| 
    \\
 &   \leq  |\lambda| \, \Bigl| 
 \bigl\langle b(\cdot, \eta_t) \cdot \nabla_x \varphi_t ; \delta \gamma_t^{\lambda}  \bigr\rangle 
 \Bigr| \, + \,  \Bigl| \bigl\langle 
 b(\cdot, \nu_t) \cdot \nabla_x \varphi_t; \delta \gamma_t^{\lambda} - \rho_t \bigr\rangle
 \Bigr| 
    \\
    &\leq |\lambda| \norm{\nabla_x \varphi_t}_{\mathcal{C}^1_2} \bigl\| \delta \gamma_t^{\lambda}
    \bigr\|_{(\mathcal{C}^1_2)^*} \int_A (1+|a|^2) d|\eta_t|(a) 
    \\
    &\hspace{15pt}+ C  \norm{ \nabla_x \varphi_t}_{\mathcal{C}^2_2} \bigl\| \delta \gamma_t^{\lambda} - \rho_t \bigr\|_{(\mathcal{C}^2_2)^*} \int_A (1+|a|^3) d|\eta_t|(a),
\end{align*}
for a constant $C$ independent of $\lambda$ and $t$. Using the fact that $\bd{\eta} \in {\mathcal D}(t_0)$
and letting $\lambda \rightarrow 0$ in \eqref{eq:weakformdeltagammalambda25/01}, 
we deduce that $\bd{\rho}$ is solution to the linearized equation  \eqref{eq:rhotProp3.2Append}.
\end{proof}

\subsubsection{Uniqueness}

The next lemma will be useful to prove uniqueness of solutions 
to the linearized continuity equation \eqref{eq:rhotProp3.2Append}. It also explains why we require $\bd{\rho}$ to be bounded in $(\mathcal{C}^1_3)^*$ although the equation makes sense if $\bd{\rho}$ is only bounded in $(\mathcal{C}^1_2)^*$. Notice that, to obtain existence of solutions that are bounded in $(\mathcal{C}^1_3)^*$, we need $\gamma_0$ to be in $\mathcal{P}_3(\R^{d_1} \times \R^{d_2})$  and not only $\mathcal{P}_2(\R^{d_1} \times \R^{d_2})$.

\begin{lem}
\label{lem:uniquenessextension25/01}
    Take $\rho \in (\mathcal{C}^2_2)^*$ such that 
   $$ \sup_{\varphi \in \mathcal{C}^2_2, \norm{\varphi}_{\mathcal{C}^1_3} \leq 1}  | \langle \varphi ; \rho \rangle|< +\infty. $$
Then, 
\begin{equation}
\label{eq:lem:uniquenessextension25/01:condition:rho}
 \Bigl \{ \forall \phi \in \mathcal{C}^3_{2,1} \quad  \langle \phi ; \rho \rangle = 0 \Bigr \}  \quad \Rightarrow  \quad \rho =0 \quad \mbox{ in }  (\mathcal{C}^2_2)^*.
 \end{equation}
\end{lem}

\begin{proof}
     For all $R >1$, we consider a smooth cut-off function 
      $\chi_R : \R^{d_1} \times \R^{d_2} \rightarrow [0,1]$ satisfying
\begin{equation*}
    \chi_R = 
    \left \{
    \begin{array}{ll}
    1  \quad \mbox{\rm in}&B(0,R) 
    \\
    0  \quad \mbox{\rm outside}\hspace{-5pt} &B(0,R+1)
    \end{array}
    \right.,
\end{equation*}
with
 $\norm{ \chi_R}_{\mathcal{C}^1} \leq c$, for some $c>0$ independent from $R$. 
 We easily check that there is $C>0$, depending on $c$, such that, for all $R >1$ and $\phi \in \mathcal{C}^1_2$, 
 $$ \bigl\| \phi(1-\chi_R) \bigr\|_{\mathcal{C}^1_3} \leq \frac{C}{R} \norm{\phi}_{\mathcal{C}^1_2}.$$
 Now, we consider $\phi \in \mathcal{C}^2_2$ together with a sequence of  functions  
 $(\phi^n)_{n \geq 1}$ in $\mathcal{C}^3_{2,1}$, bounded in $\mathcal{C}^1_2$ and converging to $\phi$ in $\mathcal{C}^2_{\rm loc}$. Then, for all $n \in \mathbb{N}^*$ and  $R>1$, we rewrite
 $\phi$ in the form
 \begin{equation}
 \label{eq:phi:phin:chiR}
  \phi = \phi^n + (\phi - \phi^n)\chi_R + (\phi - \phi^n)(1 - \chi_R).
  \end{equation}
We now take the duality bracket with respect to 
 $\rho \in (\mathcal{C}^2_2)^*$. Assuming 
that $\rho$ satisfies the condition in the left-hand side of \eqref{eq:lem:uniquenessextension25/01:condition:rho}, we observe that $\langle \phi^n ; \rho \rangle =0$ for all $n \in \mathbb{N}^*$, because each $\phi^n$ belongs to $\mathcal{C}^3_{2,1}$. Next, we estimate the duality bracket between $\rho$
 and the second and third terms in the right-hand side of 
\eqref{eq:phi:phin:chiR}. We have
 \begin{equation*}
 \begin{split}
  &\bigl\vert \bigl\langle (\phi - \phi^n) \chi_R; \rho \bigr\rangle \bigr| \leq C \bigl\| \phi - \phi^n
  \bigr\|_{\mathcal{C}^1(B(0,R+1))},  
  \\
  &\bigl| \bigl\langle (\phi - \phi^n)(1-\chi_R) ; \rho \bigr\rangle \bigr| \leq C \bigl\| (\phi -\phi^n) (1-\chi_R) \bigr\|_{\mathcal{C}^1_3} \leq CR^{-1} \bigl(\norm{\phi}_{\mathcal{C}^1_2} + \norm{\phi^n}_{\mathcal{C}^1_2} \bigr).  
  \end{split}
  \end{equation*}
 We easily deduce, by letting first $n \rightarrow +\infty$ and then $R \rightarrow +\infty$ in 
 \eqref{eq:phi:phin:chiR},  
 that $ \langle \phi ; \rho \rangle =0.$
\end{proof}

As announced, we use Lemma 
\ref{lem:uniquenessextension25/01} to prove 
the following uniqueness result:

\begin{prop}
\label{prop:A.22}
Let $(t_0,\gamma_0) \in [0,T] \times \mathcal{P}_2(\R^{d_1} \times \R^{d_2})$, $\bd{\nu}, \bd{\eta} \in \mathcal{D}(t_0)$, and $\bd{\gamma} \in \mathcal{C}([t_0,T], \mathcal{P}_2(\R^{d_1} \times \R^{d_2}))$.
Then, the linearized equation 
\eqref{eq:rhotProp3.2Append}, understood as in 
\eqref{eq:testfn1-09/03}--\eqref{eq:testfn2-09/03}--\eqref{eq:weakeqrho12/01Section2}, 
has at most one solution $\bd{\rho}$ satisfying, in addition to 
\eqref{eq:spaceforrho12/01}, 
the condition
\begin{equation} 
\label{eq:furtherintegrability08/01}
\sup_{t \in [t_0,T]} \sup_{ \varphi \in \mathcal{C}^2_2, \norm{\varphi}_{\mathcal{C}^1_3} \leq 1} |\langle \varphi; \rho_t \rangle | < +\infty.
\end{equation}
\end{prop}

\begin{proof}
    Let $\bd{\rho}^1, \bd{\rho}^2$ be two solutions to 
    \eqref{eq:rhotProp3.2Append},    
    satisfying both 
    \eqref{eq:spaceforrho12/01}
    and \eqref{eq:furtherintegrability08/01}. Let also
    $(\bd{\nu}^n)_{n \geq 1}$ be the regularization of $\bd{\nu}$ from Lemma \ref{lem:newregularizationb23/12}. For $n \in \mathbb{N}^*$, $t_1 \in (t_0,T]$ and $\phi \in \mathcal{C}^{3}_{2,1}$, we call $\bd{\varphi}^n$ the solution to \eqref{eq:varphin08/01} with control $\bd{\nu}^n$. Then, by Lemma \ref{lem:backwardtransportsmoothcontrol23/01}, see also Remark \ref{rmk:A.9},  $\bd{\varphi}^n$ is an admissible test function for the equations for $\bd{\rho}^1$ and $\bd{\rho}^2$. Then, we have
    $$ \bigl\langle \phi; \rho^2_{t_1} - \rho^1_{t_1} \bigr\rangle = \int_{t_0}^{t_1} \bigl\langle [ b(\cdot, \nu_t ) -b(\cdot, \nu_t^n) ] \cdot \nabla_x \varphi_t^n; \rho_t^2 -\rho_t^1 \bigr\rangle dt.$$
For every $t \in [t_0,t_1]$ it holds
\begin{align*} 
\bigl| \bigl\langle [ b(\cdot, \nu_t ) -b(\cdot, \nu_t^n) ] \cdot \nabla_x \varphi_t^n; \rho_t^2 -\rho_t^1 \bigr\rangle \bigr| &\leq \norm{ b(\cdot, \nu_t ) -b(\cdot, \nu_t^n) ] \cdot \nabla_x \varphi_t^n }_{\mathcal{C}^1_2}  \norm{ \rho_t^2 - \rho_t^1}_{(\mathcal{C}^1_2)^*}  \\
& \leq \norm{ b(\cdot, \nu_t ) -b(\cdot, \nu_t^n)  }_{\mathcal{C}^1_1} \norm{ \nabla_x \varphi_t^n}_{\mathcal{C}^1_1} \norm{ \rho_t^2 - \rho_t^1}_{(\mathcal{C}^1_2)^*}.
\end{align*}
By Proposition \ref{prop:SolutionBackxard16Sept} and Lemma \ref{lem:newregularizationb23/12} we have
$$ \sup_{n \geq 1} \sup_{t \in [t_0,t_1]} \norm{ \nabla_x \varphi_t^n}_{\mathcal{C}^1_1} < +\infty  $$
and, by Lemma \ref{lem:newregularizationb23/12} again we also have
$$ \lim_{n \rightarrow +\infty} \int_{t_0}^{t_1} \norm{ b(\cdot, \nu_t ) -b(\cdot, \nu_t^n)  }_{\mathcal{C}^1_1} dt =0. $$
By combining the last four displays, we deduce that $ \langle \phi; \rho_{t_1}^2 - \rho_{t_1}^1 \rangle = 0$ for all $t_1 \in [t_0,T]$ and all $\phi \in \mathcal{C}^3_{2,1}$. We conclude using Lemma \ref{lem:uniquenessextension25/01}.
\end{proof}

Combined with Proposition \ref{prop:differentiatinggammalamnda25/01} this gives the first part of Proposition \ref{prop:differentiatinggammalambda09/02}. It remains to prove the representation formula \eqref{eq:derivativeinlambda31/01}.

\subsubsection{Representation formula}

Notice that, for any $\lambda_0 \in (-1,1)$, we can apply the result above replacing $\bd{\nu}$ by $\bd{\nu} + \lambda_0 \bd{\eta}$ and deduce that $\lambda \mapsto \bd{\gamma}^{\lambda}$ is differentiable over $(-1,1)$, 
the derivative being given, for all $\lambda \in (-1,1)$, by 
\begin{equation}
\label{eq:gamma:rho:derivative:w.r.t.:lambda}
\frac{d }{d\lambda } \bd{\gamma}^{\lambda} = \bd{\rho}^{\lambda} \quad \mbox{ in } \ \mathcal{C}([t_0,T], (\mathcal{C}^2_2)^*),
\end{equation}
where $\bd{\rho}^{\lambda}$ is the unique solution in $\mathcal{R}(t_0)$ to 
\begin{equation} 
\partial_t \rho_t^{\lambda} + \div_x(b(x, \nu_t^{\lambda})\rho_t^{\lambda}) = - \div_x (b(x,\eta_t) \gamma_t^{\lambda}) \quad \mbox{ in } (t_0,T) \times \R^{d_1} \times \R^{d_2}, \quad \rho_{t_0}^{\lambda} =0. 
\label{eq:LinRhoAnnex11/01lambda}
\end{equation}
Notice that, with this notation, the solution $\bd{\rho}$ of \eqref{eq:rhotProp3.2Append} is simply $\bd{\rho}^0$.
We conclude with a useful expression for $\bd{\rho}^{\lambda}$.When $\lambda =0$, it gives an explicit formula for $\bd{\rho}$.

\begin{prop}
    Take $(t_0,\gamma_0) \in [0,T] \times \mathcal{P}_3(\R^{d_1} \times \R^{d_2})$ and $\bd{\nu}, \bd{\eta} \in \mathcal{D}(t_0)$. For $\lambda \in (-1,1)$, let $\bd{\rho}^{\lambda}$ be the unique solution to \eqref{eq:LinRhoAnnex11/01lambda}, as given by 
    Proposition \ref{prop:differentiatinggammalambda09/02}. Then, for all $t_1 \in [t_0,T]$, $\rho^{\lambda}_{t_1}$ is given, for all $\phi \in \mathcal{C}^3_2$ by
\begin{equation} 
\langle \phi ; \rho^{\lambda}_{t_1} \rangle = \lambda \int_{t_0}^{t_1} \bigl \langle b(\cdot, \eta_t) \cdot \nabla_x \varphi_t; \rho_t^{\lambda} \bigr \rangle dt +    \int_{t_0}^{t_1} \int_{\R^{d_1} \times \R^{d_2}} b(x,\eta_t) \cdot \nabla_x \varphi_t(x,y) d\gamma^{\lambda}_t(x,y) dt, 
\label{eq:explicit25/01}
\end{equation}
where $\bd{\gamma}^{\lambda} = (\gamma^{\lambda}_t)_{t \in [t_0,t_1]}$ is the solution to the continuity equation \eqref{eq:nu:gamma:lambda} starting from $(t_0,\gamma_0)$ and controlled by $\bd{\nu} + \lambda \bd{\eta}$, and $\varphi : [t_0,t_1] \times \R^{d_1} \times \R^{d_2} \rightarrow \R$ is the solution to the transport equation
$$ - \partial_t \varphi_t - b(x, \nu_t) \cdot \nabla_x \varphi_t = 0 \quad \mbox{ in } (t_0,t_1) \times \R^{d_1} \times \R^{d_2}, \quad \varphi_{t_1} = \phi \quad \mbox{ in } \R^{d_1} \times \R^{d_2}.$$
\label{prop:Explicitformularholambda25/01}
\end{prop}
\begin{proof}
Throughout the proof, we use the same notation as in the statement.  For all $\lambda \in (-1,1)$ we have, using the equations for $\bd{\varphi}$ and $\bd{\gamma}^{\lambda}$ and applying Lemma \ref{lem:duality:1,2:F},
\begin{equation*}
     \begin{split}
     \int_{\R^{d_1} \times \R^{d_2}} \phi (x,y) d\gamma_{t_1}^{\lambda}(x,y) &= \int_{\R^{d_1} \times \R^{d_2}} \varphi_{t_0}(x,y)d\gamma_0(x,y) 
     \\
     &\hspace{15pt} + \lambda \int_{t_0}^{t_1} \int_{\R^{d_1} \times \R^{d_2}} b(x, \eta_t)\cdot \nabla_x \varphi_t(x,y) d\gamma_t^{\lambda}(x,y) dt. 
     \end{split}
\end{equation*}
The strategy is to differentiate both sides of the equality with respect to $\lambda$. To do so, we observe that, for $h \in (-1,1)$ such that $\lambda + h$ belongs to $(-1,1)$, it holds
\begin{align*}
    &\biggl| \int_{t_0}^{t_1} \Bigl\langle b(\cdot,\eta_t) \cdot \nabla_x \varphi_t; \frac{\gamma_t^{\lambda+h}-\gamma_t^{\lambda}}{h} - \rho_t^{\lambda} \Bigr\rangle dt \biggr| 
    \\
    & \leq \sup_{t \in[t_0,t_1]} \biggl\| \frac{\gamma_t^{\lambda+h}-\gamma_t^{\lambda}}{h} - \rho_t^{\lambda} \biggr\|_{(\mathcal{C}^2_2)^*} \int_{t_0}^{t_1} \norm{ b(\cdot, \eta_t) \cdot \nabla_x \varphi_t }_{\mathcal{C}^2_2} dt. 
\end{align*}
By assumptions on $b$ (see \assreg) and Proposition \ref{prop:SolutionBackxard16SeptAppendix}, we get 
\begin{align*} 
\int_{t_0}^{t_1} \norm{ b(\cdot, \eta_t) \cdot \nabla_x \varphi_t }_{\mathcal{C}^2_2} dt  &\leq C \sup_{t \in [t_0,t_1]}  \norm{ \nabla_x \varphi_t}_{\mathcal{C}^2_2} \int_{t_0}^{t_1}\int_A (1+|a|^3) d|\eta_t|(a) dt \\
&\leq C \norm{ \nabla_x \phi}_{\mathcal{C}^2_2} \int_{t_0}^{t_1}\int_A (1+|a|^3) d|\eta_t|(a) dt < +\infty.
\end{align*}
Back to the penultimate display, we deduce from
the identity
$[d/d\lambda] |_{\lambda =0} \bd{\gamma}^{\lambda} = \bd{\rho}$
in $\mathcal{C}([t_0,T], (\mathcal{C}^2_2)^*)$ (see \eqref{eq:derivative19/02}) that
$$ \frac{d}{d\lambda} \int_{t_0}^{t_1} \int_{\R^{d_1} \times \R^{d_2}} b(x,\eta_t) \cdot \nabla_x \varphi_t(x,y) d\gamma^{\lambda}_t(x,y) dt  = \int_{t_0}^{t_1} \langle b(\cdot,\eta_t) \cdot \nabla_x \varphi_t ; \rho_t^{\lambda} \rangle dt. $$
The rest follows easily.
\end{proof}

We can use this explicit representation to extend $\bd \rho$ to test functions with only one spacial derivative.

\begin{prop}
\label{prop:A.24}
     Take $(t_0,\gamma_0) \in [0,T] \times \mathcal{P}_3(\R^{d_1} \times \R^{d_2})$ and $\bd{\nu}, \bd{\eta} \in \mathcal{D}(t_0)$, and let $\bd{\rho}$ be the only solution in $\mathcal{R}(t_0)$ to \eqref{eq:rhotProp3.2Append}. Then, for all $t_1 \in [t_0,T]$, $\rho_{t_1}$ extends uniquely to $\mathcal{C}^1_2$ and it is given, for any $\phi \in \mathcal{C}^1_2$ by 
      $$ \langle \phi; \rho_{t_1} \rangle = \int_{t_0}^{t_1} \int_{\R^{d_1} \times \R^{d_2}} b(x,\eta_t) \cdot \nabla_x \varphi_t(x,y) d\gamma_t(x,y) dt, $$
where $\varphi : [t_0,t_1] \times \R^{d_1} \times \R^{d_2} \rightarrow \R$ is the solution to 
\begin{equation} 
- \partial_t \varphi_t - b(x, \nu_t) \cdot \nabla_x \varphi_t = 0 \quad \mbox{ in } [t_0,t_1] \times \R^{d_1} \times \R^{d_2}, \quad \varphi_{t_1} = \phi \quad \mbox{ in } \R^{d_1} \times \R^{d_2}.
\label{eq:22/05/2025}
\end{equation}
In particular, for all $\phi \in \mathcal{C}^1_2$ and all $t_1 \in [t_0,T]$, it holds
\begin{equation} 
\langle \phi; \rho_{t_1} \rangle =  \frac{d}{d\lambda} \Big|_{\lambda = 0} \int_{\R^{d_1} \times \R^{d_2}} \phi(x,y) d\gamma_{t_1}^{\lambda}(x,y) =  \int_{t_0}^{t_1} \int_{\R^{d_1} \times \R^{d_2}} b(x,\eta_t) \cdot \nabla_x \varphi_t(x,y) d\gamma_t(x,y) dt. 
\label{eq:derivativeinlambda31/01Ap}
\end{equation}
\end{prop}

\begin{proof}
Proposition 
\ref{prop:SolutionBackxard16SeptAppendix}
makes it possible to define a continuous linear form on 
$(\mathcal{C}^1_2)^*$, by letting, for every $\phi \in \mathcal{C}^1_2$,
    $$ \langle \phi ; \tilde{\rho_{t_1}} \rangle :=\int_{t_0}^{t_1} \int_{\R^{d_1} \times \R^{d_2}} b(x,\eta_t) \cdot \nabla_x \varphi_t(x,y) d\gamma_t(x,y) dt.  $$
Thanks to 
Lemma 
\ref{lem:A.4} and because 
$\gamma_0 \in {\mathcal P}_3({\mathbb R}^{d_1} \times {\mathbb R}^{d_3})$, we then check that 
$$ \sup_{ \phi \in \mathcal{C}^1_2, \norm{ \phi}_{\mathcal{C}^1_3} \leq 1} | \langle \phi ; \tilde{\rho_{t_1}} \rangle | < +\infty. $$
Thanks to Proposition \ref{prop:Explicitformularholambda25/01} (with $\lambda=0$), $\tilde{\rho_{t_1}}$ coincides with $\rho_{t_1}$ over $\mathcal{C}^3_2$. We then proceed similarly to Lemma \ref{lem:uniquenessextension25/01} to prove that the extension is unique. 
It remains to justify that the expression for the derivative \eqref{eq:derivativeinlambda31/01Ap} is valid for any function $\phi \in \mathcal{C}_2^1$. To this end, we use Lemma \ref{lem:duality:1,2:F} to write, for $\lambda \in (-1,1) \setminus \{ 0\}$,
$$ \int_{\R^{d_1} \times \R^{d_2}} \phi(x,y) d\bigl[ \frac{\gamma_{t_1}^{\lambda} - \gamma_{t_1}}{\lambda} \bigr](x,y) = \int_{t_0}^{t_1} \int_{\R^{d_1} \times \R^{d_2}} b(x,\eta_t)\cdot \nabla_x \varphi_t(x,y) d\gamma_t^{\lambda}(x,y)dt. $$
Therefore, it suffices to prove the continuity  at $\lambda = 0$ of 
$$ \lambda \mapsto \int_{t_0}^{t_1} \int_{\R^{d_1} \times \R^{d_2}} b(x,\eta_t)\cdot \nabla_x \varphi_t(x,y) d\gamma_t^{\lambda}(x,y)dt. $$
First, we easily justify from the ODE representation that $\lambda \mapsto \gamma_t^{\lambda}$ is continuous in $\mathcal{P}_2(\R^{d_1} \times \R^{d_2})$ for all $t \in [t_0,t_1]$. By Proposition \ref{prop:SolutionBackxard16SeptAppendix}, $(x,y) \mapsto  b(x,\eta_t) \cdot \nabla_x \varphi_t(x,y) $ is continuous with at most quadratic growth for Lebesgue almost all $t \in [t_0,t_1]$ and therefore, at 
these times,
$$ \lim_{\lambda \rightarrow 0}  \int_{\R^{d_1} \times \R^{d_2}} b(x,\eta_t) \cdot \nabla_x \varphi_t(x,y) d\gamma_t^{\lambda} (x,y) = \int_{\R^{d_1} \times \R^{d_2}} b(x,\eta_t) \cdot \nabla_x \varphi_t(x,y) d\gamma_t (x,y). $$
Moreover, by 
Lemma \ref{lem:A.4}
and
Proposition \ref{prop:SolutionBackxard16SeptAppendix}, there is a constant $C>0$ independent of $\lambda $ such that
\begin{align*} 
&\sup_{\lambda \in (-1,1)}  \biggl| \int_{\R^{d_1} \times \R^{d_2}}  b(x,\eta_t) \cdot  \nabla_x \varphi_t(x,y) d\gamma_t^{\lambda} (x,y) \biggr| 
\\
&\hspace{15pt} \leq C \norm{ \nabla_x \phi}_{\mathcal{C}^0_2}  \int_{\R^{d_1} \times \R^{d_2}} \bigl(|x|^2 + |y|^2\bigr) d\gamma_0(x,y) \int_A (1+|a|) d|\eta_t|(a). 
\end{align*}
Since the right-hand side belongs to $L^1([t_0,t_1])$, we conclude by Lebesgue dominated convergence theorem that 
$$ \lim_{\lambda \rightarrow 0} \int_{t_0}^{t_1} \int_{\R^{d_1} \times \R^{d_2}} b(x,\eta_t)\cdot \nabla_x \varphi_t(x,y) d\gamma_t^{\lambda}(x,y)dt = \int_{t_0}^{t_1} \int_{\R^{d_1} \times \R^{d_2}} b(x,\eta_t)\cdot \nabla_x \varphi_t(x,y) d\gamma_t(x,y)dt, $$
which completes the proof.
\end{proof}

\subsubsection{Stability}

We continue with some further results on the application $\lambda \in [-1,1] \mapsto \bd{\rho}^{\lambda} $ where $\bd{\rho}^{\lambda}$ is the unique solution 
of the equation 
\eqref{eq:LinRhoAnnex11/01lambda}
in the space $\mathcal{R}(t_0)$. 

First, we have the following regularity estimate with respect to $\lambda$.

\begin{lem}
\label{lem:25/01/15:27}
    There exists a non-decreasing function $\Lambda : \R_+ \rightarrow \R_+^*$ such that, for 
 $(t_0,\gamma_0) \in [0,T] \times \mathcal{P}_3(\R^{d_1} \times \R^{d_2})$ and $\bd{\nu}, \bd{\eta} \in \mathcal{D}(t_0)$, 
    $$ \sup_{t \in [t_0,T]} \Bigl\| \rho_t^{\lambda_2} - \rho_t^{\lambda_1}
    \Bigr\|_{(\mathcal{C}^2_2)^*} \leq \Lambda_{\gamma_0, \bd{\nu}, \bd{\eta}} |\lambda_2 - \lambda_1|, \quad \mbox{ for all } \lambda_1, \lambda_2 \in (-1,1),$$
where $\Lambda_{\gamma_0, \bd{\nu}, \bd{\eta}}$ is a short-hand notation for 
$$ \Lambda_{\gamma_0, \bd{\nu}, \bd{\eta}} = \Lambda \biggl( 
\norm{ \bd{\nu}}_{\mathcal{D}(t_0)} + \norm{\bd{\eta}}_{\mathcal{D}(t_0)} 
+
\int_{\R^{d_1} \times \R^{d_2}} (|x|^2 + |y|^2)^{3/2} d\gamma_0(x,y)  \biggr). $$
\end{lem}

\begin{proof}
 We argue by duality again. We first notice that, by making the difference in the equations for $\bd{\rho}^{\lambda_2}$ and $\bd{\rho}^{\lambda_1}$ (recall again \eqref{eq:LinRhoAnnex11/01lambda}), we have 
\begin{align*} 
\partial_t(\rho_t^{\lambda_2} - \rho_t^{\lambda_1}) + \div_x( b(x,\nu_t^{\lambda_1}) (\rho_t^{\lambda_2} - \rho_t^{\lambda_1})) = &-(\lambda_2 - \lambda_1) \div_x ( b(x,\eta_t) \rho_t^{\lambda_2}) \\
&- \div_x (b(x,\eta_t) (\gamma_t^{\lambda_2} - \gamma_t^{\lambda_1})) \quad \mbox{ in } (t_0,T) \times \R^{d_1} \times \R^{d_2}.  
\end{align*}
We fix $t_1 \in (t_0,T]$ and $\phi \in \mathcal{C}^2_2$ and we let $ \bd\varphi^n : [t_0,t_1] \times \R^{d_1} \times \R^{d_2} \rightarrow \R $ be the solution to the backward equation \eqref{eq:22/05/2025} with control $\bd{\nu}^{\lambda_1,n}$ --the already used regularization of $\bd{\nu}^{\lambda_1}$, 
see for instance \eqref{eq:varphi:lambda,n} together with the first step in the proof of Lemma 
\ref{lem:backthenA.21-09/02}--  with terminal condition $\phi$ at $t_1$. Using  the equation satisfied by the difference $\bd{\rho}^{\lambda_2} - \bd{\rho}^{\lambda_1}$
and 
the equation satisfied by $\bd{\varphi}^n$, we have
\begin{equation}
\label{eq:rholambda1-rholambda2}
\begin{split}
    \langle \phi; \rho_{t_1}^{\lambda_2} - \rho_{t_1}^{\lambda_1} \rangle &= \int_{t_0}^{t_1} \langle b(\cdot,\nu_t^{\lambda_1} - \nu_t^{\lambda_1,n}) \cdot \nabla_x \varphi_t^n ; \rho_t^{\lambda_2} - \rho_t^{\lambda_1}\rangle dt 
    \\
    &\hspace{15pt}
    + (\lambda_2 - \lambda_1) \int_{t_0}^{t_1} \langle b(\cdot,\eta_t) \cdot \nabla_x \varphi_t^n; \rho_t^{\lambda_2} \rangle dt \\
    &\hspace{15pt} + \int_{t_0}^{t_1} \int_{\R^{d_1} \times \R^{d_2}} b(x , \eta_t) \cdot \nabla_x \varphi_t^n(x,y) d(\gamma_t^{\lambda_2} - \gamma_t^{\lambda_1})(x,y)dt 
    \\
    &=: I^n_1 + I^n_2 + I^n_3.
\end{split}
\end{equation}
We start with $I_1^n$:
\begin{align*}
    |I_1^n|&\leq \int_{t_0}^{t_1} \bigl\| b(\cdot, \nu_t^{\lambda_1}-\nu_t^{\lambda_1,n}) \cdot \nabla_x \varphi_t^n \bigr\|_{\mathcal{C}^1_3} \bigl\| \rho_t^{\lambda_2} - \rho_t^{\lambda_1}
    \bigr\|_{(\mathcal{C}^1_3)^*} dt \\
    &\leq C \int_{t_0}^{t_1} \bigl\| b(\cdot, \nu_t^{\lambda_1} - \nu_t^{\lambda_1,n} ) \bigr\|_{\mathcal{C}^1_1} \norm{ \nabla_x \varphi_t^n}_{\mathcal{C}^1_2} \bigl\| \rho_t^{\lambda_2} - \rho_t^{\lambda_1} \bigr\|_{(\mathcal{C}^1_3)^*} dt \\
    &\leq C \sup_{t \in [t_0,t_1]}  \norm{ \nabla_x \varphi_t^n}_{\mathcal{C}^1_2} \sup_{t \in [t_0,t_1]} \bigl\| \rho_t^{\lambda_2} - \rho_t^{\lambda_1} \bigr\|_{(\mathcal{C}^1_3)^*} \int_{t_0}^{t_1} \bigl\| b(\cdot, \nu_t^{\lambda_1} - \nu_t^{\lambda_1,n} ) \bigr\|_{\mathcal{C}^1_1} dt
\end{align*}
and we use Proposition \ref{prop:SolutionBackxard16SeptAppendix} and Lemma \ref{lem:newregularizationb23/12} to conclude that $\lim_{ n \rightarrow +\infty} I_1^n =0$. We then proceed with $I_2^n$:
\begin{align*}
    |I_2^n| &\leq |\lambda_2- \lambda_1| \int_{t_0}^{t_1} \norm{ b(\cdot,\eta_t) \cdot \nabla_x \varphi_t^n }_{\mathcal{C}^1_3} \bigl\| \rho_t^{\lambda_2} \bigr\|_{(\mathcal{C}^1_3)^*} dt 
    \\
    & \leq |\lambda_2- \lambda_1| \sup_{t \in [t_0,t_1]} \norm{ \nabla_x \varphi_t^n}_{\mathcal{C}^1_3} \sup_{t\in [t_0,t_1]} \bigl\| \rho_t^{\lambda_2} \bigr\|_{(\mathcal{C}^1_3)^*} \int_{t_0}^{t_1} \int_A (1+|a|^2) d|\eta_t|(a) dt 
    \\
    &\leq \Lambda_{\bd{\nu}, \bd{\eta},\gamma_0}  |\lambda_2 -\lambda_1|,
\end{align*}
for a function $\Lambda$ as in the statement. 
Then, we use Lemma \ref{lem:backthenA.21-09/02} to handle $I_3^n$:
\begin{align*}
    |I_3^n| &\leq  |\lambda_2- \lambda_1| \sup_{t \in [t_0,t_1]} \norm{ \nabla_x \varphi_t^n}_{\mathcal{C}^1_2} \sup_{t\in [t_0,t_1]} \biggl\| 
    \frac{\gamma_t^{\lambda_2} - \gamma_t^{\lambda_1}}{\lambda_2 - \lambda_1}
    \biggr\|_{(\mathcal{C}^1_2)^*} \int_{t_0}^{t_1} \int_A (1+|a|^2) d|\eta_t|(a) dt \\
    &\leq \Lambda_{\bd{\nu}, \bd{\eta},\gamma_0}  |\lambda_2 -\lambda_1|.
\end{align*}
Inserting the last three displays in \eqref{eq:rholambda1-rholambda2} and letting 
$n$ tend to $+\infty$, we complete the proof. 
\end{proof}

We can go one step further and establish the following 
weak differentiability property of 
the function
$\lambda \in [-1,1] \mapsto \bd{\rho}^{\lambda}$, where 
we recall again that $\bd{\rho}^{\lambda}$ is the unique solution 
of the equation 
\eqref{eq:LinRhoAnnex11/01lambda}
in the space $\mathcal{R}(t_0)$:

\begin{lem}
    For all $\phi \in \mathcal{C}^3_{2,1}$, $\lambda \mapsto \langle \phi; \rho_T^{\lambda} \rangle$ is differentiable at $0$ and we have
    $$ \frac{d^2}{d\lambda^2} \Big|_{\lambda =0} \int_{\R^{d_1} \times \R^{d_2}} \phi(x,y) d\gamma_T^{\lambda}(x,y)  = \frac{d}{d\lambda }\Big|_{\lambda=0} \langle \phi; \rho_T^{\lambda} \rangle = 2 \int_{t_0}^T \bigl \langle b(\cdot, \eta_t) \cdot \nabla_x \varphi_t; \rho_t \bigr \rangle dt,$$
where $\varphi : [t_0,T] \times \R^{d_1} \times \R^{d_2} \rightarrow \R$ is the solution to 
$$ -\partial_t \varphi_t - b(x, \nu_t) \cdot \nabla_x \varphi_t = 0 \quad \mbox{ in } [t_0,T] \times \R^{d_1} \times \R^{d_2}, \quad \varphi_T = \phi \quad\mbox{ in } \R^{d_1} \times \R^{d_2}.$$
\label{lem:secondderivativeterminalcost29/01}
\end{lem}

\begin{proof}[Proof of Lemma \ref{lem:secondderivativeterminalcost29/01}.]
We make 
use of the expression 
\eqref{eq:explicit25/01}, for two different values of 
the parameter therein: $\lambda \not =0$ and $0$. 
Making the difference between these two expressions
and dividing the result by $\lambda$, we get:
\begin{equation} 
\langle \phi; \delta \rho_T^{\lambda}  \rangle = \int_{t_0}^T \bigl \langle b(\cdot , \eta_t) \cdot \nabla_x \varphi_t ; \rho_t^{\lambda} \bigr \rangle dt + \int_{t_0}^T \int_{\R^{d_1} \times \R^{d_2}} b(x, \eta_t) \cdot \nabla_x \varphi_t(x,y) d \delta \gamma_t^{\lambda}(x,y)
dt,  
\label{eq:25/01/15:29}
\end{equation}
where we used the notations $\delta \rho_T^{\lambda} := \lambda^{-1} (\rho_T^{\lambda} - \rho_T)$ and $\delta \gamma_t^{\lambda} := \lambda^{-1}(\gamma_t^{\lambda} - \gamma_t)$, 
recalling that $\bd{\rho}$ and $\bd{\gamma}$
are shortened notations for 
$\bd{\rho}^0$ and $\bd{\gamma}^0$.

It remains to pass to the limit in the right-hand side when $\lambda \rightarrow 0$. For the first term in the right-hand side of 
\eqref{eq:25/01/15:29}, we obtain
\begin{align*}
    &\biggl| \int_{t_0}^T \bigl \langle b(\cdot , \eta_t)  \cdot \nabla_x \varphi_t ; \rho_t^{\lambda} \bigr \rangle dt - \int_{t_0}^T \bigl \langle b(\cdot , \eta_t) \cdot \nabla_x \varphi_t ; \rho_t \bigr \rangle dt \biggr|   
    \\
    &= \biggl| \int_{t_0}^T \bigl \langle b(\cdot , \eta_t)  \cdot \nabla_x \varphi_t ; \rho_t^{\lambda} - \rho_t\bigr \rangle dt  \biggr| 
    \\
    &\leq C \sup_{t \in [t_0,T]} \norm{ \varphi_t}_{\mathcal{C}^3_2} \sup_{t\in [t_0,T]} 
    \bigl\| \rho_t^{\lambda} - \rho_t \bigr\|_{(\mathcal{C}^2_2)^*}  \int_{t_0}^T \int_A (1+|a|^3) d|\eta_t|(a) dt.
\end{align*}
We then use Proposition \ref{prop:SolutionBackxard16SeptAppendix} and Lemma \ref{lem:25/01/15:27} to deduce that the right-hand side goes to $0$ with $\lambda$, which concludes the analysis of the first term in the right-hand side of \eqref{eq:25/01/15:29}. 

We handle the second term in a similar manner, but using Proposition \ref{prop:differentiatinggammalamnda25/01} instead of Lemma  \ref{lem:25/01/15:27}. As a consequence, we get
$$ \lim_{\lambda \rightarrow 0} \langle \phi; \delta \rho_T^{\lambda} \rangle = 2 \int_{t_0}^T \bigl \langle b(\cdot, \eta_t ) \cdot \nabla_x \varphi_t ; \rho_t \bigr \rangle dt, $$
which is te desired result.

\end{proof}

\subsection{The Linearized Transport Equation}

\label{sec:LinTransEq}

We now turn to the linearized transport equation 
\begin{equation}
\label{eq:linearized:transport:eq}
-\partial_t v_t - b(x, \nu_t) \cdot \nabla_x v_t = b(x,\eta_t) \cdot \nabla_x u_t \quad \mbox{ in } [t_0,T] \times \R^{d_1} \times \R^{d_2}, \quad v_T =0 \quad \mbox{ in } \R^{d_1} \times \R^{d_2},
\end{equation}
where $u$ is the solution to 
\begin{equation}
-\partial_t u_t - b(x,\nu_t) \cdot \nabla_x u_t = 0 \quad \mbox{ in } [t_0,T] \times \R^{d_1} \times \R^{d_2}, \quad u_T =L. 
\label{eq:transport:reminder:for:linearization}
\end{equation}
We take $\bd{\nu}, \bd{\eta} \in \mathcal{D}(t_0) $. 
Denoting the right-hand side
of \eqref{eq:linearized:transport:eq} by 
$ f_t(x,y) = b(x, \eta_t) \cdot \nabla_x u_t(x,y)$, we deduce 
from the regularity assumptions on $b$ (see \assreg) and Proposition \ref{prop:SolutionBackxard16SeptAppendix} that 
$$ \int_{t_0}^T \norm{ f_t}_{\mathcal{C}^2_1}dt + \sup_{t_1 < t_2 \in [t_0,T]} \frac{1}{\sqrt{t_2-t_1}} \int_{t_1}^{t_2} \norm{f_t}_{\mathcal{C}^1_1}dt \leq \Lambda \bigl(  \norm{\bd{\nu}}_{\mathcal{D}(t_0)} + \norm{ \bd{\eta}}_{\mathcal{D}(t_0)} \bigr).$$
The above bound provides an indication on the space within which the equation 
\eqref{eq:linearized:transport:eq}
should be solved. Precisely,
we say that  $v \in \mathcal{C}([t_0,T], \mathcal{C}^1_{1})$ is a solution if, for all $t_1 < t_2 \in [t_0,T]$ and all $(x,y) \in \R^{d_1} \times \R^{d_2}$ it holds
\begin{equation}
 \label{eq:linearized:transport:weak:formulation}
 v_{t_1}(x,y) - v_{t_2}(x,y) = \int_{t_1}^{t_2} \bigl \{ b(x, \nu_t) \cdot \nabla_xv_t(x,y) + b(x, \eta_t) \cdot \nabla_x u_t(x,y)\bigr \}dt.  
\end{equation}

Similarly to Proposition \ref{prop:SolutionBackxard16SeptAppendix}, we have the following result (the proof is identical and thus omitted):

\begin{prop}
\label{prop:linearized:transport:equation:representation}
    There is a unique solution to the linearized transport equation
    \eqref{eq:linearized:transport:eq}. For all $t \in [t_0,T]$, $v_t$ belongs to $\mathcal{C}^2$ and we have the estimate 
    $$ \sup_{t \in [t_0,T]} \norm{ v_t }_{\mathcal{C}^{2}_{1}} + \sup_{t_1 < t_2 \in [t_0,T]}  \frac{\norm{v_{t_2} - v_{t_1} }_{\mathcal{C}^1_{1}}}{\sqrt{t_2-t_1}} \leq \Lambda \bigl(  \norm{\bd{\nu}}_{\mathcal{D}(t_0)} \bigr)  \norm{ \bd{\eta}}_{\mathcal{D}(t_0)},$$
for a non-decreasing function 
    $\Lambda : {\mathbb R}_+ \rightarrow {\mathbb R}_+$, independent of $t_0$ and 
    ${\boldsymbol \nu}$.
    
Moreover, for all $t \in [t_0,T] $ and all $(x,y) \in \R^{d_1} \times \R^{d_2}$, the solution is given by 
\begin{equation}
\label{eq:representation:appendix:v}
v_t(x,y) = \int_{t}^T b \bigl( X_s^{t,x},\eta_s) \cdot \nabla_x u_s \bigl( X_s^{t,x},y \bigr) ds, 
\end{equation}
where $(X_s^{t,x})_{s \geq t}$ is the flow of the ODE, solution to 
$$ \dot{X}^{t,x}_s = b(X_s^{t,x},\nu_s) \quad s \geq t, \quad X_t^{t,x} = x. $$
\end{prop}

When $\bd{\nu}$ and $\bd{\eta}$ satisfy some further regularity we can improve the regularity of the solution. In this case, the equation is satisfied in the classical sense.

\begin{prop}
Assume that $t \mapsto \nu_t$ and $t\mapsto \eta_t$ are continuous from $[t_0,T]$ to $\mathcal{M}_{1+|a|^3}(A)$. Then, $v$ 
(as given by \eqref{eq:representation:appendix:v})
and its gradient $\nabla_x v$ are jointly continuously differentiable in $(t,x,y)$.
\label{prop:additionalregularityvt}
\end{prop}
\begin{proof}
    We only briefly sketch the argument for $\nabla_xv$. Using 
    Proposition \ref{prop:ODE16Sept} (see also Remark \ref{rmk:regularvectorfield27/01}) and Lemma \ref{lem:backwardtransportsmoothcontrol23/01}, we can differentiate twice in space the representation formula \eqref{eq:representation:appendix:v}
    and then deduce that the functions     
$(t,x,y) \mapsto \nabla v_t$ and $(t,x,y) \mapsto \nabla^2 v_t$ are jointly continuous. 
By differentiating once in the $x$-variable the equation \eqref{eq:linearized:transport:eq} for $v_t$ we get $\partial_t \nabla_x v_t = \nabla_x \bigl \{ -b(x,\nu_t) \cdot \nabla_x v_t - b(x, \eta_t) \cdot \nabla_x u_t \bigr \}  
    $  and deduce from the time continuity 
of 
$t \mapsto \nu_t$ and $t\mapsto \eta_t$, as assumed in the statement, that $\partial_t \nabla_x v_t $ is also jointly continuous.
\end{proof}

As expected, the solution $v$
to 
\eqref{eq:linearized:transport:eq}, as provided by Proposition
\ref{prop:linearized:transport:equation:representation}, can be retrieved by linearizing the transport equation 
\eqref{eq:transport:reminder:for:linearization}
with respect to the control parameter.

\begin{lem}
    Take $t_0 \in [0,T]$ and $\bd{\nu}, \bd{\eta} \in \mathcal{D}(t_0)$. Let $\bd{u} \in \mathcal{C}([t_0,T], \mathcal{C}^1_{2,1})$ be the solution to the backward equation 
    $$ - \partial_t u_t - b(x,\nu_t) \cdot \nabla_x u_t = 0 \quad \mbox{ in } [t_0,T] \times \R^{d_1} \times \R^{d_2}, \quad u_T = L \quad \mbox{ in } \R^{d_1} \times \R^{d_2}.$$
    For all $\lambda \in (-1,1)$, with $\lambda \neq 0$, let $\bd{\nu}^{\lambda} := \bd{\nu} + \lambda \bd{ \eta}$ and $\bd{u}^{\lambda} \in \mathcal{C}([t_0,T], \mathcal{C}^1_{2,1})$ be the solution to the backward equation 
    $$ - \partial_t u^{\lambda}_t - b(x,\nu^{\lambda}_t) \cdot \nabla_x u^{\lambda}_t = 0 \quad \mbox{ in } [t_0,T] \times \R^{d_1} \times \R^{d_2}, \quad u^{\lambda}_T = L \quad \mbox{ in } \R^{d_1} \times \R^{d_2}.$$
Letting $\delta \bd{u}^{\lambda} := \lambda^{-1} (\bd u^{\lambda} - \bd u)$, for $\lambda \not =0$, we have the estimate
$$ \sup_{t \in [t_0,T]} \bigl\| \delta u^{\lambda}_t \bigr\|_{\mathcal{C}^{2}_{1}} + \sup_{t_1 < t_2 \in [t_0,T]}  \frac{\| \delta u^{\lambda}_{t_2} - \delta u^{\lambda}_{t_1} \|_{\mathcal{C}^1_{1}}}{\sqrt{t_2-t_1}} \leq \Lambda \bigl(  \norm{\bd{\nu}}_{\mathcal{D}(t_0)} ,  \norm{ \bd{\eta}}_{\mathcal{D}(t_0)} \bigr),$$
for some non-decreasing function $\Lambda : \R_+ \rightarrow \R_+$ independent 
of $\lambda \in (-1,1)$ and $t_0 \in [0,T]$.
\label{lem:diffulambda24/02:13:23}
\end{lem}
\begin{proof}
    For $t \in [t_0,T]$ and $(x,y) \in \R^{d_1} \times \R^{d_2}$, we let $(X^{\lambda,t,x}_s)_{s \in [t,T]}$ be the solution to 
    $$ \dot{X}^{\lambda,t,x}_s = b(X^{\lambda,t,x}_s, \nu^{\lambda}_s ), \quad s \in [t,T]; \quad X^{\lambda,t,x}_t = x. $$
Applying  \eqref{eq:duality:relation:statement} in Lemma \ref{lem:duality:1,2:F} with $t_0=t$, $t_1=T$, $\phi = L$, $\bd{\nu}^1 = \bd{\nu}^{\lambda}$, $\bd{\nu}^2 = \bd{\nu}$ and $\gamma_0 = \delta_{(x,y)}$, and therefore with $(\gamma^1_s)_{s \in [t_0,T]} := (\delta_{(X_s^{\lambda,t,x},y)} )_{s \in [t_0,T]}$ therein, we get
$$ \delta u_t^{\lambda}(x,y) = \int_t^T b(X_s^{\lambda,t,x},\eta_s) \cdot \nabla_x u_s(X_s^{\lambda,t,x},y)ds.  $$
We easily deduce from Proposition \ref{prop:ODE16Sept} and Proposition \ref{prop:SolutionBackxard16SeptAppendix} that 
$$ \sup_{t \in [t_0,T]} \bigl\| \delta u_t^{\lambda} \bigr\|_{\mathcal{C}^2_1} \leq \Lambda\bigl( \norm{\bd{\nu}}_{\mathcal{D}(t_0)} + \norm{ \bd{\eta}}_{\mathcal{D}(t_0)} \bigr). $$
By making the difference between the equation for $\bd{u}^{\lambda}$ and the equation for $\bd{u}$ and then dividing by $\lambda$, we also find that $\delta\bd{u}^{\lambda}$ solves
$$ -\partial_t \delta u_t^{\lambda} - b(x,\nu_t^{\lambda}) \cdot \nabla_x \delta u_t^{\lambda} = - b(x,\eta_t) \cdot \nabla_x u_t \quad \mbox{ in } [t_0,T] \times \R^{d_1} \times \R^{d_2}, \quad v_T(x,y) = 0. $$
Integrating in time and recalling 
\eqref{defn:D(t_0)}
in Definition \ref{eq:definitionadmissiblecontrols:aux}, we find
\begin{align*} 
\norm{ \delta u_{t_2} -\delta u_{t_1}}_{\mathcal{C}^1_1} &= \norm{ \int_{t_1}^{t_2} \left \{ b(\cdot, \nu_t^{\lambda}) \cdot \nabla_x \delta u_t^{\lambda} + b(\cdot,\eta_t) \cdot \nabla_x u_t \right \} dt }_{\mathcal{C}^1_1} \\
&\leq \Lambda \bigl( \norm{ \bd{\nu}}_{\mathcal{D}(t_0)} + \norm{ \bd{\eta}}_{\mathcal{D}(t_0)} \bigr) \Bigl( \sup_{t \in [t_0,T]} \norm{\nabla_x \delta u^\lambda_t}_{\mathcal{C}^1_1} + \sup_{t \in [t_0,T]} \norm{ \nabla_x u_t}_{\mathcal{C}^1_1} \Bigr)  \sqrt{t_2-t_1}, 
\end{align*}
and we conclude using the previous estimate together with Proposition \ref{prop:SolutionBackxard16SeptAppendix}.
\end{proof}

We conclude with a useful duality relation. 

\begin{prop}
   Take $(t_0,\gamma_0) \in [0,T] \times \mathcal{P}_3(\R^{d_1} \times \R^{d_2})$ and $\bd{\nu}, \bd{\eta}^1, \bd{\eta}^2 \in \mathcal{D}(t_0)$. Let $(\bd{\gamma}, \bd{u})$ be the solutions to the continuity and transport equations 
   \eqref{eq:lem:compactnessContinuityEquation24/06:statement}
   and 
   \eqref{eq:transport:reminder:for:linearization} with control $\bd{\nu}$, and $(\bd{\rho}^{i}, \bd{v}^{i})$ be the solutions to the linearized equations 
   \eqref{eq:rhotProp3.2Append}
   and
   \eqref{eq:linearized:transport:eq} 
   associated to $\bd{\eta}^{i}$ for $i=1,2$.
   Then, we have
    \begin{equation} 
\notag \int_{t_0}^T \Bigl \langle b( \cdot ,\eta^2_t) \cdot  \nabla_x u_t  ; \rho^1_t  \Bigr \rangle dt = \int_{t_0}^T  \int_{\R^{d_1} \times \R^{d_2}} b(x,\eta^1_t) \cdot \nabla_x v^2_t(x,y) d \gamma_t(x,y) dt. 
\label{eq:intheproofofProp3.21Avril}
\end{equation}
\label{prop:IPPLinearized27/01}
\end{prop}

\begin{proof}
For $\lambda \in (-1,1)$, we let $\bd{\gamma}^{\lambda}$ be the solution to the continuity equation
\eqref{eq:lem:compactnessContinuityEquation24/06:statement}
when driven by the control $\bd{\nu}^{\lambda} := \bd{\nu} + \lambda \bd{\eta}^1.$ We use Lemma
\ref{lem:newregularizationb23/12}
in order to 
regularize $\bd{\nu}$ and $\bd{\eta}^2$, 
with 
$(\bd{\nu}^n)_{n \geq 1}$ and $(\bd{\eta}^{2,n})_{n \geq 1}$ as regularized sequences.
Then, for each $n \geq 1$, we 
call  $\bd{u}^n$
   the solution to the transport equation \eqref{eq:transport:reminder:for:linearization} 
   driven by ${\boldsymbol \nu}^n$. 
Also,  we   
let $\bd{v}^{2,n}$ be the classical solution to
the equation
\eqref{eq:linearized:transport:eq}
associated to 
the input $({\boldsymbol \nu}^n,\bd{\eta}^{2,n},\bd{u}^{n})$, i.e.
$$ -\partial_t v_t^{2,n} - b(x,\nu_t^{n}) \cdot \nabla_x v_t^{2,n}  = b(x,\eta_t^{2,n}) \cdot \nabla_x u_t^{n}, \quad v_T^{2,n} = 0,$$
with the regularity of 
$\bd{v}^{2,n}$ being given by 
Proposition \ref{prop:additionalregularityvt}. 

Using the equations satisfied by $\bd{\gamma}^{\lambda}$ first and $\bd{v}^{2,n}$ next, we
can follow the derivation of 
\eqref{eq:duality:relation:proof}
and then obtain (noticing that $v_T^{2,n}$ is null) 
\begin{align}
\nonumber
     \int_{\R^{d_1} \times \R^{d_2}} v_{t_0}^{2,n}(x,y) d\gamma_0(x,y) &=- \int_{t_0}^T \int_{\R^{d_1} \times \R^{d_2}} \bigl \{ \partial_t v_t^{2,n}(x,y) + b(x,\nu_t^{\lambda}) \cdot \nabla_x v_t^{2,n}(x,y)  \bigr \} d\gamma_t^{\lambda}(x,y) dt 
     \\
    &= -  \int_{t_0}^T \int_{\R^{d_1} \times \R^{d_2}}  b(x, \nu_t^{\lambda} - \nu_t^{n}) \cdot \nabla_x v_t^{2,n}(x,y) d\gamma_t^{\lambda}(x,y) dt 
    \label{eq:v2:with:regularization}
    \\
    &\hspace{15pt} + \int_{t_0}^T \int_{\R^{d_1} \times \R^{d_2}} b(x,\eta_t^{2,n}) \cdot \nabla_xu_t^n(x,y)
    d\gamma_t^{\lambda}(x,y) dt.
\nonumber
\end{align}
The objective now is to pass to the limit as 
$n \rightarrow + \infty$. 
Denoting, 
for each $n \in {\mathbb N}^*$, $(X_t^{n,x})_{t \in [t_0,T]}$ the solution to the ODE 
\eqref{eq:ODE:appendix}      
    starting from $x$ at time $t_0$ and driven by the control ${\boldsymbol \nu}^n$ (we remove the superscript 
    $n$ when $\bd{\nu}^n$ is replaced by $\bd{\nu}$), 
    we know from Corollary 
\ref{cor:approximation:nu} that, for any $t \in [t_0,T]$,
\begin{equation*}
    \lim_{n \rightarrow +\infty} \sup_{ s \in [t,T]} \norm{ X_s^{n, \cdot} - X_s^{\cdot}}_{\mathcal{C}^3_{1}} = 0.
\end{equation*}
By the representation formula in Proposition 
\ref{prop:SolutionBackxard16SeptAppendix}, we deduce that 
$(x,y) \mapsto (\nabla_x u_t^n(x,y),\nabla^2_x u_t^n(x,y))$ converges locally uniformly to 
$(x,y) \mapsto (\nabla_xu_t(x,y),\nabla_x^2 u_t(x,y))$ as $n \rightarrow +\infty$.

Moreover, we know from Lemma 
\ref{lem:newregularizationb23/12} that
that 
\begin{equation*}
\lim_{n \rightarrow +\infty} \int_{t_0}^T 
\Bigl( 
\norm{ b(\cdot, \nu_t^n) - b(\cdot, \nu_t) }_{\mathcal{C}^3_{1}} 
+
\norm{ b(\cdot, \eta_t^n) - b(\cdot, \eta_t) }_{\mathcal{C}^3_{1}} 
\Bigr) 
dt =0.
\end{equation*}
By 
using formula \eqref{eq:representation:appendix:v} (with $X^{t,x}_s$ being replaced by $X^{n}_s$, 
$\eta_s$ by $\eta_s^{2,n}$ and $\nabla_x u_s$ by $\nabla_x u^n_s$) 
to represent $v^{2,n}$, we deduce from the last three convergence properties that, for any $t \in [t_0,T]$, 
$(x,y) \mapsto v^{2,n}_t(x,y)$ converges
to 
$(x,y) \mapsto v^{2}_t(x,y)$ 
locally uniformly. 
And then, by using the first inequality in the statement 
of Proposition 
\ref{prop:linearized:transport:equation:representation}, 
we can let $n \rightarrow +\infty$ in 
\eqref{eq:v2:with:regularization} to infer that 
\begin{align*}  
\int_{\R^{d_1} \times \R^{d_2}} v^2_{t_0}(x,y) d\gamma_0(x,y) &= -\lambda \int_{t_0}^T \int_{\R^{d_1} \times \R^{d_2}}  b(x, \eta^1_t) \cdot \nabla_x v^2_t(x,y) d\gamma_t^{\lambda}(x,y) dt \\
& +\int_{t_0}^T \int_{\R^{d_1} \times \R^{d_2}} b(x,\eta^2_t) \cdot \nabla_xu_t(x,y) d\gamma_t^{\lambda}(x,y) dt.
\end{align*}
In particular, when $\lambda =0$ we get 
$$ \int_{\R^{d_1} \times \R^{d_2}} v^2_{t_0}(x,y) d\gamma_0(x,y) = \int_{t_0}^T \int_{\R^{d_1} \times \R^{d_2}} b(x,\eta^2_t) \cdot \nabla_xu_t(x,y) d\gamma_t^{0}(x,y) dt. $$
Substracting the two previous equalities and dividing by $\lambda$ gives 
$$ \int_{t_0}^T \int_{\R^{d_1} \times \R^{d_2}} b(x,\eta_t^2) \cdot \nabla_x u_t (x,y) d \bigl[ \frac{\gamma_t^{\lambda} - \gamma_t^0}{\lambda } \bigr](x,y)dt =  \int_{t_0}^T \int_{\R^{d_1} \times \R^{d_2}}  b(x, \eta^1_t) \cdot \nabla_x v^2_t(x,y) d\gamma_t^{\lambda}(x,y) dt. $$
Thanks to Proposition 
\ref{prop:differentiatinggammalamnda25/01}, we can pass to the limit when $\lambda \rightarrow 0$ in the left-hand side, while estimate \eqref{eq:23/01:17:18} in Lemma \ref{lem:backthenA.21-09/02} together with the regularity of $\bd{v}^2$ from Proposition \ref{prop:linearized:transport:equation:representation} allows us to pass to the limit in the right-hand side.
\end{proof}

\section{Second Order Conditions}

\label{sec:SOC}

The purpose of this section is to prove the main results of Subsection 
\ref{subse:SOC04/03}, and in particular Theorem 
\ref{prop:SecondOrderConditions3Avril}
and 
Proposition 
\ref{prop:reg:v,rho}.

\subsection{Second Order Variations}

We start with following statement, which corresponds to the first part of Theorem 
\ref{prop:SecondOrderConditions3Avril} (recall Definition \ref{def:admissible:perturbation}
for the definition of the set $\mathcal{A}^\ell(t_1)$ used below):

\begin{prop}
Let $(t_0,\gamma_0) \in [0,T] \times {\mathcal P}_3
({\mathbb R}^{d_1} \times {\mathbb R}^{d_2})$ and
   $\bd{\nu}^*$ be minimum of  $J ((t_0,\gamma_0), \cdot )$ with associated optimal trajectory $\bd{\gamma}^*$ and $t_1 \in [t_0,T)$. Then, $ \mathcal{J} ( (t_1, \gamma^*_{t_1}),  \bd{\nu}^*, \bd{\eta} ) \geq 0$
for all $\bd{\eta} \in \mathcal{A}^\ell(t_1)$. 
\label{prop:mathcalJnonnegative26/06}
\end{prop}

\begin{proof}
{ \ }

\textit{Step 1.}
    We first assume that $\eta \in \mathcal{A}^\ell(t_1)$ also satisfies 
    \begin{equation}
    \bd{\eta} \mbox{ belongs to } L^{\infty}([t_1,T] \times A) \mbox{ with }  \bd{\eta} =0 \mbox{ outside } [t_1,T] \times B(0,R) \mbox{ for some }R >1.
\label{ass:furtherassumptionfornu}
\end{equation}
Thanks to the lower bound on $\bd\nu^*$ from Proposition \ref{prop:regularityfromOC}, $\bd{\nu}^{\lambda} := \bd{\nu}^* + \lambda \bd{\eta}$ belongs 
to $\mathcal{A}(t_1)$ for all $\lambda \in (-\lambda_0,\lambda_0)$, for some $\lambda_0 \in (0,1)$ small enough
(recall Definition \ref{def:admissible:control} for the definition of 
${\mathcal A}(t_1)$). We denote by $\bd{\gamma}^{\lambda} = (\gamma^{\lambda}_t)_{t \in [t_1,T]}$ the resulting curve, solution to 
$$ \partial_t \gamma^{\lambda}_t + \div_x \bigl( b(x,\nu_t^{\lambda}) \gamma_t^{\lambda} \bigr) = 0, \quad \mbox{ in } (t_1,T) \times \R^{d_1} \times \R^{d_2}, \quad \gamma_{t_1}^{\lambda} = \gamma^*_{t_1}.$$
By dynamic programming, $\bd{\nu}^*$ (once restricted to $[t_1,T]$) is optimal for $J ( (t_1,\gamma^*_{t_1}),\cdot )$ and therefore, 
\begin{equation} 
J \bigl( (t_1, \gamma^*_{t_1}), \bd\nu^{\lambda} \bigr) \geq J \bigl( (t_1, \gamma^*_{t_1}), \bd\nu^* \bigr), \quad \forall \lambda \in (-\lambda_0,\lambda_0). 
\label{eq:1Avriloptimalitytildenu}
\end{equation}
We are going to compute the second order variation of $\lambda \in (-\lambda_0,\lambda_0)\mapsto J ( (t_1, \gamma^*_{t_1}), \bd\nu^{\lambda} )$. 
The terminal cost is handled by Lemma \ref{lem:secondderivativeterminalcost29/01} (which relies on the notation introduced in 
\eqref{eq:gamma:rho:derivative:w.r.t.:lambda}, with the interval $(-1,1)$ being replaced by 
$(-\lambda_0,\lambda_0)$). 
It holds
\begin{equation*} 
\frac{d^2}{d\lambda^2} \Big|_{\lambda = 0} \int_{\R^{d_1} \times \R^{d_2}} 
L(x,y) d\gamma_{T}^{\lambda}(x,y) =  2 \int_{t_0}^{T} \int_{\R^{d_1} \times \R^{d_2}} b(x,\eta_t) \cdot \nabla_x u_t^*(x,y) d\gamma_t^*(x,y) dt,
\end{equation*}
where ${\bd u}^*$ solves the 
transport equation 
in \eqref{eq:NOCforutgammat2Avril}.

On the other hand, using Lebesgue's dominated convergence theorem together with condition \eqref{ass:furtherassumptionfornu} and Proposition \ref{prop:regularityfromOC} (which supplies us with the proper integrability conditions), we have
$$ \frac{d}{d\lambda} \int_{t_1}^T \mathcal{E} (\nu_t^{\lambda} | \nu^{\infty}) dt = \int_{t_1}^T \int_A \ell(a) \eta_t(a) da dt +  \int_{t_1}^T \int_{A} \log (\nu_t^{\lambda} (a)) d\eta_t(a) dt$$
and,
\begin{equation} 
\frac{d^2}{d\lambda^2}\Big|_{\lambda = 0} \int_{t_1}^T \mathcal{E} \bigl( \nu_t^{\lambda} | \nu^{\infty} \bigr)dt = \int_{t_1}^T \int_{A} \frac{\eta_t(a)^2}{ \nu^*_t(a)} dadt.
\label{eq:secondderivativerunningcost}
\end{equation}
Using \eqref{eq:1Avriloptimalitytildenu}
(and recalling the expression 
\eqref{eq:deftotalcost} for
the cost $J$
and the formula 
\eqref{eq:minformathcalJ1Avril}
for ${\mathcal J}$), we deduce from the last two displays that, for all $\bd \eta \in \mathcal{A}^\ell(t_1)$ satisfying the condition \eqref{ass:furtherassumptionfornu} it holds
\begin{equation}  
\mathcal{J} \bigl( (t_1, \gamma_{t_1}^*), \bd{\nu}^* , \bd{\eta} \bigr)=\frac{d^2}{d\lambda^2}\Big|_{\lambda = 0} J \bigl( (t_1, \gamma^*_{t_1}), \bd\nu^{\lambda} \bigr) \geq 0.
\label{eq:SecondOrderVariation24Sept}
\end{equation}
\vskip 2pt

\textit{Step 2.}
It remains to drop Condition \eqref{ass:furtherassumptionfornu} 
in order to get 
the result for any ${\bd \eta} \in {\mathcal A}^\ell(t_1)$. We proceed by
an approximation argument. For that we take 
$\bd \eta \in \mathcal{A}^\ell(t_1)$ with associated curve $\bd{\rho} \in \mathcal{R}(t_1)$ solution to \eqref{eq:rhotProp3.2} 
(with $\bd \nu$ being understood as $\bd \nu^*$ therein). We can assume that $\eta_t$ has a density for all $t \in [t_1,T]$ and
\begin{equation} 
\int_{t_1}^T \int_{A} \frac{\eta_t(a)^2}{\nu^*_t(a)} da dt <+\infty,
\label{eq:integrability1Avril}
\end{equation}
as otherwise $\mathcal{J} \bigl( t_1, \bd\nu^*,   \bd{\eta}  \bigl) = +\infty$ (see 
\eqref{eq:minformathcalJ1Avril}) and 
\eqref{eq:SecondOrderVariation24Sept}
trivially holds true. For $R>0$, we 
consider the cutoff function $\varphi_R : \R \rightarrow \R$:
\begin{equation}
\varphi_R(r) := 
\left \{
\begin{array}{ll}
r & \mbox{ if } |r| \leq R, \\
0 &\mbox{ otherwise.}
\end{array}
\right.
\end{equation}
By the lower bound 
\eqref{eq:bound:nu(a):exp(-ell(a))}
for $\bd \nu^*$, we can find $R_0 >0$ such that, for all $R \geq R_0$ and 
$t \in [t_0,T]$, it holds $\int_{B(0,R)} \nu^*_t(a')da' \geq 1/2$. Then, we define the following approximation 
$\bd \eta^R$ of $\bd \eta$, by letting
$$\eta_t^R(a) := \varphi_R\bigl(\eta_t(a)\bigr) \mathds{1}_{B(0,R)}(a) - \frac{\int_{B(0,R)}\varphi_R(\eta_t(a'))da'}{\int_{B(0,R)} \nu^*_t(a')da'} \nu^*_t(a) \mathds{1}_{B(0,R)}(a), \quad (t,a) \in [t_1,T] \times A, $$
where $B(0,R)$ of the open ball of center $0$ and radius 
$R$ in $A$. We easily verify that $\bd\eta^R$ belongs to $\mathcal{A}^\ell(t_1)$ for all $R \geq R_0$, satisfies the condition \eqref{ass:furtherassumptionfornu} and converges point-wise to $\bd\eta$ when $R \rightarrow +\infty$ (because $\bd \eta$ is centered). Moreover,
we have, for all $(t,a) \in [t_1,T] \times A$,
\begin{equation} 
\bigl|\eta_t^R(a)\bigr| \leq |\eta_t(a)| + 2 \biggl( \int_A |\eta_t(a')| da' \biggr) \nu^*_t(a). 
\label{eq:uniforminR4Sept}
\end{equation}
By Lebesgue dominated convergence theorem,
we deduce that $\bd{\eta}^R$ converges to $\bd{\eta}$ in $\mathcal{D}(t_1)$ equipped with the norm $\| \cdot \|_{\mathcal{D}(t_1)}$ 
defined in 
\eqref{eq:definitionadmissiblecontrols:aux} (since 
$\bd \eta$ and 
$\bd \nu^*$ themselves belong to 
${\mathcal D}(t_1)$). Thanks to \eqref{eq:integrability1Avril} (and again to 
the domination property \eqref{eq:uniforminR4Sept}), we can also conclude that 
\begin{equation} 
\lim_{R \rightarrow +\infty} \int_{t_1}^T \int_A \frac{|\eta_t^R(a)|^2}{\nu^*_t(a)}dadt = \int_{t_1}^T \int_A \frac{|\eta_t(a)|^2}{\nu^*_t(a)}dadt.
\label{eq:limnutR}
\end{equation}
On the other hand, if we let $\bd\rho^R$ be the solution to \eqref{eq:rhotProp3.2} associated to $\bd\eta^R$, we can apply the estimate 
\eqref{eq:bound:for:rho:useful}
of Proposition \ref{prop:differentiatinggammalambda09/02} (after noticing that $\bd{\rho}^R- \bd{\rho}$ is the solution associated to $\bd{\eta}^R- \bd{\eta}$) to find $C>0$ independent from $R>0$ such that (recall 
that $\gamma_{t_1}^*$ belongs to 
${\mathcal P}_3
({\mathbb R}^{d_1} \times {\mathbb R}^{d_2})$
because $\gamma_0$ does)
$$ \norm{\bd{\rho}^R - \bd{\rho}}_{\mathcal{R}(t_1)} \leq C \norm{ \bd{\eta}^R - \bd{\eta} }_{\mathcal{D}(t_1)}, $$
and then both sides tend to $0$ as $R$ tends to $\infty$. Using \eqref{eq:testfn2foru}, we deduce
\begin{align*}
&\biggl| \int_{t_1}^T  \langle b(\cdot,\eta_t^R) \cdot \nabla_x u_t^*; \rho_t^R \rangle dt -  \int_{t_1}^T \langle b(\cdot,\eta_t) \cdot \nabla_x u_t^*; \rho_t \rangle dt \biggr| 
\\
&\hspace{15pt} \leq \biggl| \int_{t_1}^T \langle b(\cdot,\eta_t^R-\eta_t) \cdot \nabla_x u_t^*; \rho_t^R \rangle dt \biggr| + \biggl| \int_{t_1}^T \langle b(\cdot,\eta_t) \cdot \nabla_x u_t^*; \rho_t^R- \rho_t \rangle dt \biggr| 
\\
&\hspace{15pt} \leq C \norm{\bd{\rho}^R}_{\mathcal{R}(t_1)} \int_{t_1}^T \int_A (1+|a|^2)d|\eta_t^R - \eta_t|(a)dt + C \norm{\bd{\rho}^R - \bd{\rho}}_{\mathcal{R}(t_1)} \int_{t_1}^T \int_A (1+|a|^2) d|\eta_t|(a)dt 
\\
&\hspace{15pt} \leq C \Bigl( \norm{ \bd{\eta}^R - \bd{\eta}}_{\mathcal{D}(t_1)} + \norm{\bd{\rho}^R - \bd{\rho}}_{\mathcal{R}(t_1)} \Bigr).
\end{align*}
We deduce that 
$$  \int_{t_1}^T \langle b(\cdot,\eta_t) \cdot \nabla_x u^*_t; \rho_t \rangle dt =   \lim_{R \rightarrow +\infty}  \int_{t_1}^T \langle b( \cdot,\eta^R_t) \cdot \nabla_x u_t ;\rho^R_t \rangle dt. $$
Combining the above display with \eqref{eq:limnutR}
and recalling once
again
the formula 
\eqref{eq:minformathcalJ1Avril}
for ${\mathcal J}$, 
we obtain
$$ \mathcal{J} \bigl( (t_1, \gamma^*_{t_1}) ,\bd\nu^*, \bd\eta \bigr) = \lim_{ R \rightarrow +\infty} \mathcal{J} \bigl( ( t_1, \gamma^*_{t_1}) ,\bd{\nu}^*, \bd\eta^R \bigr) \geq 0,$$
with the inequality in the right-hand side following from
the first step. This concludes the proof of the proposition.
\end{proof}

\subsection{Necessary and Sufficient Conditions}
We now establish the second part of Theorem 
\ref{prop:SecondOrderConditions3Avril}.

\begin{prop}
Given 
$(t_0, \gamma_0) \in [0,T] \times \mathcal{P}_3(\R^{d_1} \times \R^{d_2})$,  
let $\bd{\nu}^*$  be a minimum of  $J 
( (t_0,\gamma_0), \cdot )$ with associated optimal trajectory and multiplier $(\bd{\gamma}^*, \bd{u}^*)$. Then,
for 
$t_1 \in [t_0,T)$
and for $\bd{\eta} \in \mathcal{A}^{ \ell}(t_1)$ with associated solution $\bd{\rho} \in \mathcal{R}(t_1)$ to the advection equation \eqref{eq:rhotProp3.2}, $\mathcal{J} ( (t_1, \gamma^*_{t_1}) , \bd\nu^*, \bd{\eta}) =0$ if and only if there exists $ \bd{v} \in \mathcal{C}([t_1,T], \mathcal{C}^1_{1})$  such that $(\bd\eta, \bd\rho, \bd v)$ solves \eqref{eq:nutProp3.2}-\eqref{eq:vtrhotProp3.2}.
\label{prop:NSCond26/06}
\end{prop}

\begin{proof}
 { \ }
 \vskip 4pt
    \textit{Step 1.} In this step, we prove that the 
    conditions are sufficient. To do so, we first assume that there exists $\bd{v} \in \mathcal{C}([t_1,T], \mathcal{C}^1_{1})$ such that $(\bd\nu, \bd\rho, \bd{v})$ solves \eqref{eq:nutProp3.2} and \eqref{eq:vtrhotProp3.2}. 
Using the equation 
\eqref{eq:nutProp3.2} satisfied by $\bd\eta$ and recalling that, for all $t \in [t_1,T]$, $\eta_t(A) =0$, we have, for $dt$-almost all $t \in [t_1,T]$,
\begin{equation*}
\begin{split}
&\int_{A} \frac{(\eta_t(a))^2 }{\nu^*_t(a)}da = \int_A \frac{\eta_t(a)}{\nu^*_t(a)} d\eta_t(a) 
\\
&=    -\frac{1}{\epsilon} \int_A \biggl( \Bigl \langle b(\cdot,a) \cdot \nabla_x u^*_t; \rho_t \Bigr \rangle +  \int_{\R^{d_1} \times \R^{d_2}} b(x,a) \cdot  \nabla_x v_t(x,y) d\gamma^*_t (x,y) \biggr) d\eta_t(a).
\end{split}
\end{equation*}
Integrating in time and combining the above display with Proposition \ref{prop:IPPLinearized27/01} and the formula \eqref{eq:minformathcalJ1Avril}, 
we get
$$ \mathcal{J} \bigl( (t_1, \gamma_{t_1}^*), \bd\nu^*, \bd{\eta} \bigr) = 0.$$
Since Proposition \ref{prop:mathcalJnonnegative26/06} shows that the left-hand side is always non-negative, $\bd{\eta}$ is indeed a minimum of $\mathcal{J} ( ( t_1,  \gamma_{t_1}^*),\bd{\nu}^*, \cdot )$. 
\vskip 4pt

\textit{Step 2.} We now prove that the conditions are necessary. Let us consider a minimum 
$\tilde{\bd \eta}=( \tilde{\eta}_t)_{t_1 \le t \le T}$ of 
$\mathcal{J} ( (t_1, \gamma_{t_1}^*), \bd\nu^*, \cdot)$ in
\eqref{eq:minformathcalJ1Avril}
in 
${\mathcal A}^\ell(t_1)$
with associated curve 
$\tilde{\bd \rho}=( \tilde{\rho}_t)_{t_1 \le t \le T}$
in ${\mathcal R}(t_1)$. We also let 
$ \tilde{\bd{v}} =  (\tilde{v}_t)_{t_1 \le t \le T}$ be the solution to
\begin{equation} 
-\partial_t  \tilde{v}_t  -  b(x,\nu^*_t) \cdot \nabla_x \tilde{v}_t =  b(x,\tilde{\eta}_t)  \cdot \nabla_x u^*_t \mbox{ in } (t_1,T) \times \R^{d_1} \times \R^{d_2}, \quad \tilde{v}_T = 0. 
\label{eq:blte1}
\end{equation}
Take another admissible couple 
$(\bd{\eta}, \bd{\rho}) \in {\mathcal A}^\ell(t_1) \times 
{\mathcal R}(t_1)$ with finite cost $\mathcal{J} ( (t_1, \gamma_{t_1}^*), \bd\nu^*, \bd \eta)$. For $\lambda \in (-1,1)$ we let $(\bd \eta^{\lambda} , \bd \rho^{\lambda}) := (1- \lambda) (\tilde{\bd \eta}, \tilde{\bd \rho}) + \lambda (\bd \eta, \bd \rho)$ which is also admissible. Noticing that 
$$ \lambda \mapsto \int_{t_1}^T \int_A \Bigl \langle b(\cdot,a) \cdot \nabla_x u^*_t ; \rho_t^{\lambda} \Bigr \rangle d\eta_t^{\lambda}(a)dt $$
is polynomial (of order two) in $\lambda$, we easily compute
\begin{align*}
    \frac{d}{d\lambda} \Big|_{\lambda = 0}  \int_{t_1}^T \int_A \Bigl \langle b(\cdot,a) \cdot \nabla_x u^*_t ; \rho_t^{\lambda} \Bigr \rangle d\eta_t^{\lambda}(a)dt   &= \int_{t_1}^T \int_A \Bigl \langle b( \cdot,a) \cdot \nabla_x u^*_t;  \rho_t - \tilde{\rho}_t \Bigr \rangle d\tilde{\eta}_t(a) dt  \\
    &\hspace{15pt} + \int_{t_1}^T \int_A \Bigl \langle b(\cdot,a) \cdot \nabla_x u^*_t  ; \tilde{\rho}_t \Bigr \rangle d(\eta_t - \tilde{\eta}_t)(a) dt. 
\end{align*}
In order to handle 
the first term in the right-hand side, we use the equations satisfied by $\tilde{\bd \rho}$ and $\tilde{\bd v}$ together with Proposition \ref{prop:IPPLinearized27/01} with 
$(\bd{\rho}^1,\bd{\eta}^1) = 
(\bd{\rho}- \tilde{\bd \rho},
\bd{\eta} - \tilde{\bd \eta})$ and 
$(\bd{\rho}^2,\bd{\eta}^2) = (\tilde{\bd \rho}, \bd{\tilde{\eta}})$ therein
and Proposition \ref{prop:measurable+fubini15/01}, to get
\begin{equation}
\label{IPP:29Sept}
\begin{split}
&\int_{t_1}^T  \int_{A} \Bigl \langle b( \cdot,a)  \cdot \nabla_x u^*_t ; \rho_t - \tilde{\rho}_t \Bigr \rangle d\tilde{\eta}_t(a)dt 
\\
   &=  \int_{t_1}^T \int_A \int_{\R^{d_1} \times \R^{d_2}}  b( x ,a) \cdot \nabla_x \tilde{v}_t(x,y) d \gamma^*_t(x,y) d(\eta_t - \tilde{\eta}_t)(a)dt. 
\end{split}
\end{equation}
This implies
\begin{equation}
\label{eq:derivativerunnigcostlambda9sept}
\begin{split}
&\frac{d}{d\lambda} \Big|_{\lambda = 0}  \int_{t_1}^T \int_A \Bigl \langle b(\cdot,a) \cdot \nabla_x u^*_t ;\rho_t^{\lambda} \Bigr \rangle d\eta_t^{\lambda}(a)dt   
\\
    &= \int_{t_1}^T \int_A \Bigl( \Bigl \langle b( \cdot,a) \cdot \nabla_x u^*_t ; \tilde{\rho}_t \Bigr \rangle  + \int_{\R^{d_1} \times \R^{d_2}}  b(x,a) \cdot \nabla_x \tilde{v}_t(x,y) d\gamma^*_t(x,y) \Bigr) d(\eta_t - \tilde{\eta}_t)(a)dt,
\end{split}
\end{equation}
which permits to handle the derivative w.r.t. 
$\lambda$ of the second term in the definition of the cost
$\mathcal{J} ( (t_1, \gamma_{t_1}^*), \bd\nu^*, \bd \eta^\lambda)$
in \eqref{eq:minformathcalJ1Avril}.

Now, we deal with the first part of the cost $\mathcal{J} ( (t_1, \gamma_{t_1}^*), \bd\nu^*, \bd \eta^\lambda)$
in \eqref{eq:minformathcalJ1Avril}. Since  $ \mathcal{J}( (t_1,\gamma_{t_1}^*), \bd\nu^*, \bd\eta )$ is finite, $\bd \eta$ is absolutely continuous with respect to the Lebesgue measure on $[t_1,T] \times A$ with a density (still written) $(t,a) \mapsto \eta(t,a) $ in $L^2((\bd{\nu}^*)^{-1})$. This allows us to differentiate the quadratic map $$\lambda \mapsto \int_{t_1}^T \int_A \frac{(\eta_t^{\lambda})^2(a)}{\nu^*_t(a)}da dt,$$
and find
\begin{equation}
\label{eq:derivativerunnigcostlambda9sept:2}
\frac{d}{d \lambda }\Big|_{\lambda=0} \int_{t_1}^T \int_A \frac{(\eta_t^{\lambda})^2(a)}{\nu^*_t(a)}da dt = 2 \int_{t_1}^T \int_A \frac{\tilde{\eta}_t(a) (\eta_t(a) -\tilde{\eta}_t(a))}{\nu^*_t(a)}dadt.
\end{equation}
By minimality of $\tilde{\bd \eta}$ we deduce from 
\eqref{eq:derivativerunnigcostlambda9sept}
and
\eqref{eq:derivativerunnigcostlambda9sept:2} that, for all admissible $\boldsymbol{\eta}$ with finite cost, 
$$ \int_{t_1}^T \int_A \Bigl(  l_t(a) + \frac{\tilde{\eta}_t(a)}{\nu^*_t(a)} \Bigr) d(\eta_t - \tilde{\eta}_t)(a) dt =0,$$
where the function $l$ is defined by
\begin{equation}
\label{def:l:order2:sufficient}
l_t(a) := \Bigl \langle b( \cdot , a) \cdot \nabla_x u^*_t ; \tilde{\rho}_t \Bigr \rangle + \int_{\R^{d_1} \times \R^{d_2}} b(x,a) \cdot \nabla_x \tilde{v}_t(x,y) d\gamma^*_t(x,y), \quad (t,a) \in [t_1,T] \times A.
\end{equation}
Taking $\boldsymbol{\eta}$ in the form 
$$\eta_t(a) = \tilde{\eta}_t(a) + \xi_t(a) - \int_A \xi_t(a')da', \quad (t,a) \in [t_1,T] \times A,$$ 
for an arbitrary  smooth function $\xi : [t_1,T] \times A \rightarrow \R$ with compact support, we get (thanks to Fubini's theorem)
\begin{align*} 
&\int_{t_1}^T \int_A \biggl[ l_t(a) 
+
\frac{\tilde{\eta}_t(a)}{\nu_t^*(a)} - \int_A \biggl(l_t(a') 
+ \frac{\tilde{\eta}_t(a')}{\nu_t^*(a')} \biggr) da'  \biggr] \xi_t(a)dadt 
\\
&= \int_{t_1}^T \int_A \Bigl( l_t(a) + \frac{\tilde{\eta}_t(a)}{\nu_t^*(a)} \Bigr) \biggl(\xi_t(a) - \int_A \xi_t(a')da' \biggr)da dt =0, 
\end{align*}
and we deduce that 
\begin{equation} 
l_t(a) + \frac{\tilde{\eta}_t(a)}{\nu^*_t(a)} = \int_A \Bigl( l_t(a') +\frac{\tilde{\eta}_t(a')}{\nu^*_t(a')} \Bigr) da',
\label{eq:29Sept14:23}
\end{equation}
for Lebesgue almost every $(t,a) \in [t_1,T] \times A$. Recalling that for $dt$-almost every $t \in [t_1,T]$, $\eta_t^*(A) =0$ and $\nu^*_t(A)=1$, we deduce from \eqref{eq:29Sept14:23} (after multiplying both sides by $\nu^*_t(a)$ and integrating in $a$) that
$$ \int_A l_t(a) d\nu^*_t(a) =  \int_A \Bigl( l_t(a') +\frac{\tilde{\eta}_t(a')}{\nu^*_t(a')} \Bigr) da', \quad  \mbox{for Lebesgue almost every } t  \in [t_1,T]. $$
Getting back to \eqref{eq:29Sept14:23} and recalling the definition 
\eqref{def:l:order2:sufficient}
of $l$, we get the result.
\end{proof}

\subsection{Additional Regularity}

Similarly to the first order conditions we can prove some more regularity for solutions to the linearized system.

\begin{proof}[Proof of Proposition \ref{prop:reg:v,rho}]
By Proposition \ref{prop:regularityfromOC} we know that $t \mapsto \nu_t \in \mathcal{M}_{1+|a|^3}(A)$ is continuous.  Thanks to Proposition \ref{prop:additionalregularityvt}, it therefore suffices to prove that $t \mapsto \eta_t \in \mathcal{M}_{1+|a|^3}(A)$ is continuous. From the exponential integrability of $\bd{\nu}$ given by Proposition \ref{prop:regularityfromOC}, the regularity of  $\nabla_x u$ from \eqref{eq:testfn2foru} in Proposition \ref{prop:regularityfromOC}, the facts that $\bd{\rho}$ belongs to $\mathcal{R}(t_1)$ (recall Definition \ref{def:R(t_0)})  and the fact that ${\boldsymbol v}$ belongs to $ \mathcal{C}([t_1,T], \mathcal{C}^1_{1})$ by Proposition \ref{prop:SolnVt19/02}, we first deduce from the explicit formula \eqref{eq:nutProp3.2} for $\bd{\eta}$ that 
\begin{equation}
\sup_{t \in [t_1,T]} \int_A (1+|a|^m)d|\eta_t|(a) <+\infty,
\label{19/02:19:43}
\end{equation}
for any $m \geq 1$. Using the representation formula \eqref{eq:representationformulavt} for $\bd{v}$ we easily deduce that $\sup_{t \in [t_1,T]}\norm{ \nabla_xv_t}_{\mathcal{C}^1_{2,1}} $ is finite. Using the time regularity of $t \mapsto \nu_t$ from Proposition \ref{prop:regularityfromOC} and Lemma \ref{lem:BackThenLemA.17} we deduce that $t \mapsto (a \mapsto \nu_t(a) \langle b(\cdot,a) \cdot \nabla_x u_t;\rho_t \rangle ) \in \mathcal{M}_{1+|a|^3}(A)$ is continuous over $[t_1,T]$. Now for $t_3 \geq t_2 \in [t_1,T]$ we have
\begin{align*}
   & \Big| \int_{\R^{d_1} \times \R^{d_2}} b(x,a) \cdot \nabla_x v_{t_3} (x,y) d\gamma_{t_3}(x,y) - \int_{\R^{d_1} \times \R^{d_2}} b(x,a) \cdot \nabla_x v_{t_2} (x,y) d\gamma_{t_2}(x,y) \Big| \\
   &\leq \Big| \int_{\R^{d_1} \times \R^{d_2}} b(x,a) \cdot \nabla_x(v_{t_3} -v_{t_2})d\gamma_{t_3}(x,y) \Big| + \Big|\int_{\R^{d_1} \times \R^{d_2}} b(x,a)\cdot \nabla_xv_{t_2}(x,y)d(\gamma_{t_3} - \gamma_{t_2})(x,y) \Big| \\
   &\leq C (1+|a|) \norm{ \nabla_x v_{t_3} - \nabla_xv_{t_2}}_{\mathcal{C}^0_3} \int_{\R^{d_1} \times \R^{d_2}} (|x|^2 +|y|^2)^{3/2} d\gamma_{t_3}(x,y)  \\
   &+ C(1+|a|^2) \norm{\nabla_x v_{t_2}}_{\mathcal{C}^1_{2,1}} \norm{ \gamma_{t_3} - \gamma_{t_2}}_{(\mathcal{C}^1_{2,1})^*}
\end{align*}
and we conclude from Proposition \ref{prop:SolnVt19/02}, Proposition \ref{lem:compactnessContinuityEquation24/06} and Lemma \ref{lem:comparisonNorms} together with the regularity of $t \mapsto \nu_t$ that $t \mapsto (a \mapsto \nu_t(a) \int_{\R^{d_1} \times \R^{d_2}}b(x,a) \cdot \nabla_x v_t(x,y) d\gamma_t(x,y)) $ is continuous from $[t_1,T]$ to $\mathcal{M}_{1+|a|^3}(A)$. The continuity of $t \mapsto c_t$ the normalizing constant is handled similarly. 
\end{proof}

\section{Proofs of Section \ref{sec:preparatorywork} }

\label{sec:ProofsofSection3}

In this section we prove the various results presented in Section \ref{sec:preparatorywork}. We make repeated use of the objects and notation introduced therein.

We start with

\begin{proof}[Proof of Proposition \ref{prop:preliminaryworkprop1-28/02}]
    \textit{Step 1. Convergence of $({\boldsymbol \eta}^n)_{ n \in \mathbb{N}}$.}

For each $n \in {\mathbb N}$, 
we consider ${\boldsymbol \eta}^n$ as a finite signed measure on $[t_0^n,T] \times A$, which we can decompose (with standard notations) into 
${\boldsymbol \eta}^n = {\boldsymbol \eta}^{n,+} - {\boldsymbol \eta}^{n,-}$. 
Without any difficulty, 
${\boldsymbol \eta}^{n,\pm}$ can be extended in a trivial manner to 
(the Borel subsets of) 
$[0,T] \times A$, by letting for any Borel subset $B$ of $[0,T] \times A$, 
${\boldsymbol \eta}^{n,\pm}(B)={\boldsymbol \eta}^{n,\pm}(B \cap ([t_0^n,T] \times A))$. 
We can extend ${\boldsymbol \eta}^n$ to $[0,T] \times A$ in a similar manner while preserving the decomposition 
${\boldsymbol \eta}^n={\boldsymbol \eta}^{n,+} - {\boldsymbol \eta}^{n,-}$. 
From Lemma \ref{lem:timetogotoGallia02/03}, we have 
\begin{equation}
\sup_{n \in {\mathbb N}}
\int_0^T \int_A \bigl( 1 + \vert a \vert^4 \bigr) d{\boldsymbol \eta}^{n,\pm}(t,a) 
< \infty.
\label{eq:04/03:10:37}
\end{equation}
And then, we deduce that, up to a subsequence, 
the two sequences of positive measures $({\boldsymbol \eta}^{n,+})_{n \in {\mathbb N}}$
and 
$({\boldsymbol \eta}^{n,-})_{n \in {\mathbb N}}$
converge weakly to some positive measures ${\boldsymbol \eta}^+$ and 
${\boldsymbol \eta}^-$ (on $[0,T] \times A$). 
It is standard to see that 
\begin{equation}
\int_0^T \int_A \bigl( 1 + \vert a \vert^4 \bigr) d{\boldsymbol \eta}^{\pm}(t,a) \leq  \liminf_{n \rightarrow +\infty} \int_0^T \int_A \bigl( 1 + \vert a \vert^4 \bigr) d{\boldsymbol \eta}^{n,\pm}(t,a) \leq  C.
\label{eq:tobecombinedwith04/03}
\end{equation}
And, by a uniform integrability argument, it holds, for all continuous 
function $f \in \mathcal{C}([0,T] \times A)$ satisfying  
$|f_t(a) | \leq C(1+|a|^{3+\delta})$
for some $(C,\delta) \in \R^+ \times (0,1)$, 
\begin{equation}
\label{eq:lim:etan}
\lim_{n \rightarrow +\infty} \int_{0}^T \int_A f_t(a) 
d{\boldsymbol \eta}^{n,\pm}(t,a) 
= \int_{0}^T \int_A f_t(a) d{\boldsymbol \eta}^{\pm}(t,a).
\end{equation}

It remains to see that the time marginal laws
of ${\boldsymbol \eta}^+$ and ${\boldsymbol \eta}^-$
are absolutely continuous (with respect to the Lebesgue measure 
on $[0,T]$). Repeating the second part of the proof of Lemma \ref{lem:PinskerSecFOC28/02} word for word, 
we deduce that there exists a constant $C$ such that for any 
$n \in {\mathbb N}$ and any Borel subset $I$ of $[0,T]$, 
\begin{equation}
\int_{I \times A}
d {\boldsymbol \eta}^{n,\pm}(t,a) \leq C  \sqrt{\vert I\vert},
\label{eq:09/12/2024}
\end{equation}
where $\vert I \vert$ denotes the Lebesgue measure of $I$. By weak convergence, \eqref{eq:09/12/2024} still holds at the limit for $\bd{\eta}^+$ and $\bd{\eta}^-$.  In particular we can find 
two positive functions $g^+$ and $g^-$ in $L^1([0,T])$ and two collections of 
regular conditional probability distributions $(n_t^+(\cdot))_{t \in [0,T]}$
and 
$(n_t^-(\cdot))_{t \in [0,T]}$ on $A$ such that, for any bounded 
Borel function $f : [0,T] \times A \rightarrow {\mathbb R}$,
\begin{equation*}
\int_{[0,T] \times A} f(t,a) d {\boldsymbol \eta}^{\pm}(t,a)  = \int_0^T \biggl( \int_A f_t(a) dn_t^{\pm}(a) \biggr) g_t^{\pm} dt.  
\end{equation*}
Letting $\eta_t^{\pm} := g_t^{\pm} \cdot n_t^{\pm}$, we deduce that 
\begin{equation*}
\int_{[0,T] \times A} f(t,a) d {\boldsymbol \eta}^{\pm}(t,a)  = \int_0^T \biggl( \int_A f_t(a)d \eta_t^{\pm}(a) \biggr)  dt, 
\end{equation*}
which is the required identity. In particular, letting $\eta_t:=\eta_t^+-\eta_t^-$, for any 
$t \in [0,T]$, we recover the fact that $d{\boldsymbol \eta}(t,a)$
can be disintegrated as $d\eta_t(a) dt$. 
Choosing 
$f_t(a)$ in \eqref{eq:lim:etan} as a function independent of $a$ 
and using the fact that $\eta^n_t(A)=0$ for all $n \in {\mathbb N}$
and $t \in [0,T]$, we get that $\eta_t(A)=0$ for (almost every) $t \in [0,T]$. 

Also, choosing $f_t(a)$ to be zero for $t > t_0-\epsilon$ for $\epsilon>0$ arbitrarily small and using the fact that $\eta^{\pm,n}_t=0$ for all $n \in {\mathbb N}$ and 
(almost every) $t \in [0,t_0)$, we deduce from \eqref{eq:lim:etan} that $\eta^{\pm}_t=0$ for $t < t_0$.

Now we take $t_1 < t_2 \in (t_0,T]$ and $f$ a bounded, continuous function over $[t_1,t_2] \times A$. Approximating the function $\mathbf{1}_{[t_1,t_2]}$ by continuous functions and using  the uniform bound \eqref{eq:09/12/2024} we can show that
$$\lim_{n \rightarrow +\infty} \int_{t_1}^{t_2} \int_A f_t(a) d\eta^n_t(a) dt = \int_{t_1}^{t_2} \int_A f_t(a) d\eta_t(a) dt. $$
By a uniform integrability argument, using the uniform bound \eqref{eq:04/03:10:37} we show that this limit holds when $f$ has growth of order $(1+|a|^{3+\delta})$ for $\delta \in (0,1)$. In particular, for all $t_2 > t_1$ in $(t_0,T]$ we have 
$$\frac{1}{\sqrt{t_2 - t_1}} \int_{t_1}^{t_2} \int_A (1+|a|^2) d \eta^{\pm}_t(a) dt = \lim_{n \rightarrow +\infty} \frac{1}{\sqrt{t_2-t_1}} \int_{t_1}^{t_2}\int_A (1+|a|^2) d \eta^{n,\pm}_t(a) dt < +\infty.$$
Combined with \eqref{eq:tobecombinedwith04/03} this shows that $\bd{\eta}$ belongs to $\mathcal{D}(t_0)$ and 
$$ \norm{ \bd{\eta}}_{\mathcal{D}(t_0)} \leq \liminf_{n \rightarrow +\infty} \norm{ \bd{\eta}^n}_{\mathcal{D}(t_0^n)}. $$

\color{black}

\color{black}
\vskip 4pt

\textit{Step 2. Convergence of $({\boldsymbol v}^n)_{n \in \mathbb{N}}$.} 
Following
\textit{Step 1}, we first extend 
${\boldsymbol v}^n$, for each 
$n \in {\mathbb N}$, 
to $[0,T]$ by letting 
$v^n_t=v^n_{t_0^n}$
for $t \in [0,t_0^n)$. This does not change the estimate from Lemma \ref{lem:timetogotoGallia02/03}. 
Then, the combination of the two bounds for $\bd{v}^n$ in \eqref{lem:timetogotoGallia02/03}
shows that 
the functions 
$(v^n)_{n \in {\mathbb N}}$ and 
$(\nabla_x v^n)_{n \in {\mathbb N}}$
are uniformly bounded and 
uniformly 
continuous 
on $[0,T] \times B$ for any closed ball $B$ of $\R^{d_1} \times \R^{d_2}$. It follows by Ascoli's theorem that there is $v \in \mathcal{C}([t_0,T],\mathcal{C}^1_1)$ such that 
$$ \lim_{n \rightarrow +\infty} \sup_{t \in [t_0 \vee t_0^n,T]} \norm{ v_t^n-v_t}_{\mathcal{C}^1(B)}=0. $$
In particular, we easily deduce that 
$$ \sup_{t \in [t_0,T]} \norm{v_t}_{\mathcal{C}^1_1} \leq \liminf_{n \rightarrow +\infty}  \sup_{t \in [t_0,T]} \norm{v^n_t}_{\mathcal{C}^1_1}. $$
Now we can consider functions $(\chi_R)_{ R >0}$ such that
\begin{equation*}
    \chi_R(x,y) = \left \{ 
\begin{array}{ll}
1 & \mbox{ if } (x,y) \in B(0,R) \\
0 & \mbox{ if } (x,y) \notin B(0,R+1)
\end{array}
\right.
\end{equation*}
and $\norm{ \chi_R}_{\mathcal{C}^1_b} \leq c$ for some $c >0$ independent from $R$, where $B(0,R)$ is the ball of $\R^{d_1} \times \R^{d_2}$ of radius $R$, centered at the origin. Then, for every $R >0$ and every $n \geq 1$ it holds
\begin{align*} 
\sup_{t \in [t_0^n \vee t_0,T] } \norm{v_t^n -v_t}_{\mathcal{C}^1_2} &\leq \sup_{t \in [t_0^n \vee t_0,T] } \norm{ \chi_R(v_t^n -v_t)}_{\mathcal{C}^1_2} + \sup_{t \in [t_0^n \vee t_0,T] } \norm{(1- \chi_R)(v_t^n -v_t)}_{\mathcal{C}^1_2} \\
&\leq C \sup_{t \in[t_0^n \vee t_0,T]} \norm{ v_t^n - v_t }_{\mathcal{C}^1(B(0,R+1))} + \frac{C}{R} (\sup_{t \in[t_0^n,T]} \norm{ v_t^n}_{\mathcal{C}^1_1} + \sup_{t \in[ t_0,T]  } \norm{v_t}_{\mathcal{C}^1_1} ).  
\end{align*}
The claim follows by taking first  $R \rightarrow +\infty$ and then $n \rightarrow +\infty$.
\vskip 4pt

\textit{Step 3. Convergence of $({\boldsymbol \rho}^n)_{ n \in \mathbb{N}}$.}
Following \textit{Step 2}, we extend ${\boldsymbol \rho}^n$, for each $n \in {\mathbb N}$, to $[0,T]$ by letting 
${\rho}_t^n = {\rho}^n_{t_0^n}$ for 
$t \in [0,t_0^n)$. Then, thanks to the uniform estimate on $\norm{ \bd{\rho}^n}_{\mathcal{R}(t_0^n)}$ from Lemma \ref{lem:timetogotoGallia02/03} we can apply the compactness result of Lemma \ref{lem:StrongCompactnessCurves} to find ${\boldsymbol \rho} \in \mathcal{R}(t_0)$ such that
$$\lim_{n \rightarrow +\infty} \sup_{t \in [t_0,T]} \norm{ \rho_t^n - \rho_t}_{(\mathcal{C}^2_{2})^*} =0.$$
\vskip 4pt

\textit{Step 4. Proof of 
\eqref{eq:v:rho:proof:Meta2.2-02/03}.} It remains to show 
that $({\boldsymbol \rho},{\boldsymbol v})$
solves \eqref{eq:v:rho:proof:Meta2.2-02/03}. Obviously, the strategy is to pass to the limit in 
\eqref{eq:systemrhonvn29Mai-28/02}. 

By combining  Lemma \ref{lem:convergenceifuniqueness}
and the uniform bound on $\norm{\bd{\eta}^n}_{\mathcal{D}(t_0^n)}$ from Lemma \ref{lem:timetogotoGallia02/03}, we first observe that 
\begin{equation}
\label{eq:proof:5.22:vun:tildenu}
\lim_{n \rightarrow 
+ \infty}
\int_{t_0^n \vee t_0}^T \int_A (1+|a|^4)| \nu_t^n(a) - 
 \nu^*_t(a)|da
dt
= 0. 
\end{equation}
Following the proof of Lemma \ref{lem:diffulambda24/02:13:23} we can rewrite, for all $t \in (t_0,T]$ and all $(x,y) \in \R^{d_1} \times \R^{d_2}$ and all $n$ large enough such that $t_0^n <t$,
\begin{equation}
v_t^n(x,y) = \int_t^T b(X_s^{n,t,x},\eta_s^n) \cdot \nabla_x u^{*,n}_s(X_s^{n,t,x},y) ds 
\label{eq:step3:proof(5.18):representation:vn}
\end{equation}
where $(X_s^{n,t,x})_{s \in [t,T]}$ is the solution to
\begin{equation*}
\dot{X}_s^{n,t,x} = b(X_s^{n,t,x},\nu_s^n),
\quad 
s \in [t,T]; \quad X_t^{n,t,x}=x. 
\end{equation*}
Similarly we can introduce the solution $(X_s^{*,t,x})_{s \in [t,T]}$ to the ODE
\begin{equation*}
\dot{X}_s^{*,t,x}
= b(X_s^{,x},\nu^{*}_s), \quad s \in [t,T]; 
\quad X_t^{*,t,x}=x. 
\end{equation*}
Using the regularity properties of 
$b$ together with 
\eqref{eq:proof:5.22:vun:tildenu}, we get 
\begin{equation*}
\lim_{n \rightarrow + \infty}
\sup_{s \in [t, T]}
\vert X_s^{n,t,x} - X_s^{*,t,x} \vert 
= 0. 
\end{equation*}
Recall that, by item $iii)$ in  Property {\bf $(\mathcal{Q}_0)$}, $(\bd{\nu}^{*,n}, \bd{\gamma}^{*,n}, \bd{u}^{*,n})_{n \in \mathbb{N}}$ is assumed to converge to $(\bd{\nu}^*, \bd{\gamma}^{*}, \bd{u}^*)$ in the sense of Lemma \ref{lem:convergenceifuniqueness}. Thanks to the convergence of $(\bd{\eta}^n)_{n \in \mathbb{N}}$ from \textit{Step 1}
it is then quite easy to pass to the limit in 
\eqref{eq:step3:proof(5.18):representation:vn}
and to get 
the formula 
\begin{equation*}
v_t(x,y) = \int_t^T b(X_s^{t,x}, \eta_s) \cdot 
\nabla_x u^*_s(X_{s}^{t,x},y) ds, \quad (t,x,y) \in (t_0,T] \times \R^{d_1} \times \R^{d_2}.
\end{equation*}
By a continuity argument, the representation formula also holds at $t=t_0$. The first equation in  
\eqref{eq:v:rho:proof:Meta2.2-02/03}
follows. 

We now address the second equation in \eqref{eq:v:rho:proof:Meta2.2-02/03}.
To prove that 
it is satisfied, we
fix $t_1 \in (t_0,T]$ and return to the second equation 
in \eqref{eq:systemrhonvn29Mai-28/02}(for $n$ sufficiently large so that $t_0^n<t_1$). We take a test function $\varphi : [t_0,T] \times \R^{d_1} \times \R^{d_2} \rightarrow \R$ satisfying \eqref{eq:testfn1-09/03}- \eqref{eq:testfn2-09/03} and some time $t_2 \in (t_1,T]$. Given the weak formulation of the equation we have
$$ \langle \varphi_{t_2}; \rho_{t_2}^n \rangle = \langle \varphi_{t_1}; \rho_{t_1}^n \rangle + \int_{t_1}^{t_2} \langle \partial_t \varphi_t + b(\cdot,\nu_t^n) \cdot \nabla_x \varphi_t; \rho_t^n\rangle dt + \int_{t_1}^{t_2} \int_{\R^{d_1} \times \R^{d_2}} b(x,\eta_t^n) \cdot \nabla_x \varphi_t (x,y) d\gamma_t^{*,n}(x,y) dt. $$
Using the regularity of $\varphi, \partial_t \varphi$ and $\nabla_x \varphi$, together with the regularity of $b$ and the convergence of $(\bd{\rho}^n)_{ n \in \mathbb{N}}$ to $\bd{\rho}$ from Step $iii)$ of the proof we first obtain that 
\begin{align*} \lim_{n \rightarrow +\infty} \Bigl( \langle \varphi_{t_2}; \rho_{t_2}^n \rangle, \langle \varphi_{t_1}; \rho_{t_1}^n \rangle, &\int_{t_1}^{t_2} \langle \partial_t \varphi_t + b(\cdot , \nu_t^*) \cdot \nabla_x \varphi_t; \rho_t^n \rangle dt \Bigr) \\
&= \Bigl( \langle \varphi_{t_2}; \rho_{t_2} \rangle, \langle \varphi_{t_1}; \rho_{t_1} \rangle, \int_{t_1}^{t_2} \langle \partial_t \varphi_t + b(\cdot , \nu_t^*) \cdot \nabla_x \varphi_t; \rho_t \rangle dt \Bigr).
\end{align*}
Now we observe that $(\bd{\rho}^n)_{ n \in \mathbb{N}}$ being bounded in $\mathcal{R}(t_0)$ there is $C_{\varphi} >0$ independent from $n$ such that 
\begin{align} 
\notag \Bigl| \int_{t_1}^{t_2} \langle b(\cdot,\nu_t^n- \nu_t^*) \cdot & \nabla_x \varphi_t ; \rho_t^n \rangle dt \Bigr| \leq C_{\varphi} \int_{t_1}^{t_2} \norm{ b(\cdot, \nu_t^n - \nu_t^*)}_{\mathcal{C}^1_b} dt  \leq C_{\varphi} \int_{t_1}^{t_2} \int_A (1+|a|^2) d|\nu_t^n - \nu_t^*|(a)dt   \\
&\leq C_{\varphi} \lambda_n \int_{t_1}^{t_2} \int_A (1+|a|^2) d|\eta_t^n|(a)dt + C_{\varphi} \int_{t_1}^{t_2} \int_A (1+|a|^2)d|\nu_t^{*,n} - \nu_t^*|(a)dt \label{eq:25/02-09:59}
\end{align}
The first term in the second line is handle by the boundness of $(\bd{\eta}^n)_{n \in \mathbb{N}}$ in $\mathcal{D}(t_0)$ and the fact that $\lambda_n \rightarrow 0$ as $n \rightarrow +\infty$. The second term is handled item $iii)$ in {\bf Property $(\mathcal{Q}_0)$}. Combined with \eqref{eq:25/02-09:59} this shows that 
$$ \lim_{n \rightarrow +\infty} 
\int_{t_1}^{t_2} \langle b(x,\nu_t^n) \cdot \nabla_x \varphi_t ; \rho_t^n \rangle dt = \int_{t_1}^{t_2} \langle b(x,\nu_t^*) \cdot \nabla_x \varphi_t ; \rho_t \rangle dt. $$
Now we handle the term involving $\bd{\gamma}^{*,n}$ that we decompose into
\begin{align}
\notag     \int_{t_1}^{t_2} \int_{\R^{d_1} \times \R^{d_2}} & b(x,\eta_t^n) \cdot \nabla_x \varphi_t(x,y) d\gamma_t^{*,n}(x,y) dt -  \int_{t_1}^{t_2} \int_{\R^{d_1} \times \R^{d_2}} b(x,\eta_t) \cdot \nabla_x \varphi_t(x,y) d\gamma_t^{*}(x,y) dt  \\
\notag    &=  \int_{t_1}^{t_2} \int_{\R^{d_1} \times \R^{d_2}} b(x,\eta_t^n) \cdot \nabla_x \varphi_t(x,y) d(\gamma_t^{*,n} - \gamma_t^*)(x,y) dt  \\
    &+  \int_{t_1}^{t_2} \int_{\R^{d_1} \times \R^{d_2}} b(x,\eta_t^n-\eta_t) \cdot \nabla_x \varphi_t(x,y) d\gamma_t^*(x,y)dt  \label{eq:25/02-10:22}
\end{align}
We estimate the first term by 
\begin{align*}
    & \Bigl| \int_{t_1}^{t_2} \int_{\R^{d_1} \times \R^{d_2}} b(x,\eta_t^n) \cdot \nabla_x \varphi_t(x,y) d(\gamma_t^{*,n} - \gamma_t^*)(x,y) dt \Bigr|  \\
    &\leq \sup_{t \in [t_1,t_2]} \norm{ \gamma_t^{*,n} - \gamma_t^*}_{(\mathcal{C}^1_{2,1})^*} \int_{t_1}^{t_2} \norm{ b(\cdot, \eta_t^n) \cdot \nabla_x \varphi_t}_{\mathcal{C}^1_{2,1}} dt \\
    &\leq C_{\varphi} \sup_{t \in [t_1,t_2]} \norm{ \gamma_t^{*,n} - \gamma_t^*}_{(\mathcal{C}^1_{2,1})^*} \int_{t_1}^{t_2} \int_A (1+|a|^2) d|\eta^n_t|(a)dt.
\end{align*}
By item $iii)$ in {\bf Property $(\mathcal{Q}_0)$} in  and Lemma \ref{lem:comparisonNorms} $\sup_{t \in [t_1,t_2]} \norm{ \gamma_t^{*,n} - \gamma_t^*}_{(\mathcal{C}^1_{2,1})^*} $ goes to $0$ as $n \rightarrow +\infty$ while $\int_{t_1}^{t_2} \int_A (1+|a|^2) d|\eta_t^n|(a)dt$ is bounded independently from $n$ by Lemma \ref{lem:timetogotoGallia02/03}. For the last term in \eqref{eq:25/02-10:22} we use Fubini's theorem, see in particular Appendix \ref{sec:AppendixFubini} to rewrite it as 
\begin{equation} \int_{t_1}^{t_2} \int_A \Bigl \{ \int_{\R^{d_1} \times \R^{d_2}} b(x,a) \cdot \nabla_x \varphi_t(x,y) d\gamma_t^*(x,y) \Bigr \} d(\eta_t^n-\eta_t)(a)dt 
\label{eq:25/02-10:43}
\end{equation}
and observe that the integrand is continuous with linear growth in $a$, uniformly in $t \in [t_1,t_2]$. This makes it possible to apply the conclusions of the first step of this proof to justify that \eqref{eq:25/02-10:43} converges to $0$ as $n \rightarrow +\infty$.
We can go back to \eqref{eq:25/02-10:22} to infer that 
$$ \lim_{n \rightarrow +\infty}  \int_{t_1}^{t_2} \int_{\R^{d_1} \times \R^{d_2}}  b(x,\eta_t^n) \cdot \nabla_x \varphi_t(x,y) d\gamma_t^{*,n}(x,y) dt =  \int_{t_1}^{t_2} \int_{\R^{d_1} \times \R^{d_2}} b(x,\eta_t) \cdot \nabla_x \varphi_t(x,y) d\gamma_t^{*}(x,y) dt$$
an conclude that 
$$ \langle \varphi_{t_2}; \rho_{t_2} \rangle = \langle \varphi_{t_1}; \rho_{t_1} \rangle + \int_{t_1}^{t_2} \langle \partial_t \varphi_t + b(\cdot,\nu_t) \cdot \nabla_x \varphi_t; \rho_t\rangle dt + \int_{t_1}^{t_2} \int_{\R^{d_1} \times \R^{d_2}} b(x,\eta_t) \cdot \nabla_x \varphi_t (x,y) d\gamma_t^{*}(x,y) dt $$
for all $t_1,t_2 \in (t_0,T]$ with $t_1 < t_2$.  By a continuity argument we then propagate this equality to $t_1=t_0$. It remains to notice that $\langle \varphi_{t_0}; \rho_{t_0} \rangle =0$. But this follows from the uniform bound on $\norm{\bd{\rho}^n}_{n \in \mathcal{R}(t_0^n)}$ since it gives 
$$ |\langle \varphi_{t_0}; \rho_{t_0} \rangle | = | \lim_{n \rightarrow +\infty} \langle \varphi_{t_0} ; \rho_{t_0}^n \rangle | \leq C_{\varphi} \liminf_{n \rightarrow +\infty} \norm{ \rho^n_{t_0} - \rho^n_{t_0^n}}_{(\mathcal{C}^2_2)^*} \leq C_{\varphi}  \liminf_{n \rightarrow +\infty} \bigl \{ \norm{ \bd{\rho}^n}_{\mathcal{R}(t_0^n)} \sqrt{|t_0^n -t_0|} \bigr \} =0. $$

\end{proof}

\begin{proof}[Proof of Proposition \ref{prop:convergencektn03/03}.]
\textit{Step 1. $t \mapsto k_t \in \mathcal{C}^1_2(A)$ is bounded.}    For all $t \in [t_0,T]$ we have
    $$ \norm{ k_t}_{\mathcal{C}^1_2(A)} \leq \frac{1}{\epsilon} \norm{ \int_{\R^{d_1} \times \R^{d_2}} b(x,\cdot) \cdot \nabla_x v_t(x,y) d\gamma^*_t(x,y)  }_{\mathcal{C}^1_2(A)} + \frac{1}{\epsilon} \norm{a \mapsto \langle b(\cdot,a) \cdot \nabla_x u^*_t; \rho_t \rangle }_{\mathcal{C}^1_2(A)} $$
For the first term, we use formula \eqref{eq:firstconvenientexpress02/03} with $q=1,k=0,p=3$, and for the second term, the same formula with $q=k=1$ and $p=3$ and we obtain
$$ \norm{ k_t}_{\mathcal{C}^1_2(A)} \leq C_b \bigl( \norm{ \nabla_x v_t}_{\mathcal{C}^0_2} \int_{\R^{d_1} \times \R^{d_2}}(|x|^2 + |y|^2)^{3/2}d\gamma_t^*(x,y) + \norm{ \nabla_x u_t^*}_{\mathcal{C}^1_2} \norm{ \rho_t}_{(\mathcal{C}^1_3)^*} \bigr), $$
for some constant $C_b>0$ depending only on $b$.

\textit{Step 2. Convergence of $k_t^n$ to $k_t$ in $\mathcal{C}^0_3(A)$.}
Now, for $t \in [t_0^n \vee t_0,T]$ and $a \in A$ we have, by assumptions on $b$,
\begin{equation}
\label{eq:fnt-ft:a}
\begin{split}
    |k_t^n(a) - k_t(a) | &= \frac{1}{\epsilon} \biggl| \int_{\R^{d_1} \times \R^{d_2}} b(x,a) \cdot d\Bigl\{ \bigl( \nabla_x v_t^n - \nabla_x v_t \bigr) \gamma_t^n + \nabla_x v_t\bigl( \gamma_t^n -\gamma^*_t \bigr) \Bigr \}(x,y)
    \\
    &\hspace{45pt} + \Bigl \langle    b(\cdot,a) \cdot \bigl( \nabla_x u_t^{*,n} - \nabla_x u^*_t \bigr) ; \rho_t^n \Bigr \rangle + \Bigl \langle b(\cdot,a) \cdot \nabla_x u^*_t ; \rho_t^n - \rho_t \Bigr \rangle  \biggr| \\
    & \leq C (1+|a|) \norm{ \nabla_x v_t^n - \nabla_x v_t}_{\mathcal{C}^0_2} \int_{\R^{d_1} \times \R^{d_2}} (|x|^2 + |y|^2) d\gamma_t^n(x,y) \\
    &+ C(1+|a|^2) \norm{ \nabla_x v_t}_{\mathcal{C}^1_{2,1}} \norm{ \gamma_t^n - \gamma_t^*}_{(\mathcal{C}^1_{2,1})^*}
    \\
    &  + C (1+|a|^2) \norm{\nabla_x u_t^{*,n} - \nabla_x u_t^*}_{\mathcal{C}^1_3} \norm{ \rho_t^n}_{(\mathcal{C}^1_3)^*}  +  C(1+|a|^3) \norm{ \nabla_x u_t^*}_{\mathcal{C}^2_2} \norm{ \rho_t^n - \rho_t}_{(\mathcal{C}^2_2)^*}, 
\end{split}
\end{equation}
for some $C>0$ independent from $t \in [t_0^n \vee t_0,T]$, $a \in A$ and $n \geq 1$. Recalling the regularity estimates from Proposition \ref{prop:SolnVt19/02} and Proposition \ref{prop:SolutionBackxard16Sept} for $\bd{v}$ and $\bd{u}^*$ respectively as well as Lemma \ref{lem:compactnessContinuityEquation24/06} for $\bd{\gamma}^n$ and  Proposition \ref{prop:preliminaryworkprop1-28/02} for  $\bd{\rho}^n$, we have the uniform bounds
$$ \sup_{t \in [t_0^n \vee t_0],T} \Bigl \{ \sup_{ n \in \mathbb{N}} \int_{\R^{d_1} \times \R^{d_2}} (|x|^2 + |y|^2)d\gamma_t^n(x,y) + \norm{ \nabla_x v_t}_{\mathcal{C}^1_{2,1}} + \sup_{n \in \mathbb{N} } \norm{ \rho_t^n}_{(\mathcal{C}^1_3)^*} + \norm{\nabla_x u_t^*}_{\mathcal{C}^2_2} \Bigr\} < +\infty.$$
And then, after noticing that $ \gamma_t^n- \gamma_t^* = \lambda_n \rho_t^n + \gamma_t^{*,n} - \gamma_t^*$ we have, by item $ii)$ of Property $(\mathcal{Q}_0)$ for the convergence of $(\bd{\gamma}^{*,n})_{n \in \mathbb{N}}$ to $\bd{\gamma}^*$ and $(\bd{u}^{*,n})_{n \in \mathbb{N}}$ to $\bd{u}^{*}$ and Proposition \ref{prop:preliminaryworkprop1-28/02} for the convergence of $(\bd{v}^n)_{n \in \mathbb{N}}$ to $\bd{v}$ and $(\bd{\rho}^n)_{n \in \mathbb{N}}$ to $\bd{\rho}$,
$$ \lim_{ n \rightarrow +\infty} \sup_{t \in [t_0^n\vee t_0,T] } \Bigl \{ \norm{ \nabla_x v_t^n - \nabla_x v_t}_{\mathcal{C}^0_2} + \norm{ \gamma_t^n - \gamma_t^*}_{(\mathcal{C}^1_{2,1})^*} + \norm{ \nabla_x u_t^{*,n} - \nabla_x u_t^*}_{\mathcal{C}^1_3} + \norm{ \rho_t^n - \rho_t}_{(\mathcal{C}^2_2)^*} \Bigr \} =0.$$
Combined together this leads to 
\begin{equation} 
\lim_{n \rightarrow +\infty} \sup_{t \in [t_0^n \vee t_0,T]} \norm{ k_t^n -k_t}_{\mathcal{C}^0_3(A)}=0. 
\label{eq:cvgceC03(A)03/03}
\end{equation}
Given the integrability of $\nu_t^{*,n}$ from Lemma \ref{lem:Q2-28/02} we easily deduce that 
$$ \lim_{n \rightarrow +\infty} \sup_{t \in [t_0^n \vee t_0,T]} \int_A |k_t^n(a) - k_t(a)|d\nu_t^{*,n}(a) =0. $$

\textit{Step 3. Higher-order convergence when the limit is $(0,0,0)$.}
We now assume that $(\bd{\eta},\bd{\rho}, \bd{v}) = (0,0,0)$. In particular, this implies that $\bd{k} = 0$. 
Using the assumptions on $b$ we can estimate $|\nabla_ak^n_t(a)|$ as follows. First we use the regularity assumption on $b$ to find  $C >0$ independent from $t \in [t_0^n,T]$ and $n \in \mathbb{N}$ such that 
$$ |\nabla_ak_t^n(a)| \leq C \Bigl( (1+|a|) \norm{\nabla_x v_t^n}_{\mathcal{C}^0_2} \int_{\R^{d_1} \times \R^{d_2}} (|x|^2 + |y|^2)^{3/2} d\gamma_t^n(x,y) 
+ (1+|a|^3) \norm{ \nabla_x u_t^{*,n}}_{\mathcal{C}^2_1} \norm{\rho_t^n}_{(\mathcal{C}^2_2)^*} \Bigr). $$
The term involving $\bd{\gamma}^n$ is bounded independently from $t \in [t_0^n,T]$ and $n \in \mathbb{N}$ thanks to Lemma \ref{lem:timetogotoGallia02/03} and Lemma \ref{lem:compactnessContinuityEquation24/06}.
By Proposition \ref{prop:preliminaryworkprop1-28/02} and recalling that we assume $(\bd{\rho}, \bd{v}) = (0,0)$ we get, 
$$ \lim_{n \rightarrow +\infty} \Bigl \{ \sup_{t \in [t_0^n,T]} \norm{ \nabla_x v_t^n}_{\mathcal{C}^0_2} + \sup_{t \in [t_0^n,T]} \norm{ \rho_t^n}_{(\mathcal{C}^2_2)^*} \Bigr\} =0. $$
Combined with \eqref{eq:cvgceC03(A)03/03} we get 
\begin{equation} 
\lim_{n \rightarrow +\infty} \sup_{t \in [t_0^n \vee t_0,T]} \norm{ k_t^n}_{\mathcal{C}^1_3(A)} =0.   
\label{eq:C13(A)}
\end{equation}
It remains to prove that 
\begin{equation} 
\lim_{n \rightarrow +\infty} \int_{t_0^n}^T \int_A |\nabla_a k_t^n(a)|^2 d\nu_t^n(a) = 0.
\label{eq:03/03-11:17}
\end{equation}
To this end we use the definition of $\lambda_n$ to obtain 
\begin{equation} 
\int_{t_0^n}^T \int_A |\nabla_a k_t^n(a) |^2 d\nu_t^n(a) dt = \int_{t_0^n}^T \int_A |\nabla_a k_t^n(a) |^2 d\nu_t^{*,n}(a) dt + \lambda_n \int_{t_0^n}^T \int_A |\nabla_a k_t^n(a) |^2 d\eta_t^n(a) dt 
\label{eq:26/02-14:39-03/03}
\end{equation}
On the one hand, using the integrability of $\nu_t^{*,n}$ from Lemma \ref{lem:Q2-28/02} and \eqref{eq:C13(A)} we get
\begin{equation} 
\lim_{n \rightarrow +\infty} \int_{t_0^n}^T \int_A |\nabla_a k_t^n(a)|^2 d\nu_t^{*,n}(a) dt =0. 
\label{eq:Valdotextedme-03/03}
\end{equation}
On the other hand, using the bound on $\bd{k}^n$ in $\mathcal{C}^1_2(A)$ from Lemma \ref{lem:timetogotoGallia02/03} we find $C>0$ independent from $n \in \mathbb{N}$ such that 
$$ \Bigl| \int_{t_0^n}^T \int_A |\nabla_a k_t^n(a) |^2 d\eta_t^n(a) \Bigr| \leq C \int_{t_0^n}^T \int_A (1+|a|^4) d|\eta_t^n|(a) dt, $$
and the right-hand side is bounded from above independently from $n \in \mathbb{N}$ by Lemma \ref{lem:timetogotoGallia02/03}. Combined with \eqref{eq:Valdotextedme-03/03} in \eqref{eq:26/02-14:39-03/03} and recalling that $\lambda_n \rightarrow 0$ as $n \rightarrow +\infty$ this shows that  \eqref{eq:03/03-11:17} holds.
\end{proof}

\begin{proof}[Proof of Proposition \ref{prop:cvgcesumofIs03/03}]
The proof follows from Lemmas \ref{lem:sumofIs03/03}, \ref{lem:I1I2I5I6-03/03} and \ref{lem:I3I4-03/03}  below.

\end{proof}

\begin{lem}
In the setting of Proposition \ref{prop:cvgcesumofIs03/03} we have 
    \begin{align} 
\notag \limsup_{ n \rightarrow +\infty} \int_{t_0^n \vee t_0}^T \Bigl( \int_A \Bigl|  \frac{\Gamma_t[\bd{\nu}^n](a) - \nu_t^{*,n}(a)}{\lambda_n} &+ \nu_t^*(a) \bigl( k_t(a) - \int_A k_t(a') d\nu_t^*(a') \bigr) \Bigr|da \Bigr)^2 dt \\
&\leq  \limsup_{n \rightarrow +\infty} \int_{t_0^n \vee t_0}^T \Bigl(\sum_{i = 1}^6 \int_A  | I_t^{i,n}(a) |da \Bigr)^2 dt . 
\end{align}
where $\bd{I}^{i,n} := (I_t^{i,n})_{t \in [t_0^n \vee t_0,T]} $ is defined, for $1 \leq i\leq 6$ by
\begin{equation}
\label{eq:notation:lem:step:3:1:b}
\begin{split} 
    & I_t^{1,n}(a) :=  \nu^*_t(a) \biggl( \int_A k_t(a')d \nu^*_t(a') - k_t(a) \biggr) -  \nu^{*,n}_t(a) \biggl( \int_A k_t(a')d \nu^{*,n}_t(a') - k_t(a) \biggr), 
    \\
    & I_t^{2,n}(a) := \biggl[ \frac{1}{\lambda_n} \Bigl(1 - e^{-\lambda_nk_t^n(a)}
    \Bigr) -k_t^n(a) \biggr]\nu^{*,n}_t(a),  
    \\
    & I_t^{3,n}(a) := \Bigl( k_t^n(a) -k_t(a) \Bigr) \nu_t^{*,n}(a),
    \\
    & I_t^{4,n}(a) := \biggl[ \int_A \Bigl(k_t(a') - k_t^n(a') \Bigr)d\nu^{*,n}_t(a') \biggr] \nu^{*,n}_t(a),
     \\
    & I_t^{5,n}(a) := \biggl[ \frac{e^{-\lambda_n k_t^n(a)}}{\lambda_n} \biggl(1 - \frac{1}{\int_A e^{-\lambda_n k_t^n(a') }d\nu^{*,n}_t(a') }  +\lambda_n \int_A k_t^n(a')d\nu^{*,n}_t(a')  \biggr) \biggr] \nu^{*,n}_t(a), 
    \\
    &I_t^{6,n}(a) := \biggl[ \Bigl( 1- e^{-\lambda_n k_t^n(a)} \Bigr) \int_A k_t^n(a')d\nu^{*,n}_t(a') \biggr] \nu^{*,n}_t(a).
\end{split}
\end{equation}
\label{lem:sumofIs03/03}
\end{lem}

\begin{proof}
    Let $t \in [t_0^n\vee t_0,T]$ and $a \in A$. Recalling the expression \eqref{eq:rewritinggammanutn28/02} from Lemma \ref{lem:rewriting02/02} this gives
\begin{equation*}
\begin{split}
&\frac{1}{\lambda_n} \bigl( \Gamma[\bd\nu^n]_t(a) - \nu_t^{*,n}(a) \bigr) +\nu_t^*(a)(k_t(a) - c_t)
\\
&= -\frac{\nu_t^{*,n}(a)}{\lambda_n} \biggl(1-\frac{e^{-\lambda_n k_t^n(a)}}{\int_{A} e^{-\lambda_n k_t^n(a')}d\nu^{*,n}_t(a')} \biggr) - \biggl( \int_A k_t(a') d\nu^*_t(a') - k_t(a) \biggr) \nu^*_t(a).
\end{split}
\end{equation*}
It is then standard algebra to check that 
the second line above is equal to the sum $-\sum_{i=1}^6 I_t^{i,n}(a)$, i.e. 
\begin{equation}
\label{eq:sumofIs}   
\begin{split}
\frac{1}{\lambda_n} \bigl( \Gamma[\bd\nu^n]_t(a) -\nu_t^{*,n}(a)\bigr) +\nu_t^*(a)(k_t(a) - c_t) =- \sum_{i=1}^6 I_t^{i,n}(a). 
\end{split}
\end{equation}
and the result easily follows.
\end{proof}

\begin{lem}
\label{lem:I1I2I5I6-03/03}
In the setting of Proposition \ref{prop:cvgcesumofIs03/03}, we have the following convergence for $\bd{I}^{1,n}, \bd{I}^{2,n}, \bd{I}^{5,n}$ and $\bd{I}^{6,n}$
    $$ \lim_{n \rightarrow +\infty} \int_{t_0^n \vee t_0}^T \biggl( \int_A 
\Bigl[ \bigl| I_t^{1,n} (a) \bigr|+
\bigl| I_t^{2,n} (a) \bigr|+
\bigl| I_t^{5,n} (a) \bigr| + \bigl| I_t^{6,n} (a) \bigr|
\Bigr]
da \biggr)^2 dt =0.$$  
\end{lem}

\begin{proof}
 We first address the convergence of $({\boldsymbol I}^{1,n})_{n \in {\mathbb N}}$. Lemma \ref{lem:convergenceifuniqueness}
guarantees that 
$(\boldsymbol \nu^{*,n})_{n  \in {\mathbb N}}$ converges strongly toward $\boldsymbol \nu^*$, from which we easily deduce that 
$$ \lim_{n \rightarrow +\infty} \int_{t_0^n \vee t_0}^T \biggl( \int_A \bigl| I_t^{1,n} (a) \bigr| da \biggr)^2 dt =0.$$

We now address
$({\boldsymbol I}^{2,n})_{n \in {\mathbb N}}$
and 
$({\boldsymbol I}^{5,n})_{n \in {\mathbb N}}$. 
We first notice from Lemma \ref{lem:timetogotoGallia02/03} and Proposition \ref{prop:convergencektn03/03} that there is 
$C$ such that, for any 
$n \in {\mathbb N}$, 
$t \in [t_0^n,T]$
and $a \in A$, 
$\vert k^n_t(a) \vert 
\leq C(1+ \vert a\vert^2)$, and 
for any $t \in [t_0,T]$ and $a \in A$, $\vert k_t(a) \vert 
\leq C(1+ \vert a\vert^2)$.

For ${\boldsymbol I}^{2,n}$, Taylor formula yields
\begin{equation}
\label{eq:I2:tend:to:0-03/03}
\begin{split}
\Bigl\vert \frac{1}{\lambda_n}\Bigl(1 - e^{-\lambda_n k_t^n(a)}\Bigr)
- k_t^n(a) \Bigr\vert 
\nu^{*,n}_t(a)
&\leq \Bigl(\tfrac12 \lambda_n 
\vert k_t^n(a) \vert^2 
\sup_{\theta \in [0,1]}
e^{- \theta \lambda_n k_t^n(a)}
\Bigr) \nu^{*,n}_t(a)
\\
&\leq \lambda_n 
\vert k_t^n(a) \vert^2 
\Bigl( 1 + 
e^{-  \lambda_n k_t^n(a)}
\Bigr) \nu^{*,n}_t(a). 
\end{split}
\end{equation}
Using Proposition 
\ref{prop:regularityfromOC},
there is no difficulty in integrating with respect to $a$, squaring and then integrating in 
time. Since $(\lambda_n)_{n \in {\mathbb N}}$ tends
to $0$, we deduce that 
$$ \lim_{n \rightarrow +\infty} \int_{t_0^n \vee t_0}^T \biggl( \int_A \bigl| I_t^{2,n} (a) \bigr| da \biggr)^2 dt =0.$$

We now address ${\boldsymbol I}^{5,n}$. 
First, 
by integrating 
\eqref{eq:I2:tend:to:0-03/03}
with respect to $a$ and then multiplying by $\lambda_n$, we notice that 
\begin{equation*}
\biggl\vert 
\int_{A} e^{-\lambda_n k_t^n(a')}
d \nu^{*,n}_t(a')
-1
+ \lambda_n \int_A k_t^n(a') 
d \nu^{*,n}_t(a')
\biggr\vert 
\leq C \lambda_n^2 
\int_{A} \bigl( 1 +\vert a \vert^4 \bigr) 
e^{\lambda_n \vert a \vert^2}
d\nu_t^{*,n}(a) \leq C \lambda_n^2,
\end{equation*}
for a constant $C$ independent of 
$n \in {\mathbb N}$ and $t
\in [t_0^n \vee t_0,T]$. Then, 
for a new value of $C$,
\begin{equation*}
\biggl\vert 
\biggl( \int_{A} e^{-\lambda_n k_t^n(a')}
d \nu^{*,n}_t(a')
\biggr)^{-1}
-1
- \lambda_n \int_A k_t^n(a') 
d \nu^{*,n}_t(a')
\biggr\vert 
\leq C \lambda_n^2.
\end{equation*}
Dividing by 
$\lambda_n$, 
we get
$$  \Bigl| \frac{1}{\lambda_n} \Bigl(1 - \frac{1}{\int_Ae^{-\lambda_n k_t^n(a') }d\nu^{*,n}_t(a') } \Bigr)  + \int_A k_t^n(a')d\nu^{*,n}_t(a')   \Bigr| \leq C \lambda_n .$$
Multiplying 
by $\exp(-\lambda_n k_t(a)) \nu^*_t(a)$, integrating in 
$a$, squaring and then integrating time, 
we 
proceed as above to 
deduce that 
$$ \lim_{n \rightarrow +\infty} \int_{t_0^n \vee t_0}^T \biggl( \int_A \bigl| I_t^{5,n} (a) \bigr| da \biggr)^2 dt =0.$$
It thus remains to 
address ${\boldsymbol I}^{6,n}
=(I^{6,n}_t)_{t_0^n \vee t_0 \le t \le T}$. Recall that we can find a constant $C$ such that, for any $n \in {\mathbb N}$, $\sup_{t \in [t_0^n \vee t_0,T]}
\vert k^n_t(a) \vert \leq C (1+\vert a \vert^2)$ and then using  the bound 
$|1-\exp(-\lambda_n k^n_t(a))| \leq \lambda_n \vert k^n_t(a)\vert e^{\lambda_n |k_t^n(a)|}$ together with the fact that $\lambda_n$ tends to $0$ and the integrability of $\nu_t^{*,n}$ from Lemma \ref{lem:Q2-28/02}, there is no difficulty in proving that 
$$ \lim_{n \rightarrow +\infty} \int_{t_0^n \vee t_0}^T \biggl( \int_A \bigl| I_t^{6,n} (a) \bigr| da \biggr)^2 dt =0.$$
This completes the proof. 
\end{proof}

\begin{lem}
\label{lem:I3I4-03/03}
In the setting of Proposition \ref{prop:cvgcesumofIs03/03}, we have the following convergence for $\bd{I}^{3,n}, \bd{I}^{4,n}$
$$ \lim_{n \rightarrow +\infty} \int_{t_0^n \vee t_0}^T \biggl( \int_A 
\Bigl[ \bigl| I_t^{3,n} (a) \bigr|+
\bigl| I_t^{4,n} (a) \bigr|
\Bigr]
da \biggr)^2 dt =0.$$  
\end{lem}

\begin{proof}

The convergence of the term involving $\bd{I}^{3,n}$ is a direct consequence of Proposition \ref{prop:convergencektn03/03}. Since, for all $t \in [t_0^n,T]$ and all $ n \in \mathbb{N}$ we have
$$ \int_A |I_t^{4,n}(a) |da = \bigl| \int_A I_t^{3,n}(a) da \bigr| \leq \int_A |I_t^{3,n}(a)|da ,$$
the result for $\bd{I}^{4,n} = (I_t^{4,n})_{t \in [t_0^n \vee t_0,T]}$ directly follows.

\end{proof}

\appendix

\section{Some Auxiliary Statements}

\label{sec:AppendixA}

\subsection{Lagrangian approach to the continuity and transport equations.}

\label{subse:A:1}
The goal of this section is to provide various existence, stability and regularity results for the
continuity and transport equations 
\eqref{eq:Continuitygammanu} and 
\eqref{eq:adjoint:equation}. This is mostly done through the (Lagrangian) representation 
based on the ODE \eqref{eq:ODE:intro}, which we recall below: 
\begin{equation}
\label{eq:ODE:appendix}
\dot{X}_s^{t,x} = b(X_s^{t,x}, \nu_s), \quad X_t^{t,x} = x, \quad s \in [t,T].
\end{equation}
Above, $b$ satisfies $(i)$ in \assreg.
The input ${\boldsymbol \nu}$ is taken in the class $\mathcal{D}(t_0)$ defined in
Definition \ref{defn:D(t_0)}.

\subsubsection{Analysis of the ODE}

We start with the analysis of the ODE \eqref{eq:ODE:appendix}. For some initial condition $(t,x) \in [t_0,T] \times \R^{d_1}$ and some control ${\boldsymbol \nu} \in \mathcal{D}(t_0)$, we say that $(X_s^{t,x})_{t \leq s \leq T}$ is solution to the ODE 
\eqref{eq:ODE:appendix} if $s \in [t,T] \mapsto X_s^{t,x} \in {\mathbb R}^{d_1}$ is continuous and 
$$\forall s \in [t,T], \quad X_s^{t,x} = x +\int_t^s b(X_u^{t,x}, \nu_u) du.$$

\begin{prop}
\label{prop:ODE16Sept}
    Let $t_0 \in [0,T]$ and ${\boldsymbol \nu} \in {\mathcal D}(t_0)$. 
    Then, for all $(t,x) \in [t_0,T] \times \R^{d_1}$, there exists a unique solution to the ODE \eqref{eq:ODE:appendix}. For all $s \in [t,T]$, the map $x \in {\mathbb R}^{d_1} \mapsto X_s^{t,x}$ is three times differentiable and there exists a non-decreasing function 
    $\Lambda : {\mathbb R}_+ \rightarrow {\mathbb R}_+$, independent of $t_0$ and 
    ${\boldsymbol \nu}$, such that
\begin{equation} 
\sup_{t \in [t_0,T]} \Biggl \{ \sup_{s \in [t,T]} \norm{ X_s^{t,\cdot} }_{\mathcal{C}^3_{1,b}} + \sup_{t_0 \leq s_1 < s_2 \leq T} \frac{\| X_{s_2}^{t, \cdot} - X_{s_1}^{t,\cdot} \|_{\mathcal{C}^1_b}}{\sqrt{s_2-s_1}} \Biggr\} \leq \Lambda \bigl( \norm{ \bd \nu}_{\mathcal{D}(t_0)} \bigr).  
\label{eq:estimateODEAppendix19Aout}
\end{equation}

\color{black}
Moreover, if the coefficients $(t,x) \mapsto (b(x,\nu_t), \dots, \nabla_{x}^l b(x,\nu_t) )$, for $l \in \{2,3\}$, are jointly continuous and bounded, then the map $(t,s,x) \mapsto X_s^{t,x}$ (with the solution being extended to $s \in (-\infty,t)$
by the solving ODE 
\eqref{eq:ODE:appendix} backwards 
for $s<t$) 
has continuous and bounded derivatives of the form $\partial^m_t \partial_x^k X_s^{t,x}$, with $m=0,1$, $k=0,\dots,l$ and $0 <m+k \leq l$. 
\end{prop}

\begin{rmk}
    Notice that, for any $l \in \{ 0,1,2,3 \}$ it follows from the assumptions on $b$ and its derivatives that a sufficient condition for $(t,x) \mapsto \nabla_x^{l}b(x,\nu_t)$ to be bounded is to require $t \mapsto \int_A (1+|a|^{l+1})d |\nu_t|(a)$ to be bounded while a sufficient condition for $(t,x) \mapsto \nabla_x^lb(x,\nu_t)$ to be jointly continuous is to require $t \mapsto \nu_t \in \mathcal{M}_{1+|a|^{l+1}}$ to be continuous.
\label{rmk:regularvectorfield27/01}
\end{rmk}

\begin{proof} { \ }

\textit{Step 1.}
We first address the solvability of 
\eqref{eq:ODE:appendix}. 
By (i) in \assreg, 
we notice that, for Lebesgue almost every $t$ in $[t_0,T]$, 
the vector field $ x \mapsto b(x, \nu_t)$ belongs to $\mathcal{C}^3(\R^d)$ with  
\begin{equation}
\label{eq:prop:A.1:step:1:1}
\int_{t_0}^T \norm{b(\cdot,\nu_t)}_{\mathcal{C}^3_b} dt \leq C \int_{t_0}^T\int_A (1+|a|^4) d|\nu_t|(a)dt < +\infty.
\end{equation}
Therefore, we 
can apply Cauchy-Lipschitz theorem and deduce that 
the ODE 
\eqref{eq:ODE:appendix} is uniquely solvable. 
\vskip 4pt

\textit{Step 2.}
We now turn to the proof of the space regularity estimate in 
\eqref{eq:estimateODEAppendix19Aout}. 
We start with the following bound (for $s \in [t,T]$):
\begin{equation*}
    |X_s^{t,x}| \leq |x| + \biggl| \int_t^s b(X_u^{t,x},\nu_u)du \biggr| \\
    \leq |x| + C \int_s^t \int_A (1+|a|) d|\nu_u|(a)du,
\end{equation*}
and the right-hand side can be easily bounded in terms of the fourth moment of 
$|{\boldsymbol \nu}|$. This gives the bound for the first term in 
\eqref{eq:estimateODEAppendix19Aout} when the ${\mathcal C}^3_{1,b}$-norm is replaced 
the ${\mathcal C}^0_{1}$-norm.

Moreover, by 
\eqref{eq:prop:A.1:step:1:1}, it is a standard fact that, for each $s \in [t,T]$, the mapping $x \mapsto X_s^{t,x}$
is three times differentiable. 
Differentiating the flow, we have
\begin{equation} 
\nabla_x X_s^{t,x} = I_d +\int_t^s \nabla_x b(X_u^{t,x},\nu_u) \nabla_x X_u^{t,x} du \quad s \in [t,T],
\label{eq:differentiatingtheflow26/01}
\end{equation}
and the growth assumption on $\nabla_x b$ gives
$$ |\nabla_x X_s^{t,x} | \leq C + C \int_t^s
\biggl( |\nabla_x X_u^{t,x}| \int_A(1+|a|^2) d|\nu|_u(a) \biggr) du. $$
By Grönwall's lemma, 
the left-hand side can be estimated in terms of the fourth moment of 
$|{\boldsymbol \nu}|$. 

We can proceed in the same way for the higher-order derivatives.
Taking $l\in \{2,3\}$
and assuming that we have a bound for 
$\sup_{s \in [t,T]} \| \nabla_x X_s^{t,\cdot} \|_{{\mathcal C}_b^{l-1}}$, 
we can find a universal constant $c>0$ such that
\begin{align*}
|\nabla_{x}^l X_s^{t,x} | &\leq \int_{t}^s |\nabla_x b(X_u^{t,x},\nu_u)| |\nabla_{x}^l X_u^{t,x} |du 
 + 
c \sum_{k=2}^{l}
\sum_{\underset{i_1 + \cdots + i_k =l}{1 \leq i_1 \leq \cdots \leq i_k }}
\int_s^t |\nabla_{x}^{k} b(X_u^{t,x},\nu_u)| 
\prod_{j=1}^k |\nabla_x^{i_j} X_u^{t,x} | du 
\\
    &\leq C \int_t^s \biggl( |\nabla_{x}^l X_u^{t,x} | \int_A (1+|a|^2) d|\nu_u|(a) 
    \biggr) du 
    \\
&\quad    + C 
     \sup_{s \in [t,T]}     \norm{ \nabla_x X_s^{t, \cdot}}_{\mathcal{C}^{l-1}_b}  
    \sum_{k=2}^l     
    \int_s^t \int_A \bigl(1+|a|^{k+1}\bigr)d|\nu_u|(a)du,
\end{align*}
for a constant $C$ depending on the parameters in (i) in \assreg. 
Using Gr\"onwall's lemma together with 
the fact that $k+1$ above is less than 4,
we derive the first bound in 
\eqref{eq:estimateODEAppendix19Aout}.
\vskip 4pt

\textit{Step 3.} We 
now prove the time regularity estimate in  \eqref{eq:estimateODEAppendix19Aout}.
It follows from the following inequality, which is itself a consequence of 
\eqref{eq:boundfromPinsker16Sept}:
\begin{equation*}
\begin{split}
|X_{s_2}^{t,x} - X_{s_1}^{t,x}| &\leq \int_{s_1}^{s_2} | b( X_s^{t,x},\nu_s)|ds 
\\
&\leq C \int_{s_1}^{s_2} \int_A (1+|a|)d|\nu_s|(a) ds \leq C \sqrt{|s_2 -s_1|}
\norm{\bd{\nu}}_{\mathcal{D}(t_0)},
\end{split}
\end{equation*}
which holds true for all $s_1,s_2 \in [t,T]$ with $s_1 < s_2$. Similarly, recalling \eqref{eq:differentiatingtheflow26/01} we get 
\begin{align*}
   & \bigl| \nabla_x X_{s_2}^{t,x} - \nabla_x X_{s_1}^{t,x} \bigr| \leq  \int_{s_1}^{s_2} \bigl| \nabla_x b(X_u^{t,x},\nu_u) \nabla_x X_u^{t,x} \bigr| du \\
    &\leq C \sup_{u \in [s_1,s_2]} \norm{ X_u^{t,\cdot}}_{\mathcal{C}^1_{1,b}} \int_{s_1}^{s_2} (1+|a|^2) d|\nu_u|(a)du \leq C \sup_{u \in [s_1,s_2]} \norm{ X_u^{t,\cdot}}_{\mathcal{C}^1_{1,b}} \sqrt{s_2-s_1} \norm{\bd{\nu}}_{\mathcal{D}(t_0)}.
\end{align*}
Combined with \textit{Step 2}, this completes the proof of 
\eqref{eq:estimateODEAppendix19Aout}. 
\vskip 4pt

\textit{Step 4.}
As for the last claim in the statement it follows from 
\cite[Part V, Chapter 4, Theorem 4.1 \& Corollary 4.1,]{Hartman}.
\end{proof}

\subsubsection{Regularizing the vector field}
\label{subsec:regularizingvectorfield}

The analysis of the continuity
and transport equations
\eqref{eq:Continuitygammanu} 
and \eqref{eq:adjoint:equation} relies on the following regularization argument.

We start by extending $\bd{\nu}$ to $\R$ by setting
$$ \bar{\nu}_t := \left \{ 
\begin{array}{ll}
\nu_t & \mbox{ if } t \in [t_0,T], \\
0 & \mbox{ otherwise. }
\end{array}
\right. $$
The collection $(\bar \nu_t)_{t \in {\mathbb R}}$ is denoted by $\bar{\boldsymbol{\nu}}.$

Then we take a smooth, even, density $\rho : \R \rightarrow \R^+$ and let for all 
$n \in \mathbb{N}^*$, $\rho^n := n \rho(n \cdot)$. 
For a fixed $n \in {\mathbb N}^*$,
we define
\begin{equation}
    \bd{\nu}^n := \rho^n * \bar{\bd{\nu}}
    \label{eq:defnnun23/12}
\end{equation}
over $\R$. Equivalently, we have for all $t \in \R$ and all Borel subset $B$ of $A$,
$$ \nu_t^n(B) = \int_{\R} \rho^n(t-s) \bar{\nu}_s(B) ds = \int_{t_0}^T \rho^n(t-s) \nu_s(B)ds.$$
By restricting $\bd{\nu}^n$ to $[t_0,T]$, we can see it as an element of $\mathcal{D}(t_0)$. Then $\bd{\nu}^n$ satisfies the following results:

\begin{lem}
    \label{lem:propertiesnun11/01}
Take $\bd\nu \in \mathcal{D}(t_0)$ and consider
$(\bd{\nu}^n)_{n \geq 1}$ as given by
\eqref{eq:defnnun23/12}. Then we have
\begin{equation} 
\sup_{ n \in \mathbb{N}^* } \norm{ \bd{\nu}^n}_{\mathcal{D}(t_0)} \leq \norm{\bd{\nu}}_{\mathcal{D}(t_0)}, 
\label{eq:estimatenunDt0}
\end{equation}
as well as the uniform integrability condition
\begin{equation} 
\lim_{K \rightarrow +\infty} \sup_{ n \in \mathbb{N}^*} \int_{t_0}^T \int_A (1+|a|^4) 
\mathbf1_{\{|a| \geq K\}} d|\nu_t^{n}|(a)dt = 0.
\label{eq:uniformintegrability23/12}
\end{equation}
\end{lem}

\begin{proof}
We fix some $n \in {\mathbb N}^*$. We observe that, for almost every  $t \in [t_0,T]$ and all Borel subset $B$ of $A$,
$$ |\nu_t^n|(B) \leq \int_{t_0}^T \rho^n(t-s) \bigl[ |\nu_s|(B) \bigr]
ds.$$
The above inequality can be easily extended from indicator functions of Borel subsets to positive measurable functions. Therefore, for all $t_1 < t_2 \in [t_0,T]$, $k \in \mathbb{N}$ and $K>0$,
we get, by a simple change of variable,
\begin{align*}
    \int_{t_1}^{t_2} \int_A (1+|a|^k) \mathbf{1}_{\{ |a| \geq K
    \} 
    } d|\nu_t^n|(a)dt &\leq  \int_{t_1}^{t_2} \int_{t_0}^T \rho^n(t-s) \int_A (1+|a|^k)
    \mathbf{1}_{\{ |a| \geq K
    \}}    
    d|\nu_s|(a)  ds dt \\
    &\leq \int_{t_1}^{t_2} \int_{t_0-t}^{T -t} \rho^n(u) \int_A (1+|a|^k) \mathbf{1}_{\{  |a| \geq K \}}d|\nu_{u+t}|(a)dudt, 
\end{align*}
and then, we deduce the trivial bound
\begin{align*}
     \int_{t_1}^{t_2} \int_A (1+|a|^k) \mathbf{1}_{\{ |a| \geq K\}} d|\nu_t^n|(a)dt &\leq \int_{t_1}^{t_2} \int_{\R} \rho^n(u) \int_A (1+|a|^k)\mathbf{1}_{\{ |a| \geq K\} } d|\bar{\nu}_{u+t}|(a)dudt \\
    &\leq \int_{\R} \rho^n(u) \int_{t_1}^{t_2} \int_A  (1+|a|^k) \mathbf{1}_{\{ |a| \geq K\} } d|\bar{\nu}_{u+t}|(a)dt du 
\end{align*}
from which it comes, since $\int_{\R} \rho^n(u)du =1$,
\begin{align*}
     \int_{t_1}^{t_2} \int_A (1+|a|^k) \mathbf{1}_{\{ |a| \geq K\}} d|\nu_t^n|(a)dt &\leq \sup_{u \in \R} \int_{t_1}^{t_2} \int_A  (1+|a|^k) \mathbf{1}_{\{ |a| \geq K\} } d|\bar{\nu}_{u+t}|(a)dt.
\end{align*}
Now, for $k=2$ and $K=0$ we have
$$ \sup_{u \in \R} \int_{t_1}^{t_2} \int_A (1+|a|^2)d|\bar{\nu}_{t+u}|(a) du \leq \sqrt{t_2-t_1} \sup_{s_1 < s_2 \in [t_0,T]} \frac{1}{\sqrt{s_2-s_1}} \int_{s_1}^{s_2} \int_A (1+|a|^2) d|\nu_t|(a)dt, $$
where we used that $\bar{\boldsymbol \nu} =0$ outside $[t_0,T] \times A$ and, for $k=4$,
$$ \sup_{u \in \R} \int_{t_0}^T \int_A (1+|a|^4) \mathbf{1}_{\{ |a| \geq K\} } d|\bar{\nu}_{t+u}|(a)dt = \int_{t_0}^T \int_A (1+|a|^4) \mathbf{1}_{\{ |a| \geq K\} } d|\nu_t|(a)dt. $$
Combined together we get both \eqref{eq:estimatenunDt0} and \eqref{eq:uniformintegrability23/12}.
\end{proof}

We use the above statement to derive the following properties on 
the collection of vector fields 
$((t,x) \mapsto b(x, \nu_t^n))_{n \geq 1}$:
\begin{lem}
   Take $\bd\nu \in \mathcal{D}(t_0)$ and consider
$(\bd{\nu}^n)_{n \geq 1}$ as given by
\eqref{eq:defnnun23/12}. Then, for each $n \in {\mathbb N}^*$, the vector field $(t,x) \mapsto b(x,\nu_t^n)$ is smooth in the sense that all the derivatives $ \partial_t^m \nabla_x^k [(t,x) \mapsto b(x, \nu_t^n)]$ with $m\geq 0$ and $0 \leq k \leq 3 $ are continuous and bounded. We also  have the estimate
\begin{equation} 
\sup_{n \in \mathbb{N}^*} \biggl \{ \int_{t_0}^T \norm{ b(\cdot , \nu_t^n)}_{\mathcal{C}^3_b} dt + \sup_{t_1 < t_2 \in [t_0,T]} \frac{1}{\sqrt{t_2-t_1}} \int_{t_1}^{t_2} \norm{b(\cdot,\nu_t^n)}_{\mathcal{C}^1_b}dt \biggr \} \leq C \norm{\bd{\nu}}_{\mathcal{D}(t_0)}. 
\label{eq:uniformboundnun23-12}
\end{equation}
Moreover, the vector fields $((t,x) \mapsto b(x,\nu_t^n))_{n \geq 1}$ approximate $(t,x) \mapsto b(x,\nu_t)$ in the following sense:
\begin{equation} 
\lim_{n \rightarrow +\infty} \int_{t_0}^T \norm{ b(\cdot, \nu_t^n) - b(\cdot, \nu_t) }_{\mathcal{C}^3_{1}} dt =0.
\label{eq:approxbnun23/12}
\end{equation}
\label{lem:newregularizationb23/12}
\end{lem}

\begin{proof}
\textit{Step 1.} For a fixed $n \in {\mathbb N}^*$, we reformulate 
$(t,x) \mapsto b(x,\nu_t^n)$
in the form
    $$ b(x,\nu_t^n) = \int_{t_0}^T \rho^n(t-s) \biggl[ \int_A b(x,a) d\nu_s(a) \biggr] ds, 
    \quad t \in [t_0,T] \times {\mathbb R}^{d_1}.$$ The regularity of $(t,x) \mapsto b(x,\nu_t^n)$ follows in a straightforward way from the regularity of $b$ and the growth assumption on $\nabla^k_x b(x,a)$ for $0 \leq k \leq 3$ (see \assreg) together with the integrability 
    properties of $\bd{\nu}$ (see 
    Definition \ref{defn:D(t_0)}). Moreover, for any integer $k \in \{0,1,2,3\}$,
$$ \forall t \in [t_0,T], \quad 
\norm{b(\cdot, \nu_t^n)}_{\mathcal{C}^k_b} \leq C \int_A \bigl(1+|a|^{k+1}\bigr) d|\nu_t^n|(a),$$
from which we deduce \eqref{eq:uniformboundnun23-12}.
\vskip 4pt

 \textit{Step 2.}
We now turn to the proof of 
\eqref{eq:approxbnun23/12}.
For this, we introduce another smooth density $\zeta : A \rightarrow \R^+$ with compact support and define, for all $m \in \mathbb{N}^*$, $\zeta^m(a) = m^{d'} \zeta(m \, \cdot)$. For $t \in \R$ and $n \in \mathbb{N}^*$, we let
$$ \nu^{n,m}_t := \zeta^m * \nu_t^n, \quad \quad \bar{\nu}_t^m := \zeta^m * \bar{\nu}_t,$$
where the convolution is taken with respect to the $a$-variable. 
For every $t \in \R$, the measures $\nu_t^{n,m}$ and $\bar \nu_t^m$ have densities with respect to the Lebesgue measure over $A$ and these densities are still denoted $a\mapsto \nu_t^{n,m}(a)$ and $a \mapsto \nu_t^m(a)$ respectively. Moreover, for all 
$t \in {\mathbb R}$
and 
Borel subset $B$ of $A$,
\begin{equation}
\label{eq:nunm=rhon*barnum}
\begin{split}
\nu_t^{n,m}(B) = \int_{A}
\zeta^m(a) 
\nu_t^n(B-a) da
&= \int_A \biggl( \int_{t_0}^T
\zeta^m(a) 
\rho^n(t-s) 
\nu_{s} (B-a) 
ds
\biggr)
da
\\
&= \int_{t_0}^T \biggl( \int_A 
\zeta^m(a) 
\nu_{s} (B-a) 
da
\biggr)
\rho^n(t-s) 
 ds
\\
&= \int_{\mathbb R} 
\rho^n(t-s) 
\bar \nu^m_{s} (B) 
 ds 
= \bigl( \rho^n *  
\bar \nu^m (B)\bigr)_t, 
\end{split}
\end{equation}
which permits to identify  
$\nu_t^{n,m}$ and $( \rho^n *  
\bar \nu^m)_t$ (with the latter being interpreted as a Pettis integral, even though we do not use this notion below).

For all $(t,a) \in [t_0,T]$,  $m,n \in \mathbb{N}^*$
and $k \in \{0,1,2, 3\}$, we have
\begin{align*}
    \bigl| \nabla_x^kb(x, \nu_t^{n,m}) - \nabla^k_x b(x,\nu_t^n) \bigr| &= \biggl| \int_A \bigl[ \zeta^m *\nabla_x^kb(x,\cdot)(a) - \nabla_x^kb(x,a) \bigr] d\nu_t^n(a) \biggr| \\
    &\leq \frac{C}{m}(1+|x|) \int_A (1+|a|^{k+1})d|\nu_t^n|(a), 
\end{align*}
where we used the growth assumption on $\nabla_x^k \nabla_ab$ in the last line (see \assreg). It then follows from \eqref{eq:estimatenunDt0} in Lemma \ref{lem:propertiesnun11/01} that
\begin{equation} 
\lim_{m \rightarrow +\infty} \sup_{n \in \mathbb{N}^*} \int_{t_0}^T \norm{b(\cdot,\nu_t^{n,m}) - b(\cdot,\nu_t^n)}_{\mathcal{C}^3_1} dt =0. 
\label{eq:23/12-1}
\end{equation}
Similarly, we have
\begin{equation} 
\lim_{m \rightarrow +\infty}  \int_{t_0}^T \norm{b (\cdot,\bar \nu_t^{m} ) - b(\cdot,\nu_t)}_{\mathcal{C}^3_1} dt =0. 
\label{eq:23/12-2}
\end{equation}
Now, for fixed $k \in \{1,2,3\}$ and $m,n \in \mathbb{N}^*$, we deduce from 
\eqref{eq:nunm=rhon*barnum} that, for all $(t,x) \in [t_0,T] \times \R^{d_1}$,
\begin{align*} 
\bigl| \nabla_x^kb(x, \nu_t^{n,m}) - \nabla_x^{k} b(x,\bar{\nu}_t^m) \bigr| &= \biggl| \int_{\R} \rho(s) \biggl[ \int_A \nabla_x^{k} b(x,a) \bigl[ \bar{\nu}_{t- s/n}^m(a)- \bar{\nu}_t^m(a)\bigr] da \biggr]ds \biggr| \\
&\leq C \int_{\R} \rho(s) \biggl( \int_A (1+ |a|^{k+1}) \bigl|
\bar{\nu}^m_{t - s/n}(a) - \bar{\nu}_t^m(a)\bigr|da \biggr) ds
\end{align*}
for some $C>0$ independent of  $(t,x) \in [t_0,T] \times \R^{d_1}$ and $m,n \in \mathbb{N}^*$. As a consequence, we get
\begin{equation}
\int_{t_0}^T \norm{ b(\cdot,\nu_t^{n,m}) - b(\cdot, \bar{\nu}_t^m)}_{\mathcal{C}^3_b} dt \leq C \int_{\R} \rho(s) \biggl\{ \int_A (1+|a|^4) \biggl( \int_{\R} \bigl|
\bar{\nu}^m_{t-s/n}(a) - 
\bar{\nu}_t^m(a)\bigr|dt \biggr) da \biggr\} ds.
\label{eq:togetbacjto14/01}
\end{equation}
Now, for all $m \in \mathbb{N}^*$, $a \in A$ and $t \in \R$, we have
\begin{align}
\notag    (1+|a|^4) |\bar{\nu}_t^m(a)| &\leq \int_A (1+|a|^4)\zeta^m(a-a') d|\bar{\nu}_t|(a') \\
\notag    &\leq 4 \int_A |a-a'|^4 \zeta^m(a-a')d|\bar{\nu}_t|(a') + 4 \int_A (1+|a'|^4) \zeta^m(a-a') d|\bar{\nu}_t|(a') \\
    &\leq C_m \int_A (1+|a'|^4)d |\bar{\nu}_t|(a'), \label{eq:14/01:13:48}
\end{align}
for some $C_m >0$ independent from $(t,a) \in \R \times A$, but depending on $m$. In particular, since $ \bd{\nu}$ belongs to $\mathcal{D}(t_0)$, for every $m \in \mathbb{N}^*$ and every $a \in A$, the map $t \mapsto \bar{\nu}_t^m(a)$ belongs to $L^1(\R)$. By definition of $\bar{\bd{\nu}}$, it also has compact support in $[t_0,T]$.  By a simple approximation (in $L^1(\R)$) by uniformly continuous functions, we deduce that, for all $m \in \mathbb{N}^*$ and 
$(s,a) \in \R \times A$,
$$ \lim_{n \rightarrow +\infty} \int_{\R} 
\bigl|\bar{\nu}^{m}_{t-s/n}(a) - 
\bar{\nu}^{m}_t(a)\bigr|dt =0.$$
Using the bound \eqref{eq:14/01:13:48} together with Lebesgue dominated convergence theorem and getting back to \eqref{eq:togetbacjto14/01} we deduce that, for all $ m \in \mathbb{N}^*$,
\begin{equation} 
\lim_{n \rightarrow +\infty} \int_{t_0}^T \norm{b(\cdot,\nu_t^{n,m}) - b(\cdot, 
\bar{\nu}_t^m) }_{\mathcal{C}^3_b} dt =0.
\label{eq:23/12-3}
\end{equation}
Combining together \eqref{eq:23/12-1}, \eqref{eq:23/12-2} and \eqref{eq:23/12-3}, this leads to 
$$ \lim_{n \rightarrow +\infty} \int_{t_0}^T \norm{ b(\cdot , \nu_t^n) - b(\cdot, \nu_t) }_{\mathcal{C}^3_1} dt =0,$$
which completes the proof of \eqref{eq:approxbnun23/12}. 
\end{proof}

\begin{cor}
\label{cor:approximation:nu}
Let $t_0 \in [0,T]$ and ${\boldsymbol \nu} \in {\mathcal D}(t_0)$. For 
$({\boldsymbol \nu}^n)_{n \geq 1}$ defined in \eqref{eq:defnnun23/12}, let, for each 
$n \in {\mathbb N}^*$, 
$(X^{n,t,x}_s)_{t \leq s \leq T}$
be the solution 
to the ODE 
\eqref{eq:ODE:appendix} 
driven by ${\boldsymbol \nu}^n$. 
Then, we have
\begin{equation}
    \lim_{n \rightarrow +\infty} \sup_{t \in [t_0,T]} \sup_{ s \in [t,T]} \norm{ X_s^{t, \cdot} - X_s^{n,t,\cdot}}_{\mathcal{C}^3_{1}} = 0.
\label{eq:cor:approximation:nu}
\end{equation}
\color{black}
\end{cor}

\begin{proof} { \ }
\textit{Step 1.}
Following
\eqref{eq:prop:A.1:step:1:1} and
using
Gr\"onwall's lemma, we can find 
a constant $C$ such that, for any $x \in {\mathbb R}^d$
and $n \in \N^*$,
\begin{align*}
\sup_{t \in [t_0,T]}
\sup_{s \in [t,T]}
\vert X_s^{t,x}
- X_s^{n,t,x}
\vert 
&\leq C
\sup_{t \in [t_0,T]}
\sup_{s \in [t,T]}
\biggl\vert \int_{t}^s
b\bigl(X_r^{t,x},\nu_r^n - \nu_r \bigr)  
dr
\biggr\vert. \\
&\leq C (1+|x|) \sup_{z \in \R^{d_1}} \sup_{t \in [t_0,T]} \sup_{s \in [t,T]}
\biggl[ \frac{1+|X^{t,z}_s|}{1+|z|} \int_{t_0}^T \norm{ b( \cdot, \nu_r^n) - b(\cdot,\nu_r)}_{\mathcal{C}^0_1} dr \biggr].
\end{align*}
Using the linear growth of $X_r^{t,x}$ from \eqref{eq:estimateODEAppendix19Aout} of Proposition \ref{prop:ODE16Sept} as well as the convergence of $b(\cdot,\nu^n)$ to $b(\cdot,\nu)$ from Lemma \ref{lem:newregularizationb23/12} we deduce that 
$$ \lim_{n \rightarrow +\infty} \sup_{t \in [t_0,T]} \sup_{s \in [t,T]} \sup_{x \in \R^{d_1}} \frac{|X_s^{t,x} -X_s^{n,t,x}|}{1+|x|} =0. $$
This proves 
\eqref{eq:cor:approximation:nu}
when the sum over 
$k$ therein is in fact restricted to 
$k=0$. 
\vskip 4pt

\textit{Step 2.}
We claim that, by the same argument, we can handle the convergence of $\nabla^k_x X_s^{n,t,x}$ for $k=1,2,3$.  {Since, by Proposition \ref{prop:ODE16Sept} and Lemma \ref{lem:propertiesnun11/01}, $\| \nabla_x X_s^{n,t,\cdot} \|_{\mathcal{C}^2_b} $ is bounded independently of $t,s$ and $n$, it suffices to prove that convergence holds uniformly over bounded subsets of $\R^{d_1}$}. In fact, the main difficulty is the case of the third derivative in $x$, which we prove now under the assumption that, $(\nabla_x X_s^{t,n})_{ n \geq 1}$ and  $(\nabla^2_x X_s^{t,n})_{ n \geq 1}$ converge uniformly over compact subsets of $\R^{d_1}$ (uniformly in $t \in [t_0,T]$ and $s \in [t,T]$ as well) to $\nabla_x X_s^{t,x}$ and $\nabla_x^2 X_s^{t,x}$ respectively. Following the first step in the proof of Proposition \ref{prop:ODE16Sept}, we can find a constant $C$ such that, for any $n \in {\mathbb N}^*$ and $x \in {\mathbb R}^{d_1}$, 
\begin{equation*}
\begin{split}
&\sup_{s \in [t,T]}
\bigl\vert \nabla_{x}^3 X_s^{n,t,x}
- \nabla_{x}^3 X_s^{t,x}
\bigr\vert 
\\
&\leq C \sup_{s \in [t,T]}
\biggl\vert 
\int_t^s 
\bigl[
\nabla_x b(X_u^{t,x},\nu_u)
- 
\nabla_x b(X_u^{n,t,x},\nu_u^n)
\bigr]
\otimes
\nabla_x^3 X_u^{t,x}
du 
\biggr\vert
\\
&\hspace{5pt}
+ C 
\sum_{k=2,3}
\sum_{\underset{i_1+\dots + i_k=3}{1 \leq i_1 \leq \dots \leq i_k}}
\sup_{s \in [t,T]}
\biggl\vert 
\int_t^s 
\Bigl( 
\nabla_x^k b(X_u^{t,x},\nu_u) 
\otimes
\bigotimes_{j=1}^k
\nabla_x^{i_j} X_u^{t,x}
-
\nabla_x^k b(X_u^{n,t,x},\nu_u^n) 
\otimes
\bigotimes_{j=1}^k
\nabla_x^{i_j} X_u^{n,t,x}
\Bigr) 
du 
\biggr\vert
\\
&=:C \bigl( T_1^n(t,s,x)+T_2^n(t,s,x)\bigr).
\end{split}
\end{equation*}
We explain how to handle $T_2^n(t,s,x)$. The term $T_1^n(t,s,x)$ can be handled in the same way. For 
$k \in \{2,3\}$ and  $(i_1,\dots,i_k) \in \{1,\dots,k-1\}$
with $i_1 \le \dots \leq i_k$ and $i_1+\dots+i_k=3$, we write
\begin{equation*}
\begin{split}
&\int_t^s 
\Bigl( 
\nabla_x^k b(X_u^{t,x},\nu_u) 
\otimes \bigotimes_{j=1}^k
\nabla_x^{i_j} X_u^{t,x}
-
\nabla_x^k b(X_u^{n,t,x},\nu_u^n) 
\otimes \bigotimes_{j=1}^k
\nabla_x^{i_j} X_u^{n,t,x}
\Bigr) 
du 
\\
&= \int_t^s 
\Bigl( 
\nabla_x^k b(X_u^{t,x},\nu_u - \nu_u^n) 
\otimes \bigotimes_{j=1}^k
\nabla_x^{i_j} X_u^{t,x}
\Bigr) 
du 
\\
&\hspace{15pt}
+ \int_t^s 
\Bigl( 
\nabla_x^k b(X_u^{t,x},\nu_u^n) 
\otimes \bigotimes_{j=1}^k
\nabla_x^{i_j} X_u^{t,x}
-
\nabla_x^k b(X_u^{n,t,x},\nu_u^n) 
\otimes \bigotimes_{j=1}^k
\nabla_x^{i_j} X_u^{n,t,x}
\Bigr) 
du 
\\
&=: T_{2,1}^{n,k,(i_1,\dots,i_k)}(t,s,x) + T_{2,2}^{n,k,(i_1,\dots,i_k)}(t,s,x). 
\end{split}
\end{equation*}
Using the bounds available for 
the derivatives of 
$b$ (see (i) in \assreg) and 
the derivatives of 
$X^{t,\cdot}_u$
and $X^{n,t,\cdot}_u$
(see Proposition 
\ref{prop:ODE16Sept}), we deduce that, for any $K>0$, 
\begin{equation*}
\begin{split}
    &\bigl\vert T_{2,2}^{n,k,(i_1,\dots,i_k)}(t,s,x)\bigr\vert 
    \\
    &\leq 
    \int_t^T
    \biggl(
    \int_{A}
{\mathbf 1}_{\{\vert a \vert \leq K\}}
\Bigl\vert 
\nabla_x^k b(X_u^{t,x},a) 
\otimes \bigotimes_{j=1}^k
\nabla_x^{i_j} X_u^{t,x}
-
\nabla_x^k b(X_u^{n,t,x},a) 
\otimes \bigotimes_{j=1}^k
\nabla_x^{i_j} X_u^{n,t,x}
\Bigr\vert  
d |\nu_u^n|(a)
\biggr) du 
\\
&\hspace{15pt}+
C \int_t^T \int_A (1+ \vert a \vert^4) {\mathbf 1}_{\{ \vert a \vert \ge K\}} d |{\boldsymbol \nu}^n|(t,a),
\end{split}
\end{equation*}
the constant $C$ in the last line being independent of $(t,s,x,n)$. 
By \eqref{eq:uniformintegrability23/12} in 
Lemma 
\ref{lem:newregularizationb23/12}, the second term in the right-hand side tends to $0$ as 
$K$ tends to $\infty$, uniformly in $n \geq 1$. 
Then, for a given $K$ we can use the continuity of $\nabla_x^k b$ and the fact that 
\eqref{eq:cor:approximation:nu} holds true for the derivatives of order 0, 1 and 2 to deduce that the first term in the right-hand side tends to $0$ as $n$ tends to $\infty$, uniformly with respect to $x$ in compact subsets. 
\vskip 4pt

\textit{Step 3.} We now address 
$T_{2,1}^{n,k,(i_1,\dots,i_k)}(t,s,x)$ that can be estimated as follows, thanks to Proposition \ref{prop:ODE16Sept}, 
\begin{align*}
    \bigl| T_{2,1}^{n,k,(i_1,\dots,i_k)}(t,s,x) \bigr| &\leq C(1+|x|) \int_{t_0}^T \norm{b(\cdot,\nu_u) - b(\cdot, \nu_u^n)}_{\mathcal{C}^3_1} du.
\end{align*}
The proof is easily completed using Lemma \ref{lem:newregularizationb23/12}.
\end{proof}

\subsubsection{Well-posedness of the Continuity and Transport Equations}

\begin{lem}
Take $t_0,t_1 \in [0,T]$ with $t_1 >t_0$ and $\bd{\nu} \in \mathcal{D}(t_0)$. Assume moreover that $ [t_0,t_1] \ni t \mapsto \nu_t \in \mathcal{M}_{1+|a|}(A)$ is continuous. Then, for any $p \geq q \geq 0$ and any $\phi \in \mathcal{C}^1_{p,q}$ there is a unique classical solution $\varphi : [t_0,t_1] \times \R^{d_1} \times \R^{d_2} \rightarrow \R$ to the backward transport equation 
\begin{equation} 
-\partial_t \varphi_t - b(x, \nu_t) \cdot \nabla_x \varphi_t = 0 \quad \mbox{ in } [t_0,t_1] \times \R^{d_1} \times \R^{d_2}, \quad \varphi_{t_1} = \phi \mbox{ in } \R^{d_1} \times \R^{d_2}.
\label{eq:varphin08/01}
\end{equation}
It is given by $ \varphi_t(x,y) = \phi(X_{t_1}^{t,x},y)$, for all $(t,x,y) \in [t_0,t_1] \times \R^{d_1} \times \R^{d_2}$, where $(X_s^{t,x})_{s \in [t,t_1]}$ is the solution to the ODE 
$$ \dot{X}_s^{t,x} = b(X_s^{t,x}, \nu_s) \quad  s \in [t,t_1], \quad X_t^{t,x} = x.$$
Moreover, 
there exists a non-decreasing function 
    $\Lambda : {\mathbb R}_+ \rightarrow {\mathbb R}_+$, independent of $t_0$, $\phi$ and 
    ${\boldsymbol \nu}$, such that 
\begin{align}
\label{lem:A.6:2}
&\sup_{t \in [t_0,t_1]} \norm{ \varphi_t}_{\mathcal{C}^0_p}  \leq \Lambda \bigl(\norm{\bd{\nu}}_{\mathcal{D}(t_0)}\bigr) \norm{ \phi}_{\mathcal{C}^0_p},
\\
&\sup_{t \in [t_0,t_1]} \norm{ \nabla \varphi_t}_{\mathcal{C}^0_q} +\sup_{t_2 <t_3 \in [t_0,t_1]} \frac{\norm{\varphi_{t_3} - \varphi_{t_2}}_{\mathcal{C}^0_{q}}}{\sqrt{t_3 -t_2}}  \leq \Lambda \bigl(\norm{\bd{\nu}}_{\mathcal{D}(t_0)}\bigr) \norm{ \nabla \phi}_{\mathcal{C}^0_q}. 
\label{lem:A.6:2:b}
\end{align}
If we also assume that, for $l=2$ or $l=3$, $\phi$ belongs to $\mathcal{C}^{l}_{p,q}$  and $[t_0,t_1] \ni t \mapsto \nu_t \in\mathcal{M}_{1+|a|^{l}}(A)$ is continuous, then $\phi$ has continuous derivatives of the form $\partial_t^m \nabla^k_x \varphi$ with $m=0,1$, $k=0,\dots,l$ and $m+k \leq l$. 
Moreover, for a possibly new choice of $\Lambda$, we have the estimates
\begin{equation}
\label{lem:A.6:3}
 \sup_{t \in [t_0,T]} \norm{ \nabla \varphi_t}_{\mathcal{C}^{l-1}_q} + \sup_{t_2 <t_2 \in [t_0,t_1]} \frac{\norm{\varphi_{t_3}- \varphi_{t_2}}_{\mathcal{C}^1_q}}{\sqrt{t_3-t_2}} \leq \Lambda \bigl( \norm{\bd{\nu}}_{\mathcal{D}(t_0)} \bigr) \norm{ \nabla \phi}_{\mathcal{C}^{l-1}_q}. 
 \end{equation} 
In that case, $\partial_t \varphi$ and $\nabla_x \varphi$ also belong to $\mathcal{C}([t_0,t_1], \mathcal{C}^{l-2}_q) \cap \mathcal{L}^{\infty}([t_0,t_1], \mathcal{C}^{l-1}_q)$.

\label{lem:backwardtransportsmoothcontrol23/01}
\end{lem}

\begin{rmk}
  Notice that, for any $\bd{\nu} \in \mathcal{D}(t_0)$ and any $n \in \mathbb{N}^*$, the regularization $\bd{\nu}^n$ from Lemma \ref{lem:newregularizationb23/12} satisfies the continuity condition required in Lemma \ref{lem:backwardtransportsmoothcontrol23/01}.
\label{rmk:23/01}
\end{rmk}

\begin{rmk}
    As an application of Lemma \ref{lem:backwardtransportsmoothcontrol23/01} (with $l=3$), we find that, if $[t_0,T] \ni t \mapsto \nu_t \in \mathcal{M}_{1+|a|^3}(A)$ is continuous and $\phi$ belongs to $\mathcal{C}^3_{2,1}$ then the solution $\varphi : [t_0,T] \times \R^{d_1} \times \R^{d_2} \rightarrow {\mathbb R}$ with terminal condition $\phi$ and control $\bd{\nu}$ is three times continuously differentiable and satisfies 
    $$ \sup_{t \in [t_0,T]} \norm{ \varphi_t}_{\mathcal{C}^3_{2,1}} \leq \Lambda \bigl( \norm{ \bd{\nu}}_{\mathcal{D}(t_0)} \bigr) \norm{\phi}_{\mathcal{C}^3_{2,1}}. $$
\end{rmk}

\begin{rmk}
    In particular, if $\phi \in \mathcal{C}^3_{2,1}$ and  $\bd{\nu} \in \mathcal{D}(t_0)$ such that $ [t_0,T] \ni t\mapsto \nu_t \in \mathcal{M}_{1+|a|^3}(A)$ is continuous, then the solution  $\varphi$ satisfies 
\begin{equation} 
\varphi \in \mathcal{C}([t_0,t_1], \mathcal{C}^1_2) \cap \mathcal{L}^{\infty}([t_0,t_1], \mathcal{C}^2_2)
\end{equation}
\begin{equation} 
\partial_t \varphi , \nabla_x \varphi \in  \mathcal{C} \bigl( [t_0,t_1], \mathcal{C}^1_{2}) \cap \mathcal{L}^{\infty}([t_0,t_1], \mathcal{C}^2_{1})
\end{equation}
This is therefore an admissible test function for the equation for $\bd{\rho}$.
\label{rmk:A.9}
\end{rmk}

\begin{proof}[Proof of Lemma \ref{lem:backwardtransportsmoothcontrol23/01}]
When \(\phi\) belongs to \(\mathcal{C}^1_p\), both the well-posedness of the equation and the representation formula follow directly from classical theory. Specifically, due to the continuity of the map \( t \mapsto \nu_t \in \mathcal{M}_{1+|a|}(A) \) for \( t \in [t_0, t_1] \), we can apply Remark \ref{rmk:regularvectorfield27/01}. This implies that the vector field \( (t,x) \mapsto b(x, \nu_t) \), which is already Lipschitz in \((x,y)\) with a time-integrable Lipschitz constant, is jointly continuous and bounded.
When $\phi$ belongs to $\mathcal{C}^{l}_{p,q}$, 
   the higher regularity properties of 
   $\varphi$ follow from the explicit representation formula for the latter together with Proposition \ref{prop:ODE16Sept} and Remark \ref{rmk:regularvectorfield27/01}.
   
The estimate in the first line of 
 \eqref{lem:A.6:2} and the bounds for the first terms in 
 the left-hand sides of 
\eqref{lem:A.6:2:b}
 and
  \eqref{lem:A.6:3} also follow from the explicit formula and Proposition \ref{prop:ODE16Sept}.
  
To handle the second terms in 
 the left-hand sides of 
\eqref{lem:A.6:2:b}
 and
  \eqref{lem:A.6:3}, 
we use \assreg\, together with the fact that $ \sup_{t \in [t_0,T]} \norm{ \nabla_x \varphi_t}_{\mathcal{C}^{l-1}_q} \leq \Lambda (\norm{\bd{\nu}}_{\mathcal{D}(t_0)}) \norm{ \nabla \phi}_{\mathcal{C}^{l-1}_q}$ and proceed as follows. With $l$ being equal to 1, 2 we have, for any $t_3 > t_2 \in [t_0,t_1]$,
\begin{align} 
\notag \norm{ \varphi_{t_3} - \varphi_{t_2}}_{\mathcal{C}^{l-1}_q} &= \norm{ \int_{t_2}^{t_3} b(\cdot,\nu_t) \cdot \nabla_x \varphi_t dt  }_{\mathcal{C}^{l-1}_q}  \leq \int_{t_2}^{t_3} \norm{ b(\cdot,\nu_t) \cdot \nabla_x \varphi_t }_{\mathcal{C}^{l-1}_q} dt \\
&\leq C \sup_{t \in [t_0,T]} \norm{ \nabla_x \varphi_t}_{\mathcal{C}^{l-1}_q} \int_{t_2}^{t_3} \int_A (1+|a|^l)d|\nu_t|(a)dt \label{eq:thederivationof24/05}\\
&\notag \leq C \sqrt{t_3 -t_2} \sup_{t \in [t_0,T]} \norm{ \nabla_x \varphi_t}_{\mathcal{C}^{l-1}_q} \norm{ \bd{\nu}}_{\mathcal{D}(t_0)},
\end{align}
for some $C>0$ depending only on the vector field $b$. 

For $l =3$, $\nabla_x \varphi$ belongs to $\mathcal{C}([t_0,t_1],\mathcal{C}_q^{1})$ as a consequence of the continuity of $[t_0,t_1] \ni t \mapsto \nu_t \in \mathcal{M}_{1+|a|^3}(A)$. Indeed following the derivation of \eqref{eq:thederivationof24/05} we have, for any $t_3 > t_2 \in [t_0,t_1]$,
\begin{align*}
    \norm{ \varphi_{t_3} - \varphi_{t_2}}_{\mathcal{C}^2_q} &\leq C \sup_{t \in [t_0,T]} \norm{ \nabla_x \varphi_t}_{\mathcal{C}^{2}_q} \int_{t_2}^{t_3} \int_A (1+|a|^3)d|\nu_t|(a)dt \\
    &\leq C (t_3-t_2) \sup_{t \in [t_0,T]} \norm{ \nabla_x \varphi_t}_{\mathcal{C}^{2}_q} \sup_{t \in [t_0,t_1]} \int_A(1+|a|^3)d|\nu_t|(a)
\end{align*}
for some $C>0$ depending only on $b$.

It remains to prove the regularity of $\partial_t \varphi$ when $l =2,3$. Using 
the equation \eqref{eq:varphin08/01} for $\partial_t \varphi$, we have
$$ \norm{ \partial_t \varphi_t}_{\mathcal{C}^{l-1}_q} = \norm{ b(\cdot,\nu_t) \cdot \nabla_x \varphi_t}_{\mathcal{C}^{l-1}_q} \leq C \norm{ \nabla_x \varphi_t }_{\mathcal{C}^{l-1}_q} \int_A (1+|a|^l) d|\nu_t|(a), $$
and,
\begin{align*}
    \bigl\| \partial_t \varphi_{t_3} - \partial_t \varphi_{t_2}\bigr\|_{\mathcal{C}^{l-2}_q} &= 
    \bigl\| b(\cdot, \nu_{t_3}) \cdot \nabla_x \varphi_{t_3} - b(\cdot, \nu_{t_2}) \cdot \nabla_x \varphi_{t_2} \bigr\|_{\mathcal{C}_q^{l-2}} \\
    &\leq \bigl\| [b(\cdot, \nu_{t_3}) - b(\cdot,\nu_{t_2})] \cdot \nabla_x \varphi_{t_3}  \bigr\|_{\mathcal{C}^{l-2}_q} + \bigl\| b(\cdot, \nu_{t_2}) \cdot \nabla_x (\varphi_{t_3} - \varphi_{t_2}) 
    \bigr\|_{\mathcal{C}^{l-2}_q} \\
    &\leq C  \norm{ \nabla_x \varphi_{t_3}}_{\mathcal{C}^{l-2}_q} \int_A (1+|a|^{l-1}) d|\nu_{t_3} - \nu_{t_2}|(a)  
    \\
    &\hspace{15pt} +  C \norm{ \nabla_x (\varphi_{t_3} - \varphi_{t_2})}_{\mathcal{C}^{l-2}_q}\int_A (1+|a|^{l-1}) d|\nu_{t_2}|(a),
\end{align*}
and we conclude using the bounds on $\norm{ \nabla_x \varphi_t}_{\mathcal{C}^{l-1}_{q}} $ and $\norm{ \varphi_{t_3} - \varphi_{t_2}}_{\mathcal{C}^{l-1}_q}$.
This completes the proof of the regularity of the lemma.
\end{proof}

Regarding
the analysis of 
\eqref{eq:lem:compactnessContinuityEquation24/06:statement}, we now turn to the following well-posedness result for the continuity equation, previously stated in Proposition
\ref{lem:compactnessContinuityEquation24/06}:

\begin{lem}
\label{lem:A.4}
    Take ${\boldsymbol \nu}=(\nu_t)_{t_0 \le t \le T} \in \mathcal{D}(t_0)$. Then, for any given $\gamma_0 \in \mathcal{P}_2(\R^{d_1} \times \R^{d_2})$, there is a unique distributional solution ${\boldsymbol \gamma} \in \mathcal{C}([t_0,T], \mathcal{P}_2(\R^{d_1} \times \R^{d_2}))$ to the continuity equation
\begin{equation}
\label{eq:lem:compactnessContinuityEquation24/06:statementAp}
 \partial_t \gamma_t + \div_x( b(x,\nu_t) \gamma_t) = 0 \quad \mbox{ in } (t_0,T) \times \R^{d_1} \times \R^{d_2}, \quad \gamma(t_0)= \gamma_0.
 \end{equation}
It is given by 
$(\gamma_t = (X_t^{t_0,\cdot}, i_d) \# \gamma_0)_{t_0 \le t \le T}$
where $(X_t^{t_0,\cdot})_{t_0 \le t \le T}$
is the flow of 
\eqref{eq:ODE:intro}.

Moreover, there exists a non-decreasing function $\Lambda : {\mathbb R}_+ \rightarrow {\mathbb R}_+$, independent of $(t_0,\gamma_0)$ and 
${\boldsymbol \nu}$, such that, for each $p \geq 1$ such that $\gamma_0$ belongs to $\mathcal{P}_p(\R^{d_1} \times \R^{d_2})$, 
\begin{equation}
\label{eq:lem:compactnessContinuityEquation24/06:statement:48Ap}
\sup_{t \in [t_0,T]}\frac{ \Bigl(\int_{\R^{d_1} \times \R^{d_2}} \bigl( |x|^2 + |y|^2 \bigr)^{p/2}  d\gamma_t(x,y) \Bigr)^{1/p}}{1 + \Bigl( \int_{\R^{d_1} \times \R^{d_2}} \bigl( |x|^2 + |y|^2 \bigr)^{p/2}  d\gamma_0(x,y) \Bigr)^{1/p}}  + \sup_{t \neq s \in [t_0,T]}  \frac{d_p(\gamma_t,\gamma_s)}{\sqrt{|t-s|}} \leq \Lambda
 \bigl( \norm{\bd\nu}_{\mathcal{D}(t_0)} \bigr).
 \end{equation} 
 
\end{lem}

\begin{proof}
    For existence, we easily check
    (with the notation of Proposition \ref{prop:ODE16Sept}) that
    the flow of probability measures $(\gamma_t = (X_t^{t_0,\cdot}, i_d) \# \gamma_0)_{t_0 \le t \le T}$ is a solution. As a consequence of 
  \eqref{eq:estimateODEAppendix19Aout}, we have, for each $ p \geq 1$, 
\begin{align*} 
 \Bigl( \int_{\R^{d_1} \times \R^{d_2}}  \bigl( |x|^2 +|y|^2 \bigr)^{p/2} d\gamma_t(x,y) \Bigr)^{1/p} &= \Bigl( \int_{\R^{d_1} \times \R^{d_2}} \bigl ( \bigl|X_t^{t_0,x} \bigr|^2 + |y|^2 \bigr)^{p/2}  d\gamma_0(x,y) \Bigr)^{1/p} \\
 &\leq \sqrt{2} \Bigl( \int_{\R^{d_1} \times \R^{d_2}} \bigl( |x|^2 +|y|^2 + \Lambda^2(\norm{\bd{\nu}}_{\mathcal{D}(t_0)}) \bigr)^{p/2} d\gamma_0(x,y) \Bigr)^{1/p} \\
 &\leq \sqrt{2} \Bigl( \int_{\R^{d_1} \times \R^{d_2}} \bigl(|x|^2+ |y|^2 \bigr)^{p/2} d\gamma_0(x,y) \Bigr)^{1/p} +\sqrt{2} \Lambda (\norm{\bd{\nu}}_{\mathcal{D}(t_0)}), 
\end{align*}
where we used Minkowski's inequality at the last line. Moreover,
$$ d_p(\gamma_t, \gamma_s) \leq \biggl( \int_{\R^{d_1}  \times \R^{d_2}} |X_t^{t_0,x_0} - X_s^{t_0,x_0} |^p d\gamma_0(x_0,y_0) \biggr)^{1/p} \leq \Lambda \bigl( \norm{\bd{\nu}}_{\mathcal{D}(t_0)} \bigr) \sqrt{|t-s|},$$
which proves 
\eqref{eq:lem:compactnessContinuityEquation24/06:statementAp}
and
\eqref{eq:lem:compactnessContinuityEquation24/06:statement:48Ap}.
    
For uniqueness, we consider two (distributional) solutions ${\boldsymbol \gamma}^1, {\boldsymbol \gamma}^2$ in 
${\mathcal C}([t_0,T],{\mathcal P}_2({\mathbb R}^{d_1} \times \R^{d_2}))$. For  given $t_1 \in (t_0,T]$ and $\phi \in \mathcal{C}^2_b(\R^{d_1} \times \R^{d_2})$, we consider $({\boldsymbol \nu}^n)_{n \geq 1}$ the  approximation of ${\boldsymbol \nu}$ given by 
\eqref{eq:defnnun23/12}. 
Recalling Lemma \ref{lem:backwardtransportsmoothcontrol23/01}, we find that for each $n \in {\mathbb N}^*$, there exists a (unique) solution ${\varphi}^n \in {\mathcal C}^{1,2}([t_0,T] \times {\mathbb R}^{d_1} \times \R^{d_2}) $ (one continuous time derivative and two continuous space derivatives) 
to the transport equation
$$ - \partial_t \varphi_t^n -b(x, \nu_t^n) \cdot \nabla_x \varphi_t^n =0 \quad \mbox{ in } [t_0,t_1] \times\R^{d_1} \times \R^{d_2}, \quad \quad \varphi_{t_1}^n = \phi \quad \mbox{ in } \R^{d_1} \times \R^{d_2}.$$
\eqref{eq:ODE:appendix}
By combining Lemma \ref{lem:backwardtransportsmoothcontrol23/01}
and \eqref{eq:estimatenunDt0} in 
Lemma 
\ref{lem:newregularizationb23/12}, 
it is easy to 
find a constant $C$ such that, for any $n \geq 1$, 
$\| \nabla_x \varphi^n \|_{{\mathcal C}_b^0} \leq C$. 
Using the definition of a solution to the continuity equation, as given in 
Subsection 
\ref{subse:3.1} (see 
\eqref{eq:CE:def}), we obtain
\begin{equation}
\label{eq:phi:T_1+T_2}
\int_{\R^{d_1} \times \R^{d_2}} \phi(x,y)  d (\gamma_{t_1}^2 - \gamma_{t_1}^1)(x,y) = \int_{t_0}^{t_1} \int_{\R^{d_1} \times \R^{d_2}} b(x,\nu_t - \nu_t^n) \cdot \nabla_x \varphi_t^n(x,y) d (\gamma_t^2 -\gamma_t^1)(x,y)dt.
\end{equation}
As a consequence, we get
\begin{align*} 
\Bigl|\int_{\R^{d_1} \times \R^{d_2}} &\phi(x,y)  d (\gamma_{t_1}^2 - \gamma_{t_1}^1)(x,y) \Bigr| 
\\
&\leq C \Bigl( \sup_{t \in [t_0,T]} \int_{\R^{d_1} \times \R^{d_2}}(1+|x|) d(\gamma_t^1 + \gamma_t^2)(x,y) \Bigr) \int_{t_0}^T \norm{ b( \cdot , \nu_t^n) - b(\cdot,\nu_t)}_{\mathcal{C}^0_{1}} dt 
\end{align*}
and we deduce from Lemma \ref{lem:newregularizationb23/12} and letting $n \rightarrow +\infty$ that
$$ \int_{\R^{d_1} \times \R^{d_2}} \phi(x,y) d (\gamma_{t_1}^2 - \gamma_{t_1}^1)(x,y) = 0.$$
Since this is true for any smooth $\phi$, we get that $\gamma^1_{t_1}= \gamma^2_{t_1}$, which implies uniqueness
because $t_1$ is arbitrary. \color{black}
\end{proof}

Using similar arguments, we now address the well-posedness of the transport equation \eqref{eq:varphin08/01}, but without the continuity assumption made in the statement of 
Lemma \ref{lem:backwardtransportsmoothcontrol23/01}. It will be convenient to cover the case where the terminal condition is only a $\mathcal{C}^1$ function, in which case $t \mapsto \nabla_x u_t(x,y)$ is not necessarily continuous. If $\phi$ belongs to $\mathcal{C}^1_p$ for some $p \geq 0$, we say that  $\varphi \in \mathcal{C}([t_0,t_1] \times \R^{d_1} \times \R^{d_2})$ is a solution to the transport equation if, for all $t \in [t_0,t_1]$, $(x,y) \mapsto  \varphi_t(x,y)$ is differentiable, $\sup_{t \in [t_0,T]} \norm{ \varphi}_{\mathcal{C}^1_{p} }$ is finite and it holds, for all $t \in [t_0,t_1]$ and all $(x,y) \in \R^{d_1}\times \R^{d_2}$ 
$$ \varphi_{t}(x,y)  = \phi(x,y) +  \int_{t}^{t_1} b(x, \nu_s) \cdot \nabla_x \varphi_s(x,y) ds.   $$

\begin{prop}
\label{prop:SolutionBackxard16SeptAppendix}
    Take  $t_0,t_1 \in [0,T]$, with $t_0<t_1$, and $\bd{\nu} \in \mathcal{D}(t_0)$. Consider a terminal condition $\phi \in \mathcal{C}^1_{p,q}$ for some $ p,q \geq 0$. Then, there is a unique solution $\varphi  \in \mathcal{C}([t_0,t_1] \times \R^{d_1} \times \R^{d_2})$ to the backward transport equation 
    \eqref{eq:varphin08/01} with terminal condition $\phi$. It is 
    given by $ \varphi_t(x,y) = \phi(X_{t_1}^{t,x},y)$, for all $(t,x,y) \in [t_0,t_1] \times \R^{d_1} \times \R^{d_2}$, where $(X_s^{t,x})_{s \in [t,t_1]}$ is the solution to the ODE 
$$ \dot{X}_s^{t,x} = b(X_s^{t,x}, \nu_s) \quad  s \in [t,t_1], \quad X_t^{t,x} = x.$$
Moreover, there exists a non-decreasing function 
    $\Lambda : {\mathbb R}_+ \rightarrow {\mathbb R}_+$, independent of $t_0$, $\phi$ and 
    ${\boldsymbol \nu}$, such that 
\begin{equation*}
    \begin{split}
&\sup_{t \in [t_0,t_1]} \norm{ \varphi_t}_{\mathcal{C}^0_p}  \leq \Lambda \bigl(\norm{\bd{\nu}}_{\mathcal{D}(t_0)}\bigr) \norm{ \phi}_{\mathcal{C}^0_p},
\\
&\sup_{t \in [t_0,t_1]} \norm{ \nabla \varphi_t}_{\mathcal{C}^0_q} +\sup_{t_2 <t_3 \in [t_0,t_1]} \frac{\norm{\varphi_{t_3} - \varphi_{t_2}}_{\mathcal{C}^0_{q}}}{\sqrt{t_3 -t_2}}  \leq \Lambda \bigl(\norm{\bd{\nu}}_{\mathcal{D}(t_0)}\bigr) \norm{ \nabla \phi}_{\mathcal{C}^0_q}. 
\end{split}
\end{equation*}
If we also assume that, for $l=2$ or $l=3$, $\phi$ belongs to $\mathcal{C}^{l}_{p,q}$,  then $\varphi_t$ belongs to $\mathcal{C}^{l}_{p,q}$ for all $t \in [t_0,T]$, and we have, for a possibly new function $\Lambda$, 
$$ \sup_{t \in [t_0,T]} \norm{ \nabla \varphi_t}_{\mathcal{C}^{l-1}_q} + \sup_{t_2 <t_3 \in [t_0,t_1]} \frac{\norm{\varphi_{t_3}- \varphi_{t_2}}_{\mathcal{C}^1_q}}{\sqrt{t_3-t_2}} \leq \Lambda \bigl( \norm{\bd{\nu}}_{\mathcal{D}(t_0)} \bigr) \norm{ \nabla \phi}_{\mathcal{C}^{l-1}_q}. $$
\end{prop}

\begin{proof}
For $\phi \in \mathcal{C}^l_{p,q}$, with $l=1,2,3$,
we let $\varphi_t(x,y) := \phi(X_T^{t,x},y)$. Regularizing $\bd{\nu}$ into $(\bd{\nu}^n)_{ n \geq 1}$ as in \eqref{eq:defnnun23/12} and combining Corollary \ref{cor:approximation:nu} and Lemma \ref{lem:backwardtransportsmoothcontrol23/01} we get the existence of a solution satisfying the different regularity estimates. It remains to prove uniqueness. Using the same sequence 
$({\boldsymbol \nu}^n)_{n \geq 1}$, 
we consider the continuity equation 
\eqref{eq:Continuitygammanu} but initialized at some $t_2 \in [t_0,t_1]$ from some
$\gamma_0 \in \mathcal{P}(\R^{d_1} \times \R^{d_2})$ 
with a smooth compactly supported density. 
Following the proof of Lemma \ref{lem:A.4}
, we know that the solution 
(when restricted to $[t_2,t_1]$)
reads
$(\gamma_t^n = (X_t^{n,t_2,\cdot}, i_d) \# \gamma_0)_{t_2 \le t \le t_1}$, where $(X^{n,t_2,x}_t)_{t_2 \leq t \le t_1}$ solves the ODE 
\eqref{eq:ODE:appendix}
 with 
${\boldsymbol \nu}$
replaced by 
${\boldsymbol \nu}^n$. In particular, each $\gamma_t^n$ has a density. Denoting by $(\chi_t^{n,\cdot})_{t_2 \le t \le t_1}$ the inverse of the flow
$(X_t^{n,t_2,\cdot})_{t_2 \le t \le t_1}$, we can write this density as (denoting the density by  $\gamma_t^n$ as itself):
\begin{equation*}
\gamma_t^n(x,y) 
= \frac{\gamma_0(\chi_t^{n,x},y)}{\vert {\rm det}(\nabla_x X_t^{n,t_2,\chi_t^{n,x}})\vert},
\quad t \in [t_2,t_1], \ (x,y) \in {\mathbb R}^{d_1} \times {\mathbb R}^{d_2},
\end{equation*}
from which we deduce that 
$({\boldsymbol \gamma}_t^n(x,y))_{t_2 \leq t \leq t_1}$ is continuously differentiable in $t$. Then, we can justify that, for any other solution ${\varphi}'$
to the transport equation (satisfying the prescription of 
Proposition \ref{prop:SolutionBackxard16Sept}),
\begin{equation}
\label{eq:IBP:F:1}
\begin{split}
&\varphi'_{ t_2}(x,y) \gamma^n_{t_2}(x,y) - \phi(x,y) \gamma^n_{ t_1}(x,y) 
\\
&= \lim_{\vert \pi \vert \rightarrow 0}
\sum_{i=1}^{n-1}
\Bigl[
\bigl( 
\varphi_{s_i}'(x,y)  - \varphi_{s_{i+1}}'(x,y) \bigr) \gamma^n_{s_i}(x,y)
+ \varphi_{s_{i+1}}'(x,y) \bigl( \gamma_{s_i}^n(x,y ) - \gamma_{s_{i+1}}^n(x,y) 
\bigr) 
\Bigr]
\\
&= \int_{t_2}^{t_1} b(x,\nu_s) \cdot \nabla_x \varphi_{s}'(x,y) 
\gamma_s^n(x,y) 
ds -   \int_{t_2}^{t_1} \varphi_s'(x,y) \partial_s {\gamma}_s^n(x,y) ds.
\end{split}
\end{equation}
where the limit is taken over the subdivision $\pi = (s_i)_{i=1,\dots,n}$ of
$[t_2,t_1]$ (with $\vert \pi \vert$ denoting the step of $\pi$). Since ${\boldsymbol \gamma}^n$ is compactly supported, we can integrate in 
$(x,y)$. Using the continuity 
equation satisfied by 
${\boldsymbol \gamma}^n$ and then performing an integration by parts, we obtain 
\begin{equation}
\label{eq:IBP:F:2}
\begin{split}
 \int_{\R^{d_1} \times \R^{d_2}} \varphi_{t_1}'(x,y) d\gamma_0(x,y) 
 &= 
 \int_{\R^{d_1} \times \R^{d_2}} \phi(x,y) d\gamma_{t_2}^n(x,y) 
 \\
&\hspace{15pt} +
 \int_{t_2}^{t_1} \int_{\R^{d_1} \times \R^{d_2}} b(x,\nu_s -\nu_s^n) \cdot \nabla_x \varphi_s'(x,y) d\gamma^n_s(x,y) ds
 \\
 &=: S_1^n + S_2^n.
 \end{split}
 \end{equation}
The first term $S_1^n$ is equal to $\int_{{\mathbb R}^{d_1} \times {\mathbb R}^{d_2}} 
\phi(X_{t_1}^{n,t_2,x},y) d \gamma_0(x,y)$.
Therefore, it converges to 
$$\lim_{n \rightarrow +\infty} S_1^n = \int_{{\mathbb R}^{d_1} \times {\mathbb R}^{d_2}} 
\phi(X_{t_1}^{t_2,x},y) d \gamma_0(x,y) = \int_{{\mathbb R}^{d_1} \times {\mathbb R}^{d_2}}
\varphi_{t_2}(x,y) d \gamma_0(x,y).$$
As for the second term, we have
\begin{align*} 
S_2^n &= \biggl| \int_{t_2}^{t_1} \int_{\R^{d_1} \times \R^{d_2}} b(x,\nu_t -\nu_t^n) \cdot \nabla_x \varphi_t'(x,y) d\gamma^n_t(x,y) dt \biggr| \\
&\leq C \sup_{t \in [t_2,t_1]} \norm{\nabla_x \varphi'_t}_{\mathcal{C}^0_q} \sup_{t \in [t_2,t_1]} \int_{\R^{d_1} \times \R^{d_2}} (|x|^2+|y|^2)^{(q+1)/2} d\gamma_t^n(x,y)  \int_{t_2}^{t_1}\norm{ b(\cdot, \nu_t^n-\nu_t)}_{\mathcal{C}^0_1}dt. 
\end{align*}
Since $\gamma_0$ is compactly supported, we easily show that the term in the middle is bounded independently of $n \in {\mathbb N}^*$ and conclude that $\lim_{n \rightarrow +\infty} S_2^n =0.$ As a consequence, for all $t_1 \in [t_0,T]$ and all smooth and compactly supported density $\gamma_0$ it holds
$$ \int_{\R^{d_1} \times \R^{d_2}} \varphi'_{t_1}(x,y) d\gamma_0(x,y) =  \int_{\R^{d_1} \times \R^{d_2}} \varphi_{t_1}(x,y) d\gamma_0(x,y) $$
from which we conclude that $\varphi'= \varphi$ and there is a unique solution to the equation.
\end{proof}

\subsubsection{Duality Relation}

\begin{lem}
\label{lem:duality:1,2:F}    Take $t_0 \in [0,T]$ and $\gamma_0 \in \mathcal{P}_2(\R^{d_1} \times \R^{d_2})$. Take as well $t_1 \in [t_0,T]$, $\phi \in \mathcal{C}^1_2$ and $\bd{\nu}^1, \bd{\nu}^2$ in $\mathcal{D}(t_0)$. Let $\bd{\gamma}^1$ be the solution to the continuity equation 
\eqref{eq:lem:compactnessContinuityEquation24/06:statementAp} starting from $(t_0,\gamma_0)$ and driven by $\bd{\nu}^1$, and ${\varphi}^2$ the solution to the backward equation 
\eqref{eq:varphin08/01} with terminal condition $\phi$ at $t_1$ and driven by $\bd{\nu}^2$. Then, we have
    \begin{equation}
\label{eq:duality:relation:statement}
        \begin{split}
     \int_{\R^{d_1} \times \R^{d_2}} \phi(x,y) d\gamma_{t_1}^1(x,y) &= \int_{\R^{d_1} \times \R^{d_2}} \varphi_{t_0}^2(x,y) d\gamma_0(x,y)  
     \\
     &\hspace{15pt}+ \int_{t_0}^{t_1} \int_{\R^{d_1} \times \R^{d_2}} b(x, \nu_t^1 -\nu_t^2) \cdot \nabla_x \varphi_t^2(x,y) d\gamma^1_t(x,y)dt.
\end{split}
\end{equation}
If we let $\bd{\gamma}^2$ be the solution to the continuity equation starting from $(t_0,\gamma_0)$ driven by $\bd{\nu}^2$ we also have $ \int_{\R^{d_1} \times \R^{d_2}} \varphi_{t_0}^2 (x,y) d\gamma_0(x,y) = \int_{\R^{d_1} \times \R^{d_2}} \phi(x,y) d\gamma_{t_1}^2(x,y) $ and, if ${\varphi}^1$ is the solution to the backward transport equation with terminal condition $\phi$ and control $\bd{\nu}^1$, we have $\int_{\R^{d_1} \times \R^{d_2}} \phi(x,y)d\gamma_{t_1}^1 (x,y) = \int_{\R^{d_1} \times \R^{d_2}} \varphi^1_{t_0}(x,y) d\gamma_0(x,y)$.
\end{lem}

\begin{proof}
We approximate $\bd{\nu}^2$ by $(\bd{\nu}^{2,n})_{n \geq 1}$ as in \eqref{eq:defnnun23/12}, and we let $({\varphi}^{2,n})_{n \geq 1}$ be the corresponding solutions to \eqref{eq:varphin08/01} (with terminal condition $\phi$ at $t_1$). 
We know from Lemma  \ref{lem:backwardtransportsmoothcontrol23/01} that ${\varphi}^{2,n}, \partial_t {\varphi}^{2,n}$ and $\nabla_x {\varphi}^{2,n}$ are continuous with quadratic growth. In particular, ${\varphi}^{2,n}$ is an admissible test function for the continuity equation (see \eqref{eq:CE:def}, which easily extends to test functions with quadratic growth, thanks to \eqref{eq:lem:compactnessContinuityEquation24/06:statement:48Ap}), and we have
\begin{equation}
\begin{split}
&\int_{\R^{d_1}  \times \R^{d_2}} \phi(x,y)d \gamma^1_{t_1}(x,y) 
\\
&= \int_{\R^{d_1} \times \R^{d_2}} \varphi_{t_0}^{2,n}(x,y) d\gamma_0(x,y) 
\\
&\hspace{15pt} + \int_{t_0}^{t_1} \int_{\R^{d_1} \times \R^{d_2}} \bigl \{ \partial_t \varphi_t^{2,n}(x,y) + b(x, \nu_t^1) \cdot \nabla_x \varphi_t^{2,n}(x,y) \bigl \} d\gamma^1_t(x,y) dt 
\\  
&= \int_{\R^{d_1} \times \R^{d_2}} \varphi_{t_0}^{2,n}(x,y) d\gamma_0(x,y) +\int_{t_0}^{t_1} \int_{\R^{d_1} \times \R^{d_2}} b(x, \nu_t^1 - \nu_t^{2,n}) \cdot \nabla_x \varphi_t^{2,n}(x,y) d\gamma^1_t(x,y)dt,
\end{split}
\label{eq:duality:relation:proof}
\end{equation}
where we used the equation satisfied by $\bd{\gamma}^1$ first and ${\varphi}^{2,n}$ second. Now, for all $t \in [t_0,T]$, $(x,y) \mapsto \varphi^{2,n}_t(x,y)$ and $(x,y) \mapsto b(x,\nu^1_t- \nu_t^{2,n})\cdot \nabla_x \varphi_t^{2,n}(x,y)$ converge pointwise as $n \rightarrow +\infty$ to $(x,y) \mapsto \varphi_t(x,y)$ and $(x,y) \mapsto b(x,\nu_t^1 - \nu_t^2) \cdot \nabla_x \varphi_t(x,y)$ respectively, see Corollary \ref{cor:approximation:nu}. Moreover, by Proposition \ref{prop:SolutionBackxard16SeptAppendix} we have the uniform bounds
$$ \sup_{n \in \N^*} \norm{ \varphi_{t_0}^{2,n}}_{\mathcal{C}^0_2} + \sup_{ n \in \mathbb{N}} \sup_{t \in [t_0,t_1]} \norm{ b(\cdot, \nu_t^1 - \nu_t^{2,n}) \cdot \nabla_x \varphi_t^{2,n}}_{\mathcal{C}^0_2} <+\infty $$
where we also used the fact that the drifts $(b(\cdot,{\boldsymbol \nu}^n))_{n \geq 1}$ 
are uniformly bounded. Since $\bd{\gamma}^1$ is bounded in $\mathcal{P}_2(\R^{d_1} \times \R^{d_2})$ by Lemma \ref{lem:A.4} we easily obtain \eqref{eq:duality:relation:statement} by taking the limit $n \rightarrow +\infty$ in \eqref{eq:duality:relation:proof} and using Lebesgue dominated convergence theorem. This proves the first part of the statement.

For the two claims in the second part of the statement, it suffices to notice from the first part that, in the special case $\bd{\nu}^1 = \bd{\nu}^2$,
$$ \int_{\R^{d_1} \times \R^{d_2}} \phi(x,y) d\gamma_{t_1}^{2}(x,y) = \int_{\R^{d_1} \times \R^{d_2}} \varphi_{t_0}^2 (x,y) d\gamma_0(x,y), $$
and similarly with the index $2$ being replaced by $1$.
\end{proof}

\subsubsection{FBODE Representation}

\begin{lem}
\label{lem:ApproxBODE21Sept}
Let $(t_0,\gamma_0) \in [0,T] \times {\mathcal P}_2({\mathbb R}^{d_1} 
\times {\mathbb R}^{d_2})$
and ${\boldsymbol \nu} \in \mathcal{A}(t_0)$ with  ${\boldsymbol u}=(u_t)_{t \in [t_0,T]}$ as corresponding solution to the backward equation
\eqref{eq:adjoint:equation}. If $(\Omega,\mathcal{F},\mathbb{P})$ is a probability space supporting a random variable $(X_0,Y_0)$ with probability distribution $\gamma_0$ and $(X_t)_{t \in [t_0,T]}$ is the solution to
the ODE
    $$ \dot{X}_t = b(X_t,\nu_t), \quad t \in [t_0,T]; \quad X_{t_0}=X_0,$$
    then $(Z_t := \nabla_x u_t(X_t,Y_0))_{t \in [t_0,T]}$ solves the backward ODE
    $$ \dot{Z}_t = - \nabla_x b(X_t,\nu_t)Z_t, \quad t \in [t_0,T]; \quad Z_T = \nabla_x L(X_T,Y_0).$$
\end{lem}

\begin{proof}
    Once again, we use 
    the regularization procedure introduced in Lemma 
    \ref{lem:newregularizationb23/12} and consider
    the same sequence $({\boldsymbol \nu}^n)_{n \geq 1}$
    as therein. 
    Following Corollary
    \ref{cor:approximation:nu}, we call, for each $n \in {\mathbb N}^*$, $(X_t^n)_{t \in [t_0,T]}$ the solution to the ODE 
\eqref{eq:ODE:appendix}      
    starting from $X_0$ and driven by the control ${\boldsymbol \nu}^n$. 
    We also denote by
    $(u_t^n)_{t \in [t_0,T]}$ the solution to the backward transport equation 
  \eqref{eq:adjoint:equation} driven by ${\boldsymbol \nu}^n$. We then let $Z_t^n := \nabla_x u_t^n(X_t^n,Y_0)$ for each $t \in [t_0,T]$. Then, 
  observing that $(t,x,y) \mapsto \nabla_x u^n_t(x,y)$
  is time differentiable (see Lemma \ref{lem:backwardtransportsmoothcontrol23/01}), we can differentiate 
  $t \in [t_0,T] \mapsto Z_t^n$. We find
    $$ \dot{Z}_t^n = \partial_t \nabla_x u_t^n(X_t^n,Y_0) + \nabla^2_{xx} u_t^n(X_t^n,Y_0) \dot{X}_t^n =\partial_t \nabla_x u_t^n(X_t^n,Y_0) + \nabla_{xx}^2u_t^n(X_t^n,Y_0) b(X_t^n, \nu_t^n).$$
    Moreover, 
differentiating in $x$ the equation for $u_t^n$ we find that
$$ -\partial_t \nabla_x u_t^n(x,y) - \nabla_x b(x,\nu_t^n) \nabla_xu_t^n(x,y) - \nabla^2_{xx} u_t^n(x,y) b(x,\nu_t^n) = 0, \quad (t,x,y) \in  
[t_0,T]\times\R^{d_1} \times \R^{d_2}.$$
Therefore, 
$$\dot{Z}_t^n = -\nabla_xb(X_t^n,Y_0)Z_t^n, \quad t\in [t_0,T]$$
which can be rewritten into
$$ Z_t^n = \nabla_xL(X_T^n,Y_0) + \int_t^T \nabla_x b(X_s^n,\nu_s^n) Z_s^n ds, \quad  t \in [t_0,T].$$

It remains to let $n$ tend to $+\infty$. To do so, we observe that, for all $n \in \mathbb{N}^*$ and $t \in [t_0,T]$, 
$ Z_t^n -Z_t = \nabla_x u_t^n(X_t^n,Y_0) - \nabla_x u_t(X_t,Y_0) $.
By Corollary \ref{cor:approximation:nu} (and the representation formula provided by 
Proposition \ref{prop:SolutionBackxard16SeptAppendix}), 
$(x,y) \mapsto \nabla_x u_t^n(x,y)$ converges locally uniformly to $(x,y) \mapsto \nabla_xu_t(x,y)$ and $X_t^n$ converges to $X_t$ as $n \rightarrow +\infty$.
Therefore, 
$Z_t^n$ converges to $Z_t$ as $n \rightarrow +\infty$. 
By 
Lemma 
    \ref{lem:newregularizationb23/12}
    and (i) in \assreg, 
we can pass to the limit in the backward ODE for $(Z_t^n)_{t \in [t_0,T]}$ and deduce that, for all $t \in [t_0,T]$, 
$$ Z_t =  \nabla_xL(X_T,Y_0) + \int_t^T \nabla_x b(X_s,\nu_s) Z_s ds,$$
which completes the proof. 
\end{proof}

\subsection{A Compactness Argument}

Recalling Definition 
\ref{def:R(t_0)} for the space $\mathcal{R}(t_0)$, we have the following lemma:

\begin{lem}
\label{lem:StrongCompactnessCurves}
Assume that for a fixed $t_0 \in [0,T]$, we are given a bounded sequence $(\bd{\rho}^n)_{n \geq 1}$ in $\mathcal{R}(t_0)$. Then, there exists $\bd{\rho} \in \mathcal{R}(t_0)$ such that, up to a subsequence, 
\begin{equation} 
\lim_{n \rightarrow +\infty} \sup_{t \in [t_0,T]} \norm{\rho_t^n - \rho_t}_{(\mathcal{C}^2_{2})^*} =0.
\label{eq:lem:A.8:3}
\end{equation}
Moreover, $\bd{\rho}$ satisfies
\begin{equation} 
\norm{\bd{\rho}}_{\mathcal{R}(t_0)} \leq \liminf_{n \rightarrow +\infty} \norm{\bd{\rho}^n}_{\mathcal{R}(t_0)}. 
\label{eq:lem:A.8:12}
\end{equation}
\end{lem}

We need a preliminary result. To state it properly, we introduce, for any $R>0$, some cut-off function $\chi_R : \R^d \times \R^d \rightarrow [0,1]$ such that
\begin{equation}
\label{eq:chi_R:appendix}
\chi_R =
\left \{
\begin{array}{ll}
1  \quad \mbox{\rm in} & B(0,R) \\
0  \quad \mbox{\rm outside} & B(0,R+1)
\end{array}
\right.,
\end{equation}
with $\norm{ \chi_R}_{\mathcal{C}^2} \leq c$ for some constant $c>0$ independent of $R$.
Then, we claim:

 \begin{lem}
     Consider the functions $(\chi_R)_{ R >0}$ as above. There is $C>0$ such that, for all $\rho \in (\mathcal{C}^2_2)^*$ and all $R  >0$
 $$ \sup_{ \varphi \in \mathcal{C}^2_2, \norm{ \varphi}_{\mathcal{C}^1_2} \leq 1} \langle \varphi; (1-\chi_R) \rho \rangle \leq \frac{C}{R+1} \sup_{\varphi \in \mathcal{C}^2_2, \norm{\varphi}_{\mathcal{C}^1_3} \leq 1 } \langle \varphi; \rho \rangle. $$    
\label{lem:whywetakeC13star}
 \end{lem}
 \begin{proof}
     A simple calculation shows that, for any $\varphi \in \mathcal{C}^2_2$,
     $$ \bigl\| (1 - \chi_R) \varphi \bigr\|_{\mathcal{C}^1_3} \leq \frac{C}{1+R} \norm{ \varphi}_{\mathcal{C}^1_2},$$
    for some $C>0$ independent from $\varphi$ and $R >0$. The result easily follows.
 \end{proof}

 We go on with the proof of Lemma \ref{lem:StrongCompactnessCurves}. Be aware that throughout the proof, 
 $\mathcal{C}^2_0$ is the space of twice continuously differentiable functions $\varphi$ that vanish at infinity together with their first and second derivatives, i.e., 
 $\lim_{R \rightarrow \infty} \sup_{\vert x \vert \geq R} \bigl( \vert \varphi(x) \vert 
 + \vert \nabla_x \varphi(x) \vert 
 + \vert \nabla^2_x \varphi(x) \vert) =0$. The space $\mathcal{C}^2_0$ is equipped with $\| \cdot \|_{{\mathcal C}^2_b}$.

\begin{proof}
The proof is divided in four steps.

\vskip 4pt

\textit{Step 1.} We first claim that there exists 
${\boldsymbol \rho} \in \mathcal{C}([t_0,T], (\mathcal{C}^{2}_0)^*, w-*)$ such that, up to a subsequence, 
$(\bd{\rho}^n)_{ n \in \N}$ converges to $\bd\rho$ in 
the following sense: for all test function $\varphi \in {\mathcal C}_0^2$, 
\begin{equation}
    \label{eq:main:assumption:lem:compactness of curves:2}
\lim_{n \rightarrow \infty}
\sup_{t \in [t_0,T]}
\Bigl\vert 
\langle \rho_t^n ;\varphi \rangle -
\langle \rho_t ; \varphi \rangle
\Bigr\vert =0, 
\end{equation}
with $\langle \cdot ; \cdot \rangle$
standing for the duality bracket between ${\mathcal C}^2_0$ and $({\mathcal C}^2_0)^*$. Indeed, by Banach-Alaoglu theorem,
the ball ${\mathcal B}^*$ of center $0$ and radius $\sup_{n \in \mathbb{N}^*} \norm{ \bd{\rho}^n}_{\mathcal{R}(t_0)}$ of $(\mathcal{C}^2_0)^*$ is 
compact for the weak-$*$ topology. Moreover, 
${\mathcal B}^*$ equipped with 
the weak-$*$ topology 
is 
metrizable, with the metric 
$d$ defined by 
$$d(\rho, \rho') := \sum_{ k \in \mathbb{N}} 2^{-k} |\langle \varphi_k ; \rho' - \rho \rangle|, \qquad \rho,\rho' \in {\mathcal B}^*,$$
where $(\varphi_k)_{ k \in {\mathbb N}}$ is a dense family of the unit ball
${\mathcal B}$ of the separable Banach space $\mathcal{C}^2_0$. 

Since we obviously have $d(\rho, \rho') \leq 2 \norm{ \rho' - \rho}_{(\mathcal{C}^2_0)^*} 
\leq 2 \norm{ \rho' - \rho}_{(\mathcal{C}^2_2)^*}$, we deduce from
Definition 
\ref{def:R(t_0)}
and the boundedness of $(\bd{\rho}^n)_{n \geq 1}$ in $\mathcal{R}(t_0)$ 
that the family $({\boldsymbol \rho}^n)_{n \geq 1}$  is uniformly continuous in time with respect to $d$, uniformly in 
$n \in {\mathbb N}^*$. We can then use Arzelà-Ascoli theorem (which requires the target space to be a metric space) to infer the existence of ${\boldsymbol \rho} \in \mathcal{C}([t_0,T], {\mathcal B}^*)$ 
(with ${\mathcal B}^*$ being equipped with $d$)
such that
$$\lim_{n \rightarrow \infty} \sup_{t \in [t_0,T]} d(\rho_t^n , \rho_t)
= 0,$$
which implies 
\eqref{eq:main:assumption:lem:compactness of curves:2}. 
\vskip 4pt

\textit{Step 2.} 
We now prove that $\boldsymbol{\rho}$ satisfies
\begin{equation} 
\label{eq:lem:A.8:11}
\sup_{t \in [t_0,T]} \sup_{ \varphi \in \mathcal{C}^2_0, \norm{\varphi}_{\mathcal{C}^1_{3}} \leq 1} \bigl |\langle \varphi; \rho_t \bigr \rangle |  + \sup_{t_2 > t_1  \in [t_0,T]} \sup_{\varphi \in \mathcal{C}^2_0, \norm{\varphi}_{\mathcal{C}^2_{2}} \leq 1} \frac{ | \bigl \langle \varphi; \rho_{t_2} - \rho_{t_1} \bigr \rangle |}{\sqrt{t_2 -t_1}} 
\leq \liminf_{n \rightarrow +\infty} \norm{\bd{\rho}^n}_{\mathcal{R}(t_0)}.
\end{equation}
Indeed, for all $t \in [t_0,T]$, $t_2 >t_1 \in [t_0,T]$, and $\varphi^1,\varphi^2 \in \mathcal{C}^2_0$ with $\norm{ \varphi^1}_{\mathcal{C}^1_{3}} \leq 1$ and $\norm{\varphi^2}_{\mathcal{C}^2_2} \leq 1$, we get from \eqref{eq:main:assumption:lem:compactness of curves:2} and the definition of $\norm{ \cdot }_{\mathcal{R}(t_0)}$ that
\begin{align*}
    \bigl| \bigl \langle \varphi^1; \rho_t \bigr\rangle \bigr| + \frac{|\bigr \langle \varphi^2; \rho_{t_2} - \rho_{t_1} \bigr \rangle | }{ \sqrt{t_2 -t_1}} &= \lim_{n \rightarrow +\infty} \left \{ | \bigl \langle \varphi^1; \rho^n_t \bigr\rangle | + \frac{|\bigr \langle \varphi^2; \rho^n_{t_2} - \rho^n_{t_1} \bigr \rangle | }{ \sqrt{t_2 -t_1}} \right \} \\
    &\leq \liminf_{ n \rightarrow +\infty} \norm{ \bd{\rho}^n}_{\mathcal{R}(t_0)}
\end{align*}
and then we take the supremum over $\varphi^1, \varphi^2$, $t,t_1,t_2$.

\vskip 4pt

\textit{Step 3.}
We then justify that for all $R >0$,
\begin{equation} 
\lim_{n \rightarrow +\infty} \sup_{t \in [t_0,T]} \bigl\| \chi_R (\rho_t^n - \rho_t) \bigr\|_{(\mathcal{C}^2_{2})^*} =0,
\label{eq:resultofstep3}
\end{equation}
with 
$\chi_R$ as in \eqref{eq:chi_R:appendix}.

To this end we fix $R>0$ and take $\epsilon >0$. 
Given the compact embedding $\mathcal{C}^2(\bar{B}(0,R+1)) \hookrightarrow\mathcal{C}^1(\bar{B}(0,R+1))$
(with $\bar{B}(0,R+1)$ denoting the $d_1+d_2$-dimensional ball of center $0$ and radius 
$R+1$), 
we can find 
$\varphi^1,\dots,\varphi^N \in \bar{B}_{\mathcal{C}^2(\bar{B}(0,R+1))}(0,1)$ such that
\begin{equation}
 \label{eq:lem:A.8:5} 
 \bar{\mathcal B}_{\mathcal{C}^2(\bar{B}(0,R+1))}(0,1) \subset \bigcup_{i=1}^N  
 \bar{\mathcal B}_{\mathcal{C}^1(\bar{B}(0,R+1))}(\varphi^{i},\epsilon),
 \end{equation}
where 
$\bar{\mathcal B}_{\mathcal{C}^2(\bar{B}(0,R+1))}(0,1)$
is the unit ball of 
$\mathcal{C}^2(\bar{B}(0,R+1))$
and, for all integer $i=1,\dots,N$, 
$\bar{\mathcal B}_{\mathcal{C}^1(\bar{B}(0,R+1))}(\varphi^{i},\epsilon)$
is the ball of 
$\mathcal{C}^1(\bar{B}(0,R+1))$ of center $\varphi^i$ and radius $\epsilon.$

Let $\varphi \in \mathcal{C}^2_2$ with $\norm{\varphi}_{\mathcal{C}^2_2} \leq 1$. We let $c_R := (1+(R+1)^2)^{-1}$ so that $c_R \varphi$ belongs to $\bar{\mathcal{B}}_{\mathcal{C}^2(\bar{B}(0,R+1))}(0,1)$. For all 
$1 \leq i \leq N$, we have
\begin{align}
\notag    \sup_{t \in [t_0,T]} \biggl| \Bigl \langle c_R \varphi; \chi_R (\rho_t^n - \rho_t) \Bigr \rangle \biggr|  & \leq \sup_{t \in [t_0,T]} \biggl| \Bigl \langle \bigl( c_R\varphi -\varphi^{i} \bigr) \chi_R ;  \rho_t^n - \rho_t \Bigr \rangle \biggr| + \sup_{t \in [t_0,T]} \biggl| \Bigl \langle \varphi^{i}  \chi_R ;\rho_t^n -\rho_{t} \Bigr \rangle \biggr| \\
    &\leq 2 \sup_{ n \in \mathbb{N}^* } \norm{ \bd \rho^n}_{\mathcal{R}(t_0)} \norm{(c_R\varphi - \varphi^{i}) \chi_R }_{\mathcal{C}^1_{3}} 
    + \sup_{t \in [t_0,T]} \biggl| \Bigl \langle \varphi^{i} \chi_R ;\rho_{t}^n - \rho_{t} \Bigr \rangle \biggr| \label{eq:lem:A:8:10}
\end{align}
with the first  term on the second line following from 
\eqref{eq:lem:A.8:11}.

Now, by 
\eqref{eq:lem:A.8:5} (and with $c$ as in \eqref{eq:chi_R:appendix}), we can always find an integer $1 \leq i \leq N$ such that 
$$ \norm{(c_R\varphi - \varphi^{i}) \chi_R }_{\mathcal{C}^1_{3}} \leq \frac{2c}{1+R^3} \norm{ (c_R\varphi - \varphi^{i})  }_{\mathcal{C}^1(\bar{B}(0,R+1))}  \leq C \epsilon,$$
for a constant $C$ independent of $\epsilon$. Moreover, by 
\eqref{eq:main:assumption:lem:compactness of curves:2}, we can find $n_0 \in \mathbb{N}^*$ such that, for all $n \geq n_0$, 
$$ \max_{ 1 \leq i \leq N} \sup_{t \in [t_0,T]}\biggl| \Bigl \langle \varphi^{i}   \chi_R; (\rho_{t}^n - \rho_{t}) \Bigr \rangle \biggr| \leq \epsilon.  $$
Inserting the latter two displays in 
\eqref{eq:lem:A:8:10}, taking the supremum over all $\varphi \in \mathcal{C}^2_{2}$ with $\norm{\varphi}_{\mathcal{C}^2_{2}} \leq 1$, dividing by $c_R$ and letting $\epsilon \rightarrow 0^+$, this shows that \eqref{eq:resultofstep3} holds.

\vskip 4pt

\textit{Step 4}.
Finally we use Lemma \ref{lem:whywetakeC13star}
to conclude.
Precisely, we show that the sequence $({\boldsymbol \rho}^n)_{n \in \mathbb{N}^*}$ is Cauchy in $\mathcal{C}([t_0,T], 
(\mathcal{C}^2_{2})^*)$.
To do so, we take $\epsilon >0$. By Lemma \ref{lem:whywetakeC13star},
we can find $R >0$ such that, for all $m,n \in \mathbb{N}^*$,
\begin{align*} 
&\sup_{t \in [t_0,T]} \bigl\| (1-\chi_R)(\rho_t^n - \rho_t^m) \bigr\|_{(\mathcal{C}_{2}^2)^*} \\
&\leq \sup_{t \in [t_0,T]} \sup_{\varphi \in \mathcal{C}^2_2, \norm{ \varphi}_{\mathcal{C}^1_2} \leq 1} \langle \varphi; (1-\chi_R) \rho_t^n \rangle + \sup_{t \in [t_0,T]} \sup_{\varphi \in \mathcal{C}^2_2, \norm{ \varphi}_{\mathcal{C}^1_2} \leq 1} \langle \varphi; (1-\chi_R) \rho_t^n \rangle \leq \frac{\epsilon}{2}.
\end{align*}
Since the sequence $(\chi_R \rho^n)_{ n \in \mathbb{N}}$ converges in $\mathcal{C}([t_0,T], (\mathcal{C}^2_{2})^*)$ (this is the result of \textit{Step 3}), we can find $n_0 \in \mathbb{N}^*$ such that for all $n,m \geq n_0$,
we have
$$ \sup_{t \in [t_0,T]} \bigl\| \chi_R(\rho_t^n - \rho_t^m)\bigr\|_{(\mathcal{C}^2_{2})^*} \leq \frac{\epsilon}{2}.$$
As consequence of the last two displays, we have for all $m,n \geq n_0$:
\begin{align*} 
\sup_{t \in [t_0,T]} \bigl\| \rho_t^n - \rho_t^m
\bigr\|_{(\mathcal{C}^2_{2})^*} &\leq \sup_{t \in [t_0,T]} \bigl\| \chi_R(\rho_t^n - \rho_t^m) \bigr\|_{(\mathcal{C}^2_{2})^*}  +\sup_{t \in [t_0,T]} \bigl\| 
(1-\chi_R)( \rho_t^n - \rho_t^m) \bigr\|_{(\mathcal{C}^2_{2})^*} \leq \epsilon.
\end{align*}
By completeness, the sequence converges toward some $\tilde{\boldsymbol \rho} \in \mathcal{C}([t_0,T], (\mathcal{C}^2_{2})^*)$. Of course, for all $t \in [t_0,T]$, 
the action of 
$\tilde{\rho}_t$ on the elements of 
$\mathcal{C}^2_0$ coincides
with the action of $\rho_t$. In a sense, 
$\tilde{\rho}_t$ extends $\rho_t$ to 
$({\mathcal C}^2_{2})^*$. The bound 
\eqref{eq:lem:A.8:12} follows by the same argument as in \textit{Step 2} of this proof.
\end{proof}

\subsection{Pinsker Inequalities}

\label{sec:PinskerAppendix}

The following result is taken from \cite[Theorem 2.1]{BolleyVillani}. 
We feel better to restate it here as it is used repeatedly throughout the text. 
\begin{thm}
\label{thm:pinsker}
Let $\mu$ and $\nu$ be two probability measures on a measurable space 
${\mathcal X}$ and $\varphi$ be a non-negative valued measurable function 
defined on ${\mathcal X}$. Then, 
\begin{align}
    \label{eq:pinsker:i}
&\bigl\| \varphi \bigl( \mu - \nu \bigr) 
\bigr\|_{\rm TV}
\leq \biggl( \frac32 + \log \int_{\mathcal X}
e^{2 \varphi(x)} d \nu (x) 
\biggr) \biggl( 
\sqrt{ {\mathcal E}(\mu \vert \nu)}
+ \frac12 {\mathcal E}(\mu \vert \nu) \biggr),
\\
\label{eq:pinsker:ii}
&\bigl\| \varphi \bigl( \mu - \nu \bigr) 
\bigr\|_{\rm TV}
\leq \sqrt{2} \biggl( 1 + \log \int_{\mathcal X}
e^{2 \varphi^2(x)} d \nu (x) 
\biggr)^{1/2} 
\sqrt{ {\mathcal E}(\mu \vert \nu)}.
\end{align}
Above, $\varphi(\mu-\nu)$ is a shorthand notation for the signed measure 
$\varphi\mu - \varphi \nu$, $\| \varphi \bigl( \mu - \nu \bigr) 
\|_{\rm TV}$ is its total variation norm and 
${\mathcal E}(\mu \vert \nu)$ denotes the relative entropy 
of $\mu$ with respect to $\nu$. 
\end{thm}

\bibliographystyle{plain} 
\bibliography{Biblio}

\end{document}